\documentclass{amsart}
\usepackage{enumerate}
\usepackage{tikz}
\usetikzlibrary{automata,positioning}
\usepackage{bm}
\usepackage{hyperref}
\usepackage{mathrsfs}
\usepackage{amsfonts}
\usepackage{latexsym}
\usepackage{epsfig}
\usepackage{indentfirst}
\graphicspath{{figures/}}
\usepackage{amsmath}
\usepackage{amsthm}
\usepackage{graphicx}
\usepackage{amssymb}
\usepackage{amscd}
\usepackage{latexsym}
\usepackage{subfigure}
\usepackage{epstopdf}
\usepackage{color}
\usepackage{mathtools}
\usepackage{thm-restate}
\input xy
\xyoption{all}

\newcommand{\Z}{\mathbb{Z}}
\newcommand{\R}{\mathbb{R}}

\newtheorem{thm}{Theorem}[section]
\newtheorem{lem}[thm]{Lemma}

\newtheorem{prop}[thm]{Proposition}

\theoremstyle{definition}
\newtheorem{defn}[thm]{Definition}
\newtheorem{exam}[thm]{Example}

\newtheorem{que}[thm]{Question}

\theoremstyle{remark}
\newtheorem{rmk}[thm]{Remark}

\numberwithin{equation}{section}

\numberwithin{equation}{section} \makeatletter

\begin{document}

\title[Satellite knots and immersed Heegaard Floer homology]{Satellite knots and immersed Heegaard Floer homology}

\author[Wenzhao Chen]{Wenzhao Chen}
\address {Department of Mathematics, University of British Columbia.\newline \it{E-mail address:} \tt{chenwzhao@math.ubc.ca}}

\author[Jonathan Hanselman]{Jonathan Hanselman}
\address {Department of Mathematics, Princeton University.\newline \it{E-mail address:} \tt{jh66@princeton.edu}}

\begin{abstract}
We describe a new method for computing the $UV = 0$ knot Floer complex of a satellite knot given the $UV = 0$ knot Floer complex 
for the companion and a doubly pointed bordered Heegaard diagram for the pattern, showing that the complex for the satellite can be computed from an immersed doubly pointed Heegaard diagram obtained from the Heegaard diagram for the pattern by overlaying the immersed curve representing the complex for the companion. This method streamlines the usual bordered Floer method of tensoring with a bimodule associated to the pattern by giving an immersed curve interpretation of that pairing, and computing the module from the immersed diagram is often easier than computing the relevant bordered bimodule. In particular, for (1,1) patterns the resulting immersed diagram is genus one, and thus the computation is combinatorial. For (1,1) patterns this generalizes previous work of the first author which showed that such immersed Heegaard diagram computes the $V=0$ knot Floer complex of the satellite. As a key technical step, which is of independent interest, we extend the construction of a bigraded complex from a doubly pointed Heegaard diagram and of an extended type D structure from a torus-boundary bordered Heegaard diagram to allow Heegaard diagrams containing an immersed alpha curve.
\end{abstract}

\maketitle
\tableofcontents
\section{Introduction}
The study of knot Floer chain complexes of satellite knots has many applications. For instance, computation of knot-Floer concordance invariants of satellite knots is instrumental in establishing a host of results in knot concordance, like  \cite{Cochran2013,Hom2015,MR3589337,Hedden2016, Ozsvath2017, Hom2021, Dai2021, Hom2022}. To further understand the behavior of knot Floer chain complexes under satellite operations,
the current paper introduces an immersed-curve technique to compute knot Floer chain complexes of satellite knots. This method subsumes most of the previous results in this direction, including  \cite{Hedden2005,MR2372849,MR2511910,VanCott2010, MR3134023,MR3217622, HLV2014, MR3589337,Hanselman2023,Chen2019}. This technique is derived from an immersed Heegaard Floer theory that is developed in this paper, which is built on the work by the second author, Rasmussen, and Watson  \cite{hanselman2016bordered}.
  
\subsection{Satellite knots and immersed curves} 
Knot Floer homology was introduced by Ozsv\'ath and Szab\'o and independently by Rasmussen \cite{MR2065507,MR2704683}. Recall that any knot can be encoded by a doubly-pointed Heegaard diagram, which is a closed oriented surface with two sets of embedded circles and two base points. Knot Floer theory, using the machinery of  Lagrangian Floer theory, associates a bigraded chain complex over $\mathbb{F}[U,V]$ to such a double pointed Heegaard diagram, and the bigraded chain homotopy type of this chain complex is an invariant of the isotopy type of the knot. The literature studies various versions of the knot Floer chain complex obtained by setting the ground ring to be a suitable quotient ring of $\mathbb{F}[U,V]$; throughout this paper we will consider the complex defined over the ground ring $\mathcal{R}=\mathbb{F}[U,V]/UV$. The knot Floer chain complex over $\mathcal{R}$ of a knot $K$ in the 3-sphere is equivalent to the bordered Floer invariant of the knot complement $S^3\backslash \nu(K)$, and it was shown in \cite{hanselman2016bordered} that this is equivalent to an immersed multicurve in the punctured torus decorated with local systems. The punctured torus we refer to here is a torus with a single puncture and a parametrization allowing us to identify it with the boundary of the knot complement with a chosen basepoint.

A satellite knot is obtained by gluing a solid torus that contains a knot $P$ (called the \textit{pattern knot}) to complement of a knot $K$ in the 3-sphere (called the \textit{companion knot}) in a compatible way, after which the glued-up manifold is a 3-sphere and the pattern knot $P$ gives rise to the satellite knot $P(K)$ in the 3-sphere. Just as knots in closed 3-manifolds are encoded by doubly-pointed Heegaard diagrams, a pattern knot in the solid torus can be represented by a doubly-pointed bordered Heegaard diagram, which is an oriented surface of some genus $g$ with one boundary component, together with two base points and a suitable collection of $g$ $\beta$-curves, $g-1$ $\alpha$-curves, and two $\alpha$ arcs.

Our technique involves constucting an immersed doubly pointed Heegaard diagram by combining a doubly pointed bordered Heegaard diagram for the pattern $P$ with the immersed curve associated with the companion $K$. More precisely, we fill in the boundary of the bordered Heegaard diagram for $P$ and remove the two $\alpha$ arcs and then add the immersed curve for $K$ to the diagram by identifying the punctured torus containing $K$ with a neighborhood of the now filled in boundary and the $\alpha$ arcs in a way dictated by the given parametrizations. The resulting diagram is just like a standard genus $g$ doubly pointed Heegaard diagram except that one of the $\alpha$ curves, which are usually embedded, is now replaced with a decorated immersed multicurve. See the top row of Figure \ref{Figure, cabling example} and \ref{Figure, proof strategy} for examples of immersed doubly pointed diagrams constructed in this way.

Our main theorem asserts that this diagram can be used to compute the knot Floer complex over $\mathcal{R}$ of $P(K)$. We state the main theorem below, with technical inputs referenced in the remark afterwards.

\begin{restatable}{thm}{MainTheorem}\label{Theorem, main theorem}
Let $\mathcal{H}_{w,z}$ be a doubly-pointed bordered Heegaard diagram for a pattern knot $P$, and let $\alpha_K$ be the immersed multicurve associated to a companion knot $K$. Let $\mathcal{H}_{w,z}(\alpha_{K})$ be the immersed doubly-pointed Heegaard diagram obtained by pairing $\mathcal{H}_{w,z}$ and $\alpha_K$, in which $\alpha_K$ is put in a z-passable position. Then the knot Floer chain complex $CFK_\mathcal{R}(\mathcal{H}_{w,z}(\alpha_{K}),\mathfrak{d})$ defined using $\mathcal{H}_{w,z}(\alpha_{K})$ and a generic choice of auxiliary data $\mathfrak{d}$ is bi-graded homotopy equivalent to the knot Floer chain complex of the satellite knot $P(K)$ over $\mathcal{R}$, where $\mathcal{R}=\mathbb{F}[U,V]/UV$. 
\end{restatable}
\begin{rmk}
The paring operation for constructing $\mathcal{H}_{w,z}(\alpha_K)$ is defined in Section \ref{Subsection, pairing construction}. The knot Floer chain complex of an immersed doubly-pointed Heegaard diagram is defined in Section \ref{Section, knot Floer chain complex of immersed Heegaard diagrams}. While the definition of the Heegaard Floer theory with immersed Heegaard diagrams is similar to that in the usual setup, it is complicated by the appearance of boundary degenerations.
The $z$-passable condition on $\alpha_K$ is a diagrammatic condition used to handle boundary degenerations; it is specified in Definition \ref{Definition, z passable} and can be arranged easily via finger moves as in Example \ref{Example, obtaning immersed curve admissibly equivalent to z-adj curve}. Moreover, the $z$-passable condition is not required when $\mathcal{H}_{w,z}$ is a genus-one diagram; see Theorem \ref{Theorem, pairing theorem restricted to genus-one patterns}. The proof of Theorem \ref{Theorem, main theorem} is separated into two stages: we first prove the ungraded version in Section \ref{Subsection, proof of the main theorem, ungraded version} and then the gradings are addressed in Section  \ref{Subsection, gradings in the main theorem}. The gradings can be combinatorially computed using an index formula established in Section \ref{Section, embedded index formula}; also see Definition \ref{Definition, w- and z- gradings}. 
\end{rmk}

\begin{figure}[htb!]
	\centering{
		\includegraphics[scale=0.5]{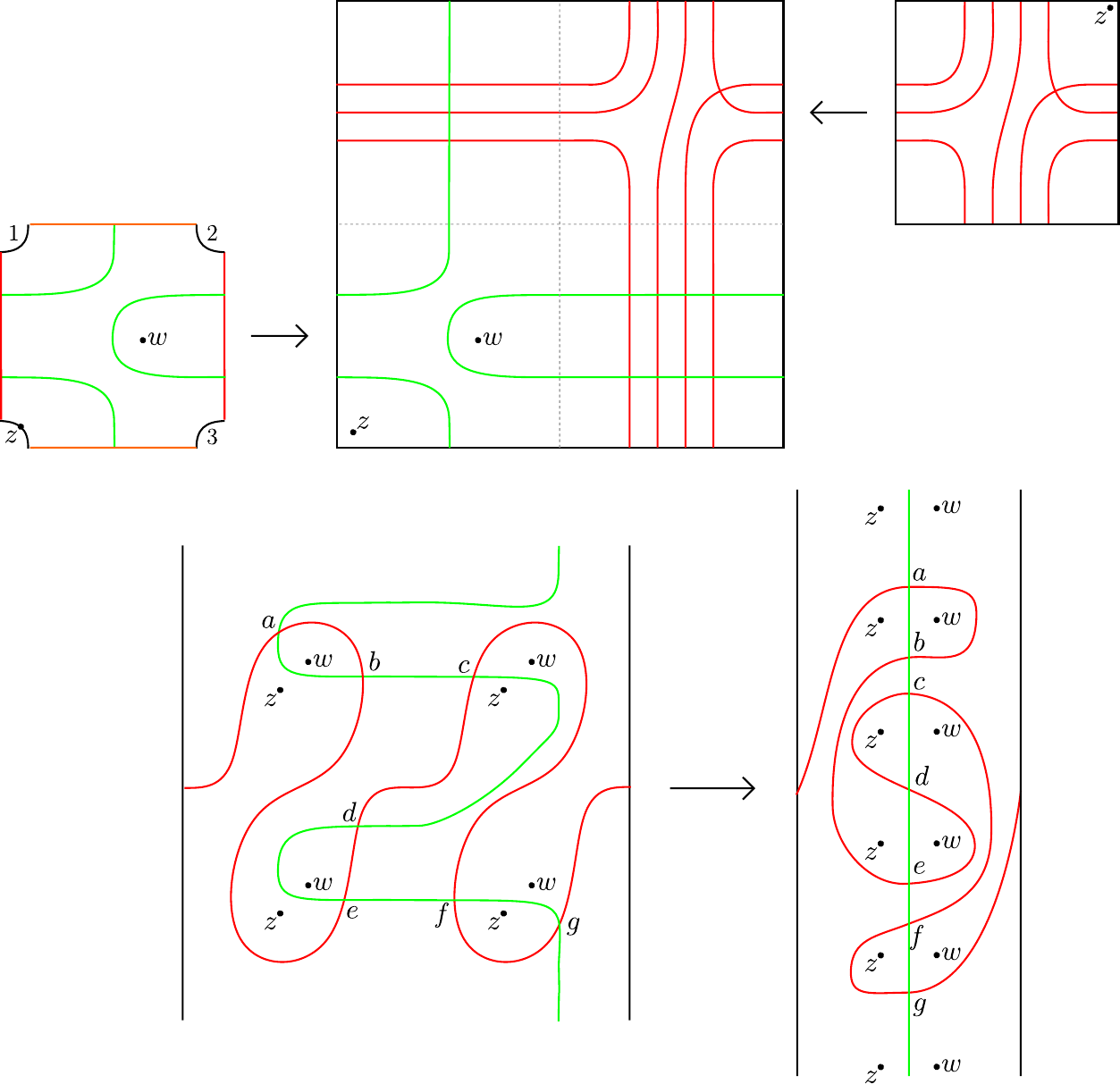}
		\caption{Obtaining the immersed curve for the $(2,1)$-cable of the trefoil $T_{2,3}$. The top row shows the pairing diagram $\mathcal{H}_{w,z}(\alpha_K)$ obtained by merging a doubly-pointed bordered Heegaard diagram $\mathcal{H}_{w,z}$ for the $(2,1)$-cable pattern and the immersed curve $\alpha_{K}$ for the trefoil $T_{2,3}$. The bottom row shows a lift of this diagram to a suitable covering space and a planar transform that sends the lift of the immersed curve for $T_{2,3}$ to that of $(T_{2,3})_{2,1}$.}
		\label{Figure, cabling example}
	}
\end{figure}
Theorem \ref{Theorem, main theorem} is especially useful when the pattern knot is a (1,1) pattern, meaning that it admits a genus one doubly pointed bordered Heegaard diagram $\mathcal{H}_{w,z}$. This is because in this setting the immersed doubly pointed Heegaard diagram $\mathcal{H}_{w,z}(\alpha_K)$ is genus one, and the complex $CFK_\mathcal{R}(\mathcal{H}_{w,z}(\alpha_{K}),\mathfrak{d})$ for such a diagram is straightforward to compute even in the presence of immersed curves; it only requires counting bigons which can be done combinatorially. An example is shown in Figure \ref{Figure, cabling example}, where we compute the knot Floer complex $CFK_\mathcal{R}$ of the $(2,1)$ cable of the trefoil $T_{2,3}$. The top row of the figure gives the pairing diagram, formed from a doubly pointed bordered Heegaard diagram $\mathcal{H}_{w,z}$ for the cable knot and the immersed curve $\alpha_K$ associated with $T_{2,3}$. The bottom left shows the curves lifted to an appropriate covering space, after a homotopy putting them in minimal position. There are seven generators, labeled as in the figure, and it is straightforward to count the bigons that cover only $w$ or only $z$ and see that the differential in $CFK_\mathcal{R}((T_{2,3})_{2,1})$ is given by
$$\partial a = 0, \quad \partial b = U a + V^2 e, \quad \partial c = V d, \quad \partial e = U d, \quad \partial d = 0, \partial f = U^2 c + V g, \quad \partial g = 0.$$
The Alexander grading changes by one decreases when traveling along the $\beta$ curve each time the $\beta$ curve crosses the short arc connecting the $z$ and $w$ basepoints, increasing if $z$ is on the left and decreasing if $z$ is on the right, so we have
$$A(a) = 2, \quad A(b) = A(c) = 1, \quad A(d)=0, \quad A(e)=A(f)=-1, \quad A(g) = -2.$$
Relative Maslov gradings can also be computed from the diagram, with the absolute grading fixed by the normalization $M(a) = 0$.

In the case of (1,1) patterns, the first author proved a weaker version of Theorem \ref{Theorem, main theorem} in \cite[Theorem 1.2]{Chen2019}, where the knot Floer chain complexes are only defined over $\mathbb{F}[U]\cong\mathbb{F}[U,V]/V$. Recall that the complex over $\mathbb{F}[U]$ does not count any disks covering the $z$ basepoint, while the complex over $\mathcal{R}$ allows disks to cover either basepoint as long as they do not cover both. \cite[Theorem 1.2]{Chen2019} can be used to recover the $\tau$-invariant formula for cable knots, Mazur satellites, and Whitehead doubles \cite{MR2372849, MR3217622,  MR3589337}.  Theorem \ref{Theorem, main theorem} generalizes this earlier result by showing that the same process recovers the complex over $\mathcal{R}$, which carries strictly more information. In particular, this version of the knot Floer complex allows one to compute the $\epsilon$-invariant introduced by Hom \cite{MR3217622} and infinitely many concordance homomorphisms $\phi_j$ ($j\in\mathbb{Z}^+$) defined by Dai-Hom-Stoffregen-Truong \cite{Dai2021}. For example, we use Theorem \ref{Theorem, main theorem} to recover and generalize the $\tau$ and $\epsilon$ formulas for cables from \cite[Theorem 2]{MR3217622} in Section \ref{sec:one-bridge-tau} and to recover the $\tau$ and $\epsilon$ formulas for Mazur patterns from \cite[Theorem 1.4]{MR3589337} in Section \ref{sec:mazur}.

The computation described above is even easier for a certain family of (1,1) patterns. In \cite[Theorem 1]{Hanselman2023}, the second author and Watson showed that the immersed multicurve of a cable knot can be obtained from that of the companion knot via a planar transform after lifting the immersed multicurves to an appropriate covering space of the marked torus. In fact, we show in Theorem \ref{Theorem, satellite knots with one-bridge braid patterns} that the same procedure works for a broader family of (1,1) patterns called 1-bridge braids (see Definition \ref{def:one-bridge-braid}). In addition to cables this class of patterns contains all Berge-Gabai knots. These patterns are specified by three integers and are denoted $B(p,q,b)$. We let $K_{p,q,b}$ denote the satellite of the companion knot $K$ with pattern $B(p,q,b)$. In Section \ref{sec:1-bridge-braids-curves} we define a diffeomorphism $f_{p,q,b}$ of $\R^2$ taking the integer lattice $\Z^2$ to itself and show that this transformation computes the immersed curve for $K_{p,q,b}$ from that of $K$.

 \begin{restatable}{thm}{PlanarTransformTheorem}\label{Theorem, satellite knots with one-bridge braid patterns}
 	Let $\gamma_K$ and $\gamma_{K_{p,q,b}}$ be the immersed multicurve associated with $K$ and $K_{p,q,b}$ respectively. Let $\tilde{\gamma}_K$ and $\tilde{\gamma}_{K_{p,q,b}}$ be the lifts of $\gamma_K$ and $\gamma_{K_{p,q,b}}$ to $\tilde{T}_{\bullet}$ respectively. Then $\tilde{\gamma}_{K_{p,q,b}}$ is homotopic to $f_{p,q,b}(\tilde{\gamma}_K)$.
 \end{restatable}
 
We demonstrate how this result is obtained from Theorem \ref{Theorem, main theorem}, in the example of $(2,1)$ cabling the trefoil, in the bottom row of Figure \ref{Figure, cabling example}.  Note that the $(2,1)$ cable pattern is the 1-bridge braid $B(2,1,0)$. On the left is the diagram $\mathcal{H}_{w,z}(\alpha_K)$, which by Theorem \ref{Theorem, main theorem} computes the complex $CFK_\mathcal{R}((T_{2,3})_{2,1})$, lifted to an appropriate covering space (specifically $(\R/p\Z)\times \R$, where $p$ is the winding number). There is a homotopy that pulls the the $\beta$ curve coming from $\mathcal{H}_{w,z}$ straight to a vertical line, sliding the basepoints and the $\alpha$ curve along the way, and rescales to obtain a different covering space of the marked torus (namely $(\R/\Z)\times \R$). This homotopy does not change the Floer complex associated with the diagram, so the new diagram still computes $CFK_\mathcal{R}((T_{2,3})_{2,1})$, and since the $\beta$ curve is vertical and passes through each pair of basepoints twice, the $\alpha$ curve in this diagram is precisely the immersed curve representing $CFK_\mathcal{R}((T_{2,3})_{2,1})$. The homotopy that pulls the $\beta$ curve to the vertical line is precisely the planar transformation $f_{2,1,0}$. We note that in the special case of cables (for which $b=0$), the transformation $f_{p,q,b}$ agrees with the planar transformation $f_{p,q}$ described in \cite[Theorem 1]{Hanselman2023}.

We will use Theorem \ref{Theorem, satellite knots with one-bridge braid patterns} to derive formulas for  $\epsilon$ and $\tau$ for 1-bridge braid satellites in Theorem \ref{thm:epsilon-for-one-bridge-braids} and Theorem \ref{thm:tau-for-one-bridge-braids}, generalizing similar formulas for cables. We also determine the precise criteria for a 1-bridge braid satellite to be an L-space knot in Theorem \ref{Theorem, L space knot criterion}; this unifies and generalizes similar results known for cables and Berge-Gabai knots \cite{Hom2011, HLV2014}.

\subsection{Immersed Heegaard Floer theory} The underlying machinery for proving Theorem \ref{Theorem, main theorem} is the bordered Heegaard Floer theory introduced by Lipshitz-Ozsv\'ath-Thurston \cite{LOT18,Lipshitz2015}. The new input is an immersed Heegaard Floer theory that we develop in this paper in which we allow Heegaard diagrams with an immersed multicurve in place of one $\alpha$ curve. We closely follow the construction of bordered invariants in \cite{LOT18}, highlighting the points at which more care is needed in this broader setting.

Bordered Heegaard Floer theory is a toolkit to compute Heegaard Floer invariants of manifolds that arise from gluing in terms of a set of relative invariants for manifolds with boundaries. In the simplest setting, assume $Y_1$ and $Y_2$ are two oriented 3-manifolds with parametrized boundary such that $\partial Y_1$ is identified with an oriented parametrized surface $\mathcal{F}$ and $\partial Y_2$ is identified with $-\mathcal{F}$ the orientation reversal of $\mathcal{F}$, and let $Y=Y_1\cup_\mathcal{F}Y_2$. Up to a  suitable notion of homotopy equivalence, the bordered Heegaard Floer theory associates to $Y_1$ a graded $A^\infty$-module $\widehat{CFA}(Y_1)$ (called type $A$ module) and associates to $Y_2$ a graded differential module $\widehat{CFD}(Y_2)$ (called type D module). Moreover, there is a box-tensor product operation $\widehat{CFA}(Y_1)\boxtimes \widehat{CFD}(Y_2)$ which produces a chain complex that is graded homotopy equivalent to the hat-version Heegaard Floer chain complex $\widehat{CF}(Y)$ of the glued-up manifold. The second author, Rasmussen, and Watson introduced an immersed-curve technique for working with these invariants for manifolds with torus boundary \cite{hanselman2016bordered}. When $\mathcal{F}$ mentioned above is a parametrized torus $T^2$, $\widehat{CFA}(Y_1)$ and $\widehat{CFD}(Y_2)$ are equivalent to immersed multicurves $\gamma_1$ and $\gamma_2$ (decorated with local systems) in the parametrized torus $T^2$ away from a marked point $z$. Moreover, the Lagrangian Floer chain complex $\widehat{CF}(T^2\backslash \{z\},\gamma_1,\gamma_2)$ is homotopy equivalent to $\widehat{CFA}(Y_1)\boxtimes \widehat{CFD}(Y_2)$, which is in turn homotopic equivalent to $\widehat{CF}(Y)$.

The bordered Heegaard Floer theory also contains a package to deal with the situation of gluing a 3-manifold $M_1$ with two parametrized boundary components $-\mathcal{F}_1$ and $\mathcal{F}_2$ to a 3-manifold $M_2$ with $\partial M_2=-\mathcal{F}_2$. It associates to $M_1$ a type $DA$ bimodule $\widehat{CFDA}(M_1)$ up to a suitable equivalence, and there is a box-tensor product $\widehat{CFDA}(M_1)\boxtimes \widehat{CFD}(M_2)$ resulting in a type D module that is homotopy equivalent to $\widehat{CFD}(M_1\cup_{\mathcal{F}_2} M_2)$. In this paper, we introduce an immersed-Heegaard-diagram
approach to recapture this bimodule pairing when the manifold boundaries are tori.

Recall that we can encode a manifold $M_1$ whose boundary are two parametrized tori by some arced bordered Heegaard diagram $\mathcal{H}_{M_1}$ (from which the type $DA$ bimodule is defined). Let $\alpha_{M_2}$ be the immersed multicurve for an oriented 3-manifold $M_2$ with a single torus boundary component. In Section \ref{Subsection, pairing construction}, we give a pairing construction that merges such an arced bordered Heegaard diagram $\mathcal{H}_{M_1}$ and an immersed multicurve $\alpha_{M_2}$ to obtain an immersed bordered Heegaard diagram $\mathcal{H}_{M_1}(\alpha_{M_2})$; see Figure \ref{Figure, proof strategy} for a schematic example of pairing an arced bordered diagram and an immersed curve. Extending the original way of defining type D modules from (non-immersed) bordered Heegaard diagrams, we define type D modules for a class of immersed bordered Heegaard diagrams that contain such pairing diagrams, and we prove the following theorem in Section \ref{Subsection, first pairing theorem}. 
\begin{restatable}{thm}{FirstPairingTheorem}\label{Theorem, paring theorem via nice diagrams}
	Let $\mathcal{H}^a$ be a left provincially admissible arced bordered Heegaard diagram, and let $\alpha_{im}$ be a $z$-adjacent immersed multicurve. Then $$\widehat{CFD}(\mathcal{H}^a({\alpha_{im}}))\cong \widehat{CFDA}(\mathcal{H}^a)\boxtimes \widehat{CFD}(\alpha_{im}).$$
\end{restatable} 

\begin{rmk}
Similar to the $z$-passable condition in Theorem \ref{Theorem, main theorem}, the $z$-adjacency is a diagrammatic condition that is used to handle boundary degeneration in immersed Heegaard Floer theory; it is defined in Section \ref{Subsection, z-adjacency} and can be easily achieved via finger moves.  
\end{rmk}
Among manifolds with torus boundary, of particular interests to us are knot complements. By the results in \cite[Section 11.4-11.5]{LOT18}, the knot Floer chain complex  $\mathcal{CFK}_\mathcal{R}(J)$ of any knot $J\subset S^3$ is equivalent to the type D module $\widehat{CFD}(S^3\backslash \nu(J))$ of the knot complement. (Consequently, $\mathcal{CFK}_\mathcal{R}(J)$ is equivalent to an immersed-multicurve in a marked torus.)

More concretely, the current state of bordered Floer theory recovers certain versions of knot Floer chain complex of a knot from the type D module of its complement as follows. Note that a knot $J$ may be obtained from the knot complement $S^3\backslash \nu(J) $ by gluing in the solid torus containing the \textit{identity pattern knot}, which is the core of the solid torus. Let $\mathcal{H}_{id}$ denote the standard doubly-pointed Heegaard diagram for the identity pattern knot; see Figure \ref{Figure, proof strategy}. In \cite{LOT18}, an $A^{\infty}$-module $CFA^-(\mathcal{H}_{id})$ is associated to $\mathcal{H}_{id}$, and it is shown that $$gCFK^-(J)\cong CFA^-(\mathcal{H}_{id})\boxtimes \widehat{CFD}(S^3\backslash \nu(J)),$$ where $gCFK^-(-)$ denotes the version of knot Floer chain complex over $\mathbb{F}[U]$ \cite[Theorem 11.9]{LOT18}. To recover knot Floer complexes over the larger ground ring $\mathcal{R}=\mathbb{F}[U,V]/UV$, we use a stronger pairing theorem which occurred implicitly in \cite{Hanselman2022} (and even more implicitly in \cite{LOT18}): There are suitable extensions $\widetilde{CFA}(\mathcal{H}_{id})$ and $\widetilde{CFD}(S^3\backslash \nu(J))$ of ${CFA^-}(\mathcal{H}_{id})$ and $\widehat{CFD}(S^3\backslash \nu(J))$ respectively such that $$CFK_{\mathcal{R}}(J)\cong \widetilde{CFA}(\mathcal{H}_{id})\boxtimes \widetilde{CFD}(S^3\backslash \nu(J)).$$

We provide an immersed-Heegaard-diagram approach to recapture the above pairing theorem as well. In Section \ref{Section, bordered invariants of immersed Heegaard diagrams}, we define the so-called weakly extended type D structures $\widetilde{CFD}(-)$ of a certain class of immersed bordered Heegaard diagrams that contains pairing diagrams $\mathcal{H}_{M_1}(\alpha_{M_2})$ mentioned earlier. In Section \ref{Section, knot Floer chain complex of immersed Heegaard diagrams}, we define knot Floer chain complexes of a class of immersed doubly-pointed Heegaard diagrams that includes any diagram $\mathcal{H}_{id}\cup \mathcal{H}_{im}$ obtained by gluing $\mathcal{H}_{id}$ and an immersed bordered Heegaard diagram $\mathcal{H}_{im}$.  Moreover, we prove the following theorem in Section \ref{Section, the second pairing theorem}.

\begin{restatable}{thm}{SecondPairingTheorem}\label{Theorem, 2nd pairing theorem}
	Let $\mathcal{H}_{im}$ be an unobstructed, bi-admissible immersed bordered Heegaard diagram, and let $\mathcal{H}_{id}$ be the standard bordered Heegaard diagram for the identity pattern. Then  
	$$CFK_{\mathcal{R}}(\mathcal{H}_{id}\cup \mathcal{H}_{im})\cong \widetilde{CFA}(\mathcal{H}_{id})\boxtimes \widetilde{CFD}(\mathcal{H}_{im}).$$
\end{restatable}

\begin{figure}[htb!]
	\centering{
		\includegraphics[scale=0.5]{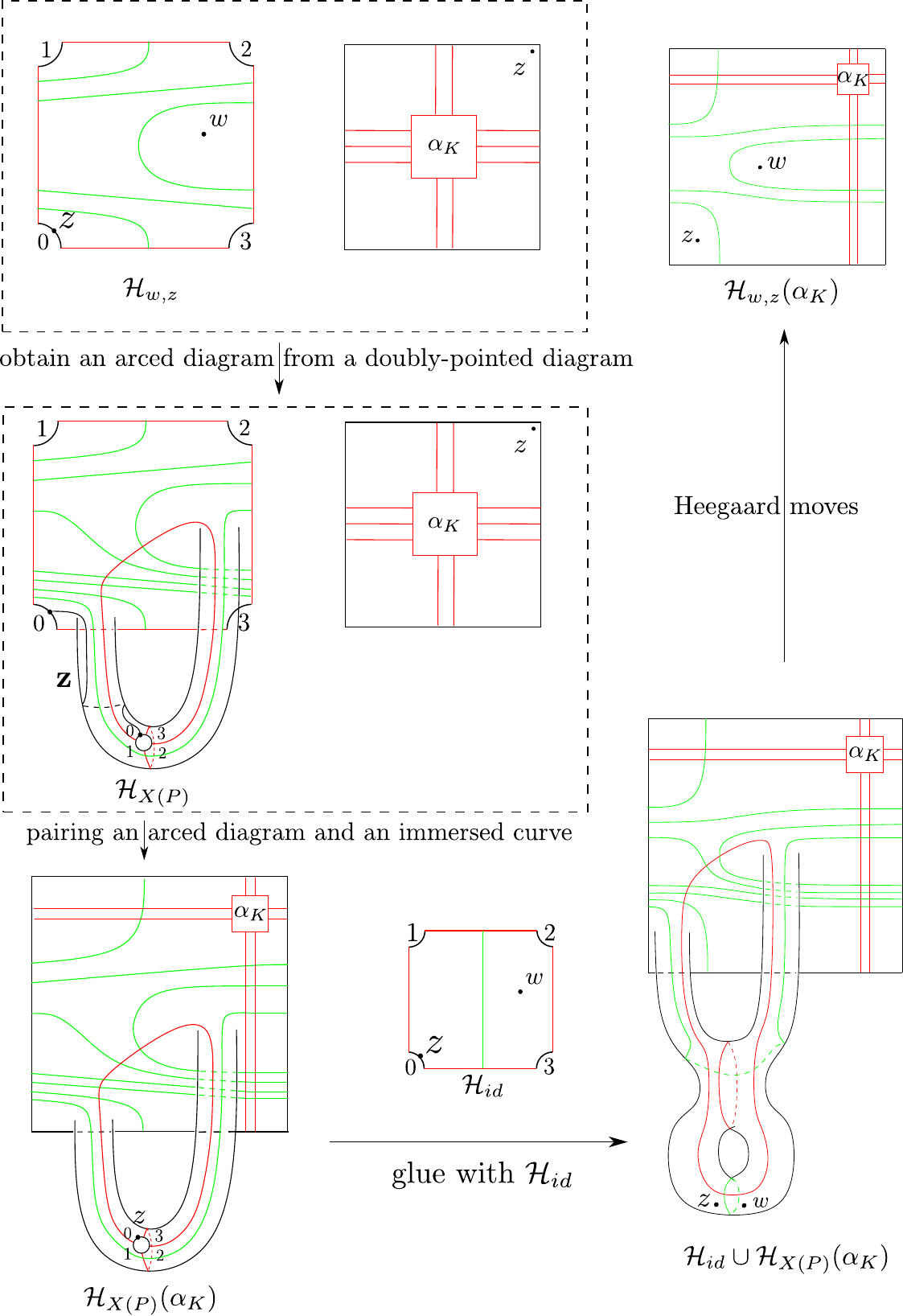}
		\caption{The pairing diagram $\mathcal{H}_{w,z}(\alpha_{K})$ can alternatively be obtained in the following three steps: First, pair an arced bordered Heegaard diagram $\mathcal{H}_{X(P)}$ obtained from $\mathcal{H}_{w,z}$ with the immersed curve $\alpha_K$; secondly, construct a closed doubly-pointed immersed Heegaard diagram $\mathcal{H}_{id}\cup \mathcal{H}_{X(P)}$; third, apply Heegaard moves to $\mathcal{H}_{id}\cup \mathcal{H}_{X(P)}$ to get $\mathcal{H}_{w,z}(\alpha_{K})$.}
		\label{Figure, proof strategy}
	}
\end{figure}
  
\subsection{Strategy to prove the main theorem}
The above theorems are used to compute the knot Floer chain complex of satellite knots as follows.  

First, the knot complement of a satellite knot $P(K)$ decomposes into the union of two 3-manifolds along a torus: the exterior $X(P)=(S^1\times D^2)\backslash \nu(P)$ of the pattern knot and the complement $S^3\backslash \nu(K)$ of the companion knot. Therefore, $$\widehat{CFD}(S^3\backslash \nu(P(K)))\cong \widehat{CFDA}(X(P))\boxtimes \widehat{CFD}(S^3\backslash \nu(K)),$$  
and hence we can apply Theorem \ref{Theorem, paring theorem via nice diagrams} to compute $\widehat{CFD}(S^3\backslash \nu(P(K)))$. More concretely, given a doubly-pointed bordered diagram $\mathcal{H}_{w,z}$ for the pattern knot $P$, one can apply a standard stabilization-and-drilling procedure to obtain an arced bordered Heegaard diagram $\mathcal{H}_{X(P)}$ for $X(P)$, which is then paired with the immersed multicurve $\alpha_{K}$ for $K$ to obtain an immersed bordered Heegaard diagram $\mathcal{H}_{X(P)}(\alpha_{K})$. The type D module $\widehat{CFD}(\mathcal{H}_{X(P)}(\alpha_{K}))$ is then homotopy equivalent to a type D module of $S^3\backslash \nu(P(K))$ by Theorem \ref{Theorem, paring theorem via nice diagrams}.

Second, one can define a weakly extended type D module $\widetilde{CFD}(\mathcal{H}_{X(P)}(\alpha_{K}))$ of the pairing diagram $\mathcal{H}_{X(P)}(\alpha_{K})$. As mentioned above, the underlying (hat-version) type D module $\widehat{CFD}(\mathcal{H}_{X(P)}(\alpha_{K}))$ defined using the same diagram is homotopy equivalent to a type D module of $S^3\backslash \nu(P(K))$. Since extensions of type D modules are unique up to homotopy, $\widetilde{CFD}(\mathcal{H}_{X(P)}(\alpha_{K}))$ is homotopy equivalent to a weakly extended type D module of $S^3\backslash \nu(P(K))$. Now Theorem \ref{Theorem, 2nd pairing theorem} implies that the knot Floer chain complex $CFK_\mathcal{R}(\mathcal{H}_{id}\cup \mathcal{H}_{X(P)}(\alpha_K))$ is homotopy equivalent to the knot Floer chain complex of $P(K)$. 

Finally, the immersed doubly-pointed Heegaard diagram $\mathcal{H}_{w,z}(\alpha_K)$ can be obtained from $\mathcal{H}_{id}\cup \mathcal{H}_{X(P)}(\alpha_K)$ via Heegaard moves. We show knot Floer chain complexes defined from immersed doubly-pointed Heegaard diagrams that differ by Heegaard moves are homotopy
equivalent, and hence $CFK_\mathcal{R}(\mathcal{H}_{w,z}(\alpha_K))$ is homotopic equivalent to a knot Floer chain complex of $P(K)$.

See Figure \ref{Figure, proof strategy} for an illustration of the operations on Heegaard diagrams involved in the strategy of the proof.

\subsection{Further discussions}
\subsubsection{Immersed Heegaard diagrams}
The work presented in this paper opens a new avenue for studying Heegaard Floer homology using immersed Heegaard diagrams. While the results in this paper already demonstrate this strategy can be useful useful for studying satellite operations, many questions remain that are worthy of further study. For example, a natural question is whether $CFD$ can be defined for a more general class of immersed Heegaard diagrams in which more than one $\alpha$ and/or $\beta$ curve may be immersed, rather than just for a single $\alpha$ curve as in the present setting. We expect this is possible but the the technical difficulties will be greater.

As a special case, one could consider doubly pointed genus one immersed Heegaard diagrams in which both the $\alpha$ and $\beta$ curve are allowed to be immersed. In this case there are no technical difficulties in defining a Floer complex from such a diagram, as the construction is combinatorial. We expect that this class of diagrams will be useful for studying satellite knots with arbitrary patterns, so that it will not be necessary to restrict to $(1,1)$ patterns to perform computations in a genus one surface. More precisely, an arbitrary pattern should give rise to an immersed (1,1)-diagram curve which can be used to recover the action of the satellite operation on knot Floer complexes, and an analog of Theorem \ref{Theorem, main theorem} should hold so that pairing the immersed (1,1)-diagram for the satellite with the immersed curve for the companion computes the knot Floer complex of a satellite. This will be explored in future work.

A related question concerns immersed diagrams for bimodules in bordered Floer theory. Stabilizing a (1,1) diagram gives a genus two arced bordered diagram for the complement of the pattern knot, which gives rise to a bordered Floer bimodule. In analogy to immersed (1,1) diagrams, we could consider arced bordered diagram with an immersed $\beta$ curve. We could ask which bimodules can be represented by such diagrams and if these diagrams are useful in determining how bimodules act on immersed curves. Just as modules over the torus algebra correspond to (decorated) immersed curves in the punctured torus $T_\bullet$, it is expected that bimodules are related to immersed surfaces in $T_\bullet\times T_\bullet$. It may be that arced bordered diagrams with immersed curves are helpful in understanding this connection.

\subsubsection{Pattern detection}
In another direction, we can ask if the nice behavior demonstrated for one-bridge braid patterns extends to any other patterns. Recall that given two patterns $P_1$ and $P_2$, we define the composition $P_1\circ P_2$ to be the pattern knot so that $(P_1\circ P_2) (K)$ is $P_1(P_2(K))$ for any companion knot $K$. Theorem \ref{Theorem, satellite knots with one-bridge braid patterns} implies that one-bridge-braid patterns and their compositions act as planar transforms on the (lifts of) immersed curves of companion knots. We wonder if this property characterize these pattern knots.
\begin{que}
	Are one-bridge braid patterns and their compositions the only pattern knots that induce satellite operations that act as planar transforms on immersed curves in the marked torus?
\end{que}
More generally, one can ask the following question.
\begin{que}
	 Which pattern knots are detected by the bordered Floer bimodule?
\end{que}
 Pattern knot detection is closely related to the pursuit of understanding which links are detected by Heegaard Floer homology, as a pattern knot $P$ is uniquely determined by the link $L_P$ in $S^3$ consisting of $P$ and the meridian of the solid torus containing $P$. For example, detection of $(n,1)$-cable patterns would follow from the corresponding link detection result on $T_{2,2n}$ by Binns-Martin \cite[Theorem 3.2]{Binns2020}. Note that the bimodule of a pattern knot complement is stronger than the knot/link Floer homology group of $L_P$, so it is also natural to wonder if one can detect patterns using bimodules that are not seen by the link detection results.

\subsection{Organization} 
The rest of the paper can be divided into two parts. 

The first part includes Section \ref{Section, bordered invariants of immersed Heegaard diagrams} to Section \ref{Section, pairing theorem section} and establishes the immersed Heegaard Floer theory outlined in the introduction. Section \ref{Section, bordered invariants of immersed Heegaard diagrams} defines bordered Heegaard Floer invariants of immersed bordered Heegaard diagrams. Section \ref{Section, knot Floer chain complex of immersed Heegaard diagrams} defines knot Floer chain complexes of immersed doubly-pointed Heegaard diagrams. Section \ref{Section, pairing theorem section} introduces the pairing constructions and proves the corresponding pairing theorems, i.e., Theorem \ref{Theorem, paring theorem via nice diagrams} and Theorem \ref{Theorem, 2nd pairing theorem}.  

The second part concerns satellite knot and includes Section \ref{Section, satellite knot pairing} and Section \ref{Section, examples}.
Section \ref{Section, satellite knot pairing} proves Theorem \ref{Theorem, main theorem} by applying the machinery established in the previous sections. Section \ref{Section, examples} applies Theorem \ref{Theorem, main theorem} to study satellite knots with $(1,1)$ patterns, in which we remove the 
$z$-passable assumption and analyze satellites with one-bridge-braid patterns and the Mazur pattern in detail. 

\subsection*{Acknowledgment} The authors would like to thank Robert Lipshitz, Adam Levine, and Liam Watson for helpful conversations while this work was in progress. The first author was partially supported by the Max Planck Institute for Mathematics and the Pacific Institute for the Mathematical Sciences during the preparation of this work; the research and findings may not reflect those of the institutes. The second author was partially supported by NSF grant DMS-2105501.

\section{Bordered Floer invariants of immersed Heegaard diagrams}\label{Section, bordered invariants of immersed Heegaard diagrams}
This section aims to define (weakly extended) type D structures\footnote{Here, weakly extended type D structures are same as type D structures with generalized coefficient maps appearing in \cite[Chapter 11.6]{LOT18}, and we call them weak since they are a quotient of the extended type D structures defined in \cite{hanselman2016bordered}.} of a certain class of immersed bordered Heegaard diagrams. Even though the Lagrangians are possibly immersed, we can still define such structures by counting holomorphic curves with smooth boundaries. The main technical complication compared to the embedded-Heegaard-diagram case are the possible appearance of boundary degeneration, which interferes with the usual proof that the differential squares to a desired element. Our method to deal with this issue is to employ a key advantage of Heegaard Floer homology: The boundary degenerations can be controlled by imposing certain diagrammatic conditions on the Heegaard diagrams (in Section \ref{Subsection, immersed bordered Heegaard diagram}), and in Section \ref{Section, pairing theorem section} we show these conditions are achievable via perturbing the $\alpha$ curves on the Heegaard diagrams that we are interested in. 

This section can be viewed as modifying \cite[Chapter 5 and Chapter 11.6]{LOT18} and \cite[Errata]{LIPSHITZ2008} to our setting. Indeed, the local results on holomorphic curves such as transversality and gluing results carry over without changes. The main differences are that (1) the embedded index formula is different (see Section \ref{Section, embedded index formula}) and that (2) we need a parity count of boundary degenerations with corners at the self-intersection points of the immersed Lagrangian (see Section \ref{Section, ends of moduli spaces of 0-p curves}--\ref{Subsection, ends of moduli space of 1P curves}). We also add in more details on counting holomorphic curves with a single interior puncture as sketched in \cite[Errata]{LIPSHITZ2008}. The counting of such curves also appeared in \cite{Hanselman2022}, and the more general case where multiple interior punctures are allowed is studied in-depth in the recent work by Lipshitz-Ozsv\'ath-Thurston on defining a $HF^{-}$ bordered invariant \cite{Lipshitz2023}; our analysis of the degeneration of one-punctured curves at east infinity is extracted from \cite{Hanselman2022}.

\subsection{Immersed bordered Heegaard diagrams}\label{Subsection, immersed bordered Heegaard diagram}

\begin{defn}
A \textit{local system} over a manifold $M$ consists of a vector bundle over $M$ and a parallel transport of the vector bundle. A trivial local system is a local system where the vector bundle is of rank $1$.
\end{defn}
\begin{defn}
An \emph{immersed bordered Heegaard diagram} is a quadruple $\mathcal{H}=(\bar{\Sigma},\bm{\bar{\alpha}},\bm{\beta},z)$, where 
\begin{itemize}
\item[(1)]$\bar{\Sigma}$ is a compact, oriented surface of genus $g$ with one boundary component;
\item[(2)]$\bm{\bar{\alpha}}=\bm{\bar{\alpha}}^a\cup \bm{\alpha}^c$. Here $\bm{\bar{\alpha}}^a=\{\bar{\alpha_1}^a,\bar{\alpha_2}^a\}$ is a set of two properly embedded arcs in $\bar{\Sigma}$, and $\bm{\alpha}^c=\{\alpha_1,\ldots,\alpha_{g-2},\alpha_{g-1}\}$, where $\alpha_1,\ldots,\alpha_{g-2}$ are embedded circles in the interior of $\bar{\Sigma}$ and $\alpha_{g-1}=\{\alpha^0_{im},\ldots,\alpha^n_{im}\}$ is an immersed multicurve with a local system in the interior of $\bar{\Sigma}$. We require that $\alpha_{im}^0$ has a trivial local system and that $\{\bar{\alpha_1^a},\bar{\alpha_2^a},\alpha_1,\ldots,\alpha_{g-2},\alpha_{im}^0\}$ are pairwise disjoint and homologically independent in $H_1(\bar{\Sigma},\partial{\bar{\Sigma}};\mathbb{Z})$. We also require the curves $\alpha_{im}^1,\ldots,\alpha_{im}^{n}$ are homologically trivial in $H_1(\bar{\Sigma},\mathbb{Z})$\footnote{To maintain this property, we will not consider handleslides of homologically-trivial curves over other embedded $\alpha$-curves.};
\item[(3)]$\bm{\beta}=\{\beta_1,\ldots,\beta_g\}$ consists of $g$ pairwise disjoint, homologically independent circles embedded in the interior of $\bar{\Sigma}$;
\item[(4)]A base point $z\in\partial\bar{\Sigma}\backslash\partial\bar{\bm{\alpha}}^a$.
\end{itemize}
\end{defn}

\begin{rmk}
	We also denote $\alpha_{g-1}$ by $\alpha_{im}$ and call $\alpha_{im}^0$ the \emph{distinguished component} of $\alpha_{im}$. Note if $\alpha_{im}$ is embedded and consists of only the distinguished component, then the immersed bordered Heegaard diagram is just an ordinary bordered Heegaard diagram representing a bordered 3-manifold with torus boundary. 
\end{rmk}
\begin{rmk}
One can define immersed bordered Heegaard diagrams that have more than one immersed multicurves. We content ourselves with the one-immersed-multicurve setting, for such diagrams occur naturally in applications and we would also like to avoid the more tedious notations incurred by allowing more immersed multicurves. One can also work with immersed bordered diagrams that generalizes regular bordered diagrams for 3-manifolds with higher-genus or multi-component boundaries. Again, we avoid such cases out of conciseness.  
\end{rmk}
We need to impose additional conditions on immersed bordered diagrams to define (weakly extended) type D structures. We give some terminology before stating the conditions. Let $\{\mathcal{D}_i\}$ be the closures of the regions in $\bar{\Sigma}\backslash(\bm{\bar{\alpha}}\cup\bm{\beta})$. A \emph{domain} $B$ is a formal linear combination of the $\mathcal{D}_i$'s, i.e., an object of the form $\sum_{i}n_i\mathcal{D}_i$ for $n_i\in\mathbb{Z}$, and the coefficient $n_i$ is called the multiplicity of $B$ at $\mathcal{D}_i$. Given a point $p\in\bar{\Sigma}\backslash(\bm{\bar{\alpha}}\cup\bm{\beta})$, $n_p(B)$ denotes the multiplicity of $B$ at the region containing $p$. A domain is called positive if $n_i\geq 0$ for all $i$. Note that a domain $B$ specifies an element $[B]$ in $H_2(\bar{\Sigma},\partial\bar{\Sigma}\cup\bm{\beta}\cup\bm{\bar{\alpha}};\mathbb{Z})$. Let $l:\amalg S^1\rightarrow \partial\bar{\Sigma}\cup\bm{\beta}\cup\bm{\bar{\alpha}}\subset\bar{\Sigma}$ be an oriented multicurve; note that this multicurve is not necessarily immersed, it can have corners at intersections of $\bar\alpha$ with $\beta$, intersections of $\bar\alpha$ with $\partial \bar\Sigma$, or self-intersections of $\alpha_{im}$. A domain $B$ is said to be bounded by $l$ if $\partial [B]=[l]$ in $H_1(\partial\bar{\Sigma}\cup\bm{\beta}\cup\bm{\bar{\alpha}};\mathbb{Z})$. A domain $B$ is called a \emph{periodic domain} if it is bounded by a (possibly empty) loop in $\partial\bar{\Sigma}\cup\bar{\bm{\alpha}}^a$ together with some copies of the $\beta$-circles and the $\alpha$-circles, where we allow at most one component of $\alpha_{im}$ to appear in $\partial B$. 

We will need some language to keep track of when the boundaries of domains include corners at self-intersection points of $\alpha_{i}$. Note that, ignoring the local systems, we can identify $\alpha_{im}=\{\alpha^0_{im},\ldots,\alpha^n_{im}\}$ as the image of an immersion $f_{im}:\amalg^{n+1} S^1 \rightarrow \bar{\Sigma}$.
\begin{defn}\label{Defition, stay-on-track or one-cornered curves}
A closed curve $l:S^1\rightarrow \alpha_{im}\subset\bar{\Sigma}$ is said to be \emph{stay-on-track} or \emph{zero-cornered} if it lifts to a map $\tilde{l}:S^1\rightarrow S^1\subset\amalg S^1$ such that $f_{im}\circ \tilde{l}=l$. Note that this is nearly the same as saying the curve $l$ is immersed; the difference is that stay-on-track paths can stop and change directions along $\alpha_{im}$. A curve $l:S^1\rightarrow \alpha_{im}$ is said to be \emph{$n$-cornered} if there exists $n$-points $\xi_1, \ldots \xi_n$ in $S^1$ dividing $S^1$ into arcs $a_1, \ldots, a_n$ (ordered cyclically, with indices understood mod $n$) such that $l|_{a_i}$ lifts through $f_{im}$ for each $i$, but $l|_{a_i\cup a_{i+1}}$ does not. Note that $l$ maps each $\xi_i$ to some self-intersection point $q_i$ of $\alpha_{im}$ and makes a sharp turn at $q_i$; we refer to this as a corner of the curve $l$. We define an arc to be either stay-on-track or $n$-cornered similarly.
\end{defn}
\begin{figure}[htb!]
	\centering{
		\includegraphics[scale=0.4]{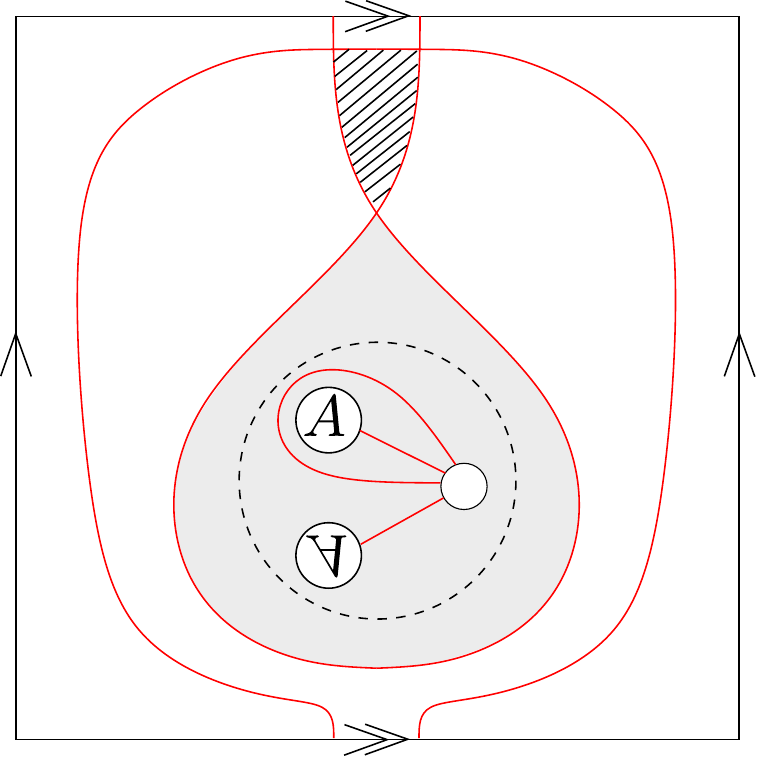}
		\caption{Examples of a $3$-cornered curve and a stabilized teardrop on a genus-2 immersed bordered Heegaard diagram, where the $\beta$ curves are omitted. The $3$-cornered curve is the boundary of the shaded triangular region. The highlighted region is a stabilized teardrop, where the dashed circle is the separating curve.}
		\label{Figure, an n-cornered curve and a stabilized teardrop}
	}
\end{figure}
See Figure \ref{Figure, an n-cornered curve and a stabilized teardrop} for an example of a $3$-cornered curve. Next, we define a class of domains on immersed bordered Heegaard diagrams. 
\begin{defn}
A domain $B$ on a genus $g$ immersed bordered Heegaard diagram is called a \emph{stabilized teardrop} if it satisfies the following conditions:
\begin{itemize}
\item[(1)]$B$ is a positive domain bounded by $\partial\bar{\Sigma}$ (with induced boundary orientation) and a one-cornered subloop of $\alpha_{im}$. (In particular, $B$ is a formal linear combination of regions of $\bar{\Sigma}\backslash\bm{\bar{\alpha}}$.)
\item[(2)]There exists a separating curve $C$ of $\bar{\Sigma}$ which does not intersect $\bm{\bar{\alpha}}$, and the local multiplicity of $B$ on the region containing $C$ is $1$. 
\item[(3)]Surgery on $(\bar{\Sigma},\bm{\bar{\alpha}})$ along $C$ produces two oriented surfaces with $\alpha$-curves: $(E_1,\bar{\bm{\alpha}}^a,\alpha_1,\ldots,\alpha_{g-1})$,  where $E_1$ is a genus-$(g-1)$ with one boundary component, and $(E_2,\alpha_{im})$, where $E_2$ a genus-one surface. The domain $B$ gives rise to two domains $B_1$ and $B_2$ on $E_1$ and $E_2$ respectively, such that $[B_1]=[E_1]$ and $B_2$ is an immersed teardrop in $E_2$ bounded by a one-cornered subloop of $\alpha_{im}$. (Here, we also allow teardrops with concave corner.)
\end{itemize}
\end{defn}
A pictorial example of a stabilized teardrop is shown in Figure \ref{Figure, an n-cornered curve and a stabilized teardrop}. For convenience, we introduce the following terminology.
\begin{defn}
A domain $B$ is said to be an \emph{$n$-cornered $\alpha$-bounded domain} if it is bounded by (possibly trivial) loops contained in $\partial\bar{\Sigma}\cup \bar{\bm{\alpha}}^a\cup \alpha_{1}^c\cup\ldots\cup \alpha_{g-2}^c$ and an $n$-cornered loop contained in some connected component of $\alpha_{im}$. 
\end{defn}
The condition on bordered diagrams needed to deal with boundary degenerations is the following. 
\begin{defn}\label{Definition, unobstructedness}
Given an immersed bordered Heegaard diagram $\mathcal{H}=(\bar{\Sigma},\bm{\bar{\alpha}},\bm{\beta},z)$, $\bm{\bar{\alpha}}$ is called \emph{unobstructed} if 
\begin{itemize}
\item[(1)]there are no positive zero- or one-cornered $\alpha$-bounded domains $B$ with $n_z(B)=0$, 
\item[(2)]the only positive zero-cornered $\alpha$-bounded domain $B$ with $n_z(B)=1$ is $[\Sigma]$,
\item[(3)]any positive one-cornered $\alpha$-bounded domain $B$ with $n_z(B)=1$ is a stabilized teardrop, and
\item[(4)]any positive two-cornered $\alpha$-bounded domain $B$ with $n_z(B)=0$ is a bigon.
\end{itemize}
Abusing the terminology, we also say an immersed Heegaard diagram is unobstructed if its $\bm{\bar{\alpha}}$ is so.
\end{defn} 

By a \textit{Reeb chord} in $(\partial\bar{\Sigma}, \partial \bar{\bm{\alpha}^a})$, we mean an oriented chord on $\partial\bar{\Sigma}$ whose endpoints are on $\partial \bar{\bm{\alpha}^a}$ and whose orientation is induced from that of $\partial\bar{\Sigma}$. When defining type D structures, it will be convenient to use Reeb chords in $(-\partial\bar{\Sigma}, \partial \bar{\bm{\alpha}^a})$. Let $\rho_0, \rho_1, \ldots,\rho_{3}$ denote the Reeb chords corresponding to the four arcs in $-\partial\bar{\Sigma}\backslash \partial\bar{\bm{\alpha}^a}$, where $\rho_0$ contains the base point $z$ and the sub-index increases according to the orientation of $-\partial\bar{\Sigma}$. We call these four Reeb chords the \textit{elementary Reeb chords}. Other Reeb chords in $(-\partial\bar{\Sigma}, \partial \bar{\bm{\alpha}^a})$ can be obtained by concatenation of the four elementary Reeb chords. For example, $\rho_{12}$ denotes the concatenation of $\rho_1$ and $\rho_2$. We use the notation $-\rho_I$ to indicate the orientation reversal of $\rho_I$, where $I$ is a string of words in $\{0,1,2,3\}$ that permits concatenation; note that $-\rho_I$ is a Reeb chord in $(\partial\bar{\Sigma}, \partial \bar{\bm{\alpha}^a})$.

 We shall need two types of admissibility on the bordered Heegaard diagrams: the first one is needed for defining the (extended) type D structures, and the second one is needed for the paring operation studied in Section \ref{Section, the second pairing theorem}. Given a domain $B$, denote by $n_{-\rho_i}(B)$ ($i=0,1,2,3$) the local multiplicity of $B$ in the region containing the Reeb chord $-\rho_i$. (In particular, $n_{-\rho_0}(B)=n_z(B)$). A periodic domain $B$ is called \emph{provincial} if $n_{-\rho_i}(B)=0$ for all $i=0,1,2,3$. 
\begin{defn}\label{Definition, provinical admissibility}
An immersed bordered Heegaard diagram is \emph{provincially admissible} if all non-trivial provincial periodic domains have both positive and negative local multiplicities.
\end{defn}

\begin{defn}\label{Definition, bi-admissibility}
An immersed bordered Heegaard diagram is \emph{bi-admissible} if any non-trivial periodic domain $B$ satisfying $n_{-\rho_0}(B)=n_{-\rho_1}(B)=0$ or $n_{-\rho_2}(B)=n_{-\rho_3}(B)=0$ has both positive and negative local multiplicities.
\end{defn}

Note that bi-admissibility implies provincial admissibility. 
\subsection{Moduli spaces of stay-on-track holomorphic curves} \label{subsection, defining moduli spaces}
In this subsection, we set up the moduli spaces that we use to define the (weakly extended) type D structures. Roughly, the moduli spaces consist of pseudo-holomorphic curves in $\Sigma\times [0,1]\times \mathbb{R}$. We define them by modifying the corresponding definitions in Section 5.2 of \cite{LOT18} with two main differences: The first one is a new constraint on the boundaries with respect to $\alpha_{im}$, and the second one is that we include pseudo-holomorphic curves with a single interior puncture for the sake of defining a weakly extended type D structure. 

\subsubsection{Definition of moduli spaces of holomorphic curves} 
\begin{defn}\label{Definition, decorated sources}
A \emph{decorated source $S^\triangleright$ of type $0$-P} is a smooth Riemann surface $S$ with boundary such that
\begin{itemize}
\item[(1)]it has boundary punctures and has no interior punctures, 
\item[(2)]there is a labeling of each puncture by $+$, $-$, or $e$, and
\item[(3)]there is a labeling of each $e$ puncture by a Reeb chord on the boundary of the immersed bordered Heegaard diagram. 
\end{itemize}
A \emph{decorated source $S^\triangleright$ of type $1$-P} is a smooth Riemann surface $S$ with boundary such that 
\begin{itemize}
\item[(1)]it has boundary punctures and a single interior puncture,
\item[(2)]there is a labeling of each boundary puncture by $+$ or $-$, and
\item[(3)]there is a labeling of the interior puncture by $e$.
\end{itemize}
By a decorated source, we mean it is either a decorated source of type $0$-P or a decorated source of type $1$-P.
\end{defn}

Let $\mathcal{H}=(\bar{\Sigma},\bm{\bar{\alpha}},\bm{\beta},z)$ be an immersed bordered Heegaard diagram. Let $\Sigma$ denote the interior of $\bar{\Sigma}$. Equip the target manifold $\Sigma\times [0,1]\times \mathbb{R}$ with an admissible almost complex structure $J$ as in Definition 5.1 of \cite{LOT18}. Let $\pi_{\mathbb{D}}$, $\pi_\Sigma$, $s$, and $t$ denote the canonical projection maps from $\Sigma\times [0,1]\times \mathbb{R}$ to $[0,1]\times \mathbb{R}$, $\Sigma$, $[0,1]$, and $\mathbb{R}$ respectively. We will count maps $$u:(S,\partial S)\rightarrow (\Sigma\times [0,1]\times \mathbb{R}, \bm{\beta}\times\{0\}\times \mathbb{R}, \bm{\bar{\alpha}} \times \{1\} \times \mathbb{R})$$ from decorated sources to the target manifold satisfying the following conditions:
\begin{itemize}
\item[(M-1)] $u$ is $(j,J)$-holomorphic, where $j$ is a complex structure on the surface $S$. 
\item[(M-2)] $u$ is proper.
\item[(M-3)] $u$ extends to a proper map $u_{\bar{e}}: S_{\bar{e}} \rightarrow \Sigma_{\bar{e}}\times[0,1]\times \mathbb{R}$, where $S_{\bar{e}}$ and $\Sigma_{\bar{e}}$ are surfaces obtained by filling in the corresponding east puncture(s).
\item[(M-4)] the map $u_{\bar{e}}$ has finite energy in the sense in \cite{BEH03}.
\item[(M-5)] $\pi_{\mathbb{D}}\circ u_{\bar{e}}$ is a $g$-fold branched cover.
\item[(M-6)] $t\circ u$ approaches $\infty$ at $+$ punctures.
\item[(M-7)] $t\circ u$ approaches $-\infty$ at $-$ punctures.
\item[(M-8)] $\pi_{\Sigma} \circ u$ approaches the labeled Reeb chord at a boundary $e$ punctures.
\item[{(M-9)}] $\pi_{\Sigma} \circ u$ covers each of the regions next to $\bar{e}\in \Sigma_{\bar{e}}$ at most once.
\item[(M-10)](Strong boundary monotonicity) For each $t\in\mathbb{R}$ , each of $u^{-1}(\beta_i\times\{0\}\times \{t\})$ $(i=1,\ldots,g)$ and $u^{-1}(\alpha^c_{i}\times\{1\}\times \{t\})$ $(i=1,\ldots,g-1)$ consists of exactly one point, and $u^{-1}(\alpha_i^a\times\{1\}\times \{t\})$ $(i=1,2)$ consists of at most one point\footnote{In \cite{LOT18}, a weak boundary monotonicity condition was introduced as well as the strong boundary monotonicity. When restricting to torus-boundary bordered manifolds, however, these two conditions are equivalent. So, we only state the strong boundary monotonicity condition here.}.

\item[{(M-11)}](Stay-on-track boundary condition) Let $A$ be the boundary component of $S$ that is mapped to the $\alpha_{im}\times\{1\}\times \mathbb{R}$. Then $\pi_{\Sigma}\circ u|_A$ is stay-on-track.
\end{itemize}   

\begin{rmk}
Only (M-9) and (M-11) are different from the corresponding conditions in \cite{LOT18}. We impose (M-9) since we aim to define an extended type D structure, which will not need holomorphic curves covering the boundary regions multiple times.
\end{rmk}

Given an immersed (bordered) Heegaard diagram, generators and homology classes of maps connecting generators are defined similarly as in the embedded-$\alpha$-curve case.
 
\begin{defn}\label{Definition, holomorphic curve in Sigma*[0,1]* R}
Let $\bm{x}$ and $\bm{y}$ be two generators and let $B\in\tilde{\pi}_2(\bm{x},\bm{y})$ be a homology class connecting  $\bm{x}$ to $\bm{y}$. $\widetilde{\mathcal{M}}^B(\bm{x},\bm{y};S^{\triangleright})$ is defined to be the moduli space of holomorphic curves with decorated source $S^{\triangleright}$, satisfying (M-1)-(M-11), asymptotic to $\bm{x}$ at $-\infty$ and $\bm{y}$ at $+\infty$, and inducing the homology class $B$. 
\end{defn}

Let $E(S^\triangleright)$ be the set of east punctures of $S^\triangleright$ lying on the boundary. Let $\widetilde{ev}\colon \widetilde{\mathcal{M}}^B(\bm{x},\bm{y};S^{\triangleright})\rightarrow \mathbb{R}^{|E(S^\triangleright)|}$ be the evaluation map given by the values of $t\circ u_{\bar{e}}$ at the east punctures; the values are called heights of the east punctures.

\begin{defn}
Let $P=\{P_i\}$ be a partition of $E$. Then $$\widetilde{\mathcal{M}}^B(\bm{x},\bm{y};S^{\triangleright};P)\coloneqq \widetilde{ev}^{-1}(\Delta_{P}),$$ where $\Delta_P:=\{(x_p)\in \mathbb R^{|E|}\mid x_p=x_q\text{ if }p,q\in P_i \text{ for some i} \}$.
\end{defn}

\begin{defn} 
Let $\overrightarrow{P}=(P_1,\ldots,P_k)$ be an ordered partition $E$, and let $P$ denote the corresponding underlying unordered partition. Define $\widetilde{\mathcal{M}}^B(\bm{x},\bm{y};S^{\triangleright};\overrightarrow{P})$ to be $\{u\in\widetilde{\mathcal{M}}^B(\bm{x},\bm{y};S^{\triangleright};P)\mid t\circ u (p)< t\circ u(q) \text{ for }p\in P_i, \text{ } q\in P_{i'},\text{ and }i< i'\}.$
\end{defn}

There is an $\mathbb{R}$-action on the above moduli spaces given by translations along the $\mathbb{R}$-coordinate of $\Sigma\times[0,1]\times \mathbb{R}$. \textit{The reduced moduli spaces} are the  quotient of the relevant moduli spaces by the $\mathbb{R}$-action; they are denoted by $\mathcal{M}^B(\bm{x},\bm{y};S^{\triangleright})$, $\mathcal{M}^B(\bm{x},\bm{y};S^{\triangleright};{P})$, and $\mathcal{M}^B(\bm{x},\bm{y};S^{\triangleright};\overrightarrow{P})$, respectively. The evaluation maps $\widetilde{ev}$ induce maps $ev$ from the reduced moduli spaces to $\mathbb{R}^{|E(S^\triangleright)|}/\mathbb{R}$, which record the relative heights between boundary east punctures.\\

\paragraph*{\textbf{Notation.}} When we need to distinguish moduli spaces of 0-P holomorphic curves without east punctures and moduli spaces of 1-P holomorphi curves, we will use the notation $\widetilde{\mathcal{M}}^B(\bm{x},\bm{y};S^{\triangleright};U)$ to emphasize the source $S^\triangleright$ is of type 1-P. 

\subsubsection{Regularity and the expected dimension}
\begin{prop}\label{Proposition, regularity of the moduli spaces}
For a generic admissible almost complex structure on $\Sigma\times [0,1]\times \mathbb{R}$, the moduli space $\widetilde{\mathcal{M}}^B(\bm{x},\bm{y};S^{\triangleright};P)$ is transversally cut out.
\end{prop}
\begin{proof}
See Proposition 5.6 of \cite{LOT18}. We point out that the $\alpha$ curves being immersed does not affect the usual proof. When analyzing the linearization of the $\bar{\partial}$-operator in the standard proof of such a result, one would be working with a pull-back bundle over $S$, on which one will not see the immersed boundary condition anymore.   
\end{proof}

\begin{prop}\label{Proposition, first index formula}
Let $B\in \widetilde{\pi}_2(\bm{x},\bm{y})$. Let $g$ denote the genus of the bordered Heegaard diagram. Let $\chi(\cdot)$ and $e(\cdot)$ denote the Euler number and Euler measure, respectively.
\begin{itemize}
\item[(1)]Let $S_0^{\triangleright}$ be a decorated source of 0-P. Then the expected dimension of the moduli space $\widetilde{\mathcal{M}}^B(\bm{x},\bm{y};S_0^{\triangleright};P)$ is 
$$\text{ind}(B,S_0,P)=g-\chi(S_0)+2e(B)+|P|.$$
\item[(2)]Let $S_1^{\triangleright}$ be a decorated source of 1-P. Then the expected dimension of the moduli space $\widetilde{\mathcal{M}}^B(\bm{x},\bm{y};S_1^{\triangleright};U)$ is 
$$\text{ind}(B,S_1)=g-\chi(S_1)+2e(B)+1.$$
\end{itemize}
\end{prop}
\begin{proof}
For (1) see Proposition 5.8 of \cite{LOT18}. For the same reason mentioned in the proof of the previous proposition, our $\alpha$-curves being immersed does not affect the proof. 

For (2), note that using the removable singularity theorem we can identify holomorphic curves in $\widetilde{\mathcal{M}}^B(\bm{x},\bm{y};S_1^{\triangleright};U)$ with holomorphic curves in $\Sigma_{\bar{e}}\times [0,1]\times \mathbb{R}$ that intersects $\{e\}\times[0,1]\times \mathbb{R}$ geometrically once. The formula then follows from the index formula in the close-Heegaard-diagram case given in Section 4.1 of \cite{MR2240908}. (Again, $\alpha_{im}$ being immersed does not affect the formula).

\end{proof}
\begin{rmk}
If we regard the interior puncture of a 1-P holomorphic curve is asymptotic to a single closed Reeb orbit, denoted by $U$, then (1) and (2) in the above proposition may be unified using the formula in (1).
\end{rmk}

\subsection{Compactification}
The moduli spaces defined in Section \ref{subsection, defining moduli spaces} admit compactification similar to that in \cite{LOT18}. The overall idea is that a sequence of holomorphic curves in $\Sigma\times[0,1]\times \mathbb{R}$ may converge to a holomorphic building in $\Sigma\times[0,1]\times \mathbb{R}$ together with some holomorphic curves in the east-$\infty$ attaching to it; such nodal holomorphic objects are called \emph{holomorphic combs}. In our setup, the degeneration in the east-$\infty$ is the same as those when the $\alpha$-curves are embedded; we recollect the relevant material (with straightforward modifications to accommodate for 1-P holomorphic curves) in Subsection \ref{subsubsection, holomorphic curves in east infinty}. However, the immersed $\alpha$ curves do complicate the situation. For example, a limit holomorphic building in $\Sigma\times[0,1]\times \mathbb{R}$ may have corners at self-intersection points of $\alpha_{im}$. We will give a precise description of this phenomenon in Subsection \ref{subsubsection, holomorphic combs}.

\subsubsection{Holomorphic curves in the end at east-infinity}\label{subsubsection, holomorphic curves in east infinty} Let $Z$ denote the oriented boundary $\partial\bar{\Sigma}$ of the bordered Heegaard surface. We define the moduli spaces of holomorphic curves in the east end $\mathbb{R}\times Z \times [0,1]\times \mathbb{R}$. They host possible degenerations of the limits of holomorphic curves at east-$\infty$. Since the closed $\alpha$ curves do not approach the cylindrical end at east-$\infty$, these moduli spaces are not affected by the closed $\alpha$ curves being immersed and their definition is the same as the usual embedded case. We first specify the sources of the holomorphic curves.
\begin{defn}
A \emph{bi-decorated source} $T^\diamond$ is a smooth Riemann surface $T$ with boundary such that 
\begin{itemize}
\item[(1)]it has boundary punctures and at most one interior puncture,
\item[(2)]the boundary punctures are labeled by $e$ or $w$,  
\item[(3)]the interior puncture, if exits, is labeled by $e$ and,
\item[(4)]the boundary punctures are also labeled by Reeb chords.
\end{itemize} 
\end{defn}

Equip $\mathbb{R}\times Z\times[0,1]\times \mathbb{R}$ with a split almost complex structure $J=j_{\mathbb{R}\times Z}\times j_\mathbb{D}$. The four points $\textbf{a}=\partial{\bm{\alpha}^a}$ on $Z$ give rise to four Lagrangians $\mathbb{R}\times \textbf{a}\times\{1\}\times\mathbb{R}$.
\begin{defn}
Given a bi-decorated source $T^\diamond$, define $\widetilde{\mathcal{N}}(T^\diamond)$ to be the moduli spaces of maps $v:(T,\partial T)\rightarrow (\mathbb{R}\times Z\times[0,1]\times \mathbb{R},\mathbb{R}\times \textbf{a}\times\{1\}\times\mathbb{R})$ satisfying the following conditions:
\begin{itemize}
\item[(N-1)] $v$ is $(j,J)$-holomorphic with respect to some complex structure $j$ on $T$.
\item[(N-2)] $v$ is proper.
\item[(N-3)] Let $T_{\bar{e}}$ and $(\mathbb{R}\times Z)_{\bar{e}}$ denote the spaces obtained from $T$ and $\mathbb{R}\times Z$ by filling in the east punctures. Then $v$ extends to a proper map $v_{\bar{e}}:T_{\bar{e}}\rightarrow (\mathbb{R}\times Z)_{\bar{e}}\times [0,1]\times \mathbb{R}$ such that $\pi_\Sigma\circ v_{\bar{e}}(e)=e$.\footnote{This condition excludes mapping the interior puncture end to the west infinity end.}  
\item[(N-4)]At each boundary $w$ puncture, $\pi_{\Sigma}\circ v$ approaches the  corresponding Reeb chords in $-\infty\times Z$ that labels $w$.
\item[(N-5)]At each boundary $e$ puncture, $\pi_{\Sigma}\circ v$ approaches the corresponding Reeb chords in $\infty\times Z$ that labels $e$.
\end{itemize}
\end{defn}

We have the following proposition regarding the regularity of the moduli spaces.

\begin{prop}[Proposition 5.16 of \cite{LOT18}]\label{Proposition, regularity for holomorphic curves in R times Z times [0,1] times R}
If all components of a bi-decorated source $T^\diamond$ are topological disks (possibly with an interior puncture), then $\widetilde{\mathcal{N}}(T^\diamond)$ is transversally cut out for any split almost complex structure on $\mathbb{R}\times Z\times[0,1]\times \mathbb{R}$. 
\end{prop}

The heights of $v\in \widetilde{\mathcal{N}}(T^\diamond)$ at east or west boundary punctures induce evaluation functions $\widetilde{ev}_e$ and $\widetilde{ev}_w$. Given partitions $P_e$ and $P_w$ of the boundary east and west punctures, one defines $\widetilde{\mathcal{N}}(T^\diamond; P_e;P_w)$ in an obvious way. One defines the reduced moduli space $\mathcal{N}$ by taking the quotient of the $\mathbb{R}\times \mathbb{R}$-action induced by translations in both $\mathbb{R}$-directions in $\mathbb{R}\times \mathcal{Z}\times [0,1] \times \mathbb{R}$. The evaluation maps $\widetilde{ev}_e$ and $\widetilde{ev}_w$ also descend to $\mathcal{N}$, taking values in $\mathbb{R}^{|E(T^\diamond)|}/\mathbb{R}$ and $\mathbb{R}^{|W(T^\diamond)|}/\mathbb{R}$, respectively.

 Given $u\in \mathcal{N}(T^\diamond)$, the open mapping theorem implies that the map $\pi_\mathbb{D}\circ u$ is constant on connected components of $T^\diamond$ (taking values in $\{1\}\times \mathbb{R}$). So, the map $u$ is determined by its projection $\pi_\Sigma\circ u$ and $t$-coordinates on connected components of $T$. Of primary interest to us are the following three types of holomorphic curves.
\begin{itemize}
\item[(Join curve)]A \emph{join component} of a bi-decorated source is a topological disk with three boundary punctures, and the punctures are labeled by $(e,\sigma)$, $(w,\sigma_1)$, and $(w,\sigma_2)$ counter-clockwise, where the Reeb chords satisfy the relation $\sigma=\sigma_1\uplus\sigma_2$ (here $\uplus$ denotes concatenation of Reeb chords). A \emph{trivial component} of a bi-decorated source is a topological disk with two boundary punctures, one $w$ puncture and one $e$ puncture, and both are labeled by the same Reeb chord. Holomorphic maps from a join component or a trivial component to $\mathbb{R}\times Z \times [0,1]\times \mathbb{R}$ exist and are unique up to translations. A \emph{join curve} is a holomorphic curve with a bi-decorated source consisting of a single join component and possibly some trivial components.
\item[(Split curve)]A \emph{split component} of a bi-decorated source is a topological disk with three boundary punctures, where the punctures are counter-clockwisely labeled by $(e,\sigma_1)$, $(e,\sigma_2)$ and $(w,\sigma)$ with $\sigma=\sigma_1\uplus\sigma_2$. Holomorphic maps from a split component to $\mathbb{R}\times Z \times [0,1]\times \mathbb{R}$ exist and are unique up to translations. A \emph{split curve} is a holomorphic curve with a bi-decorated source consisting of one or more split components and possibly some trivial components.  
\item[(Orbit curve)]An \emph{orbit component} of a bi-decorated source is a topological disk with a single boundary $w$ puncture labeled by $\sigma\in\{\ -\rho_{0123},-\rho_{1230},-\rho_{2301},-\rho_{3012}\}$ and a single interior $e$ puncture. Holomorphic maps from an orbit component to $\mathbb{R}\times Z \times [0,1]\times \mathbb{R}$ exist and are unique up to translations. An \emph{orbit curve} is holomorphic curve with a bi-decorated source consisting of a single orbit component and possibly some trivial components.
\end{itemize}

\subsubsection{Compactification by Holomorphic combs}\label{subsubsection, holomorphic combs}
We describe holomorphic combs in this subsection. We begin with a description of nodal holomorphic curves. 
\begin{defn}
A \emph{nodal decorated source} $S^\triangleright$ is a decorated source together with a set of unordered pairs of marked points $D=\{\{\overline{d}_1,\underline{d}_1\},\{\overline{d}_2,\underline{d}_2\},\ldots,\{\overline{d}_k,\underline{d}_k\}\}$. The points in $D$ are called nodes. 
\end{defn}

\begin{defn}
Let $\bm{x}$ and $\bm{y}$ be generators and let $B\in\widetilde{\pi}_2(\bm{x},\bm{y})$. Let $S^\triangleright$ be a nodal decorated source. Let $S_i$ be the components of $S\backslash\{\text{nodes}\}$. Then a nodal holomorphic curve $u$ with source $S^\triangleright$ in the homology class of $B$ is a continuous map $$u:(S,\partial S)\rightarrow (\Sigma\times [0,1]\times \mathbb{R}, \beta\times\{0\}\times \mathbb{R}, \alpha \times \{1\} \times \mathbb{R})$$ such that 
\begin{itemize}
\item[(1)] the restriction of $u$ to each $S_i$ is a map satisfying condition (M-1)-(M-11) except for (M-5), 
\item[(2)] $\lim_{p\rightarrow \overline{d}_{i}}u(p)=\lim_{p\rightarrow \underline{d}_{i}}u(p)$ for every pair of nodes, 
\item[(3)] $u$ is asymptotic to $\bm{x}$ at $-\infty$ and $\bm{y}$ at $\infty$, and 
\item[(4)] $u$ induces the homology class specified by $B$.
\end{itemize}
\end{defn}

The nodes in a nodal source $S^\triangleright$ induce punctures on the connected components $S_i$ of $S\backslash\{\text{nodes}\}$; they can be interior punctures as well as boundary punctures.  Note that $u|_{S_i}$ extends across these punctures continuously. We further divide the boundary punctures induced by the nodes into two types.
\begin{defn}\label{Definition, type I type II node}
Let $u$ be a nodal holomorphic curve. Let $d$ be a boundary puncture induced by a node on a component $S_i$ of the nodal Riemann surface. Let $l_1$ and $l_2$ denote the components of $\partial S_i$ adjacent to $d$. If the path $\pi_{\Sigma}\circ u|_{l_1\cup\{d\}\cup l_2}$ is stay-on-track in the sense of (M-11), then we say $d$ is a \emph{type I puncture}; otherwise, we say $d$ is a \emph{type II puncture}.    
\end{defn}

There are only type I punctures when the attaching curves are embedded. Type II punctures naturally appear in our setup since we have an immersed $\alpha$ curve. 

One can still define the evaluation map from the space of nodal holomorphic maps to $\mathbb{R}^{|E(S^\triangleright)|}$, where the value at a nodal holomorphic curve is the heights of the east punctures. 
\begin{defn}
A \emph{holomorphic story} is an ordered $(k+1)$-tuple $(u,v_{1},\ldots,v_{k})$ for some $k\geq 0$ such that 
\begin{itemize}
\item[(1)]$u$ is a (possibly nodal) holomorphic curve in $\Sigma\times [0,1]\times\mathbb{R}$,
\item[(2)]each $v_i$ ($i=1,2,\ldots,k$) is a holomorphic curve in $\mathbb{R}\times \mathcal{Z}\times [0,1]\times \mathbb{R}$,
\item[(3)]the boundary east punctures of $u$ match up with the boundary west punctures of $v_1$ (i.e., the two sets of punctures are identified by a one-to-one map such that both the Reeb-chord labels and the relative heights are the same under this one-to-one correspondence), and
\item[(4)]the boundary east punctures of $v_i$ match up with the boundary west punctures of $v_{i+1}$ for $i=1,2,\ldots,k-1$.
\end{itemize}
\end{defn}

\begin{defn}\label{Definition, holomorphic comb}
Let $N\geq 1$ be an integer, and let $\bm{x}$ and $\bm{y}$ be two generators. A \emph{holomorphic comb} of height $N$ connecting $\bm{x}$ to $\bm{y}$ is a sequence of holomorphic stories $(u_i,v_{i,1},\ldots,v_{i,k_i})$, $i=1,2,\ldots , N$, such that $u_i$ is a (possibly nodal) stable curve in $\mathcal{M}^{B_i}(\bm{x}_i,{\bm{x}_{i+1}};S_i^\triangleright)$ for some generators ${\bm{x}_1},\ldots,{\bm{x}_{N+1}}$ such that $\bm{x}_1=\bm{x}$ and $\bm{x}_{N+1}=\bm{y}$. 
\end{defn}

Given a holomorphic comb, the underlying (nodal) decorated sources and bi-decorated sources can be glued up and deformed in an obvious way to give a smooth decorated source; it is called the \emph{preglued source} of the holomorphic comb.

\begin{defn}
Given generators $\bm{x}$ and $\bm{y}$, and a homology class $B\in \tilde{\pi}_2(\bm{x},\bm{y})$. $\overline{\overline{\mathcal{M}}}^B(\textbf{x},\textbf{y};S^\triangleright)$ is defined to be the space of all holomorphic combs with preglued source $S^\triangleright$, in the homology class of $B$, and connecting $\bm{x}$ to $\bm{y}$. $\overline{\mathcal{M}}^B(\bm{x},\bm{y};S^\triangleright)$ is defined to be the closure of $\mathcal{M}^B(\bm{x},\bm{y};S^\triangleright)$ in $\overline{\overline{\mathcal{M}}}^B(\bm{x},\bm{y};S^\triangleright)$. $\overline{\mathcal{M}}^B(\bm{x},\bm{y};S^\triangleright; P)$ and $\overline{\mathcal{M}}^B(\bm{x},\bm{y};S^\triangleright; \overrightarrow{P})$ are defined to be the closure of $\mathcal{M}^B(\bm{x},\bm{y};S^\triangleright; P)$ and $\mathcal{M}^B(\bm{x},\bm{y};S^\triangleright; \overrightarrow{P}$ in $\overline{\mathcal{M}}^B(\bm{x},\bm{y};S^\triangleright)$ respectively.
\end{defn}

The compactness result is stated below. 
\begin{prop}\label{Proposition, compactness}
The moduli space $\overline{\mathcal{M}}^B(\bm{x},\bm{y};S^\triangleright)$ is compact. The same statement holds for $\overline{\mathcal{M}}^B(\bm{x},\bm{y};S^\triangleright; P)$ and $\overline{\mathcal{M}}^B(\bm{x},\bm{y};S^\triangleright; \overrightarrow{P})$.
\end{prop}
We omit the proof of the above proposition and remark that the proof of \cite[Proposition 5.24]{LOT18} adapts to our setup easily. 

\subsection{Gluing results}
As the regularity results, gluing results in pseudo-holomorphic curve theory are proved by analyzing the $\bar{\partial}$-operator over certain section space of some pull-back bundles over the underlying source surfaces. In particular, having immersed $\alpha$ curves does not affect the proof of such results. We hence recall the following results that we shall need without giving the proof. 

\begin{prop}[Proposition 5.30 of \cite{LOT18}]
Let $(u_1,u_2)$ be a two-story holomorphic building with $u_1\in \mathcal{M}^{B_1}(\bm{x},\bm{y};S_1^\triangleright;P_1)$ and $u_2\in \mathcal{M}^{B_2}(\bm{y},\bm{z};S_2^\triangleright;P_2)$. Assume the moduli spaces are transversally cut out. Then for sufficiently small neighborhood $U_i$ of $u_i$ ($i=1,2$), there is neighborhood of $(u_1,u_2)$ in $\overline{\mathcal{M}}^{B_1+B_2}(\bm{x},\bm{z};S_1^\triangleright\natural S_2^\triangleright; P_1\cup P_2)$ homeomorphic to $U_1\times U_2 \times [0,1)$.
\end{prop}

\begin{defn}
	A holomorphic comb is said to be \emph{simple} if it only has a single story and it is of the form $(u,v_1)$, where $u$ is a non-nodal holomorphic curve.
\end{defn}

\begin{prop}[Proposition 5.31 of \cite{LOT18}]\label{Proposition, gluing holomorphic curves and east ends}
Let $(u,v)$ be a simple holomorphic comb with $u\in \mathcal{M}^{B}(\bm{x},\bm{y};S^\triangleright)$ and $v\in \mathcal{N}(T^\diamond;P_e)$. Let $m$ denote the number of east punctures of $S^\triangleright$. Assume the moduli spaces are transversally cut out, and the evaluation maps $ev:\mathcal{M}^{B}(\bm{x},\bm{y};S^\triangleright)\rightarrow \mathbb{R}^m/\mathbb{R}$ and $ev_w: \mathcal{N}(T^\diamond;P_e)\rightarrow \mathbb{R}^m/\mathbb{R}$ are transverse at $(u,v)$. Then for sufficiently small neighborhood $U_u$ of $u$ and $U_v$ of $v$, there is a neighborhood of $(u,v)$ in $\overline{\mathcal{M}}^{B}(\bm{x},\bm{y};S^\triangleright\natural T^\diamond; P_e)$ homeomorphic to $U_u\times_{ev} U_v \times [0,1)$.
\end{prop}

\subsection{Degeneration of moduli spaces}\label{Section, degeneration of moduli space}
This subsection provides constraints on the degeneration in 1-dimensional moduli spaces using the index formulas and strong boundary monotonicity. The results here are simpler than the corresponding results in Section 5.6 of \cite{LOT18} since we restrict to the torus-boundary case. However, as mentioned earlier, nodal curves with corners at self-intersection points of $\alpha_{im}$ may occur in the compactification; we defer the further analysis of these degenerations to later subsections. Readers who wish to skip the details in this subsection are referred to Definition \ref{Definition, boundary degeneration} and Proposition \ref{Proposition, summary of degeneration of moduli spaces} for a quick summary.

The index formulas in Proposition \ref{Proposition, first index formula} lead to the following constraints on the ends of moduli spaces of 0-P curves.
\begin{prop}[cf.\ Proposition 5.43 of \cite{LOT18}]\label{Proposition, first constraint on the boundary of 1-dimensional moduli space}
Let $B\in \tilde{\pi}_2(\bm{x},\bm{y})$, let $S^\triangleright$ be a decorated source of 0-P, and let $P$ be a discrete partition of the east punctures of $S^\triangleright$. Suppose that $Ind(B,S^{\triangleright},P)=2$. Then for a generic almost complex structure $J$,
every holomorphic comb in $\partial\overline{\mathcal{M}}^{B}(\bm{x},\bm{y};S^{\triangleright};P)=\overline{\mathcal{M}}^{B}(\bm{x},\bm{y};S^{\triangleright};P)-\mathcal{M}^{B}(\bm{x},\bm{y};S^{\triangleright};P)$ has one of the following forms:
\begin{itemize}
\item[(1)]a two-story holomorphic building $(u_1, u_2)$;
\item[(2)]a simple holomorphic comb $(u; v)$ where $v$ is a join curve;
\item[(3)]a simple holomorphic comb $(u; v)$ where $v$ is a split curve with a single split component; 
\item[(4)]a nodal holomorphic comb, obtained by degenerating some arcs with ends on $\partial S$.
\end{itemize}
\end{prop}

\begin{proof}
This proposition is proved the same way as Proposition 5.43 of \cite{LOT18}: It is a consequence of compactness, transversality, index formula, and gluing results. Note that our statement is simpler as we restrict to a discrete partition $P$: the shuffle-curve end and the multi-component split-curve end that appear in \cite{LOT18} do not occur in our setting. 
\end{proof}
For moduli spaces of 1-P curves, we have the following proposition.
\begin{prop}\label{Proposition, first constraints on degeneration of 1-P curves}
Let $B\in \tilde{\pi}_2(\bm{x},\bm{y})$ and let $S^\triangleright$ be a decorated surface of type 1-P. Suppose $Ind(B,S^\triangleright)=2$. Then for a generic almost complex structure $J$, every holomorphic comb in $\partial\overline{\mathcal{M}}^{B}(\bm{x},\bm{y};S^{\triangleright};U)$ has one of the following forms:
\begin{itemize}
\item[(1)]a two-story holomorphic building $(u_1,u_2)$;
\item[(2)]a simple holomorphic comb $(u; v)$ where $v$ is an orbit curve;
\item[(3)]a nodal holomorphic comb, obtained by degenerating some arcs with boundary on $\partial S$.
\end{itemize}
\end{prop}

\begin{proof}
Suppose a given holomorphic comb in $\partial\overline{\mathcal{M}}^{B}(\bm{x},\bm{y};S^{\triangleright};U)$ is not nodal, then it possibly has degeneration at the east infinity and level splittings. The form of such holomorphic combs is analyzed in the Proof of Proposition 42 in \cite{Hanselman2022}; the results are precisely item (1) and (2) in the above statement. 
\end{proof}

We will only be interested in those moduli spaces prescribed by an ordered discrete partition $\overrightarrow{P}$; this is a void requirement for 1-P holomorphic curves. This condition is also automatic for 0-P; it follows easily from boundary monotonicity and holomorphicity (see, e.g., Lemma 5.51 of \cite{LOT18}).

The strong boundary monotonicity imposes further constraints on the degeneration of one-dimensional moduli spaces. We first describe the constraints on nodal holomorphic curves. Recall there are three types of nodal holomorphic curves.

\begin{defn}\label{Definition, boundary degeneration}
A nodal holomorphic comb $u$ is called a \emph{boundary degeneration} if it has an irreducible component $S_0$ that contains no $\pm$-punctures and $\pi_{\Sigma}\circ u|_{S_0}$ is non-constant. 
\end{defn}
\begin{defn}
A \emph{boundary double point} is a holomorphic comb with a boundary node $p$ such that the projection to $[0,1]\times \mathbb{R}$ is not constant near either preimage point $p_1$ or $p_2$ of $p$ in the normalization of the nodal curve.
\end{defn}
\begin{defn}
A holomorphic comb $u$ is called \emph{haunted} if there is a component $S_0$ of the source such that $u|_{S_0}$ is constant.
\end{defn}

\begin{prop}\label{Proposition, no boundary double point or haunted curve}
Let $J$ be a generic almost complex structure. For one-dimensional moduli spaces $\mathcal{M}^B(\bm{x},\bm{y};S_0^\triangleright;\overrightarrow{P})$ of 0-P curves and for one-dimensional moduli spaces $\mathcal{M}^B(\bm{x},\bm{y};S_1^\triangleright;U)$ of 1-P curves, boundary double points and haunted curves do not appear in $\overline{\mathcal{M}}^B(\bm{x},\bm{y};S_0^\triangleright;\overrightarrow{P})$ and $\overline{\mathcal{M}}^B(\bm{x},\bm{y};S_1^\triangleright;U)$.
\end{prop}

\begin{proof}
This is Lemma 5.56 and Lemma 5.57 of \cite{LOT18}.
\end{proof}

In summary, the only nodal holomorphic combs that could possibly appear are boundary degenerations. We defer the analysis of such degenerations to later subsections. Instead, we conclude this subsection with some further constraints in $\overline{\mathcal{M}}^B(\bm{x},\bm{y};S_0^\triangleright;\overrightarrow{P})$ obtained by combining the strong boundary monotonicity and the torus-boundary condition. 

\begin{prop}\label{Proposition, no join curve end}
For one-dimensional moduli spaces $\mathcal{M}^B(\bm{x},\bm{y};S^\triangleright;\overrightarrow{P})$ of 0-P curves, join curve ends do not appear in $\overline{\mathcal{M}}^B(\bm{x},\bm{y};S^\triangleright;\overrightarrow{P})$.
\end{prop}

\begin{proof}
Suppose this is not true. Write $\overrightarrow{P}=(\sigma_1,\ldots,\sigma_k)$ for some $k\geq 1$. The appearance of join curve end means there is a holomorphic comb $(u;v)$ such that $u\in \mathcal{M}^B(\bm{x},\bm{y};{S^\triangleright}';(\sigma_1,\ldots,\{\sigma_i',\sigma_i''\},\ldots,\sigma_k))$ where $\sigma_i=\sigma_i'\uplus\sigma_i''$. The strong boundary monotonicity condition implies $\bm{x}$ has one and only one component that lies on exactly one of the $\alpha$ arcs. Hence all the east punctures of $u$ are of different heights due to holomorphicity. This contradicts that the east punctures marked by $\sigma_i'$ and $\sigma_i''$ are of the same height. 
\end{proof}

\begin{prop}\label{Proposition, no collision of levels}
Let $\partial\overline{\mathcal{M}}^B(\bm{x},\bm{y};S^\triangleright;\overrightarrow{P})=\overline{\mathcal{M}}^B(\bm{x},\bm{y};S^\triangleright;\overrightarrow{P})-\mathcal{M}^B(\bm{x},\bm{y};S^\triangleright;\overrightarrow{P})$. Then $\partial\overline{\mathcal{M}}^B(\bm{x},\bm{y};S^\triangleright;\overrightarrow{P})\cap \mathcal{M}^B(\bm{x},\bm{y};S^\triangleright;P)=\emptyset$.
\end{prop}

\begin{proof}
If not, then collision of levels appears, i.e., there is a sequence of $u_i\in\mathcal{M}^B(\bm{x},\bm{y};S^\triangleright;\overrightarrow{P})$ converging to a holomorphic curve $u\in \mathcal{M}^B(\bm{x},\bm{y};S^\triangleright;P)$ such that at least two of the east punctures are of the same height. However, we already observed such $u$ does not exist in the proof of Proposition \ref{Proposition, no join curve end}. In other words, $\mathcal{M}^B(\bm{x},\bm{y};S^\triangleright;P)$ is equal to $\mathcal{M}^B(\bm{x},\bm{y};S^\triangleright;\overrightarrow{P})$ for a particular order determined by $S^\triangleright$ when we restrict to the torus-boundary case.
\end{proof}

We summarize the results above into the following proposition for convenience.
\begin{prop}\label{Proposition, summary of degeneration of moduli spaces}
Let $B\in \tilde{\pi}_2(\bm{x},\bm{y})$ and let $J$ be a generic almost complex structure. For a one-dimensional moduli space $\mathcal{M}^{B}(\bm{x},\bm{y};S^{\triangleright};\overrightarrow{P})$ of 0-P curves, where $\overrightarrow{P}$ is a discrete ordered partition of the east punctures of $S^\triangleright$, every holomorphic comb in $\partial\overline{\mathcal{M}}^{B}(\bm{x},\bm{y};S^{\triangleright};\overrightarrow{P})$ has one of the following forms:
\begin{itemize}
\item[(1)]a two-story holomorphic building $(u_1; u_2)$;
\item[(2)]a simple holomorphic comb $(u; v)$ where $v$ is a split curve with a single split component; 
\item[(3)]a boundary degeneration.
\end{itemize}

For a one-dimensional moduli space $\mathcal{M}^{B}(\bm{x},\bm{y};S^{\triangleright};U)$ of 1-P curves, every holomorphic comb in $\partial\overline{\mathcal{M}}^{B}(\bm{x},\bm{y};S^{\triangleright};U)$ has one of the following forms:
\begin{itemize}
\item[(1)]a two-story holomorphic building $(u_1,u_2)$;
\item[(2)]a simple holomorphic comb $(u; v)$ where $v$ is an orbit curve;
\item[(3)]a boundary degeneration.
\end{itemize}
\end{prop}

\begin{rmk}
The restriction to manifolds with torus boundary greatly simplifies the study of moduli spaces. In higher-genus cases, join ends, collision of levels, and split curve with many splitting components may appear. In particular, the latter prevents one from proving the compactified moduli spaces are manifolds with boundaries (as the gluing results fail to apply). In \cite{LOT18}, notions like smeared neighborhood, cropped moduli spaces, and formal ends were employed to circumvent this difficulty. The results in this subsection allow us to avoid introducing any of these terms. 
\end{rmk}
\subsection{Embedded holomorphic curves}\label{Section, embedded index formula}
We only use embedded holomorphic curves when defining (weakly extended) type D structures. When $\alpha$-curves are embedded, a holomorphic curve is embedded if and only if the Euler characteristic of its underlying source surface is equal to the one given by an \emph{embedded Euler characteristic formula}. This is still true in the current setup, but the embedded Euler characteristic formula needs to be generalized to take care of new phenomena caused by immersed $\alpha$-multicurves. 

The embedded Euler characteristic formula involves signs of self-intersections of oriented immersed arcs in our setup. Here we fix the sign convention. Let $f:(0,1)\rightarrow \Sigma$ be an immersed arc with transverse double self-intersection at $p=f(a)=f(b)$ with $0<a<b<1$. Then the sign of the intersection point at $p$ is positive if the orientation at $T_p\Sigma$ coincides with the one specified by the ordered pair $(f'(a),f'(b))$; otherwise, the sign of the intersection point is negative. Let $f\cdot f$ denote the signed count of self-intersection points of the arc $f$. Let $B\in\pi_2(\bm{x},\bm{y};\overrightarrow{P})$ be a domain. Then up to homotopy within $\alpha_{im}$, the boundary of $B$ at $\alpha_{im}$ gives rise to a curve $\partial_{\alpha_{im}}B$. We will restrict to those domains for which $\partial_{\alpha_{im}} B$ has a single component. We define $s(\partial_{\alpha_{im}}B)$ to be $(\partial_{\alpha_{im}} B) \cdot (\partial_{\alpha_{im}} B)$, the signed count of transverse double points in $\partial_{\alpha_{im}}B$.  

For convenience, we also define the \emph{length} of each elementary Reeb chord to be one and the \emph{length} $|\sigma|$ of a general Reeb chord $\sigma$ to be the number of elementary Reeb chords it consists of. Note if $\overrightarrow{\rho}$ is a sequence of Reeb chords appearing as the east boundary of a holomorphic curve, then by Condition (M-9) we know for any $\sigma \in \overrightarrow{\rho}$, $|\sigma|\leq 4$.

With these terminologies set, our proposition is the following. 

\begin{prop}\label{Proposition, embedded Euler characteristic}
A holomorphic curve $u\in \mathcal{M}^B(\bm{x},\bm{y};S^\triangleright,\overrightarrow{P})$ is embedded if and only if $\chi(S)=g+e(B)-n_{\bm{x}}(B)-n_{\bm{y}}(B)-\iota([\overrightarrow{P}])+s(\partial_{\alpha_{im}}B)$.
\end{prop}

Here, $S^\triangleright$ can either be a 0-P source or a 1-P source. If $S^\triangleright$ is of 0-P, then $[\overrightarrow{P}]$ stands for the sequence of Reeb chords obtained from the labels of the east punctures, and the term $\iota([\overrightarrow{P}])$ is defined in Formula (5.65) of \cite{LOT18} (in the non-extended case) and Section 4.1 of \cite{Lipshitz2021} (for Reeb chords of length $4$). In particular, if $\overrightarrow{P}$ contains a Reeb chord of length $4$, then it consists of a single Reeb chord by (M-9), and hence the only new case we need to know is $\iota((\sigma))=-1$ when $|\sigma|=4$. If $S^\triangleright$ is of 1-P, then by Condition (M-9) $\overrightarrow{P}=\emptyset$ and the term $\iota([\overrightarrow{P}])$ vanishes.

\begin{proof}[Proof of Proposition \ref{Proposition, embedded Euler characteristic}]
The proof is adapted from the corresponding proof in the embedded alpha curve case, \cite[Proposition 5.69]{LOT18}. To keep it concise, we state the main steps and skip the details that can be found in \cite{LOT18}, but we give details on the modifications. The proof is divided into four steps; only the third step is signicantly different from the embedded case.

Let $u\in \mathcal{M}^B(\bm{x},\bm{y};S^\triangleright;\overrightarrow{P})$ be an embedded curve.

\textbf{Step 1.} Apply the Riemann-Hurwitz formula to express $\chi(S)$ in terms of $e(B)$ and $br(u)$, where $br(u)$ stands for the total number of branch points of the projection $\pi_\Sigma\circ u$, counted with multiplicities: 
\begin{equation}\label{Equation, Euler char of sources of embedded curve}
\chi(S)=e(S)+\frac{g}{2}+\sum_{i} \frac{|P_i|}{2}=e(B)-br(u)+\frac{g}{2}+\sum_{i} \frac{|P_i|}{2}.
\end{equation}

\textbf{Step 2.} Let $\tau_\epsilon(u)$ be the translation of $u$ by $\epsilon$ in the $\mathbb{R}$-direction. Since we are in the torus-boundary case, we may always assume $\overrightarrow{P}$ is discrete. There is only one case where there is a branch point escaped to east infinity: This occurs when $\overrightarrow{P}=\{\sigma\}$ for some $|\sigma|=4$. Therefore, 
\begin{equation}\label{Equation, branch points of holomorphic curve projection}
br(u)=\tau_\epsilon(u)\cdot u-\frac{\vert\{\sigma|{\sigma}\in \overrightarrow{P},\ |\sigma|=4\}\vert}{2} 
\end{equation}
when $\epsilon$ is sufficiently small. This is because when $u$ is embedded and $\epsilon$ is small, excluding those intersection points which escape to east infinity when $\epsilon\rightarrow0$, the remaining intersection points correspond to points at which $u$ is tangent to $\frac{\partial}{\partial t}$, which are exactly the branch points of $\pi_{\Sigma}\circ u$. 

\textbf{Step 3.} Compute $\tau_\epsilon(u)\cdot u$. For all sufficiently large $R$, 
\begin{equation}\label{Equation, intersection number, large R}
\tau_R(u)\cdot u=n_x(B)+n_y(B)-\frac{g}{2}.
\end{equation}

To see this, note that the intersection of $u$ and $\tau_R(u)$ looks like the intersection of $u$ with the trivial disk from the generator $\bm{x}$ to itself when $t$ is small and it looks like the intersection of $u$ with the trivial disk corresponding to the generator $\bm{y}$ when $t$ is large (see Proposition 4.2 of \cite{MR2240908} for the computation).

To recover $\tau_\epsilon(u)\cdot u$, we need to understand how $\tau_t(u)\cdot u$ changes as $t$ varies. In the closed-manifold and embedded-alpha-curve case, $\tau_t(u)\cdot u$ does not change. In the bordered-manifold case, there is a change in $\tau_\epsilon(u)\cdot u-\tau_R(u)\cdot u$ caused by the relative position change of $\partial u$ and $\partial \tau_t(u)$ at east infinity. This change is captured by the term $\iota([\overrightarrow{P}])$. More precisely, when the $\alpha$-curves are embedded, $$\tau_\epsilon(u)\cdot u-\tau_R(u)\cdot u=\iota([\overrightarrow{P}])+\frac{|\overrightarrow{P}|}{2}+\frac{\vert\{\sigma|{\sigma}\in \overrightarrow{P},\ |\sigma|=4\}\vert}{2}.$$ 

When the $\alpha$-curves are immersed, there is another source of change in $\tau_t(u)\cdot u$ corresponding to self-intersection points of $\partial_{\alpha_{im}}u$. We explain such change in the examples below. As one shall see, such phenomenon is local (i.e., depends only on the behavior of $u$ near the pre-images of self-intersection points of $\partial_{\alpha_{im}}u$). Therefore, the examples also speak for the general situation. 
\begin{figure}[htb!]
\centering{
\includegraphics[scale=0.5]{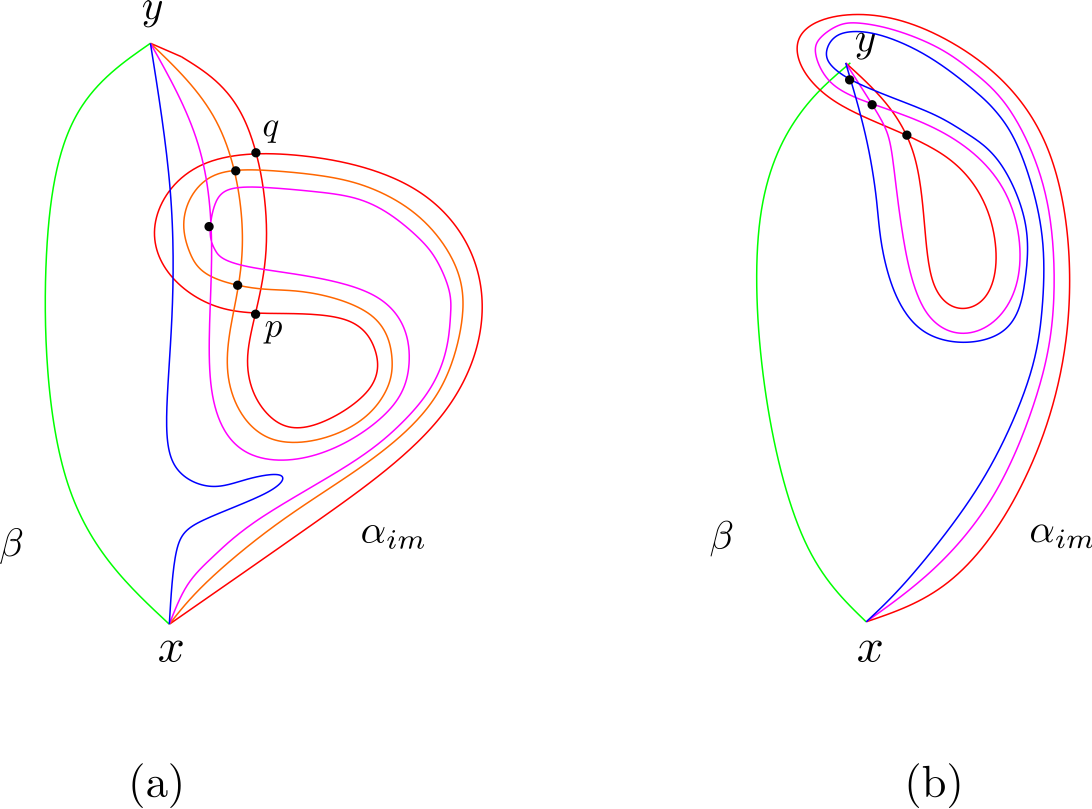}
\caption{Examples of embedded disks in $\Sigma\times[0,1]\times\mathbb{R}$. The $\alpha$-curve is drawn in red, and the $\beta$-curves are in green. We also depicted a few $s$-grid lines; points on a single $s$-grid line have constant $s$-coordinate.}
\label{Figure, example embedded_index}
}
\end{figure}
We spell out the example. Consider an embedded disk $u$ in $\Sigma\times [0,1] \times \mathbb{R}$ as shown in Figure \ref{Figure, example embedded_index} (a). In the figure, $\pi_\Sigma\circ u$ is shown by its domain, and we also introduced immersed $s$-grid lines to help visualize the $[0,1]\times\mathbb{R}$-coordinates: the points on a single immersed $s$-grid line are all mapped to the same value in $[0,1]$, and as we move from $x$ to $y$, the $t$-coordinate on an $s$-grid line increases from $-\infty$ to $\infty$. In the figure, we highlighted a few $s$-grid lines and self-intersection points on them. Despite having these self-intersection points on the $s$-grid lines, $u$ is still embedded as each such self-intersection point corresponds to two points in $\Sigma\times [0,1] \times \mathbb{R}$ with different $t$-coordinate. However, each self-intersection point gives rise to an intersection point in $\tau_{t}u\cdot u$ for an appropriate $t$. In our example, the projection to $\Sigma$ of $\partial_{\alpha} u$ has two self-intersection points $p$ and $q$. There are $t_i\in \mathbb{R}$ ($i=1,2,3,4$), $t_1<t_2<t_3<t_4$ so that $p$ has $t$-coordinate $t_2$ and $t_3$, and $q$ has coordinate $t_1$ and $t_4$. We first examine what is incurred by the negative intersection point $p$. Note $\tau_t u\cdot u$ does not change for $t<t_3-t_2$, and at $t=t_3-t_2$, $\tau_t u\cap u$ picks up a boundary intersection point $(p,1,t_3)$. Inspecting the example, we can see for $t\in[t_3-t_2-\epsilon,t_3-t_2+\epsilon]$ for a small $\epsilon$, a boundary intersection point appears and then enters to become an interior intersection point. So $\tau_{t_3-t_2+\epsilon}u \cdot u-\tau_{t_3-t_2-\epsilon}u \cdot u=1$. Similarly, for the positive self-intersection point $q$, for $t\in[t_4-t_1-\epsilon,t_4-t_1]$, we see an interior intersection point of $\tau_t u\cdot u$ hits the boundary and then disappears.

In general, an intersection point $p$ of $\partial_{\alpha_{im}} u$ contributes a boundary intersection of $u$ with $\tau_t u$ for some value of $t= t_p$, and intersection points of arcs for $s$-values just less than one contribute intersections of  $u$ with $\tau_t u$ for nearby values of $t$. These intersections occur at shifts $t > t_p$ if $p$ is a positive intersection point and at $t < t_p$ if $p$ is a negative intersection point, so positive intersection points always give rise to times $t$ at which an intersection point appears on the boundary and then moves into the interior, while negative intersection points give rise to times $t$ at which an interior intersection moves to the boundary and disappears. Thus the net change in $\tau_t u\cdot u$ caused by negative and positive boundary intersection points is given by $s(\partial_{\alpha_{im}}u)$. In the example above from Figure \ref{Figure, example embedded_index}(a), $s(\partial_{\alpha_{im}}u)=0$. An example with $s(\partial_{\alpha_{im}}u) = 1$ is shown in Figure \ref{Figure, example embedded_index} (b). 

Overall, taking into account the changes at the east infinity and at $\alpha_{im}\times\{1\}\times\mathbb{R}$, we can derive the following equation from Equation (\ref{Equation, intersection number, large R}):
\begin{equation}\label{Equation, self-intersection of embedded curves}
\tau_\epsilon(u)\cdot u=n_x(B)+n_y(B)-\frac{g}{2}+\iota([\overrightarrow{P}])+\frac{|\overrightarrow{P}|}{2}+\frac{\vert\{\sigma|{\sigma}\in \overrightarrow{P},\ |\sigma|=4\}\vert}{2}+s(\partial_{\alpha_{im}}B)
\end{equation}

\textbf{Step 4.} Synthesize the steps. Synthesizing Equation (\ref{Equation, Euler char of sources of embedded curve}), Equation (\ref{Equation, branch points of holomorphic curve projection}), and Equation (\ref{Equation, self-intersection of embedded curves}) gives the formula, proving the ``only if" direction. To see the ``if" direction, note if the holomorphic curve is not embedded, then in Step 2 we know $br(u)=\tau_{\epsilon}u\cdot u-2 \text{sigu}(u)$, where $\text{sigu}(u)>0$ is the order of singularity $u$. So, in this case we have $\chi(S)$ is strictly greater than $g+e(B)-n_x(B)-n_y(B)-\iota([\overrightarrow{P}])+s(\partial_{\alpha_{im}}B).$
\end{proof}

From now on, we use $\overrightarrow{\rho}$ to denote either a sequence of Reeb chords of length less than or equal to $4$ or the length one sequence containing a single closed Reeb orbit $U$ which is used for marking the interior puncture of a 1-P curve. Given such a $\overrightarrow{\rho}$, denote by $\overrightarrow{{\rho}_{\star}}$ the sub-sequence of non-closed Reeb chords, i.e., if $\overrightarrow{\rho}=(U)$ then $\overrightarrow{{\rho}_{\star}}=\emptyset$ and otherwise $\overrightarrow{{\rho}_{\star}} = \overrightarrow{\rho}$.  
A domain $B\in\tilde{\pi}_2(\bm{x},\bm{y})$ is said to be \emph{compatible} with a sequence of Reeb chords $\overrightarrow{\rho}$ if the \emph{homology class} induced by the east boundary $\partial^\partial B$ of $B$ agrees with that induced by $\overrightarrow{\rho}$, and $(\bm{x},\overrightarrow{{\rho}_{\star}})$ is strongly boundary monotone (in the sense of Definition 5.52 of \cite{LOT18}). In view of Proposition \ref{Proposition, embedded Euler characteristic}, we make the following definitions.
\begin{defn}\label{Definition, embedded Euler Char, index, and moduli space}
Given $B\in \tilde{\pi}_2(\bm{x},\bm{y})$ and a sequence of Reeb chords $\overrightarrow{\rho}$ such that $(B,\overrightarrow{\rho})$ is compatible. The \emph{embedded Euler characteristic} is defined to be $$\chi_{emb}(B,\overrightarrow{\rho}):=g+e(B)-n_x(B)-n_y(B)-\iota(\overrightarrow{{\rho}_{\star}})+s(\partial_{\alpha_{im}}B).$$
The \emph{embedded index} is defined to be $$\text{ind}(B,\overrightarrow{\rho}):=e(B)+n_x(B)+n_y(B)+|\overrightarrow{\rho}|+\iota(\overrightarrow{{\rho}_{\star}})-s(\partial_{\alpha_{im}}B).$$
The \emph{embedded moduli space} is defined to be $$\mathcal{M}^B(\bm{x},\bm{y};\overrightarrow{\rho}):=\bigcup_{\chi(S)=\chi_{emb}(B,\overrightarrow{\rho}),\  [\overrightarrow{P}]=\overrightarrow{{\rho}_{\star}}}\mathcal{M}^B(\bm{x},\bm{y};S^{\triangleright};\overrightarrow{P}).$$
\end{defn}

Clearly, the embedded moduli space $\mathcal{M}^B(\bm{x},\bm{y};\overrightarrow{\rho})$ has expected dimension $\text{ind}(B,\rho)-1$ by Proposition \ref{Proposition, embedded Euler characteristic} and Proposition \ref{Proposition, first index formula}.

\begin{prop}\label{Proposition, embedded index formula is additive}
The embedded index formula is additive, i.e., for compatible pairs $(B_i,\overrightarrow{\rho_i})$ where $B_1\in\pi_2(\bm{x},\bm{w})$ and $B_2\in \pi_2(\bm{w},\bm{y})$ we have $\text{ind}(B_1+B_2,(\overrightarrow{\rho_1},\overrightarrow{\rho_2}))=\text{ind}(B_1,\overrightarrow{\rho_1})+\text{ind}(B_2,\overrightarrow{\rho_2}).$
\end{prop}
\begin{proof}
The proof is a modification of the proof of Proposition 5.75 in \cite{LOT18}, which is adapted from the proof in the closed-manifold case in \cite{Sarkar2011}. We will only mention the modifications instead of going into the details. Given oriented arcs $a$ and $b$, the \emph{jittered intersection number} $a\cdot b$ is defined to be $\frac{1}{4}(a_{NE}+a_{NW}+a_{SE}+a_{SW})\cdot b$, where $a_{NE},\ldots ,a_{SW}$ are slight translations of the arc $a$ in directions suggested by the subscript (and we assume these translations intersect $b$ transversely). The proof of Proposition 5.75 of \cite{LOT18} uses that if $a$ and $a'$ are two arcs contained in the $\alpha$-curve, then $a\cdot a'=0$ (Lemma 5.73 of \cite{LOT18} (2)). This is no longer true in our setting. Overall\footnote{For careful readers: Specifically, Lemma 5.73 (5) of \cite{LOT18} needs to be changed to $a\cdot a'+a\cdot b'+b\cdot a'=0$ and Lemma 5.74 of \cite{LOT18} needs to be changed to $a\cdot a'+a\cdot b'+b\cdot a'=L(\partial^\partial B,\partial^\partial B')$ when $a\cdot a'$ can no longer be assumed to be zero.}, running the proof in \cite{LOT18} in our setting now gives 

\begin{equation}\label{Equation, first equation in the proof of additivity of index}
\begin{aligned}
&e(B_1+B_2)+n_{\bm{x}}(B_1+B_2)+n_{\bm{y}}(B_1+B_2)+|(\overrightarrow{\rho_1},\overrightarrow{\rho_2})|+\iota((\overrightarrow{\rho_1},\overrightarrow{\rho_2})_b)\\
=&e(B_1)+n_{\bm{x}}(B_1)+n_{\bm{y}}(B_1)+|\overrightarrow{\rho_1}|+\iota((\overrightarrow{\rho_1})_b)+e(B_2)+n_{\bm{x}}(B_2)\\
&+n_{\bm{y}}(B_2)+|\overrightarrow{\rho_2}|+\iota((\overrightarrow{\rho_2})_b)+\partial_{\alpha_{im}}B_1\cdot \partial_{\alpha_{im}}B_2.
\end{aligned}
\end{equation}
(If the $\alpha$-curves are embedded, then the term $\partial_{\alpha_{im}}(B_1)\cdot \partial_{\alpha_{im}}B_2$ vanishes, recovering the additivity of the index in that case.) And it is easy to see 
\begin{equation}\label{Equation, second equation in the proof of additivity of index}
s(\partial_{\alpha_{im}}(B_1+B_2))=s(\partial_{\alpha_{im}}B_1)+s(\partial_{\alpha_{im}}B_2)+\partial_{\alpha_{im}}B_1\cdot \partial_{\alpha_{im}}B_2.
\end{equation}
Now the additivity of the index follows readily from the definition of the embedded index, Equation (\ref{Equation, first equation in the proof of additivity of index}), and Equation (\ref{Equation, second equation in the proof of additivity of index}).
\end{proof}
 When proving the structure equations (e.g., $\partial^2=0$ for type D structures), one needs to relate the coefficient of the structural equation and the ends of 1-dimensional moduli spaces. A key result that allows us to do this is the following proposition, which is the counterpart of Lemma 5.76 of \cite{LOT18}. 

\begin{prop}\label{Proposition, levels of degeneration are embedded}
Given $(B,\overrightarrow{\rho})$ such that $\text{ind}(B,\overrightarrow{\rho})=2$. Then the two-story buildings that occur in the degeneration of $\overline{\mathcal{M}}^B(\bm{x},\bm{y};\overrightarrow{\rho})$ are embedded.  
\end{prop}
\begin{proof}
Given a two-story building $(u_1,u_2)$, by transversality we know $\text{ind}(u_i)=1$, $i=1,2$. Let $(B_i,\overrightarrow{\rho_i})$ be the corresponding pair of domain and Reeb chords of $u_i$. In view of Proposition \ref{Proposition, embedded Euler characteristic}, $\text{ind}(u_i)\leq \text{ind}(B_i,\overrightarrow{\rho_i})$, and the equality is achieved if and only if $u_i$ is embedded. Now by Proposition \ref{Proposition, embedded index formula is additive}, if some $u_i$ is not embedded, we will have $\text{ind}(B,\overrightarrow{\rho})=\text{ind}(B_1,\overrightarrow{\rho_1})+\text{ind}(B_2,\overrightarrow{\rho_2})>2,$ which contradicts our assumption.
\end{proof}

\subsection{Ends of moduli spaces of 0-P curves}\label{Section, ends of moduli spaces of 0-p curves}
We analyze boundary degenerations of moduli spaces of 0-P curves in this subsection, which was left untouched in Section \ref{Section, degeneration of moduli space}. We will separate the discussion into two cases. The first case is when $n_z(B)=0$, needed for defining the hat-version type D structure. The second case is when $n_z(B)=1$, needed for defining extended type D structures. Readers who wish to skip the details are referred to Proposition \ref{Proposition, ends of 0-p curves, n_z=0}, Proposition \ref{Propostion, bondary degeneration do not occur when one of the Reeb chord is not covered}, and Proposition \ref{Proposition, ends of 1-dimensional moduli spaces of 0-P curves, Reeb element of length 4} for the statement of the results. 

\subsubsection{Ends of $\mathcal{M}^B$ when $n_z(B)=0$}
\begin{prop}\label{Proposition, boundary degeneration do not occur when n_z=0}
If $n_z(B)=0$, then $\overline{\mathcal{M}}^B(\bm{x},\bm{y};\overrightarrow{\rho})$ does not have boundary degeneration. 
\end{prop}
\begin{proof}
Suppose boundary degeneration appears. Let $v$ denote the (union of the) component(s) of the nodal holomorphic curve with no $\pm$-punctures. The union of components that have $\pm$-punctures will be called the \textit{main component}. We say that the boundary degeneration has corners if at least one of the nodes connecting $v$ and the main component is of type II. Otherwise, a boundary degeneration is said to have no corners. (See Definition \ref{Definition, type I type II node} for type II nodes.)  

Let $B_v$ denote the domain corresponding to $v$. If the boundary degeneration has no corners, then $B_v$ is a positive zero-cornered $\alpha$-bounded domain. Such $B_v$ does not exist as $\mathcal{H}$ is unobstructed. This observation left us with the possibility of boundary degeneration with corners. If such degeneration appears, then each type II node connecting the main component and the degenerated components is obtained from pinching a properly embedded arc on the original source surface $S$, whose endpoints are on the component of $\partial S$ that is mapped to $\alpha_{im}\times[0,1]\times\mathbb{R}$. Therefore, $B_v$ is a (union of) positive one-cornered $\alpha$-bounded domain with $n_z=0$. Such domains do not exist either since $\mathcal{H}$ is unobstructed. Therefore, there is no boundary degeneration.
\end{proof}
\begin{prop}\label{Proposition, ends of 0-p curves, n_z=0}
For a generic almost complex structure $J$. Let $\mathcal{M}^B(\bm{x},\bm{y};\overrightarrow{\rho})$ be a one-dimensional moduli space with $n_z(B)=0$ and $a(-\overrightarrow{\rho})\neq 0$. Then $\overline{\mathcal{M}}^B(\bm{x},\bm{y};\overrightarrow{\rho})$ is a compact 1-manifold with boundary such that all the boundary points correspond to two-story embedded holomorphic buildings. 
\end{prop}

\begin{proof}
It follows from Proposition \ref{Proposition, summary of degeneration of moduli spaces} and Proposition \ref{Proposition, boundary degeneration do not occur when n_z=0} that the only elements that can appear in $\partial\overline{\mathcal{M}}_0^B(\bm{x},\bm{y};\overrightarrow{P})$ are split curves or two-story buildings. However, if split curves appear in the torus-boundary case and $n_z(B)=0$, a simple examination of the torus algebra implies we will always have $a(-\overrightarrow{\rho})=0$. Therefore, the only degenerations that can appear are two-story buildings. The statement that $\overline{\mathcal{M}}_0^B(\bm{x},\bm{y};\overrightarrow{P})$ is a 1-manifold with boundary then follows from the gluing results.
\end{proof}

\subsubsection{Ends of $\mathcal{M}^B$ when $n_z(B)=1$}
We separate the discussion into two sub-cases. First, we consider the sub-case where $n_{-\rho_i}(B)=0$ for at least one of $\rho_i\in \{\rho_1,\rho_2,\rho_3\}$, in which we have the following proposition.
\begin{prop}\label{Propostion, bondary degeneration do not occur when one of the Reeb chord is not covered}
If $n_z(B)=1$ and $n_{-\rho_i}(B)=0$ for at least one $\rho_i\in \{\rho_1,\rho_2,\rho_3\}$, then $\overline{\mathcal{M}}^B(\bm{x},\bm{y};\overrightarrow{\rho})$ does not have boundary degeneration. In particular, if $a(-\overrightarrow{\rho})\neq 0$, then $\overline{\mathcal{M}}^B(\bm{x},\bm{y};\overrightarrow{\rho})$ is a compact 1-manifold with boundary such that all the boundary points correspond to two-story embedded holomorphic buildings. 
\end{prop}
\begin{proof}
The proof for the first part of this is similar to that of Proposition \ref{Proposition, boundary degeneration do not occur when n_z=0}, which reduces to excluding the existence of positive zero- and one-cornered $\alpha$-bounded domains $B_v$. The existence of such a $B_v$ with $n_z(B)=0$ is excluded by unobstructedness; see Condition (1) of Definition \ref{Definition, unobstructedness}. The existence of such $B_v$ with $n_z(B)=1$ is excluded by Condition (2) and (3) of Definition \ref{Definition, unobstructedness}: If it exists, then the multiplicity of $B_v$ at all the Reeb chords are equal to one, which contradicts our assumption that $n_{-\rho_i}(B)=0$ for some $\rho_i$. The proof for the second part is identical to that of Proposition \ref{Proposition, ends of 0-p curves, n_z=0}.
\end{proof}

Next, we consider the case where the multiplicity of $B$ at all of the four regions next to the east puncture of $\Sigma$ is equal to one. Let $q$ be a self-intersection point of $\alpha_{im}$. We use $T(q)$ to denote the set of stabilized teardrops with an acute corner at $q$. We will use $\mathcal{M}^{B}(\bm{x},\bm{y};q)$ to denote the moduli space of embedded holomorphic curves whose $\alpha$-boundary projection to $\Sigma$ is allowed to take a single sharp turn at $q$; see also Definition \ref{Definition, moduli spaces with sharp turns} below.  

\begin{prop}\label{Proposition, ends of 1-dimensional moduli spaces of 0-P curves, Reeb element of length 4}
Given a compatible pair $(B,\overrightarrow{\bold{\rho}})$ with $B\in \pi_2(\bm{x},\bm{y})$, $\overrightarrow{\rho}\in\{(-\rho_0,-\rho_1,-\rho_2,-\rho_3),(-\rho_0,-\rho_{123}),(-\rho_{012},-\rho_3)\}$, and $\text{Ind}(B,\overrightarrow{\bold{\rho}})=2$, the compactified moduli space $\mathcal{M}^B(\bm{x},\bm{y};\overrightarrow{\bold{\rho}})$ is a compact 1-manifold with boundary. The ends are of the following four types:
\begin{itemize}
\item[(E-1)]Two-story ends;
\item[(E-2)]split curve ends;
\item[(E-3)]ends corresponding to boundary degeneration with corners;
\item[(E-4)]ends corresponding to boundary degeneration without corners.
\end{itemize}
Moreover, 
\begin{itemize}
\item[(a)]If split curve ends occur, then $\overrightarrow{\bold{\rho}}$ is either $(-\rho_0,-\rho_{123})$ or $(-\rho_{012},-\rho_3)$, and the number of such ends is $\#\mathcal{M}^B(\bm{x},\bm{y};-\rho_{1230})$ if $\overrightarrow{\bold{\rho}}=(-\rho_0,-\rho_{123})$; otherwise the number of ends is equal to $\#\mathcal{M}^B(\bm{x},\bm{y};-\rho_{3012})$;
\item[(b)]If ends corresponding to boundary degeneration with corners occur and $\overrightarrow{\rho}=(-\rho_0,-\rho_1,-\rho_2,-\rho_3)$, then the number of such ends is mod 2 congruent to $$\sum_{\{(B_1,\ q)|\exists B_2\in T(q),\ B_1+B_2=B\}}\#\mathcal{M}^{B_1}(\bm{x},\bm{y};q).$$
For the other two choices of $\overrightarrow{\rho}$, the numbers of ends corresponding to boundary degeneration with corners are both even;
\item[(c)]If (E-4) occurs, then $[B]=[\Sigma]$, $\bm{x}=\bm{y}$, and the number of such ends is odd.
\end{itemize}
\end{prop}
\begin{rmk}
Proposition \ref{Proposition, ends of 1-dimensional moduli spaces of 0-P curves, Reeb element of length 4} considers $\overrightarrow{\bold{\rho}}$ with $a(-\overrightarrow{\bold{\rho}})=\rho_{0123}$. Corresponding propositions hold for $a(-\overrightarrow{\bold{\rho}})\in \{\rho_{1230},\rho_{2301},\rho_{3012}\}$ by cyclically permuting the subscripts in the above statement. 
\end{rmk}
The remainder of the subsection is devoted to proving Proposition \ref{Proposition, ends of 1-dimensional moduli spaces of 0-P curves, Reeb element of length 4}.

\subsubsection{Reformulation of the moduli spaces}\label{Subsection, moduli space of 0-P curves in Symmetric products}
To apply gluing results needed for studying the ends corresponding to boundary degeneration, we need to know the moduli spaces of the degenerate components are transversely cut out. However, this transversality result is not clear as a key lemma needed for the proof, Lemma 3.3 of \cite{MR2240908}, is not available for degenerate curves. In \cite{LOT18}, this difficulty is overcome by using the formulation of moduli spaces in terms of holomorphic disks in the symmetric product $Sym^g(\Sigma)$. Such moduli spaces are identified with the moduli spaces of holomorphic curves in $\Sigma\times[0,1]\times \mathbb{R}$ through a tautological correspondence (provided one uses appropriate almost complex structures). One can prove the desired transversality results of degenerate discs in the symmetric product in a standard way. We shall employ the same strategy here. The difference is that the Lagrangians holding the boundary of holomorphic disks are no longer embedded. To cater for this change, the definition of moduli spaces we give below corresponds to the one used in Floer theory for immersed Lagrangians with clean self-intersections \cite{Alston2018, Fukaya2017}.

Specifically, equip $\Sigma$ with a Kahler structure $(j,\eta)$, where $j$ denotes a complex structure and $\eta$ is a compatible symplectic form. Let $\Sigma_{\bar{e}}$ denote the closed Riemann surface obtained from $\Sigma$ by filling in the east puncture $e$. Then $Sym^g(\Sigma)$ can be viewed as the complement of $\{e\}\times Sym^{g-1}(\Sigma_{\bar{e}})$ in $Sym^g(\Sigma_{\bar{e}})$. It is a symplectic manifold with a cylindrical end modeled on the unit normal bundle of $\{e\}\times Sym^{g-1}(\Sigma_{\bar{e}})$. In particular, there is a Reeb-like vector field $\overrightarrow{R}$ tangent to the $S^1$-fibers of the unit normal bundle.

The products $\mathbb{T}_\beta=\beta_1\times\cdots\times\beta_g$ and $\mathbb{T}_{\alpha,i}=\alpha^a_i\times \alpha^c_1\times\cdots\times\alpha^c_{g-1}$, $i=1,2$, are Lagrangian submanifolds of $Sym^g(\Sigma)$. Note $\mathbb{T}_{\alpha,i}$ is immersed with self-intersections $\alpha^a_i\times\alpha_1^c\times\ldots\times\alpha^c_{g-2}\times q$, where $q$ is some self-intersection point of $\alpha_{im}=\alpha_{g-1}^c$. We identify $\mathbb{T}_{\alpha,i}$ with the image of a map $\iota_i:(0,1)\times \mathbb{T}^{g-2} \times (\amalg S^1)\rightarrow Sym^g(\Sigma)$. 

The immersed Lagrangian $\mathbb{T}_{\alpha,i}$ ($i=1,2$) intersects the ideal boundary of $Sym^2(\Sigma)$ at $\partial\overline{\alpha^a_i} \times\alpha_1^c\times\cdots\times \alpha_{g-1}^c$. Each Reeb chord $\rho$, which connects two (possibly the same) alpha arcs, now corresponds to a one-dimensional family of $\overrightarrow{R}$-chords $\rho \times \bm{x}$ that connects two (possibly the same) $\mathbb{T}_{\alpha,i}$, parametrized by $\bm{x}\in \alpha_1^c\times\cdots\times \alpha_{g-1}^c$.  

To define pseudo-holomorphic maps, we shall work with an appropriate class of almost complex structures called \emph{nearly-symmetric almost complex structures} that restrict to $Sym^g(j)$ on the cylindrical end. (The concrete definitions do not matter for our purpose, and we refer the interested readers to Definition 3.1 of \cite{OS04} and Definition 13.1 of \cite{MR2240908}). 

In this subsection, we only give definitions to the moduli spaces relevant to the case of 0-P curves; the 1-P counterparts are postponed to the next subsection. 

\begin{defn}\label{Definition, holomorphic disks in symmetric product}
Let $J_s$, $s\in[0,1]$, be a path of nearly-symmetric almost complex structures. Let $\bm{x},\bm{y}\in \amalg \mathbb{T}_{\alpha,i} \cap \mathbb{T}_{\beta}$, let $\overrightarrow{\rho}=(\sigma_1,\ldots,\sigma_n)$ be a sequence of Reeb chords, and let $B\in \tilde{\pi}_2(\bm{x},\bm{y})$. We define $\widetilde{\mathcal{M}}_{Sym,J_s}^B(\bm{x},\bm{y};\overrightarrow{\rho})$ as the set of maps $$u:([0,1]\times \mathbb{R}\backslash \{(1,t_1),\ldots,(1,t_n)\})\rightarrow Sym^g(\Sigma)$$ such that 
\begin{itemize}
\item[(1)]$t_1<\ldots<t_n$ and are allowed to vary;
\item[(2)]$u(\{0\}\times \mathbb{R})\subset \mathbb{T}_\beta$;
\item[(3)]$u(\{1\}\times (\mathbb{R}\backslash \{t_1,\ldots,t_n\}))\subset \mathbb{T}_{\alpha,1}\cup \mathbb{T}_{\alpha,2}$. Moreover, the restriction of $u$ to any connected components of $\{1\}\times (\mathbb{R}\backslash \{t_1,\ldots,t_n\})$ lifts through $\iota_i:(0,1)\times\mathbb{T}^{g-2}\times (\amalg S^1)\rightarrow Sym^g(\Sigma)$ for an appropriate $i\in\{1,2\}$;
\item[(4)]$\lim_{t\rightarrow\infty}u(s+it)=\bm{y}$, and $\lim_{t\rightarrow -\infty}u(s+it)=\bm{x}$;
\item[(5)]$\lim_{(s,t)\rightarrow (1,t_i)}u(s+it)$ is an $\overrightarrow{R}$-chord $\sigma_i\times \bm{a}$ for some $\bm{a}\in \alpha_1^c\times\cdots\times \alpha_{g-1}^c$;
\item[(6)]$\frac{du}{ds}+J_s\frac{du}{dt}=0$;
\item[(7)]$u$ is in the homology class specified by $B$.
\end{itemize}
\end{defn}
\begin{rmk}
The only difference between our setting and the setting of embedded Lagrangians is the lifting property stated in (3). This condition ensures that $\partial u$ does not have corners at self-intersection points of $\alpha_{im}$.
\end{rmk}
The tautological correspondence between the two moduli spaces defined using two different ambient symplectic manifolds holds in our setting as well. Roughly, holomorphic disks $u$ in the symmetric product are in one-to-one correspondence with pairs $(v,\pi)$, where $v$ is a stay-on-track 0-P holomorphic curves $S\rightarrow \Sigma$, and $\pi: S_{\bar{e}}\rightarrow [0,1]\times \mathbb{R}$ is a $g$-fold branched cover where the filled-in punctures are mapped to $(1,t_i)$, $i=1,\ldots,n$. The tautological correspondence was first proved in Section 13 of \cite{MR2240908}. The proof for our case follows the same line and is omitted. From now on, we will simply denote $\widetilde{\mathcal{M}}_{Sym,J_s}^B(\bm{x},\bm{y};\overrightarrow{\rho})$ by $\widetilde{\mathcal{M}}_{J_s}^B(\bm{x},\bm{y};\overrightarrow{\rho})$. The reduced moduli space ${\mathcal{M}}_{J_s}^B(\bm{x},\bm{y};\overrightarrow{\rho})$ is the quotient of $\widetilde{\mathcal{M}}_{J_s}^B(\bm{x},\bm{y};\overrightarrow{\rho})$ by the $\mathbb{R}$-action given by vertical translation.

The moduli spaces $\mathcal{M}^B(\bm{x},\bm{y};q)$ where $q$ is some self-intersection point of $\alpha_{im}$ can be similarly defined in the symmetric-product setup (and the tautological correspondence to curves in $\Sigma\times[0,1]\times\mathbb{R}$ holds).

\begin{defn}\label{Definition, moduli spaces with sharp turns}
$\mathcal{M}_{J_s}^B(\bm{x},\bm{y};q)$ is the space of $J_s$-holomorphic maps $$u:([0,1]\times \mathbb{R}\backslash \{(1,0)\})\rightarrow Sym^g(\Sigma)$$ satisfying conditions (2)(3)(4)(6)(7) of Definition \ref{Definition, holomorphic disks in symmetric product} and $\lim_{(s,t)\rightarrow(1,0)}u(s+it)=(q,\bm{a})$ for some $\bm{a}\in \alpha^a_i\times \alpha_1^c\times\cdots\times \alpha_{g-2}^c$ for an appropriate $i\in\{1,2\}$.
\end{defn}
Note there is a natural evaluation map $ev_{J_s}:\mathcal{M}^B_{J_s}(\bm{x},\bm{y};q)\rightarrow \alpha^a_i\times\alpha_1^c\times\cdots\times \alpha_{g-2}^c$ for an appropriate $i\in\{1,2\}$, given by $u\mapsto \bm{a}$ if $\lim_{(s,t)\rightarrow(1,0)}u(s+it)=(q,\bm{a})$. 

 We call a holomorphic disc \emph{degenerate} if its boundaries are in $\mathbb{T}_{\alpha,i}$ ($i=1,2$). A degenerate holomorphic disc may be viewed as a map from the upper-half plane $\mathbb{H}$ with boundary punctures to the symmetric product. We further divide such discs into degenerate discs with or without corners based on the behavior of asymptotics at the point at infinity, corresponding to Type I and Type II nodes in Definition \ref{Definition, type I type II node}. We spell out the definitions for completeness. 

\begin{defn}[Degenerate disks without corners]
Let $J$ be a nearly-symmetric almost complex structure. Let $\bm{x}\in \mathbb{T}_{\alpha}$ and $\overrightarrow{\rho}=(\sigma_1,\ldots,\sigma_n)$. $\mathcal{N}_J(\bm{x};\overrightarrow{\rho})$ is the set of maps $v:\mathbb{H}\backslash\{t_1,\ldots,t_n\}\rightarrow Sym^g(\Sigma)$ such that 
\begin{itemize}
\item[(1)]$0=t_1<\ldots<t_n$ and are allowed to vary;
\item[(2)]$v(\mathbb{R}\backslash \{t_1,\ldots,t_n\}))\subset \mathbb{T}_{\alpha,1}\cup \mathbb{T}_{\alpha,2}$. Moreover, the restriction of $v$ to any connected components of $\mathbb{R}\backslash \{t_1,\ldots,t_n\}$ lifts through $\iota_i:(0,1)\times \mathbb{T}^{g-2}\times (\amalg S^1)\rightarrow Sym^g(\Sigma)$ for an appropriate $i\in\{1,2\}$;
\item[(3)]$\lim_{z\rightarrow \infty}v(z)=\bm{x}$, and the path obtained from $v|_{(-\infty,t_1)\cup (t_n,\infty)}$ by continuous extension at $\infty$ lifts through $\iota_i$ for an appropriate $i\in\{1,2\}$;
\item[(4)]$\lim_{z\rightarrow t_i}v(z)$ is an $\overrightarrow{R}$-chords $\sigma_i\times a$ for some $a\in \alpha_1^c\times\cdots\times \alpha_{g-1}^c$;
\item[(5)]$\frac{du}{ds}+J\frac{du}{dt}=0$.
\end{itemize}
\end{defn}
\begin{defn}[Degenerate disks with corners]
Let $J$ be a nearly-symmetric almost complex structure. Let $q$ be a self-intersection point of $\alpha_{im}$. Let $\overrightarrow{\rho}=(\sigma_1,\ldots,\sigma_n)$. $\mathcal{N}_J(q;\overrightarrow{\rho})$ is the set of maps $v:\mathbb{H}\backslash\{t_1,\ldots,t_n\}\rightarrow Sym^g(\Sigma)$ such that 
\begin{itemize}
\item[(1)]$0=t_1<\ldots<t_n$ and are allowed to vary;
\item[(2)]$v(\mathbb{R}\backslash \{t_1,\ldots,t_n\}))\subset \mathbb{T}_{\alpha,1}\cup \mathbb{T}_{\alpha,2}$. Moreover, the restriction of $v$ to any connected components of $\mathbb{R}\backslash \{t_1,\ldots,t_n\}$ lifts through $\iota_i:(0,1)\times \mathbb{T}^{g-2} \times (\amalg S^1)\rightarrow Sym^g(\Sigma)$ for an appropriate $i\in\{1,2\}$;
\item[(3)]$\lim_{z\rightarrow \infty}v(z)=(q,\bm{p})$ for some $\bm{p}\in \alpha^a_i\times\alpha_1^c\times\cdots\times \alpha_{g-2}^c$ for an appropriate $i\in\{1,2\}$, and and the path from $v|_{(-\infty,t_1)\cup (t_n,\infty)}$ by continuous extension at $\infty$ does not lift through $\iota_i$;
\item[(4)]$\lim_{z\rightarrow t_i}v(z)$ is an $\overrightarrow{R}$-chords $\sigma_i\times a$ for some $a\in \alpha_1^c\times\cdots\times \alpha_{g-1}^c$;
\item[(5)]$\frac{du}{ds}+J\frac{du}{dt}=0$.
\end{itemize}
\end{defn}
We call $q$ the corner of such a degenerate disk. We also have an evaluation map $ev_J:\mathcal{N}_J(q;\overrightarrow{\rho})\rightarrow \alpha^a_i\times\alpha_1^c\times\cdots\times \alpha_{g-2}^c$ defined by $v\mapsto \bm{p}$ if $\lim_{z\rightarrow \infty}v(z)=(\bm{p},q)$.

\subsubsection{Boundary degeneration with corners}
\begin{defn}
A \emph{simple boundary degeneration} is a boundary degeneration of the form $u\vee v$, where $u$ is a (non-nodal) holomorphic curve, and $v$ is a degenerate disk. 
\end{defn}

\begin{prop}\label{Proposition, boundary degeneration with corner must be of simple form}
If a boundary degeneration with corners appears in a one-dimensional moduli space $\overline{\mathcal{M}}^B(\bm{x},\bm{y};\overrightarrow{\rho})$ where $(B,\overrightarrow{\rho} )$ is as in Proposition \ref{Proposition, ends of 1-dimensional moduli spaces of 0-P curves, Reeb element of length 4}, the boundary degeneration must be a simple boundary degeneration. Moreover, the domain for the degenerate disk must be a stabilized teardrop with an acute corner. 
\end{prop}

\begin{proof}
First, we consider the case where we assume degeneration at east infinity, multi-story splitting, and boundary degeneration without corners do not occur simultaneously with the boundary degeneration with corners. Also, note that sphere bubbles do not occur as $\Sigma$ is punctured at the east infinity. Hence we may assume the boundary degeneration with corners is a holomorphic map $u_{\infty}:\mathbb{B}\rightarrow Sym^g(\Sigma)$, where $\mathbb{B}$ is a disc bubble tree: $\mathbb{B}$ has one main component containing the $\pm$-puncture and some other components attached to the main component or each other such that the graph obtained by turning the components of $\mathbb{B}\backslash\{\text{nodes}\}$ into vertexes and nodes into edges is a tree. (See Figure \ref{Figure, bubbletree} (a).) The vertex corresponding to the main component will be called the \textit{root}. 
\begin{figure}[htb!]
\centering{
\includegraphics[scale=0.5]{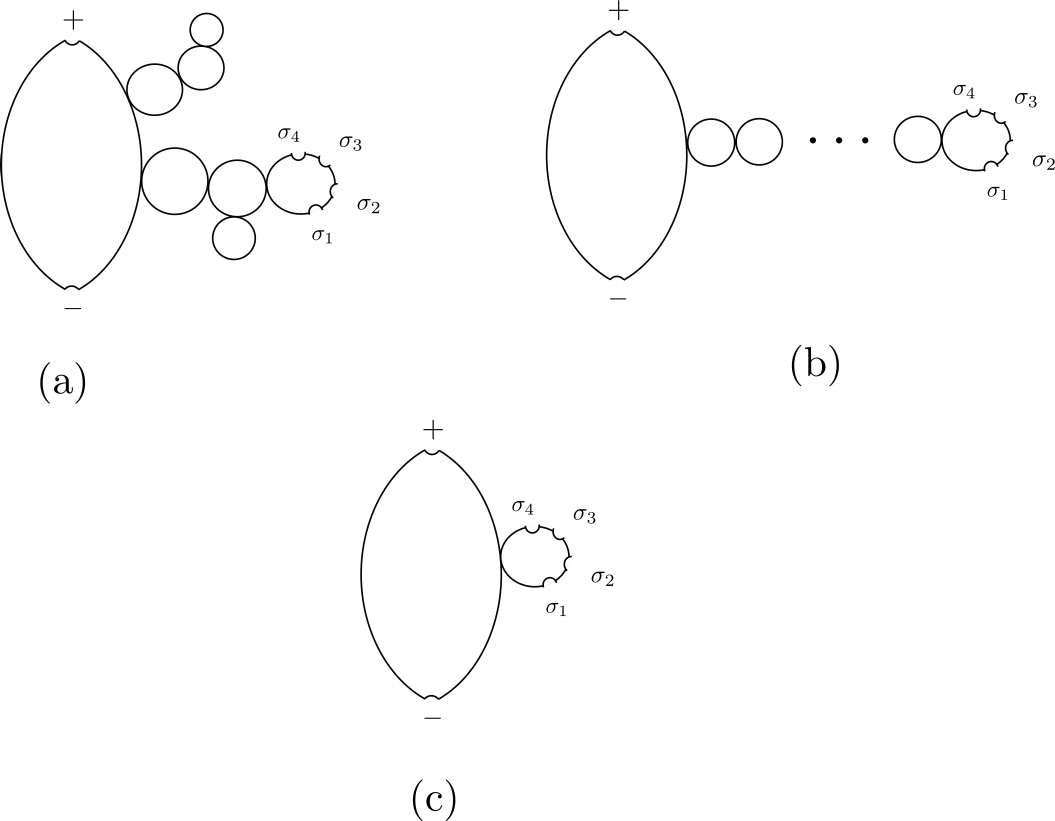}
\caption{The bubble tree $\mathbb{B}$; the $\sigma_i's$ are labels of the east punctures. Hypothetically, a bubble tree with many branches might appear, as shown in (a). In our case, we first prove the bubble tree must be of the form with a single branch as shown in (b), and then further show it must be of the simple form shown in (c).}
\label{Figure, bubbletree}
}
\end{figure}

Our first claim is that $\mathbb{B}$ must only have one leaf. (See Figure \ref{Figure, bubbletree} (b).) To see this, note a leaf corresponds to a degenerate disk whose domain is positive, and by homological consideration the domain is bounded by a one-cornered subloop of $\alpha_{im}$ and possibly $\partial \Sigma$. Note that at most one leaf would have a domain with boundary containing $\partial \Sigma$, as we assumed $a(-\overrightarrow{\rho})=\rho_{0123}$; call this the distinguished leaf. Therefore, all the other leaves, if they exist, would have positive one-cornered $\alpha$-bounded domains with $n_z=0$ (as $n_z(B)=1$ and the distinguished leaf already has multiplicity one at $z$); such domains do not exist since $\mathcal{H}$ is unobstructed. Therefore, only the distinguished leaf exists, and its domain is a stabilized teardrop since $\mathcal{H}$ is unobstructed. 

Now denote the map restricting to the main component by $u$. Let $n$ denote the number of components of $\mathbb{B}$ except the root. Denote the degenerate disc corresponding to the leaf by $v_n$, and denote those connecting the root and the leaf by $v_i$, $i=1,\ldots,n-1$. We want to prove $n=1$ and that the stabilized teardrop for the leaf has an acute corner. They follow from an index consideration as follows. 

Note that the domains $D_i$ corresponding to $v_i$, $i=1,\ldots,n-1$ are bigons: Such domains are two-cornered $\alpha$-bounded domains with $n_z=0$ and we know these domains are bigons by the assumption that $\mathcal{H}$ is unobstructed. Let $\mathcal{N}^{D_i}$ denote the reduced moduli space of holomorphic curves with domain $D_i$. Direct computation shows the virtual dimension $\text{vdim}(\mathcal{N}^{D_i})$ of the reduced moduli space satisfies $\text{vdim}(\mathcal{N}^{D_i})\geq g-1$, and the equality attained when both corners of $D_i$ are acute. Here the term $g-1$ comes from varying the constant value of the holomorphic map in $\alpha_i^a\times \alpha_1^c \times \cdots \times \alpha^{c}_{g-2}$ for some $i\in\{1,2\}$.

Now we move to consider $D_n$. We already know it is a stabilized teardrop. Depending on whether the corner of $D_n$ is acute or obtuse, the virtual dimension of the corresponding moduli space $\mathcal{N}^{D_n}$ is $g-1$ or $g$.

Let $D$ be the domain of $u$, and let $q$ denote the corner corresponding to the node. Then $$
\begin{aligned}
\text{vdim}({\mathcal{M}}^B(\bm{x},\bm{y};\overrightarrow{\rho}))=&\text{vdim}({\mathcal{M}}^D(\bm{x},\bm{y};q))+\sum_{i=1}^{n-1}\text{vdim}(\mathcal{N}^{D_i})+\text{vdim}(\mathcal{N}^{D_n})\\&-(g-1)n+n.
\end{aligned}
$$
Here, the term $-(g-1)n$ comes from the evaluation map, and the term $+n$ appears since we glued $n$ times. Note $\text{vdim}({\mathcal{M}}^B(\bm{x},\bm{y};\overrightarrow{\rho}))=1$, and hence we have $\text{vdim}({\mathcal{M}}^D(\bm{x},\bm{y};q))\leq 1-n$ since $\text{vdim}(\mathcal{N}^{D_i})\geq g-1$ for $1\leq i\leq n$. Therefore, as long as we fix a generic path of nearly-symmetric almost complex structure $\mathcal{J}_s$ so that $\mathcal{M}^D(\bm{x},\bm{y};q)$ is transversally cut out, it being non-empty implies $n=1$ (see Figure \ref{Figure, bubbletree} (c)). This also forces $\text{vdim}(\mathcal{N}^{D_n})=g-1$, which implies the corner of $D_n$ is acute. 

The above analysis shows the index of a degenerate disk with a corner is greater than or equal to $g-1$, and hence an index consideration rules out the possibility of several types of degeneration appearing simultaneously.
\end{proof}

\begin{prop}\label{Proposition, fiber of evaluation map of degenerate disks is odd when rho is cyclic}
Let $B\in T(q)$ be a stabilized teardrop with an acute corner at $q$, and let $\overrightarrow{\rho}=(-\rho_0,-\rho_1,-\rho_2,-\rho_3)$. For a generic nearly symmetric almost complex $J$, the moduli space of degenerate disks $\mathcal{N}_J^B(q;\overrightarrow{\rho})$ is a $(g-1)$-manifold, and a generic fiber of the evaluation map $ev_J:\mathcal{N}_J^B(q;\overrightarrow{\rho})\rightarrow \alpha^a_1\times \alpha^c_1\times\cdots\alpha^c_{g-2}$ consists of an odd number of points. 
\end{prop}
\begin{proof}
The argument for seeing the moduli space is smoothly cut out for a generic almost complex structure is standard, and in this case, it closely follows that of Proposition 3.14 of \cite{OS04}. 

We now study the parity of a generic fiber of the evaluation map. We first prove the relevant statements for the case where $g(\Sigma)=2$. Note that if we fix a point $p\in \alpha^a_1$, by standard arguments we may choose a generic almost complex structure $J$ so that the fiber $ev_J^{-1}(p)$ over $p$ is smoothly cut out as a $0$-manifold. Standard consideration of degeneration shows $ev_J^{-1}(p)$ is compact: limits of such maps cannot have further degenerate disks for index reasons, nor can a sequence of maps converge into a sphere bubbling as the domain is not $\Sigma$. We now claim that $|ev_J^{-1}(p)|$ is odd. A lemma is needed for this.
\begin{lem}\label{Lemma, fiber of degenerate disks over symmetric almost complex structures, rho is cyclic}
Assume $g(\Sigma)=2$. For a generic perturbation of the $\alpha$-curves, $ev_{Sym^2(j)}^{-1}(p)$ is smoothly cut out, and $|ev_{Sym^2(j)}^{-1}(p)|$ is odd. 
\end{lem}
\begin{proof}[Proof of Lemma \ref{Lemma, fiber of degenerate disks over symmetric almost complex structures, rho is cyclic}]
First note $ev_{Sym^2(j)}^{-1}(p)$ is smoothly cut out provided we perturb the $\alpha$-curves if necessary, as $B$ being a stabilized teardrop guarantees any holomorphic disks representing $[B]$ is somewhere boundary injective; see Proposition 3.9 of \cite{OS04}, or see Theorem I and Theorem 3.2 of \cite{Oh1996} for more details on the relation between boundary injectivity and regularity. 

For the second part of the statement, note by the tautological correspondence it suffices to find all the pairs $(\hat{u},\pi)$ of holomorphic maps $\hat{u}:F\rightarrow \Sigma$ and two-fold branched covers $\pi:F_{\bar{e}}\rightarrow D_{\bar{e}}$, where $F_{\bar{e}}$ stands for the surface obtained from $F$ by filling in the east punctures and $D_{\bar{e}}$ is the holomorphic disk with one boundary puncture. 
\begin{figure}[htb!]
\centering{
\includegraphics[scale=0.5]{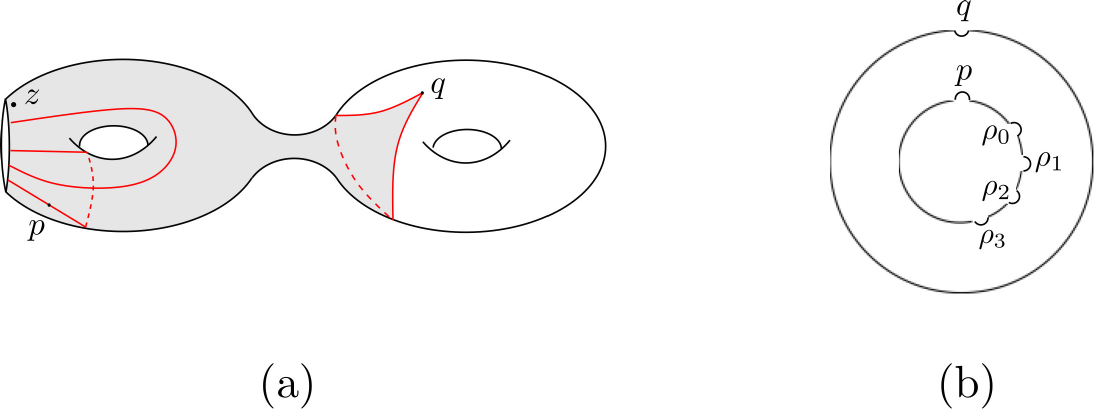}
\caption{(a) The domain $B$. (b) The annulus $F$ with punctures on the boundaries.}
\label{Figure, annulus}
}
\end{figure}
Examining the region $B$ shows that $\hat{u}$ is the obvious map from a unique holomorphic annulus $F$ with boundary punctures. (One may think $F$ is obtained by cutting $B$ open along the alpha arcs when $B$ is embedded; see Figure \ref{Figure, annulus}). Without loss of generality, we may regard $F$ is obtained from the annulus $\{z\in\mathbb{C}|\frac{1}{r}\leq |z|\leq r\}$ for some positive number $r$ by adding boundary punctures: we assume the outer boundary has only one boundary puncture and is asymptotic to $q$ under $\hat{u}$, and the inner boundary has five punctures, corresponding to $p$, $-\rho_0$, $-\rho_1$, $-\rho_2$, and $-\rho_3$ (see Figure \ref{Figure, annulus} (b)), whose relative positions depend on the complex structure induced from $j$ on $\Sigma$. There is only one involution $\iota$ on $F_{\bar{e}}$ interchanging the inner and outer boundary and swapping the boundary punctures labeled by $p$ and $q$; see \cite[Lemma 9.3]{OS04}. This involution induces $\pi:F_{\bar{e}}\rightarrow F_{\bar{e}}/\iota$, and $D$ is obtained from $F_{\bar{e}}/\iota$ by removing the boundary points corresponding to the filled-in east punctures. In summary, $ev_{Sym^2(j)}^{-1}(p)$ consists of a unique map, and hence it has odd cardinality. \end{proof}

Let $J$ be a generic almost complex structure such that both $\mathcal{N}_J^B(q;\overrightarrow{\rho})$ and $ev_{J}^{-1}(p)$ are smoothly cut out. Then a generic path $J_s$ of almost complex structures such that $J_0=J$ and $J_1=Sym^2(j)$ induces a cobordism between $ev_{J}^{-1}(p)$ and $ev_{Sym^2(j)}^{-1}(p)$. This shows $|ev_{J}^{-1}(p)|$ is odd. For a generic point $p'\in \alpha^a_1$, let $l$ be the sub-arc in $\alpha^a_1$ connecting $p$ and $p'$, then $ev^{-1}_J(l)$ is a cobordism from $ev_{J}^{-1}(p)$ to $ev_{J}^{-1}(p')$, implying $|ev_{J}^{-1}(p')|$ is odd as well. This finishes the proof of the theorem for $g(\Sigma)=2$.

The relevant statements for $g(\Sigma)>2$ can be proved inductively from the base case in which $g(\Sigma)=2$ using the neck-stretching argument in Section 10 of \cite{OS04}, which relates moduli spaces built from Heegaard diagrams differ by a stabilization.  
\end{proof}

\begin{prop}\label{Proposition, generic fiber of the evaluationo map is even when rho is not cyclic}
Let $B\in T(q)$ be a stabilized teardrop with an acute corner at $q$ and let $\overrightarrow{\rho}=(-\rho_0,-\rho_{123})$ or $(-\rho_{012},-\rho_{3})$. For a generic nearly symmetric almost complex $J$, the moduli space of degenerate disk $\mathcal{N}_J^B(q;\overrightarrow{\rho})$ is a $(g-1)$-manifold, and a generic fiber of the evaluation map $ev_J:\mathcal{N}^B(q;\overrightarrow{\rho})\rightarrow \alpha^a_1\times \alpha^c_1\times\cdots\alpha^c_{g-2}$ consists of an even number of points. 
\end{prop}
\begin{proof}
The regularities and dimensions of the moduli spaces are proved the same way as in Proposition \ref{Proposition, fiber of evaluation map of degenerate disks is odd when rho is cyclic}. 

We study the parity of a generic fiber $ev^{-1}_J(p)$ of $\mathcal{N}_J^B(q;\overrightarrow{\rho})$. Again, as in the previous proposition, we only need to study the case in which $g(\Sigma)=2$. We first prove the following lemma. 
\begin{lem}\label{Lemma, fiber of degenerate disk over sufficiently stretched complex structure, rho non-cyclic}
Assume $g(\Sigma)=2$. View $(\Sigma,\alpha_1^a,\alpha^a_2,\alpha_{im})=(E_1,\alpha_1^a,\alpha_2^a)\# (E_2,\alpha_{im})$, where $E_1$ is a punctured Riemann surface of genus one and $E_2$ is a Riemann surface of genus one. If $j$ is a sufficiently stretched complex structure on $\Sigma$, then $ev^{-1}_{Sym^2(j)}(p)$ is empty.
\end{lem}  
\begin{proof}[Proof of Lemma \ref{Lemma, fiber of degenerate disk over sufficiently stretched complex structure, rho non-cyclic}]
The proof is similar to Proposition 3.16 of \cite{OS04}. We provide a sketch. Let $j_t$ denote the complex structure on $\Sigma$, corresponding to when the connected sum tube is isometric to $S^1\times[-t,t]$. If the statement is not true, then there exists a sequence of $u_t\in ev^{-1}_{Sym^2(j_t)}(p)$ which converges to a holomorphic disk $u_{\infty}$ in $Sym^2(E_1\vee E_2)$ by Gromov compactness. In particular, the main component of $u_\infty$ is a holomorphic disk in $E_1\times E_2$. Projecting this disk to $E_1$, we would have a holomorphic disk in $E_1$ with the east punctures prescribed by $\overrightarrow{\rho}=(-\rho_0,-\rho_{123})$ or $(-\rho_{012},-\rho_{3})$, which cannot exist by direct examination. 
\end{proof}
With the above lemma at hand, a cobordism argument as in Proposition \ref{Proposition, fiber of evaluation map of degenerate disks is odd when rho is cyclic} can be applied to conclude that when the stretching parameter $t$ is sufficiently large, for a generic nearly $j_t$-symmetric almost complex structure $J_t$, the parity of a generic fiber of $ev_{J_t}$ is even. Then one can show the statement is independent of the stretching parameter $t$ as in Theorem 3.15 of \cite{OS04}: When $t_1$ and $t_2$ are sufficiently close, and $J_{t_i}$ is a $j_{t_i}$-nearly symmetric almost complex structure sufficiently close to $Sym^2(j_{t_i})$, $i=1,2$, then the moduli spaces can be identified. 
\end{proof}

\subsubsection{Boundary degeneration without corners}
\begin{prop}\label{Proposition, boundary degeneration without corner}
Under the assumption of Proposition \ref{Proposition, ends of 1-dimensional moduli spaces of 0-P curves, Reeb element of length 4}, if a boundary degeneration without corners occurs, then:
\begin{itemize}
	\item[(1)] The degenerate disk has domain $[B]=[\Sigma]$.
	\item[(2)] $\bm{x}=\bm{y}$.
	\item[(3)] Such degenerate disks do not occur simultaneously with other types of degeneration.
	\item[(4)] The number of ends corresponding to such boundary degeneration is odd.
\end{itemize}
\end{prop}
\begin{proof}
We first prove (1),(2), and (3). Suppose the limit curve $u_\infty$ has a nodal source being a disc bubble tree $\mathbb{B}$. A schematic picture of $\mathbb{B}$ to keep in mind would be an analogue of Figure \ref{Figure, bubbletree} (a). The same analysis as in the proof of Proposition \ref{Proposition, boundary degeneration with corner must be of simple form} shows there must be only one leaf in the bubble tree $\mathbb{B}$, which is the one that contains the boundary puncture. (All the others are excluded as there are no corresponding positive domains.) The degenerate disk corresponding to the leaf cannot have a corner, for otherwise, we are back to the case considered in the previous subsection by index consideration. Denote the degenerate disk corresponding to the leaf by $v$. Then, the domain $B_v$ of $v$ is a zero-cornered positive $\alpha$-bounded domain with $n_z=1$. Hence $B_v=\Sigma$ as $\mathcal{H}$ is unobstructed. Note that a degenerate disk with domain $[\Sigma]$ has Maslov index $2$, which implies the nodal curve must be of the form $u\vee v$ with $u$ being a constant curve and $v$ being the degenerate disk. We finished proving (1), (2), and (3).

(4) follows from a standard gluing argument and Proposition 11.35 of \cite{LOT18}, which states for a generic almost complex structure, the moduli space of degenerate discs at $\bm{x}$ is transversally cut out, and it consists of an odd number of points if $\overrightarrow{\rho}=(-\rho_0,\ldots,-\rho_3)$ and an even number of points for other $\overrightarrow{\rho}$. 
\end{proof}

\subsubsection{Proof of Proposition \ref{Proposition, ends of 1-dimensional moduli spaces of 0-P curves, Reeb element of length 4}}
In this subsection, we synthesize the previous results to prove Proposition \ref{Proposition, ends of 1-dimensional moduli spaces of 0-P curves, Reeb element of length 4}.
\begin{proof}[Proof of Proposition \ref{Proposition, ends of 1-dimensional moduli spaces of 0-P curves, Reeb element of length 4}]
By Proposition \ref{Proposition, summary of degeneration of moduli spaces}, we know degenerations appearing in $\overline{\mathcal{M}}^B(\bm{x},\bm{y};\overrightarrow{\rho})$ are two-story splittings, simple holomorphic combs with a single split component, or boundary degenerations. 

For a simple holomorphic comb $(u,v)$ with a single split component to appear, there must be two consecutive Reeb chords in $\overrightarrow{\rho}$ such that the endpoint of the first one is equal to the start point of the second one, this excludes the possibility of $\overrightarrow{\rho}=(-\rho_0,\ldots,-\rho_3)$. Moreover, when such a degeneration appears, as the split curve's domain is a single disk, the moduli space of the split curve $\mathcal{N}(v)$ is a transversally cut-out 0-dimensional manifold (by Proposition \ref{Proposition, regularity for holomorphic curves in R times Z times [0,1] times R}), and it consists of a single point by the Riemann mapping theorem. Therefore, the gluing result, Proposition \ref{Proposition, gluing holomorphic curves and east ends}, shows there is a neighborhood of such a holomorphic comb in $\overline{\mathcal{M}}^B(\bm{x},\bm{y};\overrightarrow{\rho})$ diffeomorphic to $(0,1]$; the count of such ends is equal to $\#\mathcal{M}^B(\bm{x},\bm{y};-\rho_{1230})$ when $\overrightarrow{\rho}=(-\rho_0,-\rho_{123})$, and is equal to $\#\mathcal{M}^B(\bm{x},\bm{y};-\rho_{3012})$ when $\overrightarrow{\rho}=(-\rho_{012},-\rho_{3})$.

By Proposition \ref{Proposition, boundary degeneration without corner} and \ref{Proposition, boundary degeneration with corner must be of simple form}, boundary degenerations are further divided into boundary degenerations with or without corners. In particular, different types of degeneration do not appear simultaneously. When boundary degeneration without corners appear, the situation is covered in Proposition \ref{Proposition, boundary degeneration without corner}. When boundary degenerations with corners appear, Proposition \ref{Proposition, fiber of evaluation map of degenerate disks is odd when rho is cyclic} and Proposition \ref{Proposition, generic fiber of the evaluationo map is even when rho is not cyclic} show the moduli space of degenerate disks are smoothly cut out. In particular, the standard gluing results can be applied to show each boundary degeneration in $\overline{\mathcal{M}}^B(\bm{x},\bm{y};\overrightarrow{\rho})$ has a neighborhood diffeomorphic to $(0,1]$. The number of such ends is equal to $$\sum_{\{(q,B_1)|\exists B_2\in T(q),B_1+B_2=B\}}\#(\mathcal{M}^{B_1}(\bm{x},\bm{y};q)\times_{ev}\mathcal{N}^{B_2}(q;\overrightarrow{\rho}))$$ (We have suppressed the almost complex structure $J_s$, which can be chosen generically so that the evaluation maps are transversal to each other.) This quantity is even when $\overrightarrow{\rho}\neq (-\rho_0,\ldots,-\rho_3)$ in view of Proposition \ref{Proposition, generic fiber of the evaluationo map is even when rho is not cyclic}. Otherwise, it has the same parity as $$\sum_{\{(q,B_1)|\exists B_2\in T(q),B_1+B_2=B\}}\#\mathcal{M}^{B_1}(\bm{x},\bm{y};q)$$ in view of Proposition \ref{Proposition, fiber of evaluation map of degenerate disks is odd when rho is cyclic}.
\end{proof}

\subsection{Ends of moduli spaces of 1-P curves}\label{Subsection, ends of moduli space of 1P curves}

This subsection characterizes the ends of one-dimensional moduli spaces of 1-P holomorphic curves. Given a generator $\bm{x}$, we say $\iota(\bm{x})=\iota_1$ if and only if $\bm{x}$ is in $\mathbb{T}_{\alpha,1}$; otherwise $\iota(\bm{x})=\iota_0$. The main result is the following.  
\begin{prop}\label{Proposition, ends of moduli space of 1-P curves}
Let $B\in\tilde{\pi}_2(\bm{x},\bm{y})$ such that $\iota(\bm{x})=\iota_1$ and $\text{ind}(B;U)=2$. Then fixing a generic almost complex structure, the compactified moduli space $\overline{\mathcal{M}}^B(\bm{x},\bm{y};U)$ is a compact 1-manifold with boundary. The boundaries are of the following types:
\begin{itemize}
\item[(1)]Two-story building
\item[(2)]simple holomorphic combs $(u,v)$ with $v$ being an orbit curve
\item[(3)]boundary degeneration with corners
\item[(4)]boundary degeneration without corners
\end{itemize}
Moreover, 
\begin{itemize}
\item[(a)]The number of type (2) ends is $\#\mathcal{M}^B(\bm{x},\bm{y};-\rho_{1230})+\#\mathcal{M}^B(\bm{x},\bm{y};-\rho_{3012})$.
\item[(b)]The number of type (3) ends is mod 2 congruent to $$\sum_{\{(B_1,\ q)|B_2\in T(q),\ B_1+B_2=B\}}\#\mathcal{M}^{B_1}(\bm{x},\bm{y};q)$$
\item[(c)]The number of type (4) ends is even.
\end{itemize}
\end{prop}
\begin{rmk}
A similar proposition holds in the case when $\iota(\bm{x})=\iota_0$. One simply needs to change the Reeb chords in (a) by a cyclic permutation of the digits in the subscript. 
\end{rmk}

\subsubsection{Reformulation of the moduli spaces}
We reformulate $\mathcal{M}^B(\bm{x},\bm{y};U)$ in terms of holomorphic disks in $Sym^g(\Sigma)$. Assume $\iota(\bm{x})=\iota_1$ throughout the rest of the section.

\begin{defn}
${\mathcal{M}}^B_{Sym}(\bm{x},\bm{y};U)$ is defined to be the space of holomorphic maps $u:[0,1]\times\mathbb{R}\backslash \{(s_0,0)\}\rightarrow Sym^g(\Sigma)$ such that:
\begin{itemize}
\item[(1)]$(s_0,0)$ is in the interior of $[0,1]\times\mathbb{R}$ and is allowed to vary;
\item[(2)]$u(\{0\}\times \mathbb{R})\subset \mathbb{T}_\beta$;
\item[(3)]$u(\{1\}\times \mathbb{R})\subset \mathbb{T}_{\alpha,1}$. Moreover, $u|_{\{1\}\times \mathbb{R}}$ lifts through $f_1:(0,1)\times\mathbb{T}^{g-2}\times (\amalg S^1)\rightarrow Sym^g(\Sigma)$;
\item[(4)]$\lim_{t\rightarrow\infty}u(s+it)=\bm{y}$, and $\lim_{t\rightarrow -\infty}u(s+it)=\bm{x}$;
\item[(5)]$\lim_{(s,t)\rightarrow (s_0,0)}u(s+it)$ is a closed $\overrightarrow{R}$-orbit $\sigma \times \bm{w}$, where $\bm{w}\in Sym^{g-1}(\Sigma)$ and $\sigma$ stands for a closed Reeb orbit that traverses $\partial \overline{\Sigma}$ once;
\item[(6)]$\frac{du}{ds}+J_s\frac{du}{dt}=0$;
\item[(7)]$u$ is in the homology class specified by $B$.
\end{itemize}
\end{defn}
Again, we have the tautological correspondence that identifies the moduli spaces defined here and the ones in Section \ref{subsection, defining moduli spaces}. Therefore, we shall no longer keep the subscript ``Sym" in the notation. 

We also define the moduli spaces of one-punctured degenerate disks (with or without corners).  

\begin{defn}[One-punctured degenerate disks without corners]
Let $J$ be a nearly-symmetric almost complex structure. Let $\bm{x}\in  \mathbb{T}_{\alpha}$. $\mathcal{N}_J(\bm{x};U)$ is the space of maps $v:\mathbb{H}\backslash\{i\}\rightarrow Sym^g(\Sigma)$ such that 
\begin{itemize}
\item[(1)]$v(\mathbb{R})\subset \mathbb{T}_{\alpha,1}$. Moreover, the restriction of $v|_{\mathbb{R}}$ lifts through $f_1:(0,1)\times \mathbb{T}^{g-2}\times(\amalg S^1)\rightarrow Sym^g(\Sigma)$;
\item[(2)]$\lim_{z\rightarrow \infty}v(z)=\bm{x}$, and the path obtained from $v|_{\partial \mathbb{H}}$ by continuous extension at $\infty$ lifts through $\iota_1$;
\item[(3)]$\lim_{z\rightarrow i}v(z)$ is some closed $\overrightarrow{R}$-orbit $\sigma \times \bm{w}$, where $\bm{w}\in Sym^{g-1}(\Sigma)$ and $\sigma$ stands for a closed Reeb orbit that traverses the $\partial \overline{\Sigma}$ once;
\item[(4)]$\frac{du}{ds}+J\frac{du}{dt}=0$.
\end{itemize}
\end{defn}
\begin{defn}[One-cornered one-punctured degenerate disks]
Let $J$ be a nearly-symmetric almost complex structure. Let $q$ be a self-intersection point of $\alpha_{im}$. $\mathcal{N}_J(q;U)$ is the space of maps $v:\mathbb{H}\backslash\{i\}\rightarrow Sym^g(\Sigma)$ such that 
\begin{itemize}
\item[(1)]$v(\mathbb{R})\subset \mathbb{T}_{\alpha,1}$. Moreover, the restriction of $v|_{\mathbb{R}}$ lifts through $f_1:(0,1)\times \mathbb{T}^{g-1}\times (\amalg S^1)\rightarrow Sym^g(\Sigma)$;
\item[(2)]$\lim_{z\rightarrow \infty}v(z)=(q,\bm{p})$ for some $p\in \alpha^a_1\times\alpha^c_1\times\cdots\times \alpha^c_{g-2}$, and and the path obtained from $v|_{\partial \mathbb{H}}$ by continuous extension at $\infty$ does not lift through $\iota_1$;
\item[(3)]$\lim_{z\rightarrow i}v(z)$ is some closed $\overrightarrow{R}$-orbit $\sigma \times \bm{w}$, where $\bm{w}\in Sym^{g-1}(\Sigma)$ and $\sigma$ stands for a closed Reeb orbit that traverses $\partial \overline{\Sigma}$ once;
\item[(4)]$\frac{du}{ds}+J\frac{du}{dt}=0$.
\end{itemize}
\end{defn}
We call $q$ the corner of such a degenerate disk. We also have an evaluation map $ev_J:\mathcal{N}_J(q;U)\rightarrow \alpha^a_i\times\alpha^c_1\times\cdots\times\alpha_{g-2}^c$ defined by $v\mapsto \bm{p}$ if $\lim_{z\rightarrow \infty}v(z)=(q,\bm{p})$.

\subsubsection{One-punctured boundary degeneration with corners}
\begin{prop}\label{Proposition, one-punctured boundary degeneration with corneres must be of simple form}
If a boundary degeneration with corners appears in the compactification of a one-dimensional moduli space $\mathcal{M}^B(\bm{x},\bm{y};U)$, then the nodal comb is of simple form, and the domain for the degenerate disk is a stabilized teardrop with an acute corner. 
\end{prop}
\begin{proof}
The proof is similar to the proof of Proposition \ref{Proposition, boundary degeneration with corner must be of simple form}. There is only one modification needed: We no longer have east boundary punctures when considering the bubble tree $\mathbb{B}$ of the nodal curve; instead, there is one and only one interior puncture. With this, the rest of the proof follows exactly as in Proposition \ref{Proposition, boundary degeneration with corner must be of simple form}. 
\end{proof}

\begin{prop}\label{Proposition, fiber of evaluation map of 1-punctured degenerate disks is odd}
Let $q$ be a self-intersection point of $\alpha_{im}$ and let $B\in T(q)$ be a stabilized teardrop with acute corner. For a generic nearly symmetric almost complex structure $J$, the moduli space of degenerate disks $\mathcal{N}_J^B(q;U)$ is a $(g-1)$-manifold, and a generic fiber of the evaluation map $ev_J:\mathcal{N}_J^B(q;U)\rightarrow \alpha^a_1\times\alpha^c_1\times\cdots\times\alpha_{g-2}^c$ is a compact 0-dimensional manifold consisting of an odd number of points. 
\end{prop}

\begin{proof}
The regularity of $\mathcal{N}_J^B(q;U)$ and compactness of a generic fiber are proved in the same way as in Proposition \ref{Proposition, fiber of evaluation map of 1-punctured degenerate disks is odd}. 

The parity of the cardinality of a generic fiber follows from a similar neck-stretching and cobordism argument as in Proposition \ref{Proposition, fiber of evaluation map of 1-punctured degenerate disks is odd}, using Lemma \ref{Lemma, fiber of a 1-punctured degenerate disks using sufficiently stretched almost complex structure} below instead of Lemma \ref{Lemma, fiber of degenerate disks over symmetric almost complex structures, rho is cyclic}.
\end{proof}

\begin{lem}\label{Lemma, fiber of a 1-punctured degenerate disks using sufficiently stretched almost complex structure}
Assume $g(\Sigma)=2$. Fix some point $p\in \alpha^a_1$. For a sufficiently stretched almost complex structure $j$ on $\Sigma$, the fiber $ev_{Sym^2(j)}^{-1}(p)$ is transversally cut out and consists of one point. 
\end{lem}
\begin{proof}
View $\Sigma$ be the connected sum $(E_1,\alpha^a_1,\alpha^a_2)$ and $(E_2,\alpha_{im})$, where $E_1$ is the punctured Riemann surface of genus one and $E_2$ is a closed Riemann surface of genus one. Let $z'$ denote the points on $E_1$ and $E_2$ where the connected sum is performed. The domain $B$ gives rise to a teardrop domain $B'$ in $E_2$ with $n_{z'}(B')=1$. The Riemann mapping theorem implies that the moduli space $\mathcal{N}^{B'}(q)$ of holomorphic disks in $E_2$ with corner $q$ and domain $B'$ is smoothly cut out and has only one element. The gluing argument in Section 10 of \cite{OS04} shows for a sufficiently stretched almost complex structure, maps in $ev_{Sym^2(j)}^{-1}(p)$ are obtained by splicing the one-punctured holomorphic sphere in $Sym^2(E_1)$ passing through $(z',p)$ and the holomorphic disk in $\mathcal{N}^{B'}(q)$.\footnote{Strictly speaking, Ozsv\'ath and Szab\'o's argument concerns splicing closed holomorphic spheres while in our case the sphere is punctured, but this does not affect the argument applying to our case. Alternatively, we may treat one-punctured holomorophic disks or spheres as the corresponding object without interior punctures, but intersecting $\{e\}\times \Sigma_{\bar{e}}$ once in $Sym^2(\Sigma_{\bar{e}})$.} In particular, $ev_{Sym^2(j)}^{-1}(p)$ is identified with $\mathcal{N}^{B'}(q)$ and hence consists of only one element. 
\end{proof}

\subsubsection{One-punctured boundary degeneration without corners}
\begin{prop}\label{Proposition, 1-P curve, boundary degeneration without corners}
In the assumption of Proposition \ref{Proposition, ends of moduli space of 1-P curves}, if a boundary degeneration without corner occurs, then: 
\begin{itemize}
\item[(1)] There is only one degenerate disk, and its domain $[B]$ is $[\Sigma]$. 
\item[(2)] $\bm{x}=\bm{y}$. 
\item[(3)] Such degenerate disk do not occur simultaneously with other types of degeneration.
\item[(4)] The number of ends corresponding to such boundary degeneration is even.
\end{itemize}
\end{prop}

\begin{proof} 
The proof of (1), (2), and (3) are straightforward modifications of that of Proposition \ref{Proposition, boundary degeneration without corner} and are omitted. (4) follows from the standard gluing result and Proposition \ref{Proposition, 1-P degenerate disks without corner appear in pairs} below, which differs from the counterpart in the 0-P case.  
\end{proof}

\begin{prop}\label{Proposition, 1-P degenerate disks without corner appear in pairs}
For a generic almost complex structure $J$, $\mathcal{N}^{[\Sigma]}_J(\bm{x};U)$ is a compact, 0-dimensional manifold that consists of an even number of points. 
\end{prop}
\begin{proof}
The argument for compactness and transversality is the same as in \cite[Proposition 3.14]{OS04}, which is the counterpart of Proposition \ref{Proposition, 1-P degenerate disks without corner appear in pairs} when the Heegaard surface is closed; we will omit this part. By a similar cobordism argument used in Proposition \ref{Proposition, fiber of evaluation map of degenerate disks is odd when rho is cyclic}, we can reduce understanding the parity of the moduli space to the base case $g(\Sigma)=2$, which is addressed in Lemma \ref{Lemma, moduli space of 1-P degenerate disks without corners with symmetric almost complex structures} below.
\end{proof}

\begin{lem}\label{Lemma, moduli space of 1-P degenerate disks without corners with symmetric almost complex structures}
Assume $g(\Sigma)=2$. View $(\Sigma,\alpha_1^a,\alpha^a_2,\alpha_{im})=(E_1,\alpha_1^a,\alpha_2^a)\# (E_2,\alpha_{im})$, where $E_1$ is a punctured Riemann surface of genus one and $E_2$ is a Riemann surface of genus one. If $j$ is a sufficiently stretched complex structure on $\Sigma$, then $\mathcal{N}^{[\Sigma]}_{Sym^2(j)}(\bm{x};U)$ is empty. 
\end{lem}  
\begin{proof}
Otherwise, the same neck-stretching procedure as in Lemma \ref{Lemma, fiber of degenerate disk over sufficiently stretched complex structure, rho non-cyclic} produces a limit nodal holomorphic curve $u_\infty:\mathbb{B}\rightarrow Sym^2(E_1\vee E_2)$. It consists of a (possibly punctured) holomorphic disk $v$ that maps to $E_1\times E_2$ with boundary in $\mathbb{T}_{\alpha,1}$ and possibly some (possibly punctured) sphere bubbles in $Sym^2(E_i)$, $i=1,2$. We claim $v$ must be a constant map. It is clear that $Pr_{E_1}\circ v$ is constant, for $\pi_2(E_{1,\bar{e}},\alpha^a_1\cup \{e\})=0$, where $E_{1,\bar{e}}$ denote the Riemann surface obtained by filling in the east puncture. We move to see $Pr_{E_2}\circ v$ is constant. Suppose $Pr_{E_2}\circ v$ is not a constant map.  Note the domain of $ (Pr_{E_2}\circ v)$ is a zero-cornered $\alpha$-bounded domain $D$ in $E_2$. Stabilizing by $E_1$, this domain induces a zero-cornered $\alpha$-bounded domain $D'$ in $\Sigma$ with $n_z(D')\leq 1$. If $n_z(D')=0$, then $D'$ does not exist as $\mathcal{H}$ is unobstructed, and hence $D$ does not exist. So $n_z(D')=1$, and hence $D'=\Sigma$ since $\mathcal{H}$ is unobstructed. This implies $D=E_2$. Therefore, $\partial(Pr_{E_2}\circ v)$ is null-homotopic in $\alpha_{im}$. So $Pr_{E_2}\circ v$ induces a nontrivial element in $\pi_2(E_2)$. This, however, contradicts that $\pi_2(E_2)=0$. Therefore, $Pr_{E_2}\circ v$ is also constant, and hence $v$ is the constant map with image $\bm{x}$. Now $\{\bm{x}\}$ intersects neither $Sym^2(E_i)$, $i=1,2$, and hence there are no sphere bubbles in $u_{\infty}$. So the Gromov limit $u_{\infty}$ is a constant map. In particular, $n_z(u_\infty)=0$. However, $n_z(u_\infty)=1$ as it is the limit of a sequence of holomorphic maps whose multiplicity at $z$ is one. This is a contradiction. Therefore, $\mathcal{N}^{[\Sigma]}_{Sym^2(j)}(\bm{x};U)$ is empty provided $j$ is sufficiently stretched. 
\end{proof}

\subsubsection{Proof of Proposition \ref{Proposition, ends of moduli space of 1-P curves}}
\begin{proof}[Proof of Proposition \ref{Proposition, ends of moduli space of 1-P curves}]
In view of Proposition \ref{Proposition, summary of degeneration of moduli spaces}, Proposition \ref{Proposition, one-punctured boundary degeneration with corneres must be of simple form}, and Proposition \ref{Proposition, 1-P curve, boundary degeneration without corners} we know the degenerations that can appear in the boundary of the compactified moduli spaces are two-story curves, simple combs with orbit curve ends, or simple boundary degenerations with or without corners. In all cases, gluing arguments can be applied to see the compactified moduli space $\overline{\mathcal{M}}^B(\bm{x},\bm{y};U)$ is a one-manifold with boundary.

For conclusion (a), note that ends of type (2) correspond to pairs of curves $(u,v)$ where $u$ is in $\mathcal{M}^B(\bm{x},\bm{y};-\rho_{1230})$ or $\mathcal{M}^B(\bm{x},\bm{y};-\rho_{3012})$ and $v$ is an orbit curve, but the moduli space of orbit curves consists of a single element by the Riemann mapping theorem so the count of type (2) boundaries agrees with $\#\mathcal{M}^B(\bm{x},\bm{y};-\rho_{1230})+\#\mathcal{M}^B(\bm{x},\bm{y};-\rho_{3012})$. For conclusion (b), standard gluing results imply that the number of such ends is equal to $$\sum_{\{(q,B_1)|\exists B_2\in T(q),B_1+B_2=B\}}\#(\mathcal{M}^{B_1}(\bm{x},\bm{y};q)\times_{ev}\mathcal{N}^{B_2}(q;U))$$ This is mod 2 equal to $$\sum_{\{(q,B_1)|\exists B_2\in T(q),B_1+B_2=B\}}\#\mathcal{M}^{B_1}(\bm{x},\bm{y};q)$$ as a generic fiber of $ev$ in $\mathcal{N}^{B_2}(q;U)$ is odd by Proposition \ref{Proposition, fiber of evaluation map of 1-punctured degenerate disks is odd}. For (c), note by gluing results the number of such ends is equal to $\#\mathcal{N}^{[\Sigma]}(\bm{x};U)$. This is even by Proposition \ref{Proposition, 1-P degenerate disks without corner appear in pairs}.
\end{proof}
\subsection{Type D structures}\label{Subsection, definition of type D structure}
We define type D structures from an immersed bordered Heeggard diagram $\mathcal{H}=(\Sigma,\bm{\beta},\bm{\bar{\alpha}},z)$ in this subsection.
\begin{figure}[htb!]
	\includegraphics[scale=0.4]{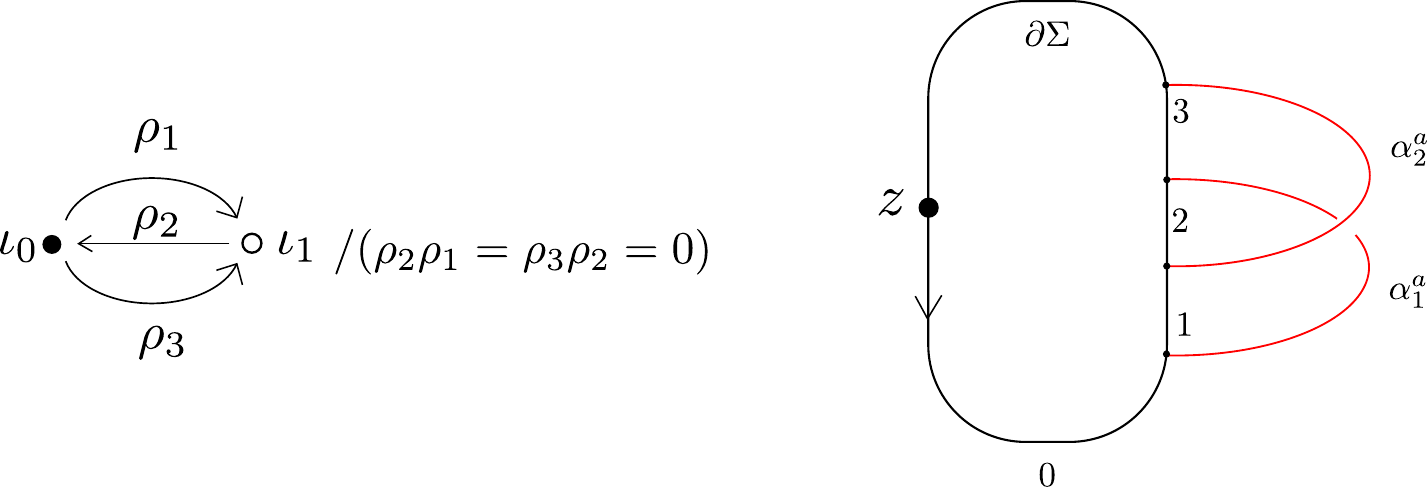}
	\caption{The quiver presentation of the torus algebra (left) and the pointed match circle of $\mathcal{H}$ with reversed boundary orientation (right).}\label{Figure, pointed match circle}
\end{figure} 
 Let $\mathcal{A}$ denote the \textit{torus algebra}, which is isomorphic to the quiver algebra of the quiver in Figure \ref{Figure, pointed match circle} (left). For $I\in\{1,2,3,12,23,123\}$, $\rho_I\in \mathcal{A}$ is understood as the product of the $\rho_i$'s for those $i$ appear in $I$. This algebra arises naturally in the context of bordered Heegaard diagrams, where $\mathcal{A}$ is associated to the pointed match circle determined by $\mathcal{H}$ with the reversed boundary orientation (Figure \ref{Figure, pointed match circle} (right)); we refer the readers to \cite[Chapter 11.1]{LOT18} for a detailed definition of the torus algebra in terms of pointed match circles, and we only point out that the element $\rho_I\in \mathcal{A}$ for $I\in\{1,2,3,12,23,123\}$ corresponds to the Reeb chord with the same label on the pointed match circle. Let $\mathcal{I}=\langle \iota_0\rangle\oplus \langle \iota_1 \rangle$ denote the ring of idempotents of $\mathcal{A}$. We recall the definition of a type D structure.
 
\begin{defn}
A type D structure over the torus algebra $\mathcal{A}$ is a left $\mathcal{I}$-module $N$ together with a linear map $\delta: N\rightarrow \mathcal{A}\otimes N$ such that the map $$\partial\coloneqq (\mu_\mathcal{A}\otimes \mathbb{I}_N)\circ(\mathbb{I}_{\mathcal{A}}\otimes \delta): \mathcal{A}\otimes N \rightarrow \mathcal{A}\otimes N$$ is a differential, i.e., $\partial^2=0$. The left differential $\mathcal{A}$-module $\mathcal{A}\otimes N$ is called the type D module of the type D structure $(N,\delta)$.
\end{defn}

Next, we spell out the construction of a type D structure from an immersed bordered Heegaard diagram. Recall $\mathbb{T}_{\beta}=\beta_1\times\cdots \times\beta_g$ and $\mathbb{T}_{\alpha,i}=\alpha^a_i\times \alpha_1^c\times\cdots\times\alpha_{g-1}^c$, $i=1,2$. Let $\mathbb{T}_{\alpha}=\mathbb{T}_{\alpha,1}\cup \mathbb{T}_{\alpha, 2}$. Let $\mathcal{G}(\mathcal{H})=\{\bm{x}|\bm{x}\in\mathbb{T}_{\alpha}\cap \mathbb{T}_{\beta}\}$. Denote the local system on $\alpha_{im}$ as a vector bundle $\mathcal{E}\rightarrow \alpha_{im}$ together with a parallel transport $\Phi$. Note that this induces a local system on $\mathbb{T}_{\alpha}$, the tensor product of $\mathcal{E}$ and the trivial local system on the other alpha curves (or arcs). Abusing notation, we still denote the local system on $\mathbb{T}_{\alpha}$ by $(\mathcal{E},\Phi)$. Now define an $\mathcal{I}$-module $X^\mathcal{E}(\mathcal{H})=\oplus_{\bm{x}\in \mathcal{G}(\mathcal{H})} \mathcal{E}|_{\bm{x}}$, where the $\mathcal{I}$-action on an element $\eta\in \mathcal{E}|_{\bm{x}}$ is specified by $$\iota_i\cdot \eta=\begin{cases} \eta, \ o(\bm{x})\equiv i \pmod 2\\ 0,\  \text{otherwise} \end{cases}$$
Here $o(\bm{x})=i$ if and only if $\bm{x}\in \mathbb{T}_{\alpha,i}$, $i=1,2$.

Given a sequence of Reeb chords $\overrightarrow{\sigma}=(\sigma_1,\ldots,\sigma_k)$ of a pointed match circle $\mathcal{Z}$, $a(-\overrightarrow{\sigma})$ is defined to be $(-\sigma_1)\cdot(-\sigma_2)\cdot \ldots (-\sigma_k)\in \mathcal{A}(-\mathcal{Z})$. Note given $B\in \pi_2(\bm{x},\bm{y})$, the parallel transport restricted to the arc $\partial_{\alpha_{im}}B\subset \alpha_{im}$ induces an isomorphism from $\mathcal{E}|_{\bm{x}}$ to $\mathcal{E}|_{\bm{y}}$, which we denote by $\Phi^{B}_{\bm{x},\bm{y}}$.

\begin{defn}\label{Definition, type D structure}
Let $\mathcal{H}$ be an unobstructed, provincially admissible, immersed bordered Heegaard diagram. Fix a generic almost complex structure on $\Sigma\times [0,1] \times \mathbb{R}$. The type D module $\widehat{CFD}(\mathcal{H})$ is defined to be the $\mathcal{A}$-module $$\mathcal{A}\otimes_{\mathcal{I}}X^\mathcal{E}(\mathcal{H})$$ together with a differential given by
$$\partial (a\otimes \eta)=a\cdot(\sum_{\bm{y}}\sum_{\{(B,\overrightarrow{\sigma})|\ n_z(B)=0,\ \text{ind}(B,\overrightarrow{\sigma})=1\}}\#\mathcal{M}^B(\bm{x},\bm{y};\overrightarrow{\sigma})a(-\overrightarrow{\sigma})\otimes \Phi^B_{\bm{x},\bm{y}}\eta),$$
where $a\in\mathcal{A}$, $\eta\in \mathcal{E}|_{\bm{x}}$, and the pairs $(B,\overrightarrow{\sigma})$ are compatible. The underlying type D structure is the pair $(X^\mathcal{E}(\mathcal{H}), \delta)$ where $\delta(\eta)\coloneqq \partial(1\otimes \eta)$ for any $\eta\in X^\mathcal{E}(\mathcal{H})$.
\end{defn}
Abusing notation, we will also use $\widehat{CFD}(\mathcal{H})$ to denote its underlying type D structure. 
\begin{rmk}
Note when the local system is trivial, we can identify $\bm{x}$ with $\mathcal{E}|_{\bm{x}}$, and the differential defined above can be more conveniently written as
$$\partial (a\otimes\bm{x})=a\cdot(\sum_{\bm{y}}\sum_{\{(B,\overrightarrow{\sigma})|\ n_z(B)=0,\ \text{ind}(B,\overrightarrow{\sigma})=1\}}\#\mathcal{M}^B(\bm{x},\bm{y};\overrightarrow{\sigma})a(-\overrightarrow{\sigma})\otimes \bm{y}).$$
\end{rmk}

\begin{prop}\label{Proposition, type D structure is well defined}
The operator $\partial$ in Definition \ref{Definition, type D structure} is well-defined and $\partial^2=0$. 
\end{prop}
\begin{proof}
We first point out $\partial$ is well-defined, i.e., the sum defining $\partial$ is finite. This reduces to the provincial admissibility of $\mathcal{H}$, which implies there are only finitely many positive domains with prescribed Reeb chords connecting any given pair of generators. The proof is standard, and we do not repeat it here. 

We move to see $\partial^2(\bm{x})=0$. For ease of explanation, we begin with the case of trivial local systems. Let $a$ be a non-zero element of $\mathcal{A}$, and let $\langle \partial^2 \bm{x},a\bm{y}\rangle \in \mathbb{F}$ denote the coefficient of the term $a\bm{y}$ in $\partial^2 \bm{x}$. Then 
\begin{equation}\label{Equation,partial square of type D structure differential}
\langle \partial^2 \bm{x},a\bm{y}\rangle=\sum_{\bm{w}\in\mathcal{G}}\sum \#\mathcal{M}^{B_1}(\bm{x},\bm{w};\overrightarrow{\sigma_1})\#\mathcal{M}^{B_2}(\bm{w},\bm{y};\overrightarrow{\sigma_2}), 
\end{equation}

where the second sum is over all the index-one compatible pairs $(B_i,\overrightarrow{\sigma_i})$ ($i=1,2$) with $a(-\overrightarrow{\sigma_1})\cdot a(-\overrightarrow{\sigma_2})=a$. 
In view of Proposition \ref{Proposition, ends of 0-p curves, n_z=0} and the gluing result, the right-hand side of Equation (\ref{Equation,partial square of type D structure differential}) is $$\sum_{\{(B,\overrightarrow{\sigma})|\text{ind}(B,\overrightarrow{\sigma})=2,\ \ a(-\overrightarrow{\sigma})=a\}}\#\partial \overline{\mathcal{M}}^B(\bm{x},\bm{y};\overrightarrow{\sigma})\equiv 0 \pmod 2$$
This finishes the proof in the case of trivial local systems. For the case of non-trivial local systems, the proof is a slight modification of the above argument. One needs to note that given $B_1\in\pi_2(\bm{x},\bm{w})$ and $B_2\in \pi_2(\bm{w},\bm{y})$, we have $\Phi_{\bm{x},\bm{w}}^{B_1}\circ \Phi_{\bm{w},\bm{y}}^{B_2}=\Phi_{\bm{x},\bm{y}}^{B_1+B_2}$. Therefore, given an $\eta\in \mathcal{E}|_{\bm{x}}$, the terms in $\partial^2(\eta)$ corresponding to two-story ends of a one-dimensional moduli space $\mathcal{M}^B(\bm{x},\bm{y};\overrightarrow{\sigma})$ are multiples of the same element in $\mathcal{E}|_{\bm{y}}$, namely $\Phi_{\bm{x},\bm{y}}^B(\eta)$, and hence the coefficient is zero mod $2$. 
\end{proof}
 
\subsection{Weakly extended Type D structures}\label{subsec:weakly-extended-type-D-structures}
We define the weakly extended type D structure $\widetilde{CFD}(\mathcal{H})$ in this subsection. The weakly extended torus algebra $\tilde{\mathcal{A}}$ can be represented by the quiver with relations shown in Figure \ref{Figure, weakly extended torus algebra}.
\begin{figure}[htb!]
\includegraphics[scale=0.6]{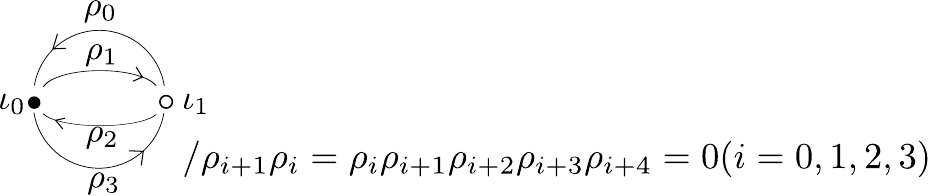}
\caption{The weakly extended torus algebra. The subscripts in the relation are understood mod $4$.}\label{Figure, weakly extended torus algebra}
\end{figure}
Note as in the torus algebra, we have the idempotent ring $\mathcal{I}=\langle \iota_0 \rangle \oplus \langle \iota_1 \rangle$. Let $\bm{U}$ be $ \rho_{0123}+\rho_{1230}+\rho_{2301}+\rho_{3012}$, which is a central element of $\tilde{\mathcal{A}}$. 
\begin{defn}
A weakly extended type D structure over $\tilde{\mathcal{A}}$ is a left $\mathcal{I}$-module $N$ together with a linear map $\tilde{\delta}: N\rightarrow \tilde{\mathcal{A}}\otimes N$ such that the map $$\tilde{\partial}\coloneqq (\mu_{\tilde{\mathcal{A}}}\otimes \mathbb{I}_N)\circ(\mathbb{I}_{\tilde{\mathcal{A}}} \otimes \tilde{\delta}): \tilde{\mathcal{A}}\otimes N \rightarrow \tilde{\mathcal{A}}\otimes N$$ squares to $\bm{U}$, i.e. $\tilde \partial^2=\bm{U}$. The curved left $\tilde{\mathcal{A}}$-module $\tilde{\mathcal{A}}\otimes N$ is called the weakly extended type D module of the weakly extended type D structure $(N,\tilde{\delta})$.
\end{defn}

Let $X^{\mathcal{E}}(\mathcal{H})$ be the $\mathcal{I}$-module defined the same way as in Section \ref{Subsection, definition of type D structure}.
\begin{defn}\label{Definition, weakly extended type D structure}
Let $\mathcal{H}$ be an unobstructed, provincially admissible, immersed bordered Heegaard diagram. Fix a generic admissible almost complex structure on $\Sigma\times [0,1] \times \mathbb{R}$. The weakly extended type D module $\widetilde{CFD}(\mathcal{H})$ is defined to be the $\tilde{\mathcal{A}}$-module $$\tilde{\mathcal{A}}\otimes_{\mathcal{I}}X^{\mathcal{E}}(\mathcal{H})$$ together with a differential given by 
$$\tilde{\partial} (a\otimes \eta)=a\cdot(\sum_{\bm{y}}\sum_{\{(B,\overrightarrow{\sigma})|\ \text{ind}(B,\overrightarrow{\sigma})=1\}}\#\mathcal{M}^B(\bm{x},\bm{y};\overrightarrow{\sigma})a(-\overrightarrow{\sigma})\otimes \Phi^B_{\bm{x},\bm{y}}\eta),$$
where $a\in\tilde{\mathcal{A}}$, $\eta\in \mathcal{E}|_{\bm{x}}$, $\overrightarrow{\sigma}$ is a sequence of Reeb chords that include the case of a single closed Reeb orbit $\{U\}$ (in which case the corresponding moduli space consists of 1-P holomorphic curves), and the pairs $(B,\overrightarrow{\sigma})$ are compatible. When $\overrightarrow{\sigma}=\{U\}$, we define $a(-U)=\bm{U}$. The underlying weakly extended type D structure is $(X^{\mathcal{E}}(\mathcal{H}),\tilde{\delta})$ where $\tilde{\delta}(\eta)\coloneqq \tilde{\partial}(1\otimes \eta)$.
\end{defn}
\begin{rmk}
Again, by abusing notation, we also use $\widetilde{CFD}(\mathcal{H})$ to denote the underlying weakly extended type D structure. When the local system is trivial, we have the following more familiar formula for the differential: 
$$\tilde{\partial} (a\otimes\bm{x})=a\cdot(\sum_{\bm{y}}\sum_{\{(B,\overrightarrow{\sigma})|\text{ind}(B,\overrightarrow{\sigma})=1\}}\#\mathcal{M}^B(\bm{x},\bm{y};\overrightarrow{\sigma})a(-\overrightarrow{\sigma})\otimes \bm{y}).$$
\end{rmk}

\begin{prop}\label{Proposition, weakly extended type D strucuture is well defined.}
The operator $\tilde{\partial}$ in Definition \ref{Definition, weakly extended type D structure} is well-defined and $\tilde{\partial}^2=\bm{U}$. 
\end{prop}
\begin{proof}
A standard argument shows that the provincial admissibility of $\mathcal{H}$ implies the sum defining $\tilde{\partial}$ in Definition \ref{Definition, weakly extended type D structure} is finite, and hence $\tilde{\partial}$ is well-defined. 

Next, we show $\tilde{\partial}^2=\bm{U}$. Once again, we first give the proof when the local system is trivial for conciseness. Recall the length of an element $a\in\tilde{\mathcal{A}}$ is the number of factors $\rho_i\in\{\rho_0,\rho_1,\rho_2,\rho_3\}$ when we write $a$ as a product of the generators $\{\iota_0,\iota_1,\rho_0,\rho_1,\rho_2,\rho_3\}$. (For example, $\rho_{123}$ has length $3$ and $\iota_0$ has length $0$.) 

For an element $a\in\tilde{\mathcal{A}}$ whose length is less than or equal to $3$, the proof of Proposition \ref{Proposition, type D structure is well defined} carries over to show $\langle \tilde{\partial}^2\bm{x}, a\bm{y}\rangle =0$ for any $\bm{x}$ and $\bm{y}$ (by permuting the region we where we put the base point $z$). 

We are left to consider the case where the algebra element is of length $4$. We claim that for a generator $\bm{x}$ such that $\iota_1\cdot \bm{x}=\bm{x}$, we have $$\langle\tilde{\partial}^2\bm{x}, \rho_{0123}\bm{y}\rangle=\begin{cases}0,\ if\  \bm{x}\neq\bm{y},\\1,\ if\ \bm{x}=\bm{y}.
\end{cases} $$
Assuming this claim, by permuting the subscripts we also have that $\langle\tilde{\partial}^2\bm{x}, \rho_{2301}\bm{y}\rangle$ is $1$ if $\bm{x}=\bm{y}$ and $0$ otherwise, and an idempotent consideration shows $\langle\tilde{\partial}^2\bm{x}, a\bm{y}\rangle=0$ when $a\in\{\rho_{1230},\rho_{3012}\}$. These together imply $\tilde{\partial}^2\bm{x}=\bm{U}\cdot \bm{x}$ when $\iota(\bm{x})=\iota_1$. A similar consideration shows this is true for $\bm{x}$ with $\iota(\bm{x})=\iota_0$ as well. This finishes the proof of the proposition modulo the claim. 

Next, we prove the claim. Note
\begin{equation}\label{Equation, coefficient of extended partial^2}
\langle\tilde{\partial}^2\bm{x}, \rho_{0123}\bm{y}\rangle=\sum_{\bm{w}}\sum_{\substack{\text{ind}(B_i,\overrightarrow{\sigma_i})=1,\\ i=1,2}} \#\mathcal{M}^{B_1}(\bm{x},\bm{w};\overrightarrow{\sigma_1})\#\mathcal{M}^{B_2}(\bm{w},\bm{y};\overrightarrow{\sigma_2}),
\end{equation}
where $(B_i,\overrightarrow{\sigma_i})$ ($i=1,2$) is compatible and $a(-\overrightarrow{\sigma_1})a(-\overrightarrow{\sigma_1})=\rho_{0123}$ or $\bm{U}$; the possible pairs of $(\overrightarrow{\sigma_1},\overrightarrow{\sigma_2})$ are listed below:
\begin{align*}
\bigg\{(\emptyset,\{-\rho_0,-\rho_1,-\rho_2,-\rho_3\}),(\{-\rho_0\},\{-\rho_1,-\rho_2,-\rho_3\}),(\{-\rho_0,-\rho_1,-\rho_2\},\{-\rho_3\}),\\
(\{-\rho_0,-\rho_1\},\{-\rho_2,-\rho_3\}),(\{-\rho_0,-\rho_1,-\rho_2,-\rho_3\},\emptyset),(\emptyset,\{-\rho_0,-\rho_{123}\}),\\(\{-\rho_0\},\{-\rho_{123}\}),(\{-\rho_0,-\rho_{123}\},\emptyset),(\emptyset,\{-\rho_{012},-\rho_{3}\}),(\{-\rho_{012}\},\{-\rho_{3}\}),\\(\{-\rho_{012},-\rho_{3}\},\emptyset),(\emptyset,\{U\}),(\{U\},\emptyset)\bigg\}.
\end{align*}
Let $$\overline{\mathcal{M}}_0\coloneqq\cup_{\text{ind}(B,\overrightarrow{\sigma})=2}\overline{\mathcal{M}}^B(\bm{x},\bm{y};\overrightarrow{\sigma}),$$ where $\overrightarrow{\sigma}\in \{(-\rho_0,-\rho_1,-\rho_2,-\rho_3),(-\rho_0,-\rho_{123}),(-\rho_{012},-\rho_{3})\}$. Let $$\overline{\mathcal{M}}_1:=\cup_{\text{ind}(B,U)=2}\overline{\mathcal{M}}^B(\bm{x},\bm{y};U).$$

Equation \ref{Equation, coefficient of extended partial^2} and the gluing result imply that $\langle\tilde{\partial}^2\bm{x}, \rho_{0123}\bm{y}\rangle$ is equal to the number of two-story ends of the moduli space $\overline{\mathcal{M}}_0\cup \overline{\mathcal{M}}_1$. 

According to Proposition \ref{Proposition, ends of 1-dimensional moduli spaces of 0-P curves, Reeb element of length 4}, the other elements in $\partial \overline{\mathcal{M}}_0$ are: 
\begin{itemize}
\item[(A-1)]simple holomorphic combs with a single split component;
\item[(A-2)]simple boundary degenerations with one corner;
\item[(A-3)]simple boundary degenerations without corners.
\end{itemize}
Proposition \ref{Proposition, ends of moduli space of 1-P curves} shows the other boundary points in $\overline{\mathcal{M}}_1$ in addition to two-story ends are: 
\begin{itemize}
\item[(B-1)]simple holomorphic combs with an orbit curve;
\item[(B-2)]simple boundary degenerations with one corner;
\item[(B-3)]simple boundary degenerations without corners.
\end{itemize}
Note by Proposition \ref{Proposition, ends of 1-dimensional moduli spaces of 0-P curves, Reeb element of length 4} and Proposition \ref{Proposition, ends of moduli space of 1-P curves}, the number of boundary points of type (A-1) is equal to that of type (B-1), both of which is $$\sum_{\text{ind}(B,-\rho_{3012})=1} \#\mathcal{M}^B(\bm{x},\bm{y};-\rho_{3012})+\sum_{\text{ind}(B,-\rho_{1230})=1} \#\mathcal{M}^B(\bm{x},\bm{y};-\rho_{1230}).$$
The parity of the number of boundary points of type (A-2) is equal to that of type (B-2), which are both mod $2$ equal to $$\sum_ {q}\sum_{\{(B_1,B_2)|B_2\in T(q),\ \text{ind}(B_1+B_2;U)=2\}} \#\mathcal{M}^{B_1}(\bm{x},\bm{y};q),$$
where $q$ ranges over self-intersection points of $\alpha_{im}$, and $T(q)$ denotes the set of stabilized teardrops at $q$. 

The parity of the number of boundary points of type (B-3) is even according to Proposition \ref{Proposition, ends of moduli space of 1-P curves}. 

In summary, the parity of the number of boundary points of $\overline{\mathcal{M}}_0\cup \overline{\mathcal{M}}_1$ corresponding to two-story ends is equal to that of type (A-3), which is odd if and only if $\bm{x}=\bm{y}$ by Proposition \ref{Proposition, ends of 1-dimensional moduli spaces of 0-P curves, Reeb element of length 4}. Therefore, $\langle\tilde{\partial}^2\bm{x}, \rho_{0123}\bm{y}\rangle$ is odd if and only if $\bm{x}=\bm{y}$, finishing the proof of the claim. 

In the presence of non-trivial local systems, we simply need to consider the above argument for each domain. For a domain $B$, let $\overline{\mathcal{M}}_0^B$ be the subset of $\overline{\mathcal{M}}_0$ consisting of holomorphic curves with domain $B$, and similarly define $\overline{\mathcal{M}}_1^B$. The two-story ends in $\overline{\mathcal{M}}_0^B\cup \overline{\mathcal{M}}_1^B$ all correspond to the same parallel transport. When ends of type (A-3) do not occur, the two-story ends cancel in pairs by the same argument as above. When ends of type (A-3) appear, we have $B=[\Sigma]$ and $\sigma=(-\rho_0,-\rho_1,-\rho_2,-\rho_3)$. In particular, $\partial_{\alpha_{im}}B=\emptyset$, which induces the identity endomorphism of $\mathcal{E}|_{\bm{x}}$. Also, the number of two-story ends is odd as the number of (A-3) ends is odd. The claim follows from these. 
\end{proof}

There is a canonical quotient map $\pi:\tilde{\mathcal{A}}\rightarrow \mathcal{A}$. We say a weakly extended type D structure $(N,\tilde{\delta})$ extends a type D structure $(N',\delta)$ if $(N',\delta)$ is isomorphic to $(N,(\pi\otimes \mathbb{I}_N)\circ \tilde{\delta})$. Clearly, $\widetilde{CFD}(\mathcal{H})$ extends $\widehat{CFD}(\mathcal{H})$ when both are defined.

\subsection{Invariance}
In this subsection, we address the invariance of the (weakly extended) type D structures. 

\begin{prop}\label{Proposition, invariance of type D structure}
The homotopy type of the type D structure defined in Definition \ref{Definition, type D structure} is independent of the choice of the almost complex structure and is invariant under isotopy of the $\alpha$- or $\beta$-curves. 
\end{prop}

\begin{rmk}
We do not need the invariance under handleslides and stabilizations for our applications. We only need to prove invariance when perturbing diagrams to obtain nice diagrams, and this only requires isotopies.
\end{rmk}
\begin{proof}[Proof of Proposition \ref{Proposition, invariance of type D structure}]
The standard proof in Section 6.3 of \cite{LOT18} carries over. For instance, to prove independence of almost complex structures, one first constructs a continuation map by counting holomorphic curves in $\Sigma\times[0,1]\times \mathbb{R}$ for a generic almost complex structure $J$ that interpolates two admissible almost complex structures $J_0$ and $J_1$. Then, one proves the continuation map is a chain map by analyzing the ends of one-dimensional moduli spaces. The only possible complication comes from boundary degenerations since $\alpha_{im}$ is immersed. However, this does not happen as $\mathcal{H}$ is unobstructed and the holomorphic curves have $n_z=0$. Therefore, no new phenomenon appears in the degeneration of moduli spaces, and hence the proof stays the same.
\end{proof}
\begin{prop}\label{Proposition, invariance of weakly type D structures}
The homotopy type of the weakly extended type D structure defined in Definition \ref{Definition, weakly extended type D structure} is independent of the choice of the almost complex structure and is invariant under isotopy of the $\alpha$- and $\beta$-curves.  
\end{prop}
\begin{proof}
One could prove this proposition similarly to the previous one. However, such an approach would require generalizing the analysis of the ends of moduli spaces in Proposition 6.20 of \cite{LOT18} and hence is slightly tedious to write down. Here we give a different approach. Let $\mathcal{H}$ denote the immersed bordered Heegaard diagram. By Proposition \ref{Proposition, invariance of type D structure}, we know the homotopy type of $\widehat{CFD}(\mathcal{H})$ is independent of the choice of almost complex structures and isotopy of the $\alpha$- or $\beta$-curves. Since $\widetilde{CFD}(\mathcal{H})$ extends $\widehat{CFD}(\mathcal{H})$ and that such extension is unique up to homotopy by Proposition 38 of \cite{hanselman2016bordered}, we know the homotopy type of $\widetilde{CFD}(\mathcal{H})$ is also independent of the choice of almost complex structures and isotopy of the $\alpha$- and $\beta$-curves.
\end{proof}
\section{Knot Floer homology of immersed Heegaard diagrams}\label{Section, knot Floer chain complex of immersed Heegaard diagrams}
This section defines knot Floer chain complexes of immersed Heegaard diagrams and proves the homotopy invariance under Heegaard moves.
\subsection{Immersed doubly-pointed Heeggard diagram}
\begin{defn}\label{Definition, generalized doubly pointed Heegaard diagram}
An \emph{immersed doubly-pointed Heegaard diagram} is a 5-tuple $\mathcal{H}_{w,z}=(\Sigma,\bm{\alpha},\bm{\beta},w,z)$ where 
\begin{itemize}
\item[(1)]$\Sigma$ is a closed oriented surface of genus $g$.
\item[(2)]$\bm{\alpha}=\{\alpha_1,\ldots,\alpha_{g-1},\alpha_g\}$, where $\alpha_1,\ldots,\alpha_{g-1}$ are embedded disjoint curves in $\Sigma$ and $\alpha_g=\{\alpha_g^1,\ldots,\alpha_g^n\}$ is a collection of immersed curves decorated with local systems. Moreover, $\alpha_i$ ($i=1,\ldots,g-1$) are disjoint from $\alpha_g$, $\alpha_g^1$ has the trivial local system, and $\{\alpha_1,\ldots,\alpha_{g-1},\alpha_{g}^1\}$ induce linearly independent elements in $H_1(\Sigma,\mathbb{Z})$. We also assume that $\alpha_g^i$ is trivial in $H_1(\Sigma, \mathbb{Z})/\langle \alpha_1, \ldots, \alpha_{g-1} \rangle$ for $i > 1$. For convenience, we also denote $\alpha_{g}$ by $\alpha_{im}$.
\item[(3)]$\bm{\beta}=\{\beta_1,\ldots,\beta_g\}$ are embedded disjoint curves in $\Sigma$ which induce linearly independent elements in $H_1(\Sigma,\mathbb{Z})$.
\item[(4)]$w$ and $z$ are base points such that they both lie in a single connected region in the complement of $\alpha$-curves as well as a single region in the complement of $\beta$-curves.  
\end{itemize}
\end{defn}
Domains, periodic domains, and $\alpha$-bounded domains are defined similarly in this setting as for bordered Heegaard diagrams (by ignoring $\alpha$-arcs and surface boundary). We make a similar but slightly different definition of unobstructedness and admissibility below. 
 
\begin{defn}\label{Definition, unobstructedness for doubly-pointed diagram}
Given an immersed doubly-pointed Heegaard diagram, $\bm{\alpha}$ is called \emph{unobstructed} if there are no nontrivial zero- or one-cornered $\alpha$-bounded domains $B$ with $n_z(B)=0$ (or equivalently $n_w(B)=0$). An immersed doubly-pointed Heegaard diagram is called unobstructed if $\bm{\alpha}$ is unobstructed. 
\end{defn}
\begin{defn}\label{Definition, admissibility for doubly-pointed diagram}
An immersed doubly-pointed Heegaard diagram is \emph{bi-admissible} if any nontrivial periodic domain $B$ with $n_z(B)=0$ or $n_w(B)=0$ has both positive and negative coefficients.
\end{defn}

We remark that the restriction to having only one immersed multicurve in the definition of immersed doubly-pointed Heegaard diagrams is not essential. 
\subsection{The knot Floer chain complex}
We define the knot Floer chain complex of an immersed Heegaard diagram similar to that in the ordinary setup. The only modification is that we only count stay-on-track holomorphic curves. The definition and analysis of moduli spaces in this setup is a straightforward modification of that in the previous section; it is even simpler as we do not need to care about east punctures. We hence do not repeat the moduli space theory but only mention the key properties when we need them. We will let $\mathcal{G}(\mathcal{H}_{w,z})$ denote the set of generators, which are $g$-tuples $(x_1,\ldots,x_g)$ such that $x_i\in \alpha_i\cap \beta_{\sigma(i)}$ $(i=1,\ldots,g)$ where $\sigma$ is a permutation of $\{1,\ldots,g\}$. Let $\mathcal{R}=\mathbb{F}[U,V]/(UV)$. Implicit in the definition below is that we choose a generic admissible almost complex structure $J$ on $\Sigma\times[0,1]\times \mathbb{R}$.
\begin{defn}\label{Definition, CFK_R}
Let $\mathcal{H}_{w,z}$ be an unobstructed and bi-admissible immersed doubly-pointed Heegaard diagram. $CFK_{\mathcal{R}}(\mathcal{H}_{w,z})$ is the free $\mathcal{R}$-module generated over $\mathcal{G}(\mathcal{H}_{w,z})$ with differential $\partial$ defined as $$\partial \bm{x} =\sum_{y}\sum_{B\in\pi_2(\bm{x},\bm{y}),\ \text{ind}(B)=1} \#\mathcal{M}^B(\bm{x},\bm{y})U^{n_w(B)}V^{n_z(B)}\bm{y},$$
where $\bm{x},\bm{y}\in\mathcal{G}$.  
\end{defn}  
\begin{rmk}
Here we only give the definition assuming the local system on $\alpha_{im}$ is trivial. The case in which the local system is non-trivial is only notationally more complicated, and we leave it for the interested readers to work out. See Definition \ref{Definition, type D structure} for an example.    
\end{rmk}
\begin{prop}\label{Proposition, well-definess of partial for CFK_R}
$(CFK_\mathcal{R}(\mathcal{H}_{w,z}),\partial)$ is a chain complex, i.e., $\partial^2=0$.
\end{prop}
\begin{proof}
The same proof for Proposition \ref{Proposition, type D structure is well defined} works here. Note we will only use moduli spaces with domains $B$ such that $n_w(B)=0$ or $n_z(B)=0$, and the unobstructedness of $\mathcal{H}_{w,z}$ excludes the possibility of boundary degeneration in the compactified 1-dimensional moduli space supported in such domains. Hence, an analogue version of Proposition \ref{Proposition, ends of 0-p curves, n_z=0} holds. With this observation, the proof of Proposition \ref{Proposition, type D structure is well defined} carries over.  
\end{proof}
\subsection{Bi-grading}
We would like to consider gradings on knot Floer chain complexes.

\begin{defn}\label{Definition, gradable doubly-pointed diagram}
A (possibly immersed) doubly-pointed Heegaard diagram is gradable if all non-trivial periodic domain $P$ satisfies $\text{ind}(P)-2n_z(P)=0$ and $\text{ind}(P)-2n_w(P)=0$, where $\text{ind}(-)$ is defined in Definition \ref{Definition, embedded Euler Char, index, and moduli space}. 
\end{defn}

If $\mathcal{H}_{w,z}$ is gradable then the knot Floer chain complex $(CFK_\mathcal{R}(\mathcal{H}_{w,z}),\partial)$ admits a relative $\mathbb{Z}\oplus\mathbb{Z}$-grading, as described below. We will be interested in diagrams  $\mathcal{H}_{w,z}$  for which $\widehat{HF}(\mathcal{H}_{w})\cong \widehat{HF}(\mathcal{H}_{z})\cong \mathbb{F}$, where $\widehat{HF}(\mathcal{H}_{w})$ and $\widehat{HF}(\mathcal{H}_{z})$ are homology groups of the chain complexes obtained from $CFK_\mathcal{R}(\mathcal{H}_{w,z})$ by setting $V=0$ and $U=1$ or $U=0$ and $V=1$, respectively. In this case we say that the horizontal and vertical homology has rank one. Gradable diagrams with this property can be given an absolute grading, as follows.

\begin{defn}\label{Definition, w- and z- gradings}
	Let $\bm{x},\bm{y}\in \mathcal{G}(\mathcal{H}_{w,z})$ be two generators. Let $B\in \tilde{\pi}_2(\bm{x},\bm{y})$ be a domain. Then the $w$-grading difference between $\bm{x}$ and $\bm{y}$ is given by $$gr_w(\bm{x})-gr_w(\bm{y})=\text{ind}(B)-2n_w(B),$$ and
	the $z$-grading difference between $\bm{x}$ and $\bm{y}$ is given by $$gr_z(\bm{x})-gr_z(\bm{y})=\text{ind}(B)-2n_z(B).$$
	If the horizontal and vertical homology of $\mathcal{H}_{w,z}$ is rank one, then the absolute $w$-grading is normalized so that $\widehat{HF}(\mathcal{H}_{w})$ is supported in $w$-grading $0$, and absolute $z$-grading is normalized so that $\widehat{HF}(\mathcal{H}_{z})$ is supported in $z$-grading $0$.
\end{defn}
Equivalently, one can equip $CFK_{\mathcal{R}}(\mathcal{H}_{w,z}(\alpha_{K}))$ with the \textit{Maslov grading} and the \textit{Alexander grading}. These two gradings can be expressed in terms of the $w$-grading and $z$-grading: The Maslov grading is equal to the $z$-grading, and the Alexander grading is given by $\frac{1}{2}(gr_w-gr_z)$.  

\begin{rmk}
The normalization conditions for the absolute gradings are chosen so that the bi-graded chain complexes model those associated to knots in the 3-sphere. 
\end{rmk}
\subsection{Invariance}
We will show knot Floer chain complexes defined over immersed Heegaard diagrams satisfy similar invariance properties when varying the almost complex structure or modifying the Heegaard diagram by isotopy, handleslides, and stabilizations. While the meaning of isotopy and stabilization are obvious for immersed Heegaard diagrams, we give a remark on handleslides. 
\begin{rmk}
When speaking of handleslides of an immersed Heegaard diagram $\mathcal{H}_{w,z}$, we only allow an $\alpha$-curve to slide over another \emph{embedded} $\alpha$-curve, not over an immersed $\alpha$-curve. Furthermore, we point out that handle-slides do not change the unobstructedness, bi-admissibility, and gradability of the diagram. To see this, note periodic domains of two Heegaard diagrams before and after a handleslide are related. A periodic domain in the old Heegaard diagram with boundary on the arc that moves in the handleslide give rise to a periodic domain in the new Heegaard diagram by boundary summing a thin annulus (whose multiplicity can be one or negative one). In particular, if we started from a somewhere negative domain $B$, then the new domain $B'$ after this procedure is still somewhere negative; it is also easy to see $\text{ind}(B)=\text{ind}(B')$, $n_z(B)=n_z(B')$, and $n_w(B)=n_w(B')$, which implies the gradability of two diagrams are the same as well.  
\end{rmk} 

\begin{prop}\label{Proposition, invariance of CFK_R} Let $\mathcal{H}_{w,z}$ be an unobstructed, bi-admissible, and gradable immersed doubly-pointed Heegaard diagram.
The bigraded chain homotopy type of $CFK_{\mathcal
R}(\mathcal{H}_{w,z})$ is invariant under varying the almost complex structure, isotopy of the $\alpha$- and $\beta$-curves, handleslides, and stabilization/destabilization.  
\end{prop}
\begin{proof}
The proof of the bigraded homotopy invariance under the variation of the almost complex structure, isotopy, and stabilization is the same as the corresponding results in the embedded-$\alpha$-curve set-up in \cite{MR2240908}. In fact, changing the $\alpha$-curves from embedded to immersed can only complicate the arguments in that boundary degeneration might appear as ends of the moduli spaces involved, yet the unobstructedness dispels such worries.  
\begin{figure}[htb!]
\centering{
\includegraphics[scale=0.6]{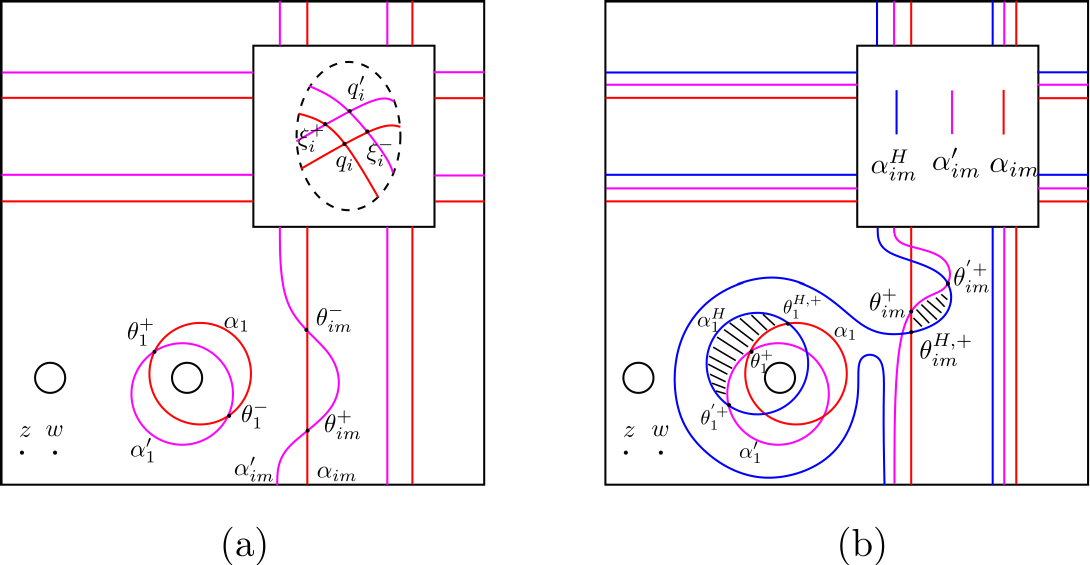}
\caption{The $\alpha$-curves in proving the handleslide invariance on a genus-two surface, which is represented as a torus obtained by identifying the edges of a square together with a handle attached to the two circles inside the square. The labels $\theta_{im}^{\pm}$ are used interchangeably with $\theta_2^{\pm}$. Similarly, $\theta_{im}^{H,\pm}$ and $\theta_{im}^{'\pm}$ are the same as $\theta_{2}^{H,\pm}$ and $\theta_{2}^{'\pm}$, respectively. (a) shows any self-intersection point $q_i$ of $\alpha_{im}$ induces two intersection points between $\alpha_{im}$ and its perturbation $\alpha'_{im}$. (b) shows the small triangles showing $F(\Theta_{\alpha',\alpha^H}\otimes\Theta_{\alpha^H,\alpha})=\Theta_{\alpha',\alpha}$}.
\label{Figure, handleslide}
}
\end{figure}

The handleslide invariance can also be proved using the same strategy as in the embedded-$\alpha$-curve case with slightly more caution. The main difference is that in the embedded-$\alpha$-curve case, there is a unique maximal graded generator in the Heegaard Floer homology of a Heegaard diagram where the set of $\alpha$-curves is a small Hamiltonian perturbation of the $\beta$-curves. In contrast, such a generator needs to be specified more carefully in our case. We spell this out in more detail.

Denote $\mathcal{H}=(\Sigma,\bm{\alpha},\bm{\beta},w,z)$. For clarity of exposition, assume $\alpha_{im}$ consists of a single component with a trivial local system and $n$ self-intersection points. We also restrict to the interesting case, in which the handleslide is sliding $\alpha_{im}$ over an embedded $\alpha$-curve. Let $\bm{\alpha'}$ denote a small Hamiltonian perturbation of $\bm{\alpha}$ so that $\alpha_i\cap \alpha_j=\emptyset$ for $i\neq j$; for $i=1,\ldots,g-1$, the embedded curves $\alpha_i$ and $\alpha_i'$ intersects exactly at two points $\{\theta_i^+,\theta_i^-\}$; $\alpha_{im}$ intersects $\alpha_{im}'$ at $2+2n$ points $\{\theta_g^+,\theta_g^-, \xi_1^+,\xi_1^-,\ldots,\xi_n^+,\xi_n^-\}$, where $\xi_i^{\pm}$ are intersection points corresponding to the self-intersection points of $\alpha_{im}$. We label the $\theta$-intersection points using the convention so that $(\theta_i^+,*)$ is of higher grading than $(\theta_i^-,*)$ in $CFK_\mathcal{R}(\Sigma,\bm{\alpha'},\bm{\alpha},w,z)$, $(i=1,\ldots,g)$ (see Figure \ref{Figure, handleslide} (a)). Let $\alpha_{im}^H$ denote the curve obtained by sliding $\alpha_{im}$ over, say, $\alpha_{g-1}$, so that $\alpha_{im}^H$ intersects each of $\alpha_{im}$ and $\alpha_{im}'$ in $2+2n$ points; denote the $\theta$-intersection points by $\{\theta_g^{H,+},\theta_g^{H,-}\}$ and $\{\theta_g^{'+},\theta_g^{'-}\}$, respectively. Let $\alpha_i^H$ ($i=1,\ldots, g-1$) be small Hamiltonian perturbations of $\alpha_i'$ so that $\alpha_i^H$ intersects each of $\alpha_i$ and $\alpha_i'$ at exactly two points, denoted by $\{\theta_i^{H,+},\theta_i^{H,-}\}$ and $\{\theta_i^{'+},\theta_i^{'-}\}$, respectively. Let $\Theta_{\alpha',\alpha}=(\theta_1^+,\ldots,\theta_g^+)$, $\Theta_{\alpha^H,\alpha}=(\theta_1^{H,+},\ldots,\theta_g^{H,+})$, and $\Theta_{\alpha',\alpha^H}=(\theta_1^{'+},\ldots,\theta_g^{'+})$. These correspond to the maximal graded intersection points used in the embedded case.\footnote{A straightforward computation would show $\Theta_{\alpha,\alpha'}$ are indeed cycles in the Floer chain complex associated to the immersed Heegaard diagram $(\Sigma, \bm{\alpha}$, $\bm{\alpha'},w,z)$; similar statements hold for $\Theta_{\alpha^H,\alpha}$ and $\Theta_{\alpha',\alpha^H}$.}

The rest of the proof is similar to the embedded case. We provide a sketch. Let $\mathcal{H}^H=(\Sigma,\bm{\alpha}^H,\bm{\beta},w,z)$ and $\mathcal{H}'=(\Sigma,\bm{\alpha'},\bm{\beta},w,z)$. By counting holomorphic triangles (with stay-on-track boundaries), one can define chain maps
$$F(\Theta_{\alpha^H,\alpha}\otimes -): CFK_\mathcal{R}(\mathcal{H})\rightarrow CFK_\mathcal{R}(\mathcal{H}^H)$$
and  
$$F(\Theta_{\alpha',\alpha^H}\otimes -): CFK_\mathcal{R}(\mathcal{H}^H)\rightarrow CFK_\mathcal{R}(\mathcal{H}')$$
Again, the usual proof which shows the above maps are chain maps carries through, as the unobstructedness excludes boundary degeneration when analyzing the ends of one-dimensional moduli spaces of holomorphic triangles. Similarly, by analyzing ends one-dimensional moduli spaces of holomorphic quadrilaterals, one can show the composition of these two maps is chain homotopic equivalent to $F(F(\Theta_{\alpha',\alpha^H}\otimes\Theta_{\alpha^H,\alpha})\otimes -)$. One can show this map is homotopic equivalent to $$F(\Theta_{\alpha',\alpha}\otimes -): CFK_\mathcal{R}(\mathcal{H})\rightarrow CFK_\mathcal{R}(\mathcal{H}')$$ by a standard computation which shows $F(\Theta_{\alpha',\alpha^H}\otimes\Theta_{\alpha^H,\alpha})=\Theta_{\alpha',\alpha}$ (see Figure \ref{Figure, handleslide} (b)). One can show that the map $F(\Theta_{\alpha',\alpha}\otimes -)$ is a chain isomorphism (using the area-filtration technique in \cite{OS04}, Proposition 9.8). 
\end{proof} 

\section{Paring theorems}\label{Section, pairing theorem section}
In Section \ref{Subsection, marked torus and immersed curves}--\ref{Subsection, pairing construction}, we introduce a pairing construction which merges a (non-immersed) bordered Heegaard diagram and an immersed multicurve to produce an immersed Heegaard diagram. After that, we establish the unobstructedness and admissibility of these pairing diagrams in Section \ref{Subsection, z-adjacency}--\ref{Subsection, unobstructeness and admissbility of pairing diagrams}, and then we prove the bordered invariant of such pairing diagrams admits a box-tensor product interpretation in Section \ref{Subsection, first pairing theorem}. Finally, in Section \ref{Section, the second pairing theorem} we prove a pairing theorem for gluing a particular type of doubly-pointed bordered Heegaard diagram and an immersed bordered Heegaard diagram; this theorem will be useful in Section \ref{Section, satellite knot pairing}. 

\subsection{Immersed curves in the marked torus}\label{Subsection, marked torus and immersed curves}
\begin{defn}
The \emph{marked torus} $T^2$ is the oriented surface $\mathbb{R}^2/\mathbb{Z}^2$ together with a base point $z$ located at $(1-\epsilon,1-\epsilon)$ for some sufficiently small $\epsilon>0$. The images of the positively oriented $x$-axis and $y$-axis are called the \emph{preferred longitude} and \emph{preferred meridian} respectively.  
\end{defn}
We will consider immersed multicurves with local systems in the marked torus. Two immersed multicurves are \textit{equivalent} if they are regularly homotopic in $T^2\backslash{z}$ and the local systems are isomorphic. Throughout this paper, we restrict to immersed multicurves $\alpha_{im}$ satisfying the following assumptions: 
\begin{itemize}
	\item[(C-1)]No component of $\alpha_{im}$ is a circle enclosing the base point $z$ once.
	\item[(C-2)]No component of the immersed multicurve is null-homotopic in $T^2\backslash\{z\}$, and the immersed multicurve is \emph{unobstructed} in the sense that it does not bound any teardrops in $T^2\backslash\{z\}$.
	\item[(C-3)]The immersed multicurve is \emph{reduced}, i.e., if we let $[0,1]\times[0,1]$ be the square obtained by cutting the marked torus open along the preferred meridian and longitude, then no sub-arcs of $\alpha_{im}$ contained in $[0,1]\times[0,1]$ have both ends on the same edge of the square.  
	\item[(C-4)]Let $\pi$ denote the projection map from $\mathbb{R}^2$ to $T^2$. Using regular homotopy, we assume all immersed curves in the marked torus are contained in the complement of $\pi([-\frac{1}{4},\frac{1}{4}]\times [-\frac{1}{4},\frac{1}{4}])$ in $T^2$, the strands contained in $\pi([-\frac{1}{4},\frac{1}{4}]\times [\frac{1}{4},\frac{3}{4}])$ are horizontal, and the strands contained in the image of $\pi([\frac{1}{4},\frac{3}{4}]\times [-\frac{1}{4},\frac{1}{4}])$ are vertical.
\end{itemize} 
An immersed multicurve in the marked torus determines a type D structure over the torus algebra as follows. First, we introduce some terminology.

\begin{figure}[htb!]
	\centering{
		\includegraphics[scale=0.35]{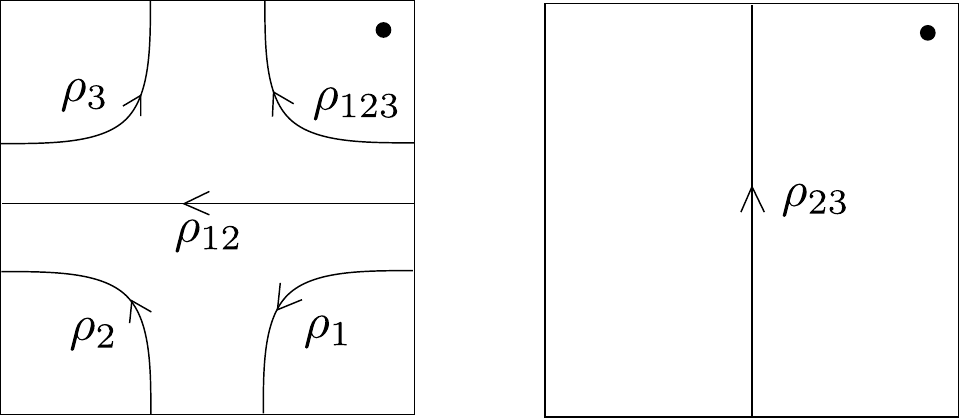}
		\caption{Six types of elementary arcs. The orientations are the so-called correct orientations.}
		\label{Figure, elemenatary arcs}
	}
\end{figure}
\begin{defn}
	An \emph{elementary arc} is an embedded arc in the marked torus ${T}^2$ such that it only intersects the preferred meridian or longitude at the endpoints. There are six types of elementary arc based on the position of the endpoints, each of which is labeled by a Reeb chord in $\{\rho_1,\rho_2,\rho_3,\rho_{12},\rho_{23},\rho_{123}\}$ as shown in Figure \ref{Figure, elemenatary arcs}.
\end{defn}
If we ignore the local systems, then any immersed multicurve is comprised of a collection of elementary arcs; one can see this by cutting $T^2$ open along the preferred longitude and meridian. Sometimes we also need to consider oriented elementary arcs.

\begin{defn}
An orientation of an elementary arc is called the correct orientation if it is the one shown in Figure \ref{Figure, elemenatary arcs}.
\end{defn}

Next, we describe how to obtain a type D structure from an immersed multicurve in terms of elementary arcs. Denote the local system on $\alpha_{im}$ by $(\mathcal{E},\Phi)$, where $\mathcal{E}$ is a vector bundle over $\alpha_{im}$ and $\Phi$ is a parallel transport. Let $\mathcal{G}(\alpha_{im})=\mathcal{G}_m\cup \mathcal{G}_l$, where $\mathcal{G}_m$ (respectively, $\mathcal{G}_l$)  is  the set of intersection points of $\alpha_{im}$ and the preferred meridian (respectively, longitude). Let $\mathcal{X}$ be the vector space $\oplus_{x\in \mathcal{G}(\alpha_{im})}\mathcal{E}|_x$. Next, we define an $\mathcal{I}$-action on $\mathcal{X}$, where $\mathcal{I}$ is the ring of idempotent of the torus algebra. If $x\in \mathcal{G}_m$, for any $\tilde{x}\in\mathcal{E}|x$, $\iota_0\cdot \tilde{x}=\tilde{x}$ and $\iota_1\cdot \tilde{x}=0$; if $x\in \mathcal{G}_l$, for any $\tilde{x}\in\mathcal{E}|x$, $\iota_0\cdot \tilde{x}=0$ and $\iota_1\cdot \tilde{x}=\tilde{x}$. The underlying $\mathcal{A}$-module for $\widehat{CFD}(\alpha_{im})$ is $\mathcal{A}\otimes_\mathcal{I}\mathcal{X}$. Finally, the differential on $\widehat{CFD}(\alpha_{im})$ decomposes linearly as maps between $\mathcal{E}|_x$ for $x\in \mathcal{G}(\alpha_{im})$. Given $x,y\in \mathcal{G}(\alpha_{im})$ and $\rho_I$ a Reeb element, there is a differential map $\mathcal{E}|_x\rightarrow \rho_{I}\otimes \mathcal{E}|_y$ if and only if $x$ and $y$ are connected by a $\rho_I$-elementary arc whose correct orientation goes from $x$ to $y$, in which case the differential is given by $\partial (\tilde{x})=\rho_{I}\otimes \Phi(\tilde{x})$ for $\tilde{x}\in \mathcal{E}|_x$. In particular, when the local system of $\alpha_{im}$ is trivial, then the generators of $\widehat{CFD}(\alpha_{im})$ are in one-to-one correspondence with the intersection points of $\alpha_{im}$ with the preferred longitude/meridian, and the differentials are in one-to-one correspondence with the elementary sub-arcs of $\alpha_{im}$.

The immersed-curve presentation of type D structures is empowered by the following result. 
\begin{thm}[\cite{hanselman2016bordered}]
Each Type D structure of a bordered 3-manifold with torus boundary is homotopic to a type D structure determined by some immersed multicurve (with local systems) in the marked torus.
\end{thm} 

\begin{rmk}\label{Remark, no circle type immersed curve for knot complements}
	All immersed multicurves arising from 3-manifolds with torus boundary satisfies the assumptions (C-1)-(C-4): (C-4) is straightforward, (C-2) and (C-3) follows from the algorithm of converting type D structures to immersed multicurves in \cite{hanselman2016bordered}, and for (C-4) see the discussion around Figure 31 and 32 in \cite{Hanselman2022}. 
\end{rmk}
We will mainly be interested in the immersed multicurves corresponding to type D structures of knot complements for knots in the 3-sphere; these immersed multicurves satisfy some further properties that we specify in Definition \ref{Definition, immersed curve in the marked torus} below, and the proofs of these properties can be found in \cite[Section 4]{Hanselman2022}. 
\begin{defn}\label{Definition, immersed curve in the marked torus}
	An immersed multicurve $\alpha_{im}=\{\alpha_{im}^0,\ldots,\alpha_{im}^{n-1}\}$ of $n$ components (for some $n\geq 1$) with a local system is called knot-like if the local system restricted to ${\alpha_{im}^0}$ is trivial, $\alpha_{im}^0$ (with some orientation) is homologous to the preferred longitude in $T^2$, and $[\alpha_{im}^i]$ for $i\geq 1$ is trivial in $H_1(T^2,\mathbb{Z})$. 
\end{defn}
From now on, we assume all immersed multicurves are knot-like. 
\subsection{Pairing diagrams}\label{Subsection, pairing construction}
We introduce a class of immersed bordered Heegaard diagrams and doubly pointed Heegaard diagrams. They are respectively obtained from two types of pairing constructions that we will define:
\begin{itemize}
\item[(1)]Pairing an immersed multicurve in the marked torus and an \emph{arced bordered Heegaard diagram with two boundary components} to construct an immersed bordered Heegaard diagram.
\item[(2)]Paring an immersed multicurve in the marked torus with a doubly pointed bordered Heegaard diagram to construct a closed immersed doubly pointed Heegaard diagram.
\end{itemize}

We begin with the first type. For convenience, we first recall the definition of arced bordered Heegaard diagrams below (in the special case where both boundaries of the corresponding bordered manifold are tori).
\begin{defn}
An arced bordered Heegaard diagram with two boundary components is a quadruple $\mathcal{H}^a=(\bar{\Sigma},\bar{\bm{\alpha}},\bm{\beta},\bm{z})$ where 
\begin{itemize}
\item[(1)]$\bar{\Sigma}$ is a compact, oriented surface of genus $g$ with two boundary components $\partial\bar{\Sigma}=\partial_L\bar{\Sigma}\cup\partial_R\bar{\Sigma}$;
\item[(2)]$\bar{\bm{\alpha}}$ is a collection of pairwise disjoint properly embedded arcs and curves $\{\alpha^{a,L}_1,\alpha^{a,L}_2,\alpha^{a,R}_1,\alpha^{a,R}_2,\alpha^c_1,\ldots,\alpha^c_{g-2}\}$. Here, $\alpha^{a,L}_1$ and $\alpha^{a,L}_2$ are two arcs with endpoints on $\partial_L\bar{\Sigma}$, $\alpha^{a,R}_1$ and $\alpha^{a,R}_2$ are two arcs with endpoints on $\partial_R\bar{\Sigma}$, and the $\alpha^c_i$'s ($i=1,\ldots,g-2$) are embedded circles. Moreover, elements in $\bar{\bm{\alpha}}$ induce linearly independent elements in $H_1(\bar{\Sigma},\partial\bar{\Sigma};\mathbb{Z})$;
\item[(3)]$\bm{\beta}$ is a set of $g$ pairwise disjoint embedded circles $\{\beta_1,\ldots,\beta_g\}$ in the interior of $\bar{\Sigma}$ that are linearly independent as elements in $H_1(\bar{\Sigma},\partial\bar{\Sigma};\mathbb{Z})$;
\item[(4)]$\bm{z}$ is a properly embedded arc in $\bar{\Sigma}\backslash(\bar{\bm{\alpha}}\cup\bm{\beta})$ with one endpoint $z_L$ on $\partial_L\bar{\Sigma}$ and the other endpoint $z_R$ on $\partial_R\bar{\Sigma}$.
\end{itemize}
\end{defn}
Periodic and provincially period domains for arced bordered Heegaard diagrams with two boundary components are defined similarly to the case of a single boundary component. In the two boundary case we will also consider periodic domains that are adjacent to only one of the boundaries.

\begin{defn}
A domian is \emph{left provincial} if the multiplicity in the regions adjacent to $\partial_L \bar\Sigma$ are zero. We say an arced bordered Heegaard diagrams with two boundary components is \emph{left provincially admissible} if all left provincial periodic domains have both positive and negative multiplicities.
\end{defn}

The pairing construction is illustrated in Figure \ref{Figure, pairing bordered diagram}, and is spelled out in Definition \ref{Definition, pairing diagram}.  
\begin{figure}[htb!]
\centering{
\includegraphics[scale=0.5]{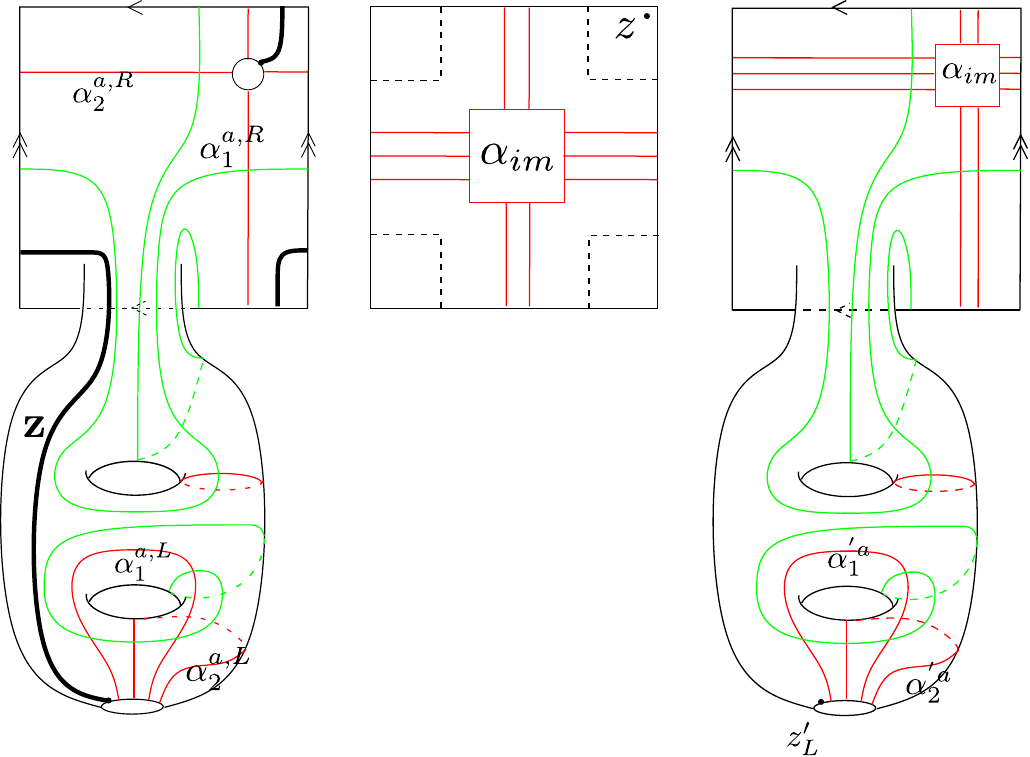}
\caption{Left: an arced bordered Heegaard diagram. Middle: an immersed multicurve in the marked torus. The dashed lines are the boundary of $\pi([-\frac{1}{4},\frac{1}{4}]\times [-\frac{1}{4},\frac{1}{4}])$. Right: a bordered Heegaard diagram obtained by the pairing construction.}
\label{Figure, pairing bordered diagram}
}
\end{figure}

\begin{defn}\label{Definition, pairing diagram}
Let $\mathcal{H}^a=(\bar{\Sigma},\bar{\bm{\alpha}},\bm{\beta},\bm{z})$ be an arced bordered Heegaard diagram with two boundary components and let $\alpha_{im}$ be an immersed multicurve in the marked torus $T^2$. The \emph{pairing diagram of $\mathcal{H}^a$ and $\alpha_{im}$}, denoted by $\mathcal{H}^a(\alpha_{im})$, is a bordered Heegaard diagram obtained through the following steps. 
\begin{itemize}
\item[(1)]Form $\bar{\Sigma}'$ from $\bar{\Sigma}$ by collapsing $\partial_R\bar{\Sigma}$. Let $\alpha'^{a}_i$ be the image of $\alpha^{a,L}_i$ ($i=1,2$), $\alpha'^{c}_i$ be the image of $\alpha^c_i$ ($i=1,\ldots,g-2$), $\bm{\beta}'$ be the image of $\bm{\beta}$, and $z'_L$ be the image of $z_L$. The images of $\alpha^{a,R}_i$ ($i=1,2$), denoted by $\tilde{\alpha}_i$, are two circles intersecting at a single point ${z}'_R$, the image of $z_R$.
\item[(2)]Take a neighborhood $U$ of $\tilde{\alpha}_1\cup\tilde{\alpha}_2$ which admits a homeomorphism $h:U\rightarrow T^2\backslash \pi([-\frac{1}{4},\frac{1}{4}]\times [-\frac{1}{4},\frac{1}{4}])$ such that $h(\tilde{\alpha}_1)=\pi(\{\frac{1}{2}\}\times[0,1])$, $h(\tilde{\alpha}_2)=\pi([0,1]\times\{\frac{1}{2}\})$, and each connected component of $h(\bm{\beta}'\cap U)$ is an arc of the form $\pi(\{x\}\times [\frac{1}{4},\frac{3}{4}])$ or $\pi([\frac{1}{4},\frac{3}{4}]\times\{y\})$ for some $x$ or $y$ in $(2\epsilon,\frac{1}{4})$.
\item[(3)]Let $\alpha_{im}'=h^{-1}(\alpha_{im})$. Let $\bar{\bm{\alpha}}'=\{\alpha'^{a}_1,\alpha'^{a}_2,\alpha'^c_1,\ldots,\alpha'^c_{g-1},\alpha_{im}'\}$.
\item[(4)]Let $\mathcal{H}^a(\alpha_{im})=(\bar{\Sigma}',\bar{\bm{\alpha}}',\bm{\beta}',z'_L)$.
\end{itemize} 
\end{defn}

Recall a \emph{doubly pointed bordered Heegaard diagram} is a bordered Heegaard diagram with an extra basepoint in the complement of the $\alpha$- and $\beta$-curves. It encodes a knot in a bordered 3-manifold. There is an entirely similar pairing construction for a doubly-pointed bordered Heegaard diagram and an immersed multicurve in the marked torus. We do not spell out the wordy definition and simply refer the readers to Figure \ref{Figure, pairing doubly pointed diagram} for an example. 

\begin{figure}[htb!]
	\centering{
		\includegraphics[scale=0.45]{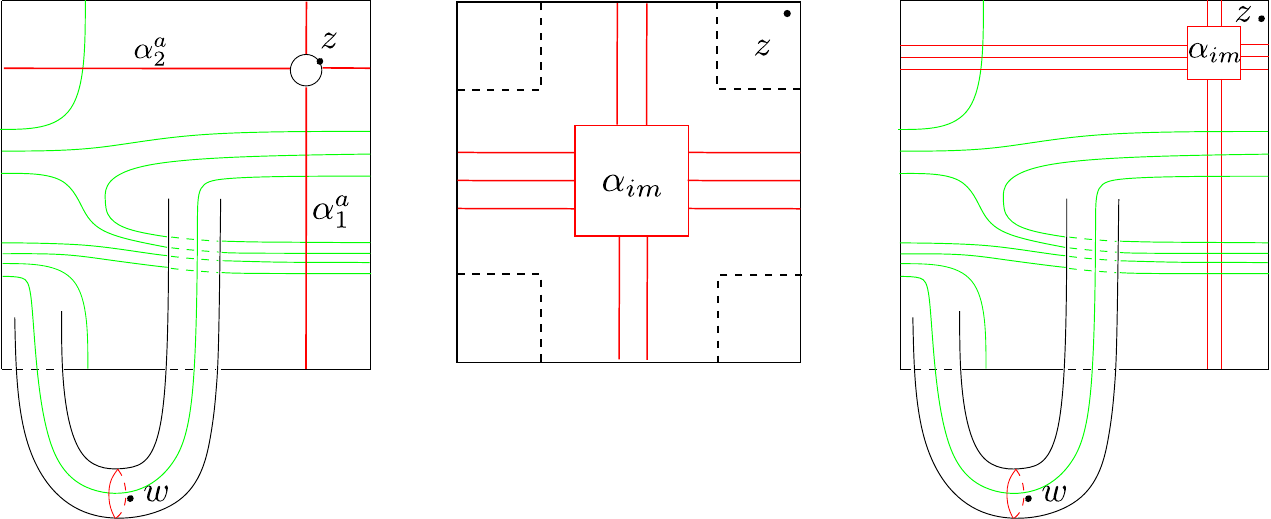}
		\caption{Pairing construction that gives rise to a doubly pointed Heegaard diagram.}
		\label{Figure, pairing doubly pointed diagram}
	}
\end{figure}
We want to establish the unobstructedness and admissibility of the immersed Heegaard diagrams obtained from pairing constructions. For that we need two tools, namely \emph{z-adjacency} and \emph{the collapsing map} introduced in the next two subsections.
\subsection{z-adjacency}\label{Subsection, z-adjacency}
We will consider a diagrammatic condition for immersed multicurves that guarantees the unobstructedness of the paring diagram; this condition can be achieved easily by finger moves. We begin by introducing some terminology for convenience. 

In the definition below, we orient the curves in $\alpha_{im}$ arbitrarily and orient the four edges of the cut-open torus using the boundary orientation. For each edge of the cut-pen torus, let $k_+$ and $k_-$ denote the number of elementary arcs intersecting a given edge positively and negatively, respectively.
\begin{figure}[htb!]
	\centering{
		\includegraphics[scale=0.6]{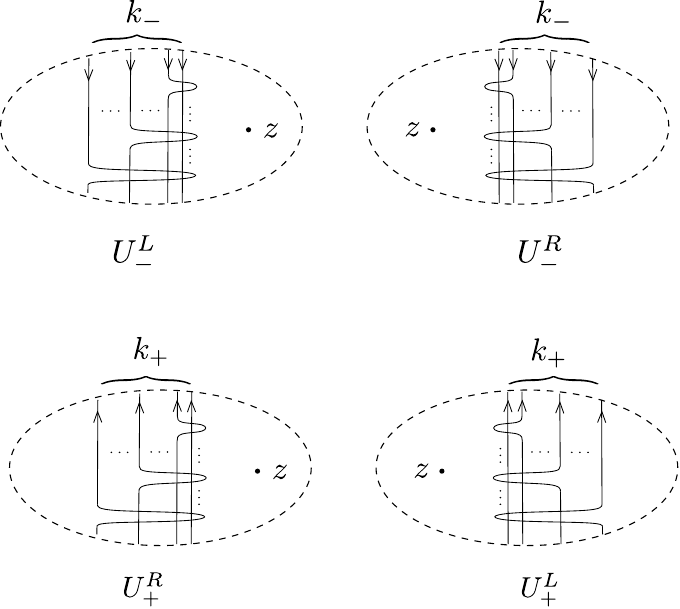}
		\caption{The disks $U^R_-$ and $U^L_-$. The superscript is chosen to suggest whether $z$ is one the left or on the right of the strands when we traverse an arc in the indicated direction.}
		\label{Figure, z adjacent neighbourhoods}
	}
\end{figure}
\begin{defn}\label{Definition, z-adjacency}
	Let $\alpha_{im}$ be an immersed multicurve in the marked torus. Then $\alpha_{im}$ is \emph{$z$-adjacent} if, for each of the four edges of the cut-open torus, there exist four open disks $U_{\pm}^R$ and $U_{\pm}^L$ in $T^2$ such that
	\begin{itemize}
		\item[(1)] $(U_-^L, U_-^L\cap(\alpha_{im}\cup \{{z}\})$, $(U_-^R, U_-^R\cap(\alpha_{im}\cup \{{z}\})$, $(U_+^L, U_+^L\cap(\alpha_{im}\cup \{{z}\})$ and $(U_+^R, U_+^R\cap(\alpha_{im}\cup \{{z}\})$ are homeomorphic to the corresponding disks in Figure \ref{Figure, z adjacent neighbourhoods}, where the arcs in the disks are sub-arcs on the $k_-$ distinct elementary arcs intersecting the given edge negatively for discs with subscript $-$ or sub-arcs on the $k_+$ distinct elementary arcs intersecting the given edge positively for discs with subscript $+$;
		\item[(2)]if the given edge is the top edge, then $U^L_-$ and $U^R_+$ are contained in $[0,1]\times [0,1]$;
		\item[(3)] if the given edge is the right edge, then $U^R_-$ and $U^L_+$ are contained in $[0,1]\times [0,1]$.
	\end{itemize}
\end{defn} 

\begin{prop}\label{Proposition, arranging a curve to be z-adjacent}
Every immersed multicurve in the marked torus is regularly homotopic to a $z$-adjacent multicurve. 
\end{prop}

\begin{proof}
	Orient $\alpha_{im}$ arbitrarily. We first define an operation on a collection of oriented parallel arcs. Assume there are $k_{+}+k_-$ arcs, where $k_{+}$-many of the arcs are oriented in one direction, and the rest are oriented in the opposite direction. 
	\begin{figure}[htb!]
		\centering{
			\includegraphics[scale=0.5]{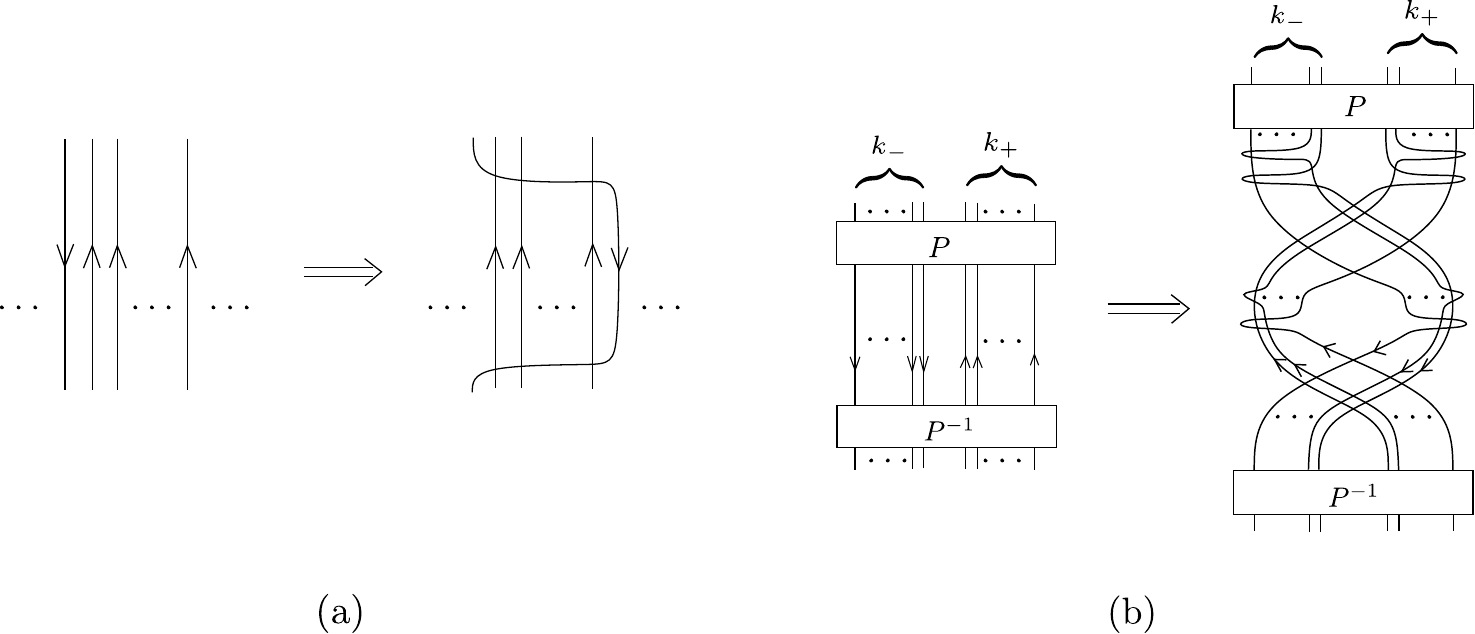}
			\caption{Finger moves on parallel strands.}
			\label{Figure, finger move on strands}
		}
	\end{figure}
	The operation is shown in Figure \ref{Figure, finger move on strands}: First, by performing the finger moves in Figure \ref{Figure, finger move on strands} (a) repeatedly, we can arrive at a collection of arcs as shown in the left of Figure \ref{Figure, finger move on strands} (b): the $P$- and $P^{-1}$-boxes indicate a pair of mutually inverse permutations, and between the $P$- and $P^{-1}$-boxes the arcs are arranged so that all $k_-$ arcs with parallel orientations are grouped on the left and all the other $k_+$ arcs with the opposite orientations are grouped on the right. Next, do a sequence of finger moves to the diagram on the left of Figure \ref{Figure, finger move on strands} (b) to arrive at the right-hand-side diagram of Figure \ref{Figure, finger move on strands} (b). Now perform this operation to the arcs of $\alpha_{im}$ near all four edges in the cut-open marked torus, then we have a z-adjacent immersed multicurve; see Figure \ref{Figure, z-adjacent immersed curve} for the desired open disks. Note that conditions $(2)$ and $(3)$ are obviously satisfied because $z$ is in the top right corner of the cut open torus. 
\end{proof}
\begin{figure}[htb!]
	\centering{
		\includegraphics[scale=0.55]{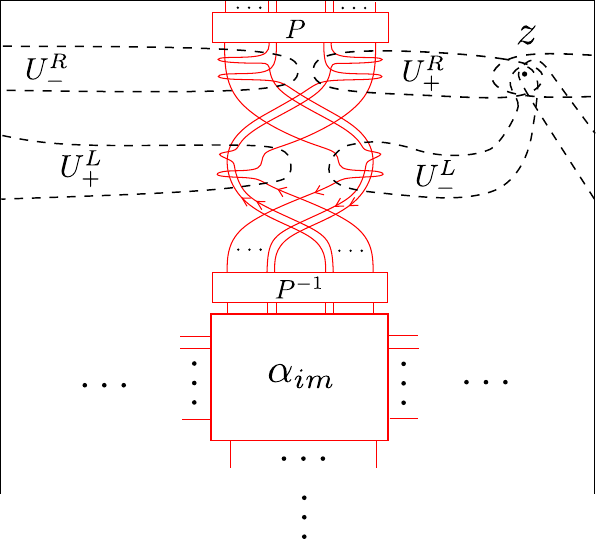}
		\caption{A $z$-adjacent immersed curve.}
		\label{Figure, z-adjacent immersed curve}
	}
\end{figure}

We shall need a technical lemma. Let $l$ be a one-cornered sub-loop of $\alpha_{im}$ with a corner $q$. If we traverse $l$ in either direction, we see it begins with an arc starting from $q$ to the meridian or longitude, then a sequence of elementary arcs, and finally, an arc starting from the meridian or longitude and ending at $q$. We call the starting and ending arcs the \emph{non-elementary sub-arcs of $l$}, and the other sub-arcs \emph{the elementary sub-arcs of $l$}. 

\begin{lem}\label{Lemma, constraints on positive domains bounded by one-cornered subloop}
Let $\alpha_{im}$ be a $z$-adjacent immersed curve. Let $D$ be a positive domain in ${T}^2$ bounded by a $k$-cornered (sub)loop of $\alpha_{im}$.
\begin{itemize}
	\item[(1)]If $n_z(D)=n$ for some $n\geq 0$ and $k=0$ or $1$, then for any side of the cut-open marked torus $[0,1]\times [0,1]$ and any sign, the number of elementary sub-arcs in $\partial D$ intersecting the given side with the given sign is less than or equal to $n$.
	\item[(2)]If $n_z(D)=0$, then for arbitrary $k\geq 0$, there are no elementary subarcs contained in $\partial D$. 
\end{itemize}
\end{lem}
\begin{proof}
We prove (1) first. We will only consider the case in which the elementary sub-arcs intersect the given edge negatively and remark that the other case is similar. 

We prove by contradiction. Suppose there are $k_->n$ elementary sub-arcs contained in $\partial D$ intersecting the given edge negatively. Since $\partial D$ is $0$- or $1$-cornered it has an orientation induced by the orientation on $\alpha_{im}$. Examining the local diagram $(U_-^L, U_-^L\cap({\partial D}\cup \{\bm{z}\}))$ in Figure \ref{Figure, z adjacent neighbourhoods} one sees $D$ has negative multiplicity $n-k_-$ in the left-most region, which contradicts our assumption that $D$ is a positive domain. Therefore, $k_-\leq n$. 

Next, we prove (2). Assume there is an elementary sub-arc in $\partial D$. Then no matter how this sub-arc is oriented, $z$ is on both the left and right of it. As $n_z(D)=0$, there is a region with $-1$ multiplicity, which contradicts that $D$ is positive. 
\end{proof}
\subsection{The collapsing operation}
To relate the domains of the pairing diagram $\mathcal{H}^a(\alpha_{im})$ and the arced bordered diagram $\mathcal{H}^a$, we define the so-called \emph{collapsing operation}. This operation was previously defined in the case of paring genus one bordered Heegaard diagrams with immersed curves \cite{Chen2019}, and we give the general case here. The operation is pictorially shown in Figure \ref{Figure, collapsing operation}, and the definition is given below.

\begin{defn}\label{Definition, collapsing operation}
The collapsing operation on $\mathcal{H}^a(\alpha_{im})$ is defined to be the composition of the following modifications of the diagram: 
\begin{itemize}
\item[(Step 1)]Extend the map $h$ in Definition \ref{Definition, pairing diagram} to identify $T^2-\pi([-\frac{1}{4}+\epsilon,\frac{1}{4}-\epsilon]\times [-\frac{1}{4}+\epsilon,\frac{1}{4}-\epsilon])$ with a slightly larger neighborhood of $U=h^{-1}(T^2-\pi([-\frac{1}{4},\frac{1}{4}]\times [-\frac{1}{4},\frac{1}{4}]))$. Here $\epsilon$ is a sufficiently small positive number.  
\item[(Step 2)]Puncture $h^{-1}((\frac{3}{4},\frac{3}{4}))$, and enlarge it to a hole so that under the identification map $h$, the boundary of the hole is a square of side length $\frac{1}{2}+2\epsilon$ and with rounded corner modeled on a quarter of a circle of radius $\epsilon$.
While enlarging the hole, we push immersed curves it encountered along the way so that part of the immersed curves are squeezed to the boundary of the hole. 
\item[(Step 3)]Collapse $h^{-1}(\pi([-\frac{1}{4}+\epsilon,\frac{1}{4}-\epsilon]\times [\frac{1}{4},\frac{3}{4}]))$ to the core $h^{-1}(\pi([-\frac{1}{4}+\epsilon,\frac{1}{4}-\epsilon]\times\{\frac{1}{2}\}))$, which is denoted $a^{a,R}_1$. Collapse $h^{-1}(\pi([\frac{1}{4},\frac{3}{4}]\times[-\frac{1}{4}+\epsilon,\frac{1}{4}-\epsilon]))$ to the core $h^{-1}(\pi(\{\frac{1}{2}\}\times[-\frac{1}{4}+\epsilon,\frac{1}{4}-\epsilon]))$, which is denoted $a^{a,R}_2$. 
\end{itemize}
\end{defn}  

\begin{rmk}\label{Remark, collapsing map}
\begin{itemize}
\item[(1)]Clearly, the outcome of the collapsing operation on $\mathcal{H}^a(\alpha_{im})$ can be identified with $\mathcal{H}^a$.
\item[(2)]Each elementary arc in $\alpha_{im}$ standing for $\rho_I\in\{\rho_1,\rho_2,\rho_3,\rho_{12},\rho_{23},\rho_{123}\}$ is mapped under the collapsing map to an arc that passes the Reeb chord $\rho_I$ in $\mathcal{Z}^R$ of $\mathcal{H}^a$. Note that an oriented elementary sub-arc is \textit{correctly oriented} if it induces a Reeb chord $\mathcal{Z}^R$ under the collapsing map, i.e., the orientations coincide.  

\item[(3)]The intersection points in $\mathcal{G}({\mathcal{H}^a(\alpha_{im})})$ are of the form $\bm{x}\otimes a$ are in one-to-one correspondence with $\mathcal{G}({\mathcal{H}^a})\otimes_{\mathcal{I}_R} \mathcal{G}({\alpha_{im}})$, where the tensor product is taken over $\mathcal{I}_R\subset \mathcal{A}(\mathcal{Z}_R)$. Indeed, given an intersection point $\xi\in \mathcal{G}({\mathcal{H}^a(\alpha_{im})})$, its image under the collapsing map yields an intersection point $\bm{x}$ in $\mathcal{H}^a$. Also, the component of $\xi$ on $\alpha_{im}$ uniquely gives rise to an intersection point $a$ of $\alpha_{im}$ as follows. By the definition of the pairing operation, when we pull back the intersection point on $\alpha_{im}$ to the marked torus, it lies in a horizontal or vertical arc as described in assumption (C-4) on immersed multicurves, which uniquely corresponds to an intersection point of $\alpha_{im}$ with the longitude or meridian. Therefore, every intersection point $\xi$ in $\mathcal{H}^a(\alpha_{im})$ can be written as $\bm{x}\otimes y$. It is easy to see this induces a one-to-one correspondence between $\mathcal{G}({\mathcal{H}^a(\alpha_{im})})$ and $\mathcal{G}({\mathcal{H}^a})\otimes_{\mathcal{I}_R} \mathcal{G}({\alpha_{im}})$.
\end{itemize}
\end{rmk}

\begin{figure}[htb!]
\includegraphics[scale=0.5]{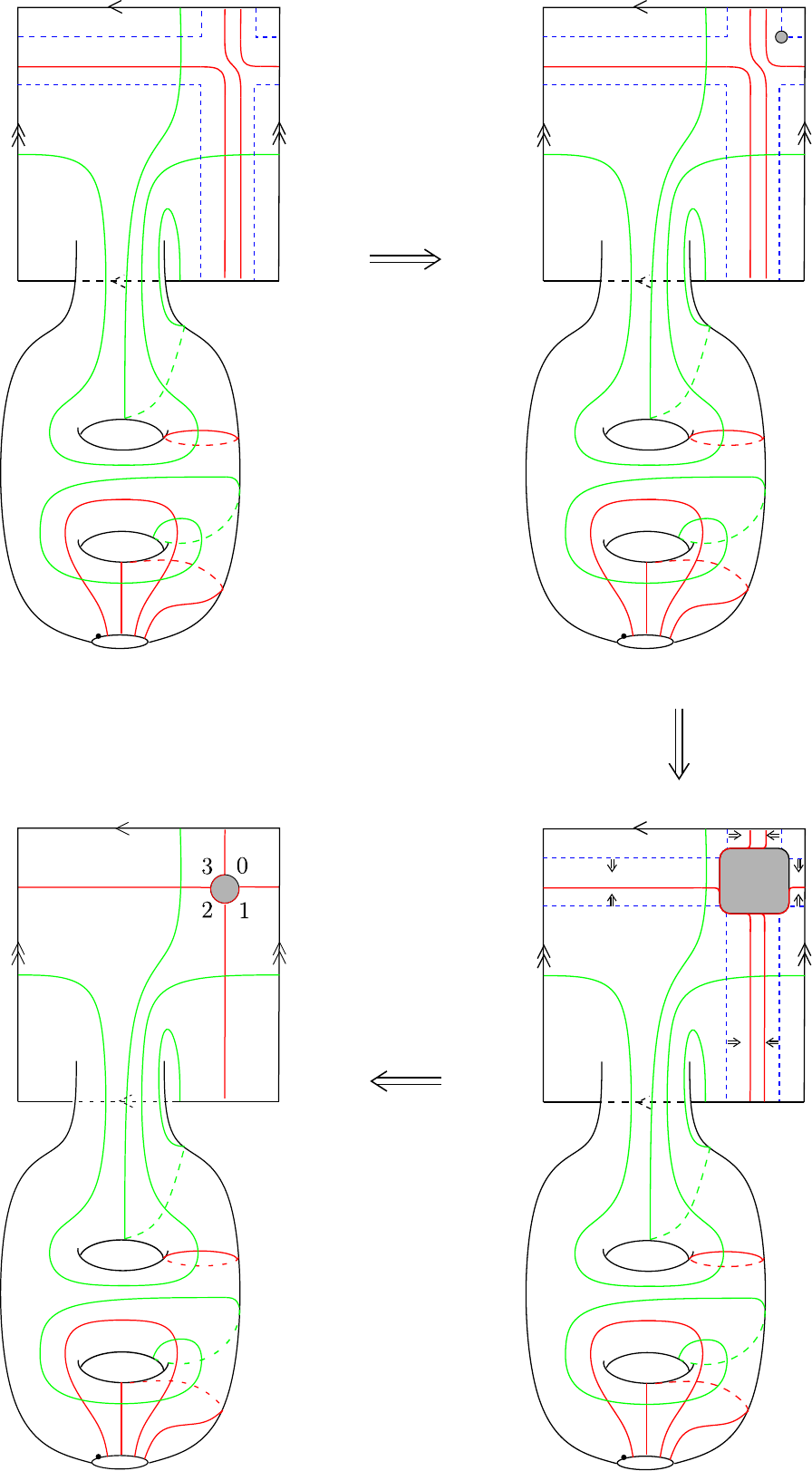}
\caption{The collapsing operation.}\label{Figure, collapsing operation}
\end{figure}

We will give a proposition relating the domains of $\mathcal{H}^a(\alpha_{im})$ and $\mathcal{H}^a$. Let $l$ be an oriented arc $l$ on $\alpha_{im}$ such that all the elementary sub-arcs are oriented correctly. We use $\overrightarrow{\rho}(l)$ to denote the sequence of Reeb chords determined by $l$.
\begin{prop}\label{Proposition, relating domains in two different Heegaard diagrams}
Assume the immersed multicurve $\alpha_{im}$ is $z$-adjacent. Let $B$ be a positive domain in $\mathcal{H}^a(\alpha_{im})$ corresponding to a homology class in $\pi_2(\bm{x}\otimes a, \bm{y}\otimes b, \overrightarrow{\sigma})$ with $n_z(B)=0$. Then the image of $B$ under the collapsing map is a positive domain $B'$ in $\mathcal{H}^a$ corresponding to a homology class $\pi_2(\bm{x},\bm{y},\overrightarrow{\rho}(\partial_{\alpha_{im}}B),\overrightarrow{\sigma})$ with $n_z(B')=0$. Here, $\partial_{\alpha_{im}}B$ refers to the arc on $\alpha_{im}$ connecting the corresponding components of $\bm{x}\otimes a$ and $\bm{y}\otimes b$. Moreover, $$e(B')=e(B)-\frac{|\overrightarrow{\rho}(\partial_{\alpha_{im}}B)|}{2}.$$    
\end{prop}

\begin{proof}
It is clear that $B'$ is positive and $n_z(B')=0$. It is also clear that $B'$ give rise to a domain connecting $\bm{x}$ and $\bm{y}$. We need to show that $B'$ has the Reeb chords $\overrightarrow{\rho}(\partial_{\alpha_{im}}B)$ at the east infinity. We claim all the elementary arcs appear in $\partial_{\alpha_{im}}B$ are correctly oriented, and hence $\partial_{\alpha_{im}}B$ gives rise to a monotonic arc (in the sense that all the Reeb chords appearing on the arc respect the boundary orientation) connecting (the components on the $\alpha$ arc of) $\bm{x}$ to $\bm{y}$ under the collapsing map. The sequence of Reeb chords appearing in this arc are exactly $\overrightarrow{\rho}(\partial_{\alpha_{im}}B)$ in view of Remark \ref{Remark, collapsing map} (2). To see the claim, note $\alpha_{im}$ is $z$-adjacent and $B$ is positive with $n_z(B)=0$. Therefore, if an elementary arc on $\partial_{im} B$ intersects the top edge or the right edge, then its orientation is forced by the positivity of domains and condition (2) and (3) in Definition \ref{Definition, z-adjacency}, and the orientation is the correct orientation. The only type of elementary arcs that intersects neither the top edge nor the right edge corresponds to $\rho_2$. If an elementary arc corresponding to $\rho_2$ on $\partial_{\alpha_{im}}B$ has a successor or precursor, then correct orientation on the successor or the precursor would induce the correct orientation on it. Otherwise, $\partial_{\alpha_{im}}B$ has only one elementary arc corresponding to $\rho_2$, in which case it is clear that the elementary arc is correctly oriented.      

Next, we compare the Euler measures. Divide the domain $B$ into two parts $B_1$ and $B_2$, along the square with rounded corners, which is the boundary of the hole in Step 2 of the collapsing operation. (See Figure \ref{Figure, dividing the domain into two parts}.) This time, we do not puncture the interior of the square. Let $B_1$ denote the part of $B$ outside of the square, and let $B_2$ denote the part inside the square (which is pushed onto the boundary circle under the collapsing map). Then $e(B_1)=e(B')$ since these two domains differ by a bunch of rectangles whose Euler measure are zero; these rectangles are collapsed in Step 3 of the collapsing operation. As $\alpha_{im}$ is $z$-adjacent, $B_2$ is positive, and $n_z(B)=0$,  we see $B_2$ can be further expressed as a sum of of simple domains determined by the elementary arcs appearing in $\partial_{\alpha_{im}}B$ (counted with multiplicity). (See Figure \ref{Figure, simple domains}.) Each simple domain of multiplicity one has Euler measure $\frac{1}{2}$, and there are $|\overrightarrow{\rho}(\partial_{\alpha_{im}}B)|$ many of them being collapsed (in Step 2 of the collapsing operation) in order to obtain $B'$. Therefore, $e(B')=e(B)-\frac{|\overrightarrow{\rho}(\partial_{\alpha_{im}}B)|}{2}.$ 
\end{proof}

\begin{figure}[htb!]
\includegraphics[scale=0.5]{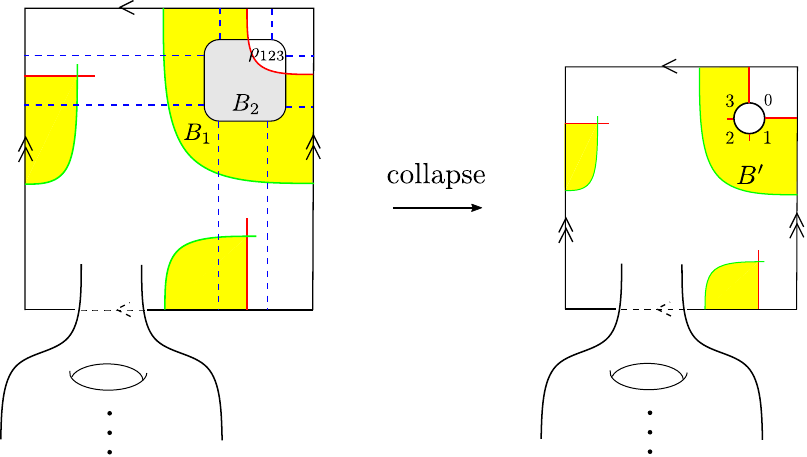}
\caption{Left: $B=B_1+B_2$. Right: $B'$.}\label{Figure, dividing the domain into two parts}
\end{figure}

\begin{figure}[htb!]
\includegraphics[scale=0.5]{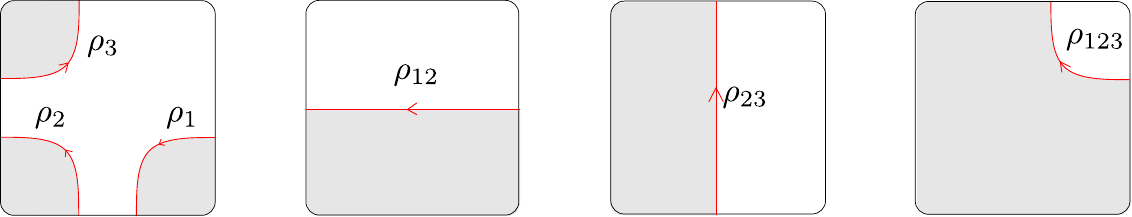}
\caption{Simple domains corresponding to Reeb elements in $\mathcal{A}$.}\label{Figure, simple domains}
\end{figure}

\subsection{Unobstructedness and admissibility of paring diagrams}\label{Subsection, unobstructeness and admissbility of pairing diagrams}
\begin{prop}\label{Proposition, paring diagrams are admissible}
Let $\alpha_{im}\subset T^2$ be a $z$-adjacent immersed multicurve. Then the pairing diagram $\mathcal{H}^a(\alpha_{im})$ of an arced bordered Heegaard diagram $\mathcal{H}^a$ and $\alpha_{im}$ is unobstructed. Furthermore, $\mathcal{H}^a(\alpha_{im})$ is provincially admissible provided $\mathcal{H}^a$ is left provincially admissible. (See Definition \ref{Definition, unobstructedness} and Definition \ref{Definition, provinical admissibility}.)
\end{prop}

\begin{proof}[Proof of Proposition \ref{Proposition, paring diagrams are admissible}] Consider the bordered Heegaard diagram $\mathcal{H}^a(\alpha_{im})=(\bar{\Sigma}',\bar{\bm{\alpha}}',\bm{\beta}',z')$ obtained from pairing an arced bordered Heegaard diagram $\mathcal{H}^a=(\bar{\Sigma},\bar{\bm{\alpha}},\bm{\beta},\bm{z})$ and a $z$-adjacent immersed multicurve $\alpha_{im}$. 

We begin by showing $\bar{\bm{\alpha}}'$ is unobstructed in the sense of Definition \ref{Definition, unobstructedness}. Let $B$ be a zero- or one-cornered $\alpha$-bounded domain. Since the curves $\{\bar{\alpha_1^a},\bar{\alpha_2^a},\alpha_1,\ldots,\alpha_{g-2},\alpha_{im}^0\}$ are pairwise disjoint and homologically independent in $H_1(\bar{\Sigma},\partial)$, $[\partial B]$ (as a homology class) is equal to a linear combination of at most one copy of $\partial\bar{\Sigma}$ and some homologically trivial zero- or one-cornered loop contained in a single connected component of $\alpha_{im}$. 

We first show there are no positive zero- or one-cornered $\alpha$-bounded domains $B$ with $n_{z'}(B)=0$. In this case, $\partial B$ is a homologically trivial zero- or one-cornered loop contained in a single connected component of $\alpha_{im}$, i.e., $\partial\bar{\Sigma}$ does not appear in $\partial B$. As the $\bm{\beta}$-curves are irrelevant to our consideration, we may assume the $\alpha$-curves of $\mathcal{H}^a$ are in standard position. Therefore, there is an obvious circle $C\subset\bar{\Sigma}$ that splits $\bar{\Sigma}$ into a genus-$(g-1)$ surface $E_1$ containing $\{\alpha^{a,L}_1,\alpha^{a,L}_2,\alpha^c_1,\ldots,\alpha^c_{g-2}\}$ and a genus-one surface $E_2$ containing $\alpha^{a,R}_1$ and $\alpha^{a,R}_2$. Let $C'$ be the corresponding curve on $\bar{\Sigma}'$. Then after surgery along $C'$, $B$ induces a positive domain $D$ in the marked torus $T^2$ (obtained from $E_2$ in an obvious way), and $D$ is bounded by a zero- or one-cornered (sub)loop of $\alpha_{im}$. According to Lemma \ref{Lemma, constraints on positive domains bounded by one-cornered subloop}, $\partial D$ contains no elementary sub-arcs, so $D$ cannot exist.

Next we show that if $n_{z'}(B)=1$, $B$ is a stabilized teardrop or $[\Sigma']$ depending on whether $\partial B$ is one-cornered or zero-cornered. In this case, after performing surgery along the same $C'$ as in the previous paragraph, $B$ gives rise to two domains: one is $[E_1']$, where $E'_1$ is the genus-$(g-1)$ surface, and the other is a positive domain $D$ contained in the marked torus $T^2$ with $n_z(D)=1$. 
We first consider the case in which $\partial D$ is zero-cornered. If $\partial D=\emptyset$, then $D=E_2$ and hence $B=[\Sigma']$. If $\partial D\neq \emptyset$, then according to Lemma \ref{Lemma, constraints on positive domains bounded by one-cornered subloop}, it consists of at most (and at least) $4$ elementary sub-arcs, and hence is a circle enclosing the $z$-basepoint once. However, such circles are assumed not to exist. When $\partial D$ is one-cornered, we claim $D$ is a teardrop. To see this, note that Lemma \ref{Lemma, constraints on positive domains bounded by one-cornered subloop} implies that $\partial D$ crosses the meridian at most three times and the longitude at most three times since each time the meridian or the longitude is crossed (except possibly the last time) the intersection is the beginning of an elementary sub-arc and there are at most two elementary sub-arcs starting on each. Because $\partial D$ is homologically trivial in $H_1({T}^2)$ it crosses both of the meridian and the longitude and even number of times, so it crosses each at most twice. It follows that $\partial D$ must circle once around $z$ and $D$ is a teardrop with $n_z(D)=1$.

Now we show any two-cornered positive $\alpha$-bounded domain $B$ with $n_z(B)=0$ is a bigon. To see, we may split $\Sigma'$ as $E_1\#E_2$ as before and regard $B$ as a domain in $E_2$ with $n_{z'}=0$. Note by Lemma \ref{Lemma, constraints on positive domains bounded by one-cornered subloop} (2), we know $\partial B$ consists of no elementary subarcs, and hence $B$ must be of the form shown in Figure \ref{Figure, two cornered domains} (up to rotation), which is a bigon. (Note we do not require the corners of the bigon $B$ to be convex.) 
  
\begin{figure}[htb!]
\centering{
\includegraphics[scale=0.35]{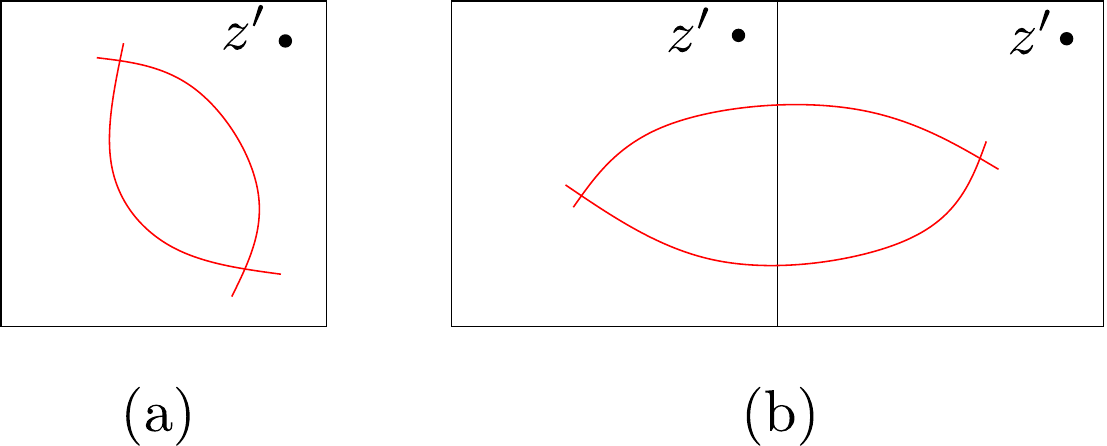}
\caption{Two-cornered positive $\alpha$-bounded domains.}
\label{Figure, two cornered domains}
}
\end{figure}

So far, we have proved $\bar{\bm{\alpha}}'$ is unobstructed. We now show there are no non-trivial positive provincial periodic domains. If not, let $B$ be a positive provincial periodic domain for $\mathcal{H}^a(\alpha_{im})$. Then by Proposition \ref{Proposition, relating domains in two different Heegaard diagrams}, $\Psi(B)$ is a positive periodic domain for $\mathcal{H}^a$, where $\Psi$ denotes the collapsing map. Note $\Psi(B)$ is left provincial. As $\mathcal{H}^a$ is left provincially admissible, we have $\Psi(B)=0$, and hence $\partial B$ has no $\beta$-curves. So, $B$ is a positive zero-cornered $\alpha$-bounded domain with $n_z(B)=0$, but such domains are already excluded by unobstructedness. 
\end{proof}

\subsection{The first paring theorem}\label{Subsection, first pairing theorem}
Recall a bordered Heegaard diagram is \emph{nice} if every connected region in the complement of the $\alpha$- and $\beta$-curves is a disk with at most four corners except for the region containing $z$. Any bordered Heegaard diagram can be turned into a nice diagram via isotopy and handleslides of the $\beta$-curves (Proposition 8.2 of \cite{LOT18}). The key property of nice Heegaard diagrams is that the Euler measure of any region away from the base point is non-negative. This property imposes great constraints on domains supporting holomorphic representatives via the index formula. Hence it opens up the combinatorial vein for proving the pairing theorem.

\FirstPairingTheorem*
\begin{proof}
In view of the homotopy equivalence of the relevant invariants under isotopy of $\beta$ curves (Proposition \ref{Proposition, invariance of type D structure}), we may assume $\mathcal{H}^a$ is a nice arced bordered Heegaard diagram. Note nice arced bordered Heegaard diagrams are automatically left provincially admissible. Therefore, $\widehat{CFDA}(\mathcal{H}^a)$ and $\widehat{CFD}(\mathcal{H}^a({\alpha_{im}}))$ are defined. In fact, a stronger admissibility condition holds for $\mathcal{H}^a$: any periodic domains with $n_z=0$ has both positive and negative local multiplicities. This implies $\widehat{CFDA}(\mathcal{H}^a)$ is bounded, and hence the box-tensor product is expressed as a finite sum. 
 
Implicit in the proof is that we will be using split almost complex structures for defining $\widehat{CFD}(\mathcal{H}^a({\alpha_{im}}))$ and $\widehat{CFDA}(\mathcal{H}^a)$. A split almost complex structure is sufficient for defining $\widehat{CFD}(\mathcal{H}^a({\alpha_{im}}))$, since all the domains involved will be bigons and rectangles. In this setting, up to a generic perturbation of the $\alpha$ and $\beta$ curves, moduli spaces defined using a split almost complex structure are transverse (c.f. \cite[Proposition 3.9]{MR2240908}).

We will call the two punctures in $\mathcal{H}^a$ the $\sigma$-puncture and the $\rho$-puncture, where the $\rho$-puncture is the one that gets capped off in the pairing diagram. For now, we assume the local systems on $\alpha_{im}$ are trivial, and we will indicate the modifications needed for dealing with nontrivial local system later on. First, the generators $\mathcal{G}(\mathcal{H}^a(\alpha_{im}))$ and $\mathcal{G}(\mathcal{H}^a)\otimes_{\mathcal{I}_R}\mathcal{G}(\alpha_{im})$ are identified as pointed out in Remark \ref{Remark, collapsing map} (3).  Next, we prove the differentials have a one-to-one correspondence.

We first show any differential incurred by the box-tensor product has a corresponding differential in $\widehat{CFD}(\mathcal{H}^a(\alpha_{im}))$. A differential arising from the box tensor product comes in two types, depending on whether it involves nontrivial differentials in $\widehat{CFD}(\alpha_{im})$. If it does not involve non-trivial differential in $\widehat{CFD}(\alpha_{im})$, then the input from $\widehat{CFDA}(\mathcal{H}^a)$ counts curves with the domain being a provincial bigon, a provincial rectangle, or a bigon with a single Reeb chord on the $\sigma$-puncture; see \cite[Proposition 8.4]{LOT18}. Such bigons or rectangles clearly have their counterparts in $\mathcal{H}^a(\alpha_{im})$, giving the corresponding differentials in $\widehat{CFD}(\mathcal{H}^a(\alpha_{im}))$. If the box-tensor differential involves differentials in $\widehat{CFD}(\alpha_{im})$, then the corresponding input from  $\widehat{CFDA}(\mathcal{H}^a)$ counts curves with the domain being a bigon with a single Reeb chord on the $\rho$-puncture \cite[Proposition 8.4]{LOT18}. As it pairs with a differential in $\widehat{CFD}(\alpha_{im})$, this bigon gives rise to a bigon in $\mathcal{H}^a(\alpha_{im})$ (which is a pre-image of the collapsing map), giving the corresponding differential in $\widehat{CFD}(\mathcal{H}^a(\alpha_{im}))$.

Next, we show that every differential in $\widehat{CFD}(\mathcal{H}^a(\alpha_{im}))$ corresponds to a differential incurred by the box-tensor product. Suppose $u\in \pi_2(\bm{x}\otimes a, \bm{y}\otimes b)$ admits a holomorphic representative contributing to a differential for $\widehat{CFD}(\mathcal{H}^a(\alpha_{im}))$. Let $B$ be the domain of $u$, and let $B'$ denote the image of $B$ under the collapsing operation. By Proposition \ref{Proposition, relating domains in two different Heegaard diagrams}, $e(B)=e(B')+\frac{|\overrightarrow{\rho}(\partial_{\alpha_{im}}B)|}{2}$. As $B'$ is a positive domain with $n_z(B')=0$ and $\mathcal{H}^a$ is a nice Heegaard diagram, we have $e(B)\geq e(B')\geq 0$. By the index formula, denoting the source surface of $u$ by $S$, we have $$\text{Ind}(u)=g-\chi(S)+2e(B)+|\overrightarrow{\sigma}|.$$ 
As $\text{Ind}(u)=1$ and $2e(B)+|\overrightarrow{\sigma}|\geq 0$, we have $\chi(S)=g$ or $g-1$. 

When $\chi(S)=g$, $S$ consists of $g$ topological disks; each disk has a $+$ and a $-$ puncture, and there is at most one $\sigma$-puncture overall since $2e(B)+|\overrightarrow{\sigma}|=1$. We separate the discussion according to the number of $\sigma$-punctures. First, if there is a $\sigma$-puncture, then the corresponding domain $B$ in $\mathcal{H}^a(\alpha_{im})$ is a bigon with a single Reeb chord on the $\sigma$-puncture and does not involve $\alpha_{im}$. This domain clearly has its counterpart in $\mathcal{H}^a$ under the collapsing map, giving rise to an operation in $\widehat{CFDA}(\mathcal{H}^a)$; the corresponding differential in the box-tensor product is obtained by pairing this $DA$-operation with an element in $\widehat{CFD}(\alpha_{im})$. Secondly, if there is no $\sigma$-puncture, then the domain $B$ is a provincial bigon in $\mathcal{H}^a(\alpha_{im})$. There are two sub-cases to consider depending on whether the $\alpha$-boundary of $B$ overlaps with $\alpha_{im}$. If the $\alpha$ boundary of $B$ is not on $\alpha_{im}$, then we argue as in the first case to see that $B$ gives a corresponding differential in the box-tensor product. If, on the other hand, the boundary of $B$ is on $\alpha_{im}$, since $1/2=e(B)=e(B')+|\overrightarrow{\rho}(\partial_{\alpha_{im}}B)|/{2}$ we have $|\overrightarrow{\rho}(\partial_{\alpha_{im}}B)|$ is either $0$ or $1$. If $|\overrightarrow{\rho}(\partial_{\alpha_{im}}B)|=0$, then $B'$ is a provincial domain, giving the type-DA operation for the corresponding differential obtained by the box-tensor product.  If $|\overrightarrow{\rho}(\partial_{\alpha_{im}}B)|=1$, then $B'$ is obtained from $B$ subtracting a simple region (as in the proof of Proposition \ref{Proposition, relating domains in two different Heegaard diagrams}) and then applying the collapsing map. We can see $B'$ is a bigon with a single Reeb chord corresponding to the Reeb chord specified by $\partial_{\alpha_{im}}B$ on the $\rho$-puncture. The DA-operation given by $B'$ and the type D operation given by $\partial_{\alpha_{im}}B$ pair up to give the corresponding differential in the box-tensor product.

When $\chi(S)=g-1$, then $S$ consists of $g-1$ topological disks; $g-2$ of the disks are bigon, while the remaining one is a rectangle. As $e(B)=0$, the bigons are mapped trivially to $\Sigma$. Therefore, the domain $B$ is a rectangle. Again, since $e(B)=e(B')+|\overrightarrow{\rho}(\partial_{\alpha_{im}}B)|/{2}$ and $e(B')\geq 0$, we have $|\overrightarrow{\rho}(\partial_{\alpha_{im}}B)|=0$. Then $B'$ is a provincial rectangular domain in $\mathcal{H}^a$, giving rise to a DA-operation that pairs with a trivial type-D operation to give the corresponding differential in the box-tensor product. We have finished the proof when the local system is trivial.

Next, we consider the case where $\alpha_{im}$ admits a non-trivial local system $(\mathcal{E},\Phi)$. The local system induces a local system on the $\alpha$ curves in the pairing diagram $H^a(\alpha_{im})$. First, the discussion above identifies the generators at the vector space level: let $\bm{x}\otimes y$ be an intersection point in $\mathcal{G}(\mathcal{H}^a(\alpha_{im}))$, where $\bm{x}\in\mathcal{G}(H^a)$ and $y\in\mathcal{G}(\alpha_{im})$; then $\bm{x}\otimes y$ corresponds to a direct summand $\mathcal{E}|_{\bm{x}\otimes y}$ of $\widehat{CFD}(\mathcal{H}^a(\alpha_{im}))$ as a vector space, and $\mathcal{E}|_{\bm{x}\otimes y}$ can be naturally identified with $\bm{x}\otimes \mathcal{E}|_y$, a summand of $\widehat{CFDA}(\mathcal{H}^a)\boxtimes \widehat{CFD}(\alpha_{im})$. Secondly, the discussion in the trivial-local-system case shows that $\widehat{CFD}(\mathcal{H}^a(\alpha_{im}))$ has a differential map between the summands $\mathcal{E}|_{\bm{x}\otimes y}\rightarrow \sigma_I\otimes \mathcal{E}|_{\bm{x'}\otimes y'}$ if and only if the box-tensor product has a differential map between the corresponding summands $\bm{x}\otimes\mathcal{E}|_{y}\rightarrow \sigma_I\otimes (\bm{x'}\otimes \mathcal{E}|_{ y'})$ in the box-tensor product, and under the natural identification between these summands both differential maps are induced by the same parallel transport from $\mathcal{E}|_{y}$ to $\mathcal{E}|_{y'}$.

\end{proof}

\subsection{The second pairing theorem}\label{Section, the second pairing theorem}
We are interested in computing knot Floer chain complexes over $\mathcal{R}=\mathbb{F}[U,V]/(UV)$ using bordered Floer homology. We have already defined an extended type-D structure, and we want to pair it with an extended type-A structure to get a bi-graded chain complex over $\mathcal{R}=\mathbb{F}[U,V]/(UV)$. Here we will only restrict to a specific extended type-A structure associated to the doubly-pointed bordered Heegaard diagram $\mathcal{H}_{id}$ given in Figure \ref{Figure, H_st}. The diagram $\mathcal{H}_{id}$ corresponds to the pattern knot given by the core of a solid torus, which is the \textit{identity pattern}. 
\begin{figure}[htb!]
\centering{
\includegraphics[scale=0.35]{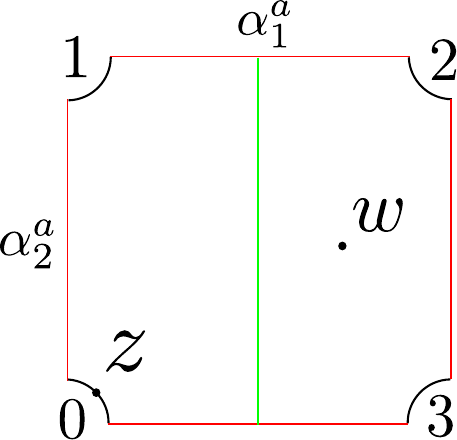}
\caption{The bordered diagram $\mathcal{H}_{id}$.}
\label{Figure, H_st}
}
\end{figure}

Recall $\tilde{\mathcal{A}}$ denotes the weakly extended torus algebra, and $\mathcal{I}\subset\tilde{\mathcal{A}}$ is the ring of idempotents (see Figure \ref{Figure, weakly extended torus algebra}). 
\begin{defn}
The extended type-A structure $\widetilde{CFA}(\mathcal{H}_{id})$ is a free $\mathcal{R}$-module generated by the single intersection point $x$ in $\mathcal{H}_{id}$. It is equipped with an $\mathcal{I}$-action given by $x\cdot \iota_0=x$ and $x\cdot \iota_1=0$, together with a family of $\mathcal{R}$-linear maps $m_{i+1}: \widetilde{CFA}(\mathcal{H}_{id})\otimes_{\mathcal{I}} \tilde{A}^{\otimes i}\rightarrow \widetilde{CFA}(\mathcal{H}_{id})$ ($i\in \mathbb{N}$), where up to $\mathcal{R}$-linearity the only non-zero relations are:
$$
\begin{aligned}
&m_2(x,1)=x,\\
&m_{3+i}(x,\rho_3,\overbrace{\rho_{23},\ldots,\rho_{23}}^{i},\rho_2)=U^ix, \quad i\in\mathbb{N},\\
&m_{3+i}(x,\rho_1,\overbrace{\rho_{01},\ldots,\rho_{01}}^{i},\rho_0)=V^ix, \quad i\in\mathbb{N}.
\end{aligned}
$$
\end{defn}
\begin{rmk}
This extends the hat-version type A-structure $\widehat{CFA}(\mathcal{H}_{id})$ by allowing Reeb chords crossing the base point $z$. 
\end{rmk}
Straightforwardly, the hat-version box-tensor product can be generalized to be an operation between the extended type A structure $\widetilde{\mathcal{M}}:=\widetilde{CFA}(\mathcal{H}_{id})$ and a weakly extended type D structure $(\widetilde{\mathcal{N}},\delta^i)$: It is the $\mathcal{R}$-module $\widetilde{\mathcal{M}}\otimes_\mathcal{I}\widetilde{\mathcal{N}}$ together with a differential $\partial_\boxtimes:=\sum_{i\geq 0} (m_{i+1}\otimes \mathbb{I}_{\widetilde{\mathcal{N}}})\circ(\mathbb{I}_{\widetilde{\mathcal{M}}}\otimes \delta^i)$; the finiteness of the sum can be guaranteed for type D structures defined using bi-admissible diagrams (see the proof of Theorem \ref{Theorem, 2nd pairing theorem} below). One may verify $\partial_\boxtimes^2=0$ algebraically using the structure equations defining the (weakly) extended type D and type A structures. We omit such computation and instead content ourselves with Theorem \ref{Theorem, 2nd pairing theorem} below, which implies that the $\partial_{\boxtimes}$ induced by gluing bordered Heegaard diagrams is indeed a differential. We further remark that $\widetilde{\mathcal{M}}\boxtimes_\mathcal{I}\widetilde{\mathcal{N}_1}$ is chain homotopic to $\widetilde{\mathcal{M}}\boxtimes_\mathcal{I}\widetilde{\mathcal{N}_2}$ provided $\widetilde{\mathcal{N}_1}$ is homotopic to $\widetilde{\mathcal{N}_2}$. The proof of this is similar to that in the hat version and is omitted.

\SecondPairingTheorem*
(See Definition \ref{Definition, unobstructedness} and \ref{Definition, bi-admissibility} for the unobstructedness and bi-admissibility for $\mathcal{H}$.)
\begin{proof}[Proof of Theorem \ref{Theorem, 2nd pairing theorem}]
Note periodic domains of $\mathcal{H}_{id}\cup \mathcal{H}_{im}$ with $n_z=0$ (respectively $n_w=0$) corresponds to periodic domains of $\mathcal{H}_{im}$ with $n_{\rho_0}=n_{\rho_1}=0$ (respectively $n_{\rho_2}=n_{\rho_3}=0$). Therefore,   since $\mathcal{H}_{im}$ is bi-admissible, $\mathcal{H}_{id}\cup \mathcal{H}_{im}$ is bi-admissible in the sense of Definition \ref{Definition, admissibility for doubly-pointed diagram}. Also, zero- or one-cornered $\alpha$-bounded domains in $\mathcal{H}_{id}\cup \mathcal{H}_{im}$ with $n_z=n_w=0$ must lie in $\mathcal{H}_{im}$. So, unobstructedness of $\mathcal{H}_{im}$ implies the unobstructedness of $\mathcal{H}_{id}\cup \mathcal{H}_{im}$. In summary, $\mathcal{H}_{id}\cup \mathcal{H}_{im}$ is bi-admissible and unobstructed, and hence $CFK_{\mathcal{R}}(\mathcal{H}_{id}\cup \mathcal{H}_{im})$ is defined. The bi-admissibility of $\mathcal{H}_{im}$ also implies $\partial_\boxtimes$ is expressed as a finite sum and hence is well-defined. To see this, note for any $\bm{x},\bm{y}\in\mathcal{G}(\mathcal{H}_{im})$, bi-admissibility implies there are only finitely many positive domains connecting $\bm{x}$ and $\bm{y}$ with a prescribed Reeb-chord sequence of the form $\rho_1,\rho_{01},\ldots\rho_{01},\rho_{0}$ and $\rho_3,\rho_{23},\ldots\rho_{23},\rho_{2}$.

Recall that the differential in $CFK_{\mathcal{R}}(\mathcal{H}_{id}\cup \mathcal{H}_{im})$ counts holomorphic curves that can only cross at most one of the $w$- and $z$-base points. Note also that both of the base points in $\mathcal{H}_{id}$ are adjacent to the east boundary. Therefore, by the symmetry of the base points $w$ and $z$, it suffices to prove the theorem for the hat version knot Floer homology, i.e., $$\widehat{CFK}(\mathcal{H}_{id}\cup \mathcal{H}_{im})\cong \widehat{CFA}(\mathcal{H}_{id})\boxtimes \widehat{CFD}(\mathcal{H}_{im}).$$
Though our Heegaard diagrams have immersed $\alpha$-multicurves, given that no boundary degeneration can occur, the proof for the embedded-$\alpha$-curves case, which uses neck stretching and time dilation, carries over without changes; see Chapter 9 of \cite{LOT18} for detail or Section 3 of \cite{Lipshitz2014} for an exposition.
\end{proof}

\section{Knot Floer homology of satellite knots}\label{Section, satellite knot pairing}
We apply the machinery developed in the previous sections to study the knot Floer homology of satellite knots. First, we introduce a gentler condition on the immersed curves than the z-adjacency condition.
\begin{defn}\label{Definition, admissibility of immersed curves}
	An immersed multicurve $\alpha_{im}$ in the marked torus $(T^2,z)$ is \textit{admissible} if there are no nontrivial zero- and one-cornered $\alpha$-bounded positive domains $B$ with $n_z(B)=0$.   
\end{defn}
\noindent Note any $z$-adjacent immersed multicurve is admissible in view of Lemma \ref{Lemma, constraints on positive domains bounded by one-cornered subloop}.

Let $\mathcal{H}_{w,z}$ be a doubly pointed bordered Heegaard diagram for a pattern knot $(S^1\times D^2,P)$. Recall that we can construct a doubly pointed immersed diagram $\mathcal{H}_{w,z}(\alpha_{im})$. The admissibility condition guarantees that $CFK_{\mathcal{R}}(\mathcal{H}_{w,z}(\alpha_{im}))$ is defined in view of the following proposition.
\begin{prop}\label{Proposition, H_w,z(alpha_im) is bi-admissible and unobstructed}
 If $\alpha_{im}$ is an admissible immersed multicurve, then $\mathcal{H}_{w,z}(\alpha_{im})$ is bi-admissible and unobstructed.  
\end{prop}
\begin{proof}
	First, we show the diagram $\mathcal{H}_{w,z}(\alpha_{im})$ is unobstructed (in the sense of Definition \ref{Definition, unobstructedness for doubly-pointed diagram}). Let $B$ be a zero- or one-cornered $\alpha$-bounded domain for $\mathcal{H}_{w,z}(\alpha_{im})$. Note $n_w(B)=n_z(B)$ since the base points $w$ and $z$ are in the same region in the complement of the $\alpha$-curves. We restrict to the domains with $n_z(B)=n_w(B)=0$. Recall we want to prove $B$ must be somewhere negative. Splitting the Heegaard surface as in the proof of Proposition \ref{Proposition, paring diagrams are admissible},  we see $B$ corresponds to a zero- or one-cornered $\alpha$-bounded domain $B'$ in the marked torus with $n_z(B')=0$. Since $\alpha_{im}$ is admissible, $B'$ is somewhere negative. Therefore, $B$ is somewhere negative. 
	
	Next we show the pairing diagram $\mathcal{H}_{w,z}(\alpha_{im})$ is bi-admissible. Recall that bi-admissibility means any nontrivial periodic domains $B$ with $n_w(B)=0$ or $n_z(B)=0$ must have both positive and negative coefficients. To see this, we first claim any given periodic domain $B$ is bounded by some multiple of a homologically trivial component of $\alpha_{im}$. (We warn the reader that the claim is no longer true if one further performs a handleslide of such a component over an embedded alpha curve.) To see the claim, note that as homology classes, the curves $[\alpha_i]$ ($i=1,\ldots,g-1$), $[\alpha^1_{im}]$, and $[\beta_i]$ ($i=1,\ldots, g$) are linearly independent, just as the attaching curves in a Heegaard diagram for $S^3$. Now, the claim implies $B$ is a zero-cornered $\alpha$-bounded domain. In view of the unobstructedness established above, $B$ is somewhere negative. 
\end{proof}

\begin{defn}
	Let $\alpha_{im}$ and $\alpha'_{im}$ be two admissible immersed multicurves. They are said to be admissibly equivalent if there exists a finite sequence of admissible immersed curves $\alpha^i_{im}$, $i=1,\ldots,n$ such that 
	\begin{itemize}
		\item[(1)]$\alpha^1_{im}=\alpha_{im}$ and $\alpha^n_{im}=\alpha'_{im}$,
		\item[(2)]For $i=1,\ldots, n-1$, $\alpha^i_{im}$ and $\alpha^{i+1}_{im}$ are related by a finger move that creates/cancels a pair of self-intersection points of the immersed curves. 
	\end{itemize}
\end{defn}

\begin{prop}\label{Proposition, finger move independence}
	Let $\alpha_{im}$ and $\alpha_{im}'$ be two admissibly equivalent immersed multicurves, and let $\mathcal{H}_{w,z}$ be a doubly pointed bordered Heegaard diagram. Then $$CFK_{\mathcal{R}}(\mathcal{H}_{w,z}(\alpha_{im}))\cong CFK_{\mathcal{R}}(\mathcal{H}_{w,z}(\alpha'_{im})).$$
\end{prop}

\begin{proof}
	The proof follows the same strategy as the usual proof of isotopy invariance. Let $\bm{\alpha_0}$ and $\bm{\alpha_1}$ be the two sets of $\alpha$-curves. For simplicity, assume they are related by a single finger move. We model the finger move using a locally supported exact Hamiltonian isotopy on $\Sigma$. The isotopy induces a family of $\alpha$-curves, $\bm{\alpha}_t$, on $\Sigma$ ($t\in \mathbb{R}$); for $t\ll 0$ (resp. $t\gg 0$), $\bm{\alpha}_t$ is constant with respect to $t$ and is identified with $\bm{\alpha}_0$ (resp.\ $\bm{\alpha}_1$). $\bm{\alpha}_t$ induces an immersed totally real submanifold $C_{\alpha}=\bm{\alpha}_t\times\{1\}\times\{t\}$ in $\Sigma\times[0,1]\times \mathbb{R}$. $C_{\alpha}$ can be realized as an immersion $\Psi_t: (\amalg_{i=1}^g S^1)\times \mathbb{R}\rightarrow \Sigma\times\{1\}\times \mathbb{R}.$ Let $C_{\beta}$ be the Lagrangian induced by the $\beta$ curves. For $\bm{x}\in \mathbb{T}_{\bm{\alpha}_0}\cap \mathbb{T}_{\bm{\beta}}$ and $\bm{y} \in\mathbb{T}_{\bm{\alpha}_1}\cap \mathbb{T}_{\bm{\beta}}$, one then define $\mathcal{M}_{\Psi_t}(\bm{x},\bm{y})$ to be the moduli space of holomorphic curves in $\Sigma\times[0,1]\times \mathbb{R}$ with boundary on $C_{\alpha}\cup C_{\beta}$ such that the $\alpha$-boundary can be lifted through $\Psi_{t}$. With this, one can define a map $\Phi_0: CFK_{\mathcal{R}}(\mathcal{H}_{w,z}(\alpha_{im}))\rightarrow CFK_{\mathcal{R}}(\mathcal{H}_{w,z}(\alpha'_{im}))$ by $$\bm{x}\mapsto \sum_{\bm{y}}\sum_{\phi \in \pi_2(\bm{x},\bm{y})}\#\mathcal{M}_{\Psi_t}(\bm{x},\bm{y})U^{n_w(\phi)}V^{n_z(\phi)}\bm{y},$$
	where $\mathcal{M}_{\Psi_t}(\bm{x},\bm{y})$ has dimension zero. Define $\Phi_1:CFK_{\mathcal{R}}(\mathcal{H}_{w,z}(\alpha'_{im})) \rightarrow CFK_{\mathcal{R}}(\mathcal{H}_{w,z}(\alpha_{im}))$ similarly. We remark that the compactness and gluing results still apply to this setup. The bi-admissibility of the diagrams obstructs the appearance of boundary degeneration in the compactification of one-dimensional moduli spaces, and hence we can still apply the usual argument to show (1) $\Phi_0$ and $\Phi_1$ are chain maps, and (2) $\Phi_0\circ \Phi_1$ and $\Phi_1\circ \Phi_0$ are homotopy equivalent to the identity map. Therefore, $\Phi_0$ and $\Phi_1$ are chain homotopy equivalences. 
\end{proof}

\begin{defn}\label{Definition, z passable}
An immersed multicurve is called $z$-passable if it is admissibly equivalent to a $z$-adjacent multicurve.
\end{defn}

\begin{rmk}
 We can easily arrange $\alpha_K$ to be a $z$-passable multicurve; see Example \ref{Example, obtaning immersed curve admissibly equivalent to z-adj curve} below. Moreover, when the pattern knot admits a genus-one doubly pointed Heegaard diagram, we can even drop the admissibility condition; see Section \ref{Subsection, 1 1 pattern knots}.
\end{rmk}

\begin{exam}\label{Example, obtaning immersed curve admissibly equivalent to z-adj curve} We give a simple way to arrange an immersed multicurve $\alpha_K$ to be $z$-passable. 
	Without loss of generality, we consider a single component $\gamma$ of $\alpha_K$ each time, and we orient $\gamma$ arbitrarily. We view the torus $T^2$ as a square as usual and position $\gamma$ such that the elementary arcs hitting the top edge are separated into two groups of arcs where the arcs in a single group intersect the top edge in the same direction; see Figure \ref{Figure, admissible immersed curve} (1). Next, we perform a Reidemeister-II-like move to the two groups as in Figure \ref{Figure, admissible immersed curve} (2). Perform the above modification for every component of $\alpha_K$. We claim the resulting multicurve, which we denote $\alpha_K'$, is a $z$-passable multicurve. 
	\begin{figure}[htb!]
		\centering{
			\includegraphics[scale=0.4]{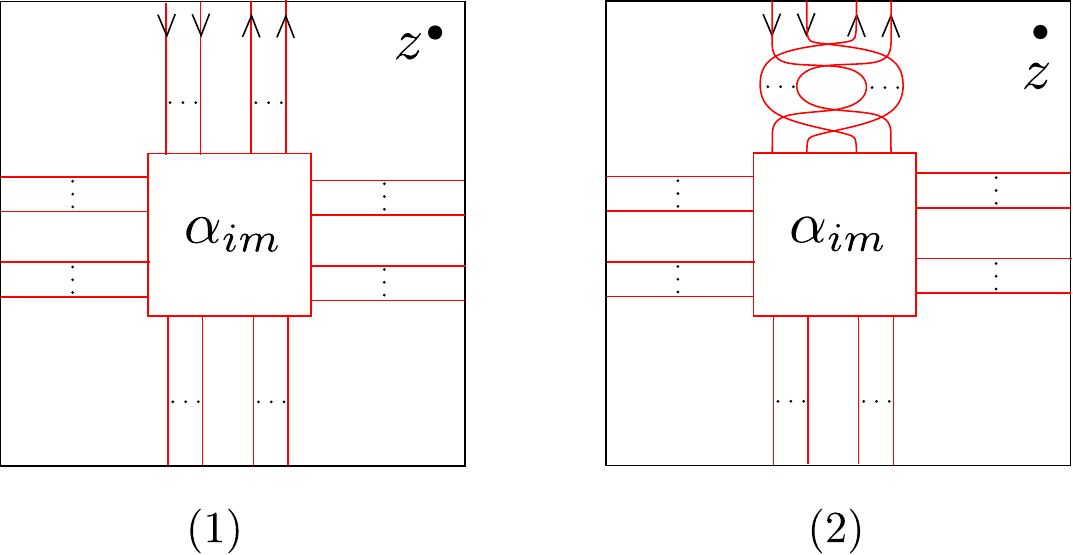}
			\caption{(2) is an $z$-passable immersed curve obtained from (1).}
			\label{Figure, admissible immersed curve}
		}
	\end{figure}
	
	We justify the claim when $\gamma$ is homologically trivial; the case where $\gamma$ is homologically essential is similar. We first check that $\alpha_K'$ is admissible by checking that there are no zero- or one-cornered $\alpha$ bounded domains $B$ with $n_z(B) = 0$. First note that for any  zero- or one-cornered $\alpha$ bounded domains $B$, $\partial B$ must include an elementary arc meeting the top edge of the square. To see this, note that $\partial B$ is a nullhomologous curve in the torus and thus lifts to a closed path in the universal cover. Cutting along (lifts of) the meridian (i.e., $\mathbb{Z}\times \mathbb{R}$) breaks $\partial B$ into pieces, with at least two of these (the leftmost and rightmost piece) forming bigons with the meridian. At least one of those two pieces has no corners (since $B$ is zero- or one-cornered). The cornerless piece must intersect the longitude because $\alpha_K'$ is reduced, and the subarc of $\partial B$ directly below this intersection with the longitude gives an elementary arc meeting the top edge of the square. Next we observe that the elementary arcs near the top edge of the square are arranged such that each arc has the base point $z$ both on its left and on its right, in each case without oppositely oriented arcs in between the arc and $z$, and this implies that no domain whose boundary includes one of these elementary arcs can have $n_z(B) = 0$.
	Having shown the immersed curve $\alpha_K'$ is admissible, it remains to check that it is $z$-passable. Recall from Proposition \ref{Proposition, arranging a curve to be z-adjacent} that we can perform a sequence of finger moves to achieve a $z$-adjacent position. Note that all the intermediate diagrams are admissible by exactly the same argument above.  
\end{exam}

\subsection{Proof of the main theorem, ungraded version}\label{Subsection, proof of the main theorem, ungraded version}
This subsection is devoted to proving the ungraded version of Theorem \ref{Theorem, main theorem}.   

A satellite knot is constructed via the so-called satellite operation that requires a pattern knot and a companion knot as input. A pattern knot is an oriented knot $P$ in an oriented solid torus $S^1\times D^2$, where an oriented meridian $\mu$ and an oriented longitude $\lambda$ are chosen for $\partial (S^1\times D^2)$ so that the orientation determined by $(\mu,\lambda)$ coincides with the induced boundary orientation. A companion knot is an oriented knot $K$ in the 3-sphere. We orient any Seifert longitude of $K$ using the parallel orientation, and orient any meridian $m$ of $K$ so that $lk(m,K)=1$. The satellite knot $P(K)$ is obtained by gluing $(S^1\times D^2, P)$ to the companion knot complement $S^3\backslash \nu(K)$ so that the chosen meridian $\mu$ is identified with a meridian of $K$ and that the chosen longitude $\lambda$ is identified with the Seifert longitude of $K$; $P(K)$ is given by viewing $P$ as a knot in the glued-up 3-sphere $(S^1\times D^2)\cup (S^3\backslash\nu(K))$.   
 
We state the main theorem again below for the readers' convenience. Recall that any pattern knot can be represented by a doubly-pointed bordered Heegaard diagram \cite[Section 11.4]{LOT18}.
\MainTheorem*

Given a doubly pointed bordered Heegaard diagram $\mathcal{H}_{w,z}$ for the pattern knot, we will construct an arced bordered Heegaard diagram $\mathcal{H}_{X(P)}$; the Heegaard diagram $\mathcal{H}_{X(P)}$ specifies a bordered 3-manifold $X(P)$ with two boundary components\footnote{Strictly speaking, an arced bordered Heegaard diagram specifies a strongly bordered 3-manifold in the sense of Definition 5.1 in \cite{Lipshitz2015}, where there is also a framed arc in addition to the underlying bordered 3-manifold. This extra structure will not be relevant to us, so we will not specify it. }, where (1) the underlying 3-manifold is $S^1\times D^2\backslash \nu(P)$, (2) the parametrization of $\partial (S^1\times D^2)$ is the standard meridian-longitude parametrization, and (3) the parametrization of interior boundary $\partial (\nu(P))$ is given by a meridian of $P$ and some longitude of $P$. (The choice of the longitude of $P$ does not matter).

We describe how to obtain $\mathcal{H}_{X(P)}$ from the doubly pointed bordered Heegaard diagram $\mathcal{H}_{w,z}$. This is a standard construction, similar to the one appearing in \cite{LOT18} Section 11.7; the reader familiar with it may skip this paragraph and consult Figure \ref{Figure, obtaining H_P} for an overview.
\begin{figure}[htb!]
	\centering{
		\includegraphics[scale=0.48]{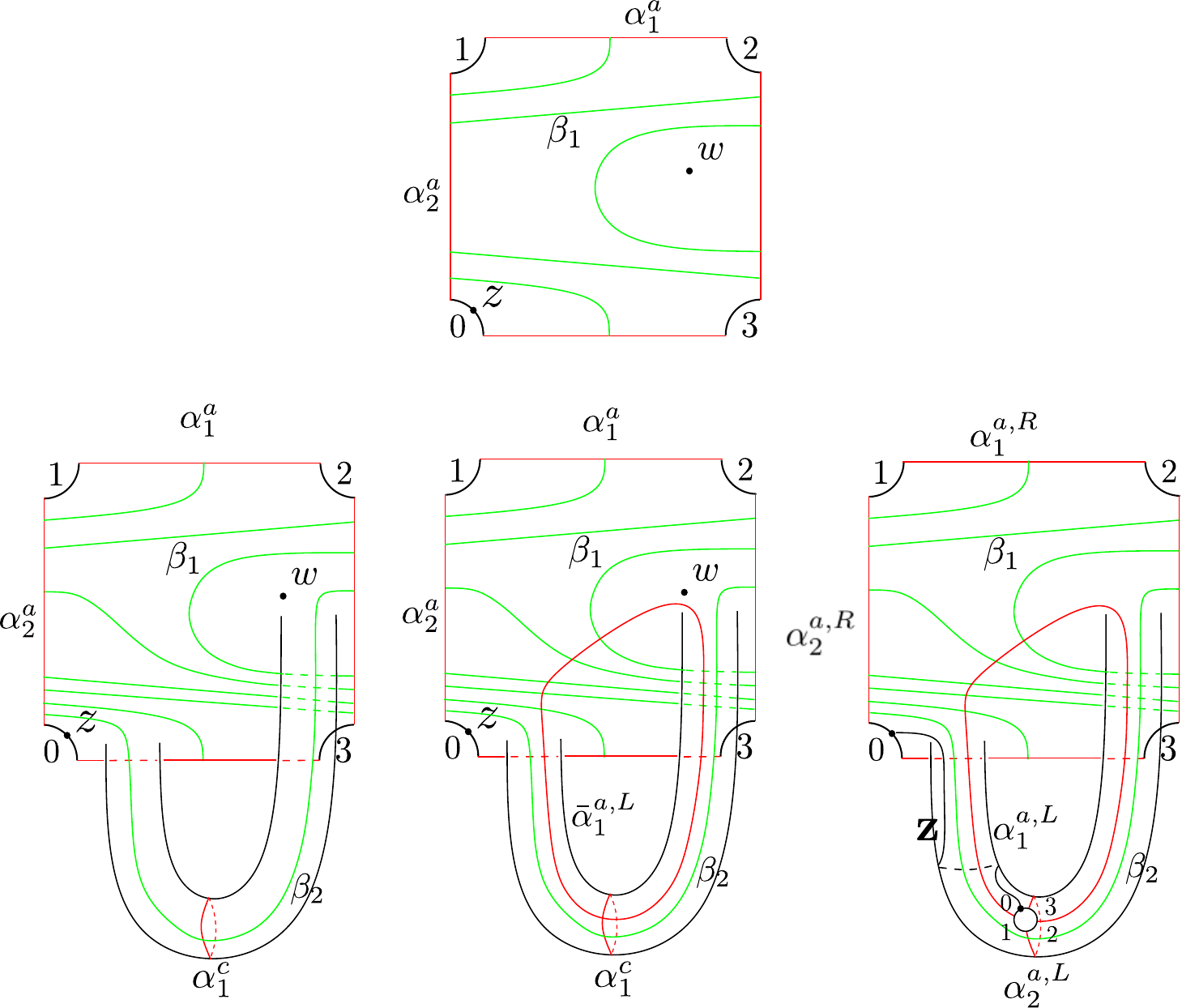}
		\caption{An example of obtaining $\mathcal{H}_{X(P)}$ from $\mathcal{H}_{w,z}$. Here, $\mathcal{H}_{w,z}$ is showed on the top row; it is a genus-one Heegaard diagram for the $(3,1)$-cable pattern. $\mathcal{H}_{X(P)}$ is the rightmost diagram on the second row.  }
		\label{Figure, obtaining H_P}
	}
\end{figure}
Assume $\mathcal{H}_{w,z}$ is of genus $g$. First, we stabilize $\mathcal{H}_{w,z}=(\bar{\Sigma},\bm{\bar{\alpha}},\bm{\beta},w,z)$ to get a new doubly pointed bordered Heegaard diagram $\mathcal{H}'_{w,z}=(\bar{\Sigma}',\bm{\bar{\alpha}}\cup\{\alpha^c_{g}\},\bm{\beta}\cup\{\beta_{g+1}\},w,z)$. More concretely, $\bar{\Sigma}'$ is obtained from $\bar{\Sigma}$ by attaching a two-dimensional one-handle, with feet near the base points $w$ and $z$. Parametrize the new one-handle by $S^1\times[0,1]$, where $S^1\times\{0\}$ is the feet circle near $z$, and $S^1\times\{1\}$ is the feet circle near $w$. We also parametrize $S^1$ by $[0,2\pi]/(0\sim 2\pi)$. The new $\alpha$-circle $\alpha^c_{g}$ is the belt circle $S^1\times\{1/2\}$ of the new one-handle. Let $p_1=(0,0)$ and $p_2=(0,1)$ be two points on the two feet circles of the one-handle. The new $\beta$ circle $\beta_{g+1}$ is the union of two arcs $l_1$ and $l_2$ connecting $p_1$ and $p_2$, where $l_1$ is an arc in $\bar{\Sigma}\backslash \bm{\beta}$ and $l_2$ is the arc $\{(0,t)|t\in[0,1]\}$ in new one-handle. Next, introduce a new curve $\bar{\alpha}_{1}^{a,L}$ as follows. Let $l_z$ be an arc from $z$ to the point $(-1,0)\in S^1\times\{0\}$ does not intersect any of the $\alpha$- and $\beta$-curves. Let $l_2'$ be the arc $\{(1,t)|t\in[0,1]\}$ in the one-handle; denote the endpoints of $l_2'$ by $p_1'$ and $p_2'$.  Let $l_1'$ be an arc connecting $p_1'$ and $p_2'$ in $\bar{\Sigma}\backslash \{\bm{\bar{\alpha}}\cup l_z\}$. Let $\bar{\alpha}_{1}^{a,L}=l_1'\cup l_2'$. Then $\bar{\alpha}_{1}^{a,L}$ intersects $\alpha^c_{g}$ geometrically once at a point $p$. Note $\alpha_g^c$ is the meridian of $P$, and $\bar{\alpha}_1^{a,L}$ is a longitude of $P$. Let $\bar{\Sigma}''$ be the circle compactification of $\bar{\Sigma}'\backslash\{p\}$. Denote the new boundary circle by $\partial_L\bar{\Sigma}''$, and denote the boundary circle inherited from $\partial\bar{\Sigma}$ by $\partial_R\bar{\Sigma}''$. Let ${\alpha}_{1}^{a,L}=\bar{\alpha}_{1}^{a,L}\backslash\{p\}$, and let ${\alpha}_{2}^{a,L}=\alpha^c_{g}\backslash\{p\}$. Let $\alpha_1^{a,R}=\alpha_1^a$, and let $\alpha_2^{a,R}=\alpha_2^a$. Let $\bm{\bar{\alpha}}''=\{\alpha^{a,L}_1,\alpha^{a,L}_2,\alpha^{a,R}_1,\alpha^{a,R}_2,\alpha^c_1,\ldots,\alpha^c_{g-1}\}$. Let $\bm{\beta}''=\bm{\beta}\cup\{\beta_{g+1}\}$. Label the Reeb chords corresponding to the new boundary circle $\partial_L\bar{\Sigma}''$ by $\sigma_i$ ($i=0,1,2,3$) so that $\sigma_2$ and $\sigma_3$ lie on the side attached to the feet near $w$, and $\sigma_0$ and $\sigma_1$ lie on the side attached to the feet near $z$. Let $z_R=z$, and let $z_L$ be a point on $\sigma_0$.  Let $\bm{z}$ be an arc connecting $z_R$ and $z_L$ in the complement of $\bm{\bar{\alpha}}'' \cup \bm{\beta}''$; $\bm{z}$ exists since we can obtain such an arc by extending $l_z$. Finally, we let $\mathcal{H}_{X(P)}=(\bar{\Sigma}'',\bm{\bar{\alpha}}'',\bm{\beta}'',\bm{z})$. See Figure \ref{Figure, obtaining H_P}.

\begin{lem}\label{Lemma, bi-admissibility of H_P(alpha_im)}
Let $\mathcal{H}_{X(P)}$ be the arced bordered Heegaard diagram obtained from $\mathcal{H}_{w,z}$ via the above procedure. Let $\alpha_{im}$ be a $z$-adjacent multicurve. Then $\mathcal{H}_{X(P)}(\alpha_{im})$ is unobstructed and bi-admissible.
\end{lem}
\begin{proof}
The unobstructedness follows from Proposition \ref{Proposition, paring diagrams are admissible}. We move to see bi-admissibility in the sense of Definition \ref{Definition, bi-admissibility}. Note that periodic domains $B$ for $\mathcal{H}_{X(P)}(\alpha_{im})$ with $n_{\sigma_0}(B)=n_{\sigma_1}(B)=0$ (respectively $n_{\sigma_2}(B)=n_{\sigma_3}(B)=0$) correspond to periodic domains $B'$ for $\mathcal{H}_{w,z}(\alpha_{im})$ with $n_z(B')=0$ (respectively $n_w(B')=0$). Therefore, the bi-admissibility of $\mathcal{H}_{w,z}(\alpha_{im})$, which was shown in Proposition \ref{Proposition, H_w,z(alpha_im) is bi-admissible and unobstructed}, implies the bi-admissibility of $\mathcal{H}_{X(P)}$.  
\end{proof}

Recall $\mathcal{H}_{id}$ is the standard doubly pointed bordered Heegaard diagram for the identity pattern knot. 

\begin{lem}\label{Lemma, handleslides and destabilization diagrams}
$CFK_{\mathcal{R}}(\mathcal{H}_{w,z}(\alpha_{im}))$ is chain homotopy equivalent to $CFK_{\mathcal{R}}(\mathcal{H}_{id}\cup\mathcal{H}_{X(P)}(\alpha_{im}))$.
\end{lem}
\begin{proof}
Note the doubly pointed Heegaard diagram $\mathcal{H}_{id}\cup \mathcal{H}_{X(P)}(\alpha_{im})$ is obtained from $\mathcal{H}_{w,z}(\alpha_{im})$ by two stabilizations; see Figure \ref{Figure, double stabilization}. 
\begin{figure}[htb!]
	\centering{
		\includegraphics[scale=0.40]{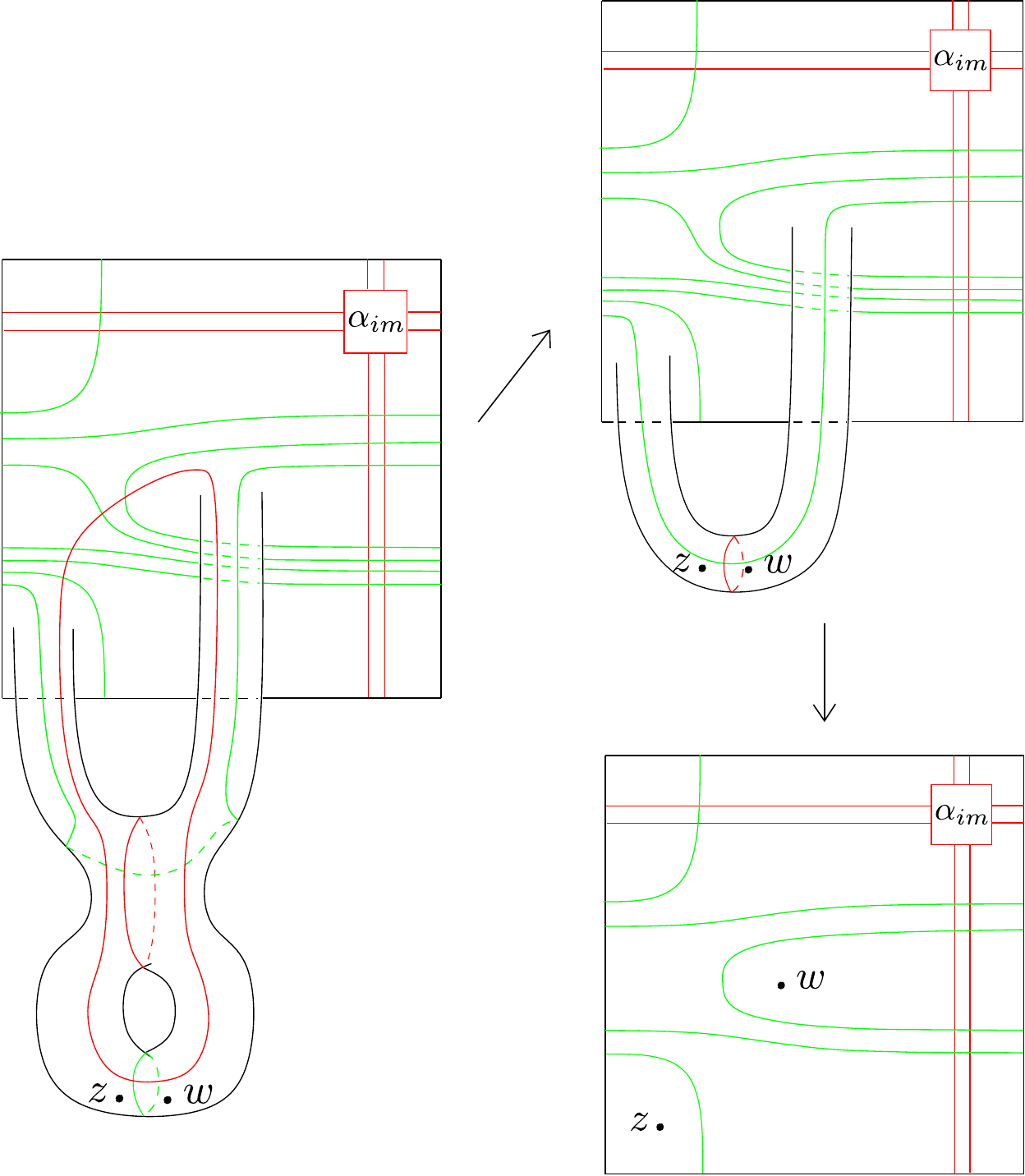}
		\caption{An example of $\mathcal{H}_{id}\cup\mathcal{H}_{X(P)}(\alpha_{im})$ (left) and $\mathcal{H}_{w,z}(\alpha_{im})$ (lower right). Here, $P$ is the $(3,1)$-cable. These two diagrams are related via handleslides and destabilizations, where the handleslides do not involve sliding over the immersed $\alpha$-curve.}
		\label{Figure, double stabilization}
	}
\end{figure}
In particular, it is also bi-admissible, and hence one can define
$CFK_{\mathcal{R}}(\mathcal{H}_{id}\cup\mathcal{H}_{X(P)}(\alpha_{im}))$. We claim there is a sequence of Heegaard moves relating $\mathcal{H}_{id}\cup\mathcal{H}_{X(P)}(\alpha_{im})$ and $\mathcal{H}_{w,z}(\alpha_{im})$ which do not involve sliding $\alpha$ curves over $\alpha_{im}$. To see this, note that on $\mathcal{H}_{id}\cup\mathcal{H}_{X(P)}(\alpha_{im})$ there is a $\beta$-circle between the $w$ and $z$ base points that intersects an $\alpha$-circle geometrically once; denote these curves by $\beta_{g+2}$ and $\alpha_{g+2}$ respectively. After sliding other beta curves over $\beta_{g+2}$ if necessary, we may assume $\alpha_{g+2}$ does not intersect other beta curves, and hence we can destabilize $\mathcal{H}_{id}\cup\mathcal{H}_{X(P)}(\alpha_{im})$ along $\alpha_{g+2}$ and $\beta_{g+2}$. Now we arrive at an intermediate Heegaard diagram; see Figure \ref{Figure, double stabilization} (upper right). It is a stabilization of $\mathcal{H}_{w,z}(\alpha_{im})$. On this intermediate Heegaard diagram, there is an $\alpha$-circle $\alpha_{g+1}$ that intersects only one $\beta$-circle $\beta_{g+1}$, and the geometric intersection number is one. So, we may slide other $\alpha$-curves over $\alpha_{g+1}$ if necessary so that $\beta_{g+1}$ do not intersect other $\alpha$-curves. After this, we destabilize the Heegaard diagram along $\alpha_{g+1}$ and $\beta_{g+1}$, and the resulting Heegaard diagram is $\mathcal{H}_{w,z}(\alpha_{im})$. The homotopy equivalence between $CFK_{\mathcal{R}}(\mathcal{H}_{w,z}(\alpha_{im}))$ and $CFK_{\mathcal{R}}(\mathcal{H}_{id}\cup\mathcal{H}_{X(P)}(\alpha_{im}))$ follows from the homotopy invariance of knot Floer chain complexes established in Proposition \ref{Proposition, invariance of CFK_R}.
\end{proof}

With these lemmas at hand, we now prove the ungraded version of Theorem \ref{Theorem, main theorem}.
\begin{proof}[Proof of Theorem \ref{Theorem, main theorem}, ungraded version]
In view of Proposition \ref{Proposition, finger move independence}, we may assume the immersed multicurve $\alpha_K$ for the knot complement of $K$ is z-adjacent. Let $\mathcal{H}_{X(P)}$ be the arced bordered Heegaard diagram obtained from $\mathcal{H}_{w,z}$ via the ``punctured-stabilization procedure". Throughout, when referring to the type D structure of a knot complement, we use the meridian and Seifert longitude to parametrize the boundary. By standard arguments, we can arrange that  $\mathcal{H}_{X(P)}$ is left provincially admissible at the cost of isotopy of the $\beta$ curves. By Theorem \ref{Theorem, paring theorem via nice diagrams}, we have 
\begin{displaymath}
\begin{aligned}
\widehat{CFD}(\mathcal{H}_{X(P)}(\alpha_K))&\cong \widehat{CFDA}(\mathcal{H}_{X(P)})\boxtimes \widehat{CFD}(\alpha_K)\\
&\cong\widehat{CFDA}(\mathcal{H}_{X(P)})\boxtimes \widehat{CFD}(S^3 \backslash \nu(K))\\
&\cong \widehat{CFD}(S^3\backslash \nu(P(K)))
\end{aligned}
\end{displaymath}

Therefore, up to homotopy equivalence, the extended type D structure $\widetilde{CFD}(\mathcal{H}_{X(P)}(\alpha_K))$ extends $\widehat{CFD}(S^3\backslash \nu(P(K)))$. Consequently, we have the following:
\begin{displaymath}
\begin{aligned}
CFK_{\mathcal{R}}(P(K))&\cong \widetilde{CFA}(\mathcal{H}_{id}){\boxtimes}\widetilde{CFD}(S^3\backslash \nu(P(K)))\\
&\cong \widetilde{CFA}(\mathcal{H}_{id}){\boxtimes}\widetilde{CFD}(\mathcal{H}_{X(P)}(\alpha_K))\\
&\cong CFK_{\mathcal{R}}(\mathcal{H}_{id}\cup \mathcal{H}_{X(P)}(\alpha_K))
\end{aligned}
\end{displaymath}
Here, the last equality follows from applying Theorem \ref{Theorem, 2nd pairing theorem}.
Note ${\boxtimes}$ in the above equation is well-defined since $\mathcal{H}_{X(P)}(\alpha_K)$ is bi-admissible by Lemma \ref{Lemma, bi-admissibility of H_P(alpha_im)}. Now, by Lemma \ref{Lemma, handleslides and destabilization diagrams}, $CFK_{\mathcal{R}}(P(K))$ is chain homotopy equivalent to $CFK_{\mathcal{R}}(\mathcal{H}_{w,z}(\alpha_K))$. 
\end{proof}

\subsection{$\mathcal{H}_{w,z}(\alpha_{K})$ is gradable}
We want to show that the chain homotopy equivalence established in the previous subsection preserves the $w$-grading and $z$-grading of knot Floer chain complexes. As the first step, we need to show that $\mathcal{H}_{w,z}(\alpha_{K})$ is gradable (in the sense of Definition \ref{Definition, gradable doubly-pointed diagram}).

\begin{prop}\label{Proposition, Maslov index of periodic domain}
The diagram $\mathcal{H}_{w,z}(\alpha_{K})$ is gradable.
\end{prop}
\noindent In addition to being gradable, note that the results in the previous subsection also imply that $\widehat{HF}(\mathcal{H}_{w}(\alpha_{K}))\cong \widehat{HF}(\mathcal{H}_{z}(\alpha_{K}))\cong \mathbb{F}$. Therefore we can define an absolute bigrading on $CFK_{\mathcal{R}}(\mathcal{H}_{w,z}(\alpha_{K}))$.

We will reduce the proof of Proposition \ref{Proposition, Maslov index of periodic domain} to the case where $\mathcal{H}_{w,z}$ is of genus one. If $\mathcal{H}_{w,z}$ is a genus-one bordered Heegaard diagram, then one can define a Maslov grading $m(-)$ on $CFK_{\mathcal{R}}(\mathcal{H}_{w,z}(\alpha_{K}))$ as follows. Given any two generators $x$ and $y$, let $p_0$ and $p_1$ be two paths from $x$ to $y$ in $\alpha_K$ and $\beta$ respectively such that $p_0-p_1$ lifts to a closed path $\gamma$ in the universal cover $\mathbb{R}^2$ of the genus-one Heegaard surface. Up to perturbing the curves, we may assume that $p_0$ and $p_1$ intersect in right angles at $x$ and $y$. Then $m(x)-m(y)$ is equal to $\frac{1}{\pi}$ times the total counterclockwise rotation along the smooth segments of $\gamma$ minus twice the number of the (lifts of) base point $z$ enclosed by $\gamma$; see \cite[Definition 35]{Hanselman2022}. This Maslov grading is also defined (by the same definition) when the $\beta$ curve is only immersed. In \cite{Hanselman2022}, it is shown that the Maslov grading thus defined on a pairing diagram of two immersed curves agrees with the Maslov grading computed using the grading package of bordered Heegaard Floer homology. Next, we show this Maslov grading can be equivalently defined in terms of the index of domains. 

\begin{prop}\label{Proposition, identify two ways to compute the Maslov index for genus one diagram}
	Let $\mathcal{H}_{w,z}$ be a genus-one bordered Heegaard diagram and let $m(-)$ be the Maslov grading on $\mathcal{G}(\mathcal{H}(\alpha_{K}))$ mentioned above. Let $B\in \pi_2(x,y)$ be a domain connecting $x$ and $y$ with $\partial B= p_0-p_1$. Then $m(x)-m(y)=\text{ind}(B)-2n_z(B)$. Moreover, this result extends to the case where the $\beta$ is immersed, in which we define the index of $B$ by $$\text{ind}(B)=e(B)+n_{x}(B)+n_{y}(B)-s(\partial _{\alpha_K} B)-s(\partial_{\beta}B).$$ 
	(Here $s(-)$ denotes the self-intersection number of an oriented immersed arc as defined in Section \ref{Section, embedded index formula})
\end{prop}
Before proving Proposition \ref{Proposition, identify two ways to compute the Maslov index for genus one diagram} we introduce some terminology. It will be clear later that we can assume $p_0-p_1$ is immersed and only has discrete double points. 

\begin{defn}
A \textit{cornered immersed loop} in $T^2$ is the union of two oriented immersed arcs $p_0$ and $p_1$ with at most discrete double points such that 
\begin{itemize}
\item[(1)] $p_0$ and $p_1$ share common endpoints,
\item[(2)] the interior of $p_0$ and $p_1$ intersect transversally, 
\item[(3)] $p_0-p_1$ is an oriented loop which is null-homologous,
\item[(4)] $p_0$ and $p_1$ intersect transversally at the endpoints if $p_0$ and $p_1$ are non-degenerate (i.e., not a point), and
\item[(5)] if one of $p_0$ and $p_1$ is degenerate, the remaining arc forms a smooth loop after identifying the endpoints. 
\end{itemize}
The endpoints of $p_0$ (or equivalently, $p_1$) are called \textit{corners} of the cornered immersed loop.
\end{defn}

\begin{defn}
Two cornered immersed loops $p_0-p_1$ and $p_0'-p_1'$ in $T^2$ are called \textit{cornered identical} if they share the same set of corners $\{x,y\}$ (or $\{x\}$ if the loops have degenerate arcs) and there are arbitrarily small neighborhoods $N_x$ and $N_y$ of $x$ and $y$ respectively such that $(p_0-p_1)|_{N_x}=(p'_0-p'_1)|_{N_x}$ and $(p_0-p_1)|_{N_y}=(p'_0-p'_1)|_{N_y}$.
\end{defn}
	\begin{figure}[htb!]
	\centering{
		\includegraphics[scale=0.6]{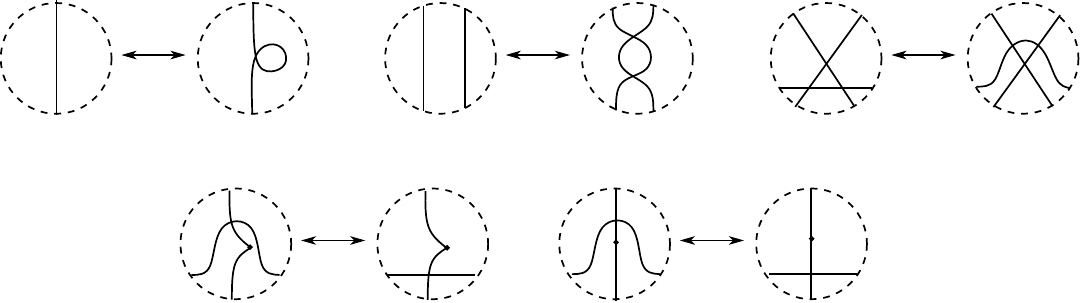}
		\caption{Upper row from left to right: Reidemeister I, II, and III move. Lower row from left to right: an isotopy that crosses a non-degenerate corner and a degenerate corner.}
		\label{Figure, Reidemeister moves and istopy}
	}
\end{figure}
\begin{lem}\label{Lemma, Reidemeister Moves Of Immersed Loop}
If two cornered immersed loops $p_0-p_1$ and $p_0'-p_1'$ are cornered identical, then they are related by a finite sequence of moves of the following types:
\begin{itemize}
\item[(1)]Reidemeister moves that do not involve the corners and
\item[(2)]isotopy that possibly cross the corners.
\end{itemize} 
(See Figure \ref{Figure, Reidemeister moves and istopy}) Here, we require $(p_0-p_1)|_{N_x}$ and $(p_0-p_1)|_{N_y}$ are fixed throughout the modification for some sufficiently small neighborhoods 
${N_x}$ and ${N_y}$ of the corners. 
\end{lem}
\begin{proof}
One can prove this applying the usual Reidemeister-move equivalence of knot diagrams (by treating both immersed loops as diagrams for the unknot via imposing proper crossing information); note that any Reidemeister move involving a corner can be traded by an isotopy crossing the corner and a Reidemeister move that does not involve the corner.
\end{proof}

\begin{defn}\label{Definition, index and net rotaiont number of cornered immersed loop}
Given a cornered immersed loop $p_0-p_1$ in $T^2$. Let $\tilde{p}_0-\tilde{p}_1$ be a lift of $p_0-p_1$ in $\mathbb{R}^2$ and let $\tilde{B}$ be the bounded domain in $\mathbb{R}^2$ such that $\partial \tilde{B}=\tilde{p}_0-\tilde{p}_1$. Let $B$ be the domain in $T^2$ obtained from $\tilde{B}$ by applying the covering projection. Define the \textit{index of the cornered immersed loop} as $$\text{ind}(p_0-p_1)=e(B)+n_{x}(B)+n_{y}(B)-s(\partial _{p_0} B)-s(\partial_{p_1}B),$$  
where $x$ and $y$ are the corners. 

Define the \textit{net rotation number} $nr(p_0-p_1)$ to be
$\frac{1}{\pi}$ times the counterclockwise net rotation along the smooth segments $p_0$ and $p_1$.
\end{defn}

\begin{lem}\label{Lemma, syncing net rotation and index}
Suppose $p_0-p_1$ and $p_0'-p_1'$ are cornered immersed loops differ by an isotopy or a Reidemeister move. Then 
$$\text{ind}(p_0-p_1)-\text{ind}(p'_0-p'_1)=nr(p_0-p_1)-nr(p'_0-p'_1).$$
\end{lem}
\begin{figure}[htb!]
	\centering{
		\includegraphics[scale=0.9]{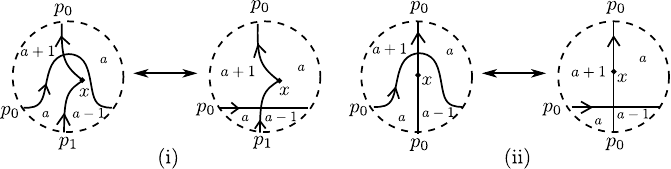}
		\caption{Local diagrams for isotopies that cross a corner. The numbers $a$, $a-1$, and $a+1$ indicate the multiplicities of the regions.}
		\label{Figure, local isotpy}
	}
\end{figure}
\begin{proof}
First, we examine the effect of an isotopy on both quantities. Clearly, the net rotation number is unchanged. We separate the discussion of the index into two cases according to whether the isotopy crosses corners or not. If the isotopy does not cross the corners, it clearly does not change the index as well whence we are done. If the isotopy crosses a corner, then we claim the local multiplicity and the self-intersection numbers change in a way that cancel each other, leaving the index unchanged. This claim can be seen by examining local diagrams, which are further divided into two cases according to whether the corner is degenerate or not. When the corner is non-degenerate, the local diagram of one case is shown in Figure \ref{Figure, local isotpy} (i); all the other cases can be obtained from this case by swapping the labels and orientations of the arcs, and the analysis of all cases are similar. In the case shown in Figure \ref{Figure, local isotpy} (i), only $n_x(B)$ and $s(\partial_{p_0}B)$ change: the diagram on the left has $n_x(B)=\frac{a+(a-1)+(a-1)+(a-1)}{4}$ and the local self-intersection of $p_0$ contributes $s_{p_0}=-1$; the diagram on the right has $n_x(B)=\frac{(a+1)+a+a+a}{4}$ and there are no self-intersections of the arcs in the local diagram so the local contribution $s_{p_0}=0$. In both diagrams we have $n_x(B)-s_{p_0}=\frac{4a+1}{4}$, and hence the index is unchanged. When the corner is degenerate, one of the cases is shown in Figure \ref{Figure, local isotpy} (ii). In this case, only $n_x$ and the self-intersection of $p_0$ change: the diagram on the left has $n_x=\frac{a+a+(a-1)+(a-1)}{4}$ and a local contribution of the self-intersection of $p_0$ given by $s_{p_0}=-1$; the diagram on the right has $n_x=\frac{(a+1)+(a+1)+a+a}{4}$ and a local contribution of the self-intersection of $p_0$ given by $s_{p_0}=1$. In both local diagrams we have $n_x(B)+n_x(B)-s_{p_0}=2a$, and hence the index is unchanged. All other cases can be obtained from this case by swapping the labels and orientations of the arcs, and the analysis of all cases are similar.

\begin{figure}[htb!]
	\centering{
		\includegraphics[scale=0.5]{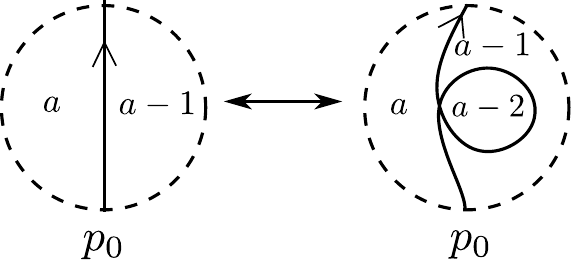}
		\caption{The local diagram for Reidemeister I move. The numbers $a$, $a-1$, and $a-2$ indicate the multiplicities of the regions.}
		\label{Figure, Reidemesiter One}
	}
\end{figure}
Next, we examine the effect of Reidemeister I move. Up to swapping orientations and the labels, we may assume the local diagram is as shown in Figure \ref{Figure, Reidemesiter One} The net rotation number of the diagram on the right is 2 less than that of the diagram on the left. For the index comparison, the Euler measure of the local domain on the right is $1$ less than that of the left diagram and the self-intersection number $s(\partial_{p_0}B)$ of the right diagram is $1$ more than that of the left diagram; in total the index of the diagram on the right is 2 less than that of the diagram on the left. Therefore, the changes in the net rotation and in the index are the same after doing a Reidemeister I move. 

Next, we examine the effect of Reidemeister II moves. It does not change the net rotation number. Also, it does not affect the Euler measure and the local multiplicities at the corners. A Reidemeister II move creates/annihilates a pair of self-intersection points whose signs cancel each other if both arcs involved are on $p_0$ or $p_1$, and otherwise does not involve self-intersections; in both cases the self-intersection numbers are unchanged. So, the index does not change as well.

Finally, it is easy to see that a Reidemeister III move does not change the net rotation number. It is also easy to see a Reidemeister III move does not change the Euler measure, local multiplicities at the corners, or self-intersections, and hence it does not change the index either.   
 \end{proof}

\begin{prop}\label{Proposition, index equals to net rotation}
Let $p_0-p_1$ be a cornered immersed loop. Then $\text{ind}(p_0-p_1)=nr(p_0-p_1)$.
\end{prop}
\begin{proof}
By Lemma \ref{Lemma, syncing net rotation and index} and Lemma \ref{Lemma, Reidemeister Moves Of Immersed Loop}, it suffices to show that $p_0-p_1$ is cornered identical with some cornered immersed loop whose index coincides with the net rotation number.

If at least one of $p_0$ and $p_1$ is degenerate, $p_0-p_1$ is cornered identical with an embedded circle that passes the corner, and it is easy to see the index and the net rotation number coincide on an embedded circle with a degenerate corner. 

\begin{figure}[htb!]
	\centering{
		\includegraphics[scale=0.55]{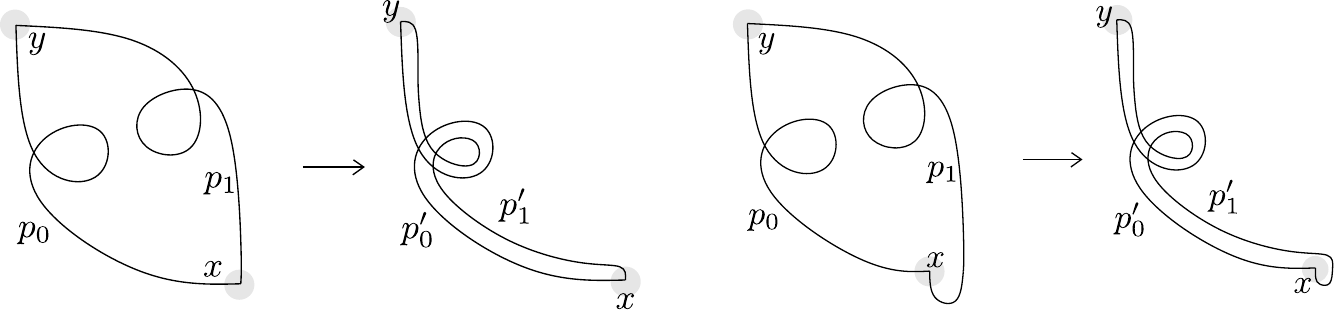}
		\caption{Deforming $p_0-p_1$.}
		\label{Figure, corner identical loops}
	}
\end{figure}
 Next, we discuss the case where $p_0-p_1$ is non-degenerate. We first construct a cornered immersed loop $p_0'-p_1'$ that is cornered identical to $p_0-p_1$ as follows. Let $p_0'=p_0$. We shall construct $p_1'$ to be a path which is almost a parallel push-off of $p_0$. (See Figure \ref{Figure, corner identical loops} for examples.) To spell out the construction, let $f_0:[0,1]\rightarrow T^2$ be an immersion such that $f_0([0,1])=p_0$.  Let $\hat{N}$ be a sufficiently small tubular neighborhood of $p_0$ such that it can realized as the image of an extension of $f_0$, i.e., there exits an immersion $\tilde{f_0}:[0,1]\times[-\epsilon,\epsilon]\rightarrow T^2$ such that $\tilde{f_0}|_{[0,1]\times\{0\}}=f_0$ and $\tilde{f_0}([0,1]\times \{pt\})$ is a parallel push-off of $p_0$ for any $pt\in[-\epsilon,0)\cup (0,\epsilon]$. We can further assume near the two corners $x=f_0(0)$ and $y=f_0(1)$, the other arc $p_1$ is contained in $\tilde{f_0}(\{0,1\}\times [-\epsilon,\epsilon])$; denote these two arcs on $p_1$ near $x$ and $y$ by $p_x$ and $p_y$ respectively. We construct $p_1'$ in two different cases. In the first case, both $p_x$ and $p_y$ are on the same side of $p_0$, say, $p_x=\tilde{f_0}(\{0\}\times [0,\epsilon])$ and $p_y=\tilde{f_0}(\{1\}\times [0,\epsilon])$. Then we let $p_1'$ be the path obtained from $p_x\cup p_y \cup \tilde{f_0}([0,1]\times\{\epsilon\})$ by smoothing the corners; see Figure \ref{Figure, corner identical loops} (left) for an example. In the second case, $p_x$ and $p_y$ are on different sides of $p_0$, say, $p_x=\tilde{f_0}(\{0\}\times [0,\epsilon])$ and $p_y=\tilde{f_0}(\{1\}\times [-\epsilon,0])$. In this case, we extend $\tilde{f_0}$ near $x$ slightly to an immersion $F_0:[-\delta,1]\times[-\epsilon,\epsilon]\rightarrow T^2$ for some $\delta>0$ such that $F_0|_{([-\delta,0]\times [-\epsilon,\epsilon])}$ is an embedding and its image intersects $N$ at $\tilde{f_0}(\{0\}\times[-\epsilon,\epsilon])$. We will let $p_1'$ be the path obtained from $p_x\cup F_0([-\delta,0]\times \{\epsilon\})\cup F_0(\{-\delta\}\times[-\epsilon,\epsilon])\cup F_0([-\delta,1]\times \{-\epsilon\})\cup p_y$ by smoothing the corners; see Figure \ref{Figure, corner identical loops} (right) for an example. Note that in both cases, $p_0'-p_1'$ bounds an immersed disk $B$ in $T^2$, and $s(\partial_{p_0'}B)+s(\partial_{p_1'}B)=0$ since self-intersections of $p_0'$ and $p_1'$ are in one-to-one correspondence and have opposite signs. In the case where $p_x$ and $p_y$ are on the same side of $p_0$, both corners of $B$ are acute, and hence we have $e(B)=1-\frac{1}{4}-\frac{1}{4}=\frac{1}{2}$ and  $n_x=n_y=\frac{1}{4}$. Therefore, $\text{ind}(p_0'-p_1')=\frac{1}{2}+\frac{1}{4}+\frac{1}{4}=1$, which is equal to the net rotation number that can be easily computed. In the case where here $p_x$ and $p_y$ are on different sides of $p_0$, one of the corners of $B$ is obtuse and the other one is acute, and we have $e(B)=1-\frac{1}{4}+\frac{1}{4}=1$, $n_x=\frac{3}{4}$, and $n_y=\frac{1}{4}$. Therefore, $\text{ind}(p_0'-p_1')=1+\frac{3}{4}+\frac{1}{4}=2$, which is again equal to the net rotation number that can be computed easily. So, $\text{ind}(p_0'-p_1')=nr(p_0'-p_1')$. 
\end{proof}

We are ready to prove Proposition \ref{Proposition, identify two ways to compute the Maslov index for genus one diagram}.
\begin{proof}[Proof of Proposition \ref{Proposition, identify two ways to compute the Maslov index for genus one diagram}]
Throughout the discussion, we may assume $p_0-p_1$ is an immersed loop with only discrete double points; if not, we can perturb $p_0-p_1$ slightly to achieve such a loop and keep both $m(x)-m(y)$ and $\text{ind}(B)-2n_z(B)$ unchanged. 

Note that it does not matter which domain $B$ we use to compute $\text{ind}(B)-2n_z(B)$ since any two such domains differ by multiples of $[T^2]$, which does not change the quantity. 
For convenience, we take $B$ to be the domain as specified in Definition \ref{Definition, index and net rotaiont number of cornered immersed loop}. It is clear that $n_z(B)$ is equal to the number of base point $z$ enclosed by a lift of $p_0-p_1$ in $\mathbb{R}^2$. By Proposition \ref{Proposition, index equals to net rotation}, we also have $\text{ind}(B)=nr(p_0-p_1)$. So, $m(x)-m(y)=\text{ind}(B)-2n_z(B)$. 
\end{proof}
Next, we prove Proposition \ref{Proposition, Maslov index of periodic domain}.
\begin{proof}[Proof of Proposition \ref{Proposition, Maslov index of periodic domain}]
Let $\bm{x}$ be a generator in $\mathcal{H}_{w,z}(\alpha_{K})$ and let $P\in \tilde{\pi}_2(\bm{x},\bm{x})$ be a non-trivial periodic domain; we need to show that  $\text{ind}(P)-2n_z(P)=0$ and $\text{ind}(P)-2n_w(P)=0$, where $\text{ind}(-)$ is defined in Definition \ref{Definition, embedded Euler Char, index, and moduli space}. Note that $\partial _{\alpha_{K}}P$ sits on a single connected component of $\alpha_{K}$. If it sits on the distinguished component, then since the $\alpha$-curves and $\beta$-curves are homologically independent in $H_1(\Sigma,\mathbb{Q})$ we know $\partial P=\emptyset$, which means $P$ is a multiple of $\Sigma$.  In this case it is celar that $\text{ind}(P) - 2 n_z(P) = 0$. Otherwise, $\partial_{\alpha_{K}}P$ must sit in some null-homologous component of $\alpha_{K}$, and by homological considerations $\partial P$ must be some non-zero multiple of the this component. 
	
	Note that when $\Sigma$ has genus greater than or equal to $2$, the domain $P$ can be viewed as a stabilization of a domain $P'$ in the marked torus $(T^2,z)$ bounded by the null-homologous component. In particular, $n_z(P)=n_w(P)$, and hence $\text{ind}(P)-2n_z(P)=\text{ind}(P)-2n_w(P)$. Moreover, a straightforward computation using the definition of the index shows $\text{ind}(P)-2n_z(P)=0$ if and only if $\text{ind}(P')-2n_z(P')=0$. The latter follows from Proposition \ref{Proposition, identify two ways to compute the Maslov index for genus one diagram} since $\text{ind}(P')-2n_z(P')=m(x)-m(x)=0$, where $x$ is the component of $\bm{x}$ in the genus-one diagram.	
\end{proof}

\subsection{Gradings in the main theorem}\label{Subsection, gradings in the main theorem} We now show the chain homotopy equivalence established in the main theorem preserves the $w$-grading and $z$-grading of knot Floer chain complexes. To do so, it suffices to consider a simpler version of knot Floer chain complexes. 

According to \cite[Theorem 11.19]{LOT18}, ${CFA}^-(\mathcal{H}_{w,z})\boxtimes \widehat{CFD}(\alpha_K)$ is a bi-graded chain complex over $\mathbb{F}[U]$ representing ${gCFK}^-(P(K))$; here ${gCFK}^-(-)$ refers to the version of knot Floer chain complex whose differential only count holomorphic disks that do not cross the $z$ base point. We shall prove the following theorem. 

\begin{thm}\label{Theorem, graded isomorphism of gCFKminus}
$gCFK^-(\mathcal{H}_{w,z}(\alpha_{K}))$ is isomorphic to ${CFA}^-(\mathcal{H}_{w,z})\boxtimes \widehat{CFD}(\alpha_K)$ as a bi-graded chain complex over $\mathbb{F}[U]$.
\end{thm}

 When $\mathcal{H}_{w,z}$ is a genus-one Heegaard diagram, Theorem \ref{Theorem, graded isomorphism of gCFKminus} is true by Proposition \ref{Proposition, identify two ways to compute the Maslov index for genus one diagram} and \cite[Theorem 1.2]{Chen2019}: Generalizing the corresponding argument in \cite{Hanselman2022},  \cite[Theorem 1.2]{Chen2019} shows the Maslov grading on ${CFA}^-(\mathcal{H}_{w,z})\boxtimes {\widehat{CFD}(\alpha_K)}$ identifies with the Maslov grading on $gCFK^-(\mathcal{H}_{w,z}(\alpha_{K}))$ defined via the net rotation numbers. (Strictly speaking, \cite{Chen2019} works with $\widehat{CFK}(-)$ instead of $gCFK^{-}$, but these two versions of knot Floer complexes are equivalent to each other.) Our strategy for proving Theorem \ref{Theorem, graded isomorphism of gCFKminus} is to reduce the higher-genus case to the genus-one case. 
 
\begin{proof}[Proof of Theorem \ref{Theorem, graded isomorphism of gCFKminus}]

First, we claim that $gCFK^-(\mathcal{H}_{w,z}(\alpha_{K}))$ is isomorphic to ${CFA}^-(\mathcal{H}_{w,z})\boxtimes \widehat{CFD}(\alpha_K)$ as ungraded chain complexes; this follows from the ungraded version of Theorem \ref{Theorem, main theorem}. Next, we will show that the $z$-gradings are identified. Since the $z$-gradings on both chain complexes are normalized by the same rule, it suffices to prove the relative $z$-gradings are the same. 

We set up some notation. Let $\Sigma$ and $\Sigma'$ denote the Heegaard surface for $\mathcal{H}_{w,z}(\alpha_{K})$ and $\mathcal{H}_{w,z}$ respectively. Let $\bm{x}_i$ (for $i\in \{1,2\}$) denote a generator of $CFA^-(\mathcal{H}_{w,z})$ and let $y_i$ (for $i\in \{1,2\}$) denote a generator of $\widehat{CFD}(\alpha_K)$; correspondingly, we use $\bm{x}_i\otimes y_i$ to denote a generator of $gCFK^-(\mathcal{H}_{w,z}(\alpha_{K}))$ (where we assume the relevant idempotents match up). Let $B\in \tilde{\pi}_2(\bm{x}_1\otimes y_1,\bm{x}_2\otimes y_2)$ be a domain. For a technical reason, we will assume $B$ is a positive domain; this can always be achieved by adding sufficiently many copies of the Heegaard surface $[\Sigma]$. We will use $B$ to compute the relative $z$-gradings of  $\bm{x}_1\otimes y_1$ and $\bm{x}_2\otimes y_2$ in two ways and compare them: one in Definition \ref{Definition, w- and z- gradings} and one via the grading package in bordered Floer homology; we refer the readers to \cite[Section 2]{Chen2019} for a brief summary of this grading package.

Let $k=n_z(B)$ and write $B=B_0+k[\Sigma]$. Let $B_0'=\Phi(B_0)$ where $\Phi$ denotes the collapsing map, and let $B'=B_0'+k[\Sigma']$. Let $\partial_{\alpha_K} B$ denote the portion of the boundary of $B$ lying on $\alpha_K$. Note $\partial_{\alpha_{K}}(B)=\partial_{\alpha_K}(B_0)$.

We compute the $z$-grading difference using Definition \ref{Definition, w- and z- gradings}.
By Definition \ref{Definition, embedded Euler Char, index, and moduli space}, $\text{ind}(B)=e(B)+n_{\bm{x}_1\otimes y_1}(B)+n_{\bm{x}_2\otimes y_2}(B)-s(\partial_{\alpha_{K}}B)$. Therefore, we have the following equation. 
\begin{equation}\label{Equation, grading difference via domains}
	gr_z(\bm{x}_2\otimes y_2)-gr_z(\bm{x}_1\otimes y_1)=-e(B)-n_{\bm{x}_1\otimes y_1}(B)-n_{\bm{x}_2\otimes y_2}(B)+s(\partial_{\alpha_{K}}B)+2n_z(B).
\end{equation}

We now compute the $z$-grading obtained from the box-tensor product, and we will use $gr_z^{\boxtimes}$ to distinguish it with the $z$-grading computed above. Note $B'_0$ is a domain in $\Sigma'$ with $n_z(B'_0)=0$ and it connects $\bm{x}_1$ to $\bm{x}_2$. Let $gr_A$ denote the grading function for $CFA^-$ that consists of the Maslov component and the $Spin^c$-component. Then 
$$
gr_A(\bm{x}_2)=gr_A(\bm{x}_1)\cdot(-e(B'_0)-n_{\bm{x}_1}(B'_0)-n_{\bm{x}_2}(B'_0),[\partial^\partial B'_0]).
$$
Here, $\partial^\partial B_0'$ denotes the portion of the oriented boundary of $B_0'$ on $\partial \overline{\Sigma'}$, and $[\partial^\partial B_0']$ denotes the $Spin^c$- component of $gr_A$ determined by the homology class of $\partial^\partial B_0'$. 

Note $\partial_{\alpha_K}(B)$ determines a sequence of type D operations connecting $y_1$ to $y_2$, giving rise to 
$$
gr_D(y_2)=(m(\partial_{\alpha_{K}}B),-[\partial_{\alpha_{K}}B])\cdot gr_D(y_1)
$$
Here $gr_D$ denotes the grading function on $\widehat{CFD}(\alpha_{K})$ and $m(\partial_{\alpha_{K}}B)$ denotes the Maslov component of the grading; we will not need a specific formula for $m(\partial_{\alpha_{K}}B)$.

Note $[\partial^\partial B_0]=[\partial_{\alpha_{K}}B]$ in view of the definition of the collapsing map. Therefore,
$$gr_A(\bm{x}_2)gr_D(y_2)=(-e(B'_0)-n_{\bm{x}_1}(B'_0)-n_{\bm{x}_2}(B'_0)+m(\partial_{\alpha_K}B),0)\cdot gr_A(\bm{x}_1)gr_D(y_1).$$
Hence, 
$$
gr^\boxtimes_z(\bm{x}_2\otimes y_2)-gr^\boxtimes_z(\bm{x}_1\otimes y_1)=-e(B'_0)-n_{\bm{x}_1}(B'_0)-n_{\bm{x}_2}(B'_0)+m(\partial_{\alpha_K}B).
$$
Since $B'=B_0'+k[\Sigma']$, the above equation is equivalent to 
\begin{equation}\label{Equation, grading difference from boxtensor}
gr^\boxtimes_z(\bm{x}_2\otimes y_2)-gr^\boxtimes_z(\bm{x}_1\otimes y_1)=-e(B')-n_{\bm{x}_1}(B')-n_{\bm{x}_2}(B')+m(\partial_{\alpha_K}B)+n_z(B)
\end{equation}
Comparing Equation \ref{Equation, grading difference from boxtensor} and \ref{Equation, grading difference via domains}, identifying both $z$-gradings is equivalent to proving the following equation:
\begin{equation}\label{Equation, grading comparison target equation}
	\begin{aligned}
		-e(B')-n_{\bm{x}_1}(B')-&n_{\bm{x}_2}(B')+m(\partial_{\alpha_K}B)\\&=-e(B)-n_{\bm{x}_1\otimes y_1}(B)-n_{\bm{x}_2\otimes y_2}(B)+n_z(B)+s(\partial_{\alpha_{K}}B).
	\end{aligned}
\end{equation}
Equation \ref{Equation, grading comparison target equation} is true when $\mathcal{H}_{w,z}$ is a genus-one Heegaard diagram by \cite{Chen2019} and Proposition \ref{Proposition, identify two ways to compute the Maslov index for genus one diagram}; see the discussion before the proof. We will reduce the proof of the higher-genus case to the genus-one case. To do so, we will first crop both $B$ and $B'$, leaving only a portion of the domains near $\alpha_K$ and the $\alpha$-arcs respectively; we can reduce the proof of the grading identification to a claim involving some quantities of the cropped domains; later on, we will extend the cropped domains to domains in genus-one Heegaard diagrams, where we can use the grading identification to derive the desired claim on the cropped domains. (See Figure \ref{Figure, CropExtend} for an illustration of the cropping and extending procedures.) 

	\begin{figure}[htb!]
	\centering{
		\includegraphics[scale=0.5]{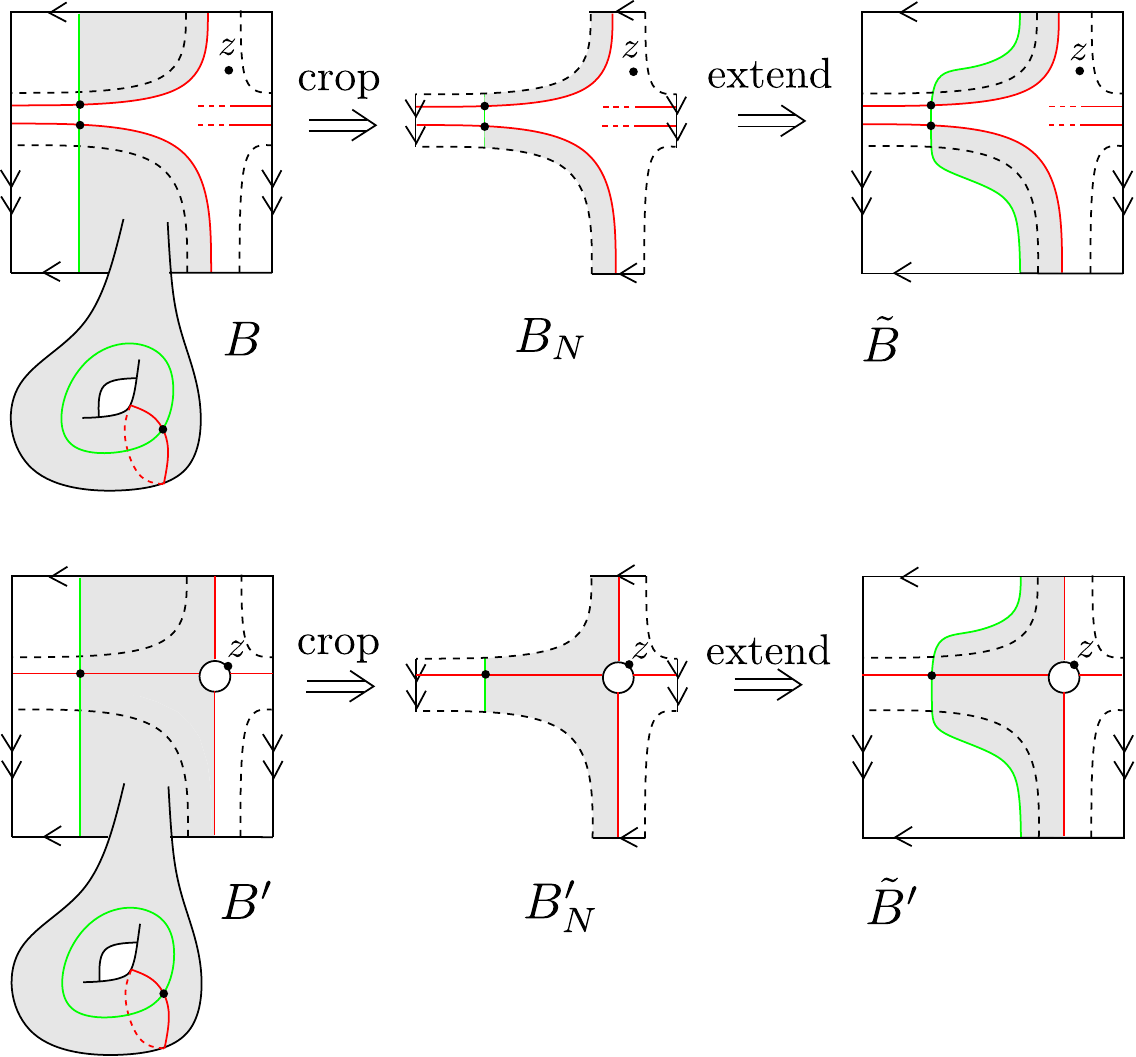}
		\caption{The cropping and extending procedures. The lower row is obtained from the upper row by the collapsing operation.}
		\label{Figure, CropExtend}
	}
\end{figure}

We spell out the cropping procedure. Let $N$ be a closed subset in $\overline{\Sigma'}$ given by the union of three subsets $R\cup S_1\cup S_2$ satisfying the following requirements: 
\begin{itemize}
\item[(1)] $R$ is a collar neighborhood of $\partial\overline{\Sigma'}$, and $S_i$ ($i=1,2$) is a neighborhood of $\alpha^a_i$ homeomorphic to $[0,1]\times[0,1]$ where $\alpha_i^a$ is identified with $[0,1]\times\{0\}$;
\item[(2)] $\beta$-curves do not intersect $R$;
\item[(3)] if a $\beta$-curve intersects some $S_i$ ($i=1,2$), the intersections are arcs of the form $\{p\}\times[0,1]$ for some $p\in[0,1]$. 
\end{itemize} 

One can think about $N$ as the image under the collapsing map of a slightly larger neighborhood than the one specified in Definition \ref{Definition, collapsing operation} (Step 1). Abusing the notation, we will also use $N$ to denote the inverse image of $N$ in $\Sigma$ under the collapsing map.

Let $B_N=B\cap N$ and let $B_{N^c}=\overline{B-B_N}$. Then $B=B_N+B_{N^c}$. Similarly we can define $B'_N$ and $B'_{N^c}$ and have $B'=B'_N+B'_{N^c}$. Let $x_i^a$ $(i=1,2)$ be the component of $\bm{x}_i$ on the $\alpha$-arcs and let $\hat{\bm{x}}_i$ denote the remaining components; we have $\bm{x}_i=\hat{\bm{x}}_i\cup \{x_i^a\}$. Similarly, we have $\bm{x}_i\otimes y_i=\hat{\bm{x}}_i\cup \{x_i^a\otimes y_i\}$. Now we claim that to prove Equation \ref{Equation, grading comparison target equation}, we only need to prove it over $N$, i.e., proving the following equation:
\begin{equation}\label{Equation, grading comparsion equation cropped}
\begin{aligned}
	-e(B_N')-&n_{x_1^a}(B_N')-n_{x_2^a}(B_N')+m(\partial_{\alpha_K}B_N)\\&=-e(B_N)-n_{x_1^a \otimes y_1}(B_N)-n_{x_2^a \otimes y_2}(B_N)+n_z(B_N)+s(\partial_{\alpha_{K}}B_N).
\end{aligned}
\end{equation}
The claim follows from Equation \ref{Equation, helper 1}--\ref{Equation, helper 3} below.
\begin{equation}\label{Equation, helper 1}
	\begin{aligned}
-e(B')-n_{\bm{x}_1}(B')-n_{\bm{x}_2}(B')+&m(\partial_{\alpha_K}B)=-e(B_{N^c}')-n_{\hat{\bm{x}}_1}(B_{N^c}')-n_{\hat{\bm{x}}_2}(B_{N^c}')\\
&-e(B_N')-n_{x_1^a}(B_N')-n_{x_2^a}(B_N')+m(\partial_{\alpha_K}B_N)
	\end{aligned}
\end{equation}
Here we used $\partial_{\alpha_K}B_N= \partial_{\alpha_K}B$ and the additivity of the other terms.

\begin{equation}\label{Equation, helper 2}
\begin{aligned}
	-&e(B)-n_{\bm{x}_1\otimes y_1}(B)-n_{\bm{x}_2\otimes y_2}(B)+n_z(B)+s(\partial_{\alpha_{K}}B)=-e(B_{N^c})-n_{\hat{\bm{x}}_1}(B_{N^c})\\
	&-n_{\hat{\bm{x}}_2}(B_{N^c})-e(B_N)-n_{x_1^a \otimes y_1}(B_N)-n_{x_2^a \otimes y_2}(B_N)+n_z(B_N)+s(\partial_{\alpha_{K}}B_N)
\end{aligned}
\end{equation}
Here we used $n_z(B)=n_z(B_N)$ and  $\partial_{\alpha_K}B_N= \partial_{\alpha_K}B$.

\begin{equation}\label{Equation, helper 3}
-e(B_{N^c}')-n_{\hat{\bm{x}}_1}(B_{N^c}')-n_{\hat{\bm{x}}_2}(B_{N^c}')=-e(B_{N^c})-n_{\hat{\bm{x}}_1}(B_{N^c})-n_{\hat{\bm{x}}_2}(B_{N^c})
\end{equation}
Here we used the identification of $B_{N^c}$ and $B_{N^c}'$ under the collapsing map. 

With Equation \ref{Equation, helper 1}--\ref{Equation, helper 3}, it is clear that Equation \ref{Equation, grading comparison target equation} is equivalent to Equation \ref{Equation, grading comparsion equation cropped}.

We need to further process the domains $B_N$ and $B_N'$ before we can appeal to the genus-one case. 

\begin{lem}\label{Lemma, main sub decomposion of the cropped region}
	$B_N=B_{main}+B_{sub}$ where $B_{main}$ is an immersed disk in $\Sigma$ such that $\partial_{\alpha_{K}}B_{main}=\partial_{\alpha_{K}}B$ and $B_{sub}$ is a domain whose boundaries consist of arcs in $\partial N$ or connected components of $(\beta\text{-curves} \cap N)$. In particular, the above decomposition also induces a decomposition $B'_N=B'_{main}+B'_{sub}$ via applying the collapsing map.
\end{lem} 

\begin{proof}[Proof of Lemma \ref{Lemma, main sub decomposion of the cropped region}]
Recall that $B$ is a positive domain. By \cite[Lemma $4.1'$]{Lipshitz2014a}\footnote{Strictly speaking, we need a version of \cite[Lemma $4.1'$]{Lipshitz2014a} where the $\alpha$ curves are only immersed, but this can be proved the same way as if the $\alpha$ curves are embedded.}, there a smooth map $u:S\rightarrow \Sigma\times[0,1]\times \mathbb{R}$ representing the homology class $B$ such that $u^{-1}(\alpha_K\times\{1\}\times \mathbb{R})$ consists of a single arc in $\partial S$ and the map $\pi_{\Sigma}\circ u$ is a branched covering. Let $D_{main}$ be the connected component of $(\pi_{\Sigma}\circ u)^{-1}(N)$ that contains the boundary arc of $S$ that maps to $\alpha_{K}$. Then up to shrinking $N$, we may assume $D_{main}$ is homeomorphic to a disk. Let the domain $B_{main}$ be the image of $\pi_{\Sigma}\circ u$ restricted to $D_{main}$. Let $B_{sub}=B-B_{main}$. By construction, $\partial_{\alpha_{K}}B_{main}=\partial_{\alpha_{K}}B$. Therefore, the boundaries of the regions in $B_{sub}$ do not involve $\alpha_{K}$ and hence must consist of arcs on $\partial N$ and the $\beta$-curves. 
\end{proof}

\begin{lem}\label{Lemma, grading comparison reduction}
Equation \ref{Equation, grading comparsion equation cropped} is equivalent to the following equation:
\begin{equation}\label{Equation, grading comparision reduction}
	\begin{aligned}
	-&e(B_{main}')-n_{x_1^a}(B_{main}')-n_{x_2^a}(B_{main}')+m(\partial_{\alpha_K}B_{main})\\&=-e(B_{main})-n_{x_1^a \otimes y_1}(B_{main})-n_{x_2^a \otimes y_2}(B_{main})+n_z(B_{main})+ s(\partial_{\alpha_{K}}B_{main}).
	\end{aligned}
\end{equation}
\end{lem}
\begin{proof}[Proof of Lemma \ref{Lemma, grading comparison reduction}]
The lemma will follow from verifying that $B_{sub}'$ and $B_{sub}$ contribute equally to the left- and right-hand side of Equation \ref{Equation, grading comparsion equation cropped} respectively. 

First, $-e(B_{sub}')=-e(B_{sub})+n_z(B_{sub})$ since $B_{sub}'$ is obtained from $B_{sub}$ by removing $n_z(B_{sub})$ many disks near $z$ (and doing collapses that do not affect the Euler measures). Secondly, $n_{x_i^a}(B_{sub}')=n_{x_i^a\otimes y_i}(B_{sub})$ for $i=1,2$. To see this, note that $x_i^a\otimes y_i$ is either in its interior or on some beta arc that appears on $\partial B_{sub}$; the collapsing map sends $x_i^a\otimes y_i$ to $x_i$ which lies in the interior of $B_{sub}'$ or on some beta-arc boundary correspondingly; the local multiplicities of $B_{sub}$ at $x_i^a\otimes y_i$ and local multiplicities of $B'_{sub}$ at $x_i$ are the same. Finally, $m(\partial_{\alpha_K}B_{main})=m(\partial_{\alpha_K}B_{N})$ since $\partial_{\alpha_K}B_{main}=\partial_{\alpha_K}B_{N}$. Lemma \ref{Lemma, grading comparison reduction} follows readily from combining these observations and the following two equations:

\begin{equation}
\begin{aligned}
-e(B_N')-&n_{x_1^a}(B_N')-n_{x_2^a}(B_N')+m(\partial_{\alpha_K}B_N)\\&=-e(B_{main}')-n_{x_1^a}(B_{main}')-n_{x_2^a}(B_{main}')+m(\partial_{\alpha_K}B_{main})\\&-e(B_{sub}')-n_{x_1^a}(B_{sub}')-n_{x_2^a}(B_{sub}')
\end{aligned}
\end{equation}

\begin{equation}
\begin{aligned}
-&e(B_N)-n_{x_1^a \otimes y_1}(B_N)-n_{x_2^a \otimes y_2}(B_N)+n_z(B_N)+s(\partial_{\alpha_{K}}B_N)\\&=-e(B_{main})-n_{x_1^a \otimes y_1}(B_{main})-n_{x_2^a \otimes y_2}(B_{main})+n_z(B_{main})+ s(\partial_{\alpha_{K}}B_{main})\\&-e(B_{sub})-n_{x_1^a \otimes y_1}(B_{sub})-n_{x_2^a \otimes y_2}(B_{sub})+n_z(B_{sub})
\end{aligned}
\end{equation}
\end{proof}

Next, we prove Equation \ref{Equation, grading comparision reduction}. We say $(B_{main},N)$ is \textit{extendable to a genus-one Heegaard diagram} if there exists a genus-one bordered Heegaard $\mathcal{H}_1$ and a domain $\tilde{B}$ in $\mathcal{H}_1(\alpha_{K})$ connecting a pair of intersection points such that the cropped domain $\tilde{B}_N$ can be identified with $B_{main}$ for some suitable chosen region $N$ in $\mathcal{H}_1(\alpha_{K})$. In this case, $B_{main}'$ can be identified with the image $\tilde{B}'_N$ of $\tilde{B}_N$ under the collapsing map. Moreover, the $\beta$ curve in $\mathcal{H}_1$ is allowed to be immersed, in which case we require $\beta$ and $\alpha_K^1$ induce linearly independent homology class so that one can define a $z$-grading on $CFK_\mathcal{R}(\mathcal{H}_1(\alpha_{K}))$ as in Definition \ref{Definition, w- and z- gradings}.

 Equation \ref{Equation, grading comparision reduction} holds as long as $(B_{main}, N)$ is extendable to a genus-one Heegaard diagram: Equation \ref{Equation, grading comparison target equation} holds for $\tilde{B}$ and $\tilde{B}'$ since this is in the genus-one case and hence Equation \ref{Equation, grading comparsion equation cropped} holds for $\tilde{B}_N$ and $\tilde{B}'_N$; as $\tilde{B}_N$ and $\tilde{B}'_N$ are identified with $B_{main}$ and $B'_{main}$ respectively, Equation \ref{Equation, grading comparision reduction} is true\footnote{Strictly speaking, when the $\beta$ curve is immersed, we need a version of Equation \ref{Equation, grading comparison target equation} that takes the $s(\partial_\beta B)$ terms into account, but this is a straightforward modification; the validity of this modified equation follows from Proposition \ref{Proposition, identify two ways to compute the Maslov index for genus one diagram} and an extension of grading identification in \cite{Chen2019} to include the case where the $\beta$-curve is immersed, which uses the same proof when $\beta$ is embedded.}. Therefore, we are left to prove the following lemma to show the $z$-gradings on $gCFK^-(\mathcal{H}_{w,z}(\alpha_{K}))$ and ${CFA}^-(\mathcal{H}_{w,z})\boxtimes \widehat{CFD}(\alpha_K)$ are the same.  

\begin{lem}\label{Lemma, extending B main}
$(B_{main},N)$ is extendable to a genus-one Heegaard diagram.
\end{lem}
\begin{proof}[Proof of Lemma \ref{Lemma, extending B main}]
We can embed $N$ in a genus-one doubly-pointed Riemann surface $\widetilde{\Sigma}$. (The $w$ base point can be placed arbitrarily and will not affect our discussion since we are dealing with the $z$-grading.) Recall $B_{main}$ is an immersed disk $f:D_{main}\rightarrow N$. In particular, $\partial D_{main}$ can be decomposed into the union of two connected sub-arcs $b_1\cup b_2$, where $b_1$ is mapped to $\alpha_K$ and $b_2$ is mapped to $\beta$ arcs and $\partial N$ alternatively. We simply perturb $f$ near the portion of $\partial D_{main}$ that is mapped to $\partial N$ to obtain a new map $\tilde{f}: D_{main}\rightarrow \widetilde{\Sigma}$ so that $\tilde{f}({b_2})$ is an immersed arc and that $\tilde{f}(D_{main})$ contains $f(D_{main})$ as a subdomain. Then, we extend $\tilde{f}({b_2})$ to a closed $\beta$-curve (which is possibly immersed). Now, the doubly-pointed genus-one Riemann surface, the newly constructed $\beta$-curve, and $\alpha_{K}$ constitute a genus-one Heegaard diagram and we can take $\tilde{B}$ to be $\tilde{f}(D_{main})$.    
\end{proof}

The above discussion finishes the identification of the $z$-gradings. Next, we show the $w$-gradings can also be identified. Equivalently, we show the Alexander gradings are identified since $A=\frac{1}{2}(gr_w-gr_z)$ and we already know the $z$-gradings are the same. Again, we will only need to show the relative gradings are the same since we will normalize both gradings via the same rule.

The corresponding proof in \cite{Chen2019} can be adapted to the current setting even though our Heegaard diagram might have a higher genus. More specifically, using the previous notations, let $\bm{x}_1\otimes y_1$ and $\bm{x}_2\otimes y_2$ be two generators of $gCFK^-(\mathcal{H}_{w,z}(\alpha_{K}))$ and let $B\in {\pi}_2(\bm{x}_1\otimes y_1,\bm{x}_2\otimes y_2)$ be a domain; we no longer require $B$ to be positive, but for convenience we assume $n_z(B)=0$. Then the Alexander grading difference of $\bm{x}_1\otimes y_1$ and $\bm{x}_2\otimes y_2$ in $gCFK^-(\mathcal{H}_{w,z}(\alpha_{K}))$ is:

$$A(\bm{x}_2\otimes y_2)-A(\bm{x}_1\otimes y_1)=n_w(B)$$ 
Next we show the corresponding Alexander grading difference in ${CFA}^-(\mathcal{H}_{w,z})\boxtimes \widehat{CFD}(\alpha_K)$ is also equal to $n_w(B)$.
Let $B'\in \pi_2(\bm{x}_1,\bm{x}_2)$ be the domain obtained from $B$ by applying the collapsing map. We will use $\tilde{gr}_A$ denote the grading function on $CFA^-(H_{w,z})$ that consists of the $Spin^c$ component and the Alexander component. (See \cite[Section 11.4]{LOT18} for details on the grading function.) Then 
\begin{equation}\label{Equation, type A Alexander grading difference}
\tilde{gr}_A(\bm{x}_2)=\tilde{gr}_A({\bm{x}}_1)\cdot([\partial^\partial B'], n_w(B'))
\end{equation}
  
Let $\tilde{gr}_D$ denote the grading function on $\widehat{CFD}(\alpha_{K})$ consisting of the $Spin^c$ component and the Alexander component; note that in this case the value the Alexander component of $\tilde{gr}_D$ is always zero. 
The boundary $\partial_{\alpha_K}(B)$ determines a sequence of type D operations connecting $y_1$ to $y_2$, giving rise to 
\begin{equation}\label{Equation, type D Alexander grading difference}
\tilde{gr}_D(y_2)=(-[\partial_{\alpha_{K}}B],0)\cdot gr_D(y_1)
\end{equation}
As before, we have $[\partial_{\alpha_{K}}B]=[\partial^\partial B']$. So, combining Equation \ref{Equation, type A Alexander grading difference} and \ref{Equation, type D Alexander grading difference}, we have 
$$\tilde{gr}_A(\bm{x}_2)\tilde{gr}_D(y_2)=\tilde{gr}_A(\bm{x}_1)\tilde{gr}_D(y_1)\cdot(0,n_w(B'))$$

This implies the Alexander grading difference computed using the box-tensor product is equal to $n_w(B')$, which is equal to $n_w(B)$ since the collapsing map preserves the multiplicity of the domain at $w$.
\end{proof}

In establishing the grading correspondence, the version of the knot chain complex is not important. In particular, we have the following.
\begin{proof}[Proof of the main theorem, with gradings]
	The chain homotopy equivalence in Theorem \ref{Theorem, main theorem} established in the Subsection \ref{Subsection, proof of the main theorem, ungraded version} also preserve the $w$- and $z$-gradings by Theorem \ref{Theorem, graded isomorphism of gCFKminus}. 
\end{proof}

\section{(1,1) Patterns}\label{Section, examples}

\subsection{$(1,1)$ diagrams }\label{Subsection, 1 1 diagrams}
Theorem \ref{Theorem, main theorem} is particularly useful for patterns admitting a genus one Heegaard diagram $\mathcal{H}_{w,z}$ because in this case the paired diagram $\mathcal{H}_{w,z}(\alpha_K)$ associated to a satellite knot is genus one and computing the Floer chain complex from this diagram is combinatorial. We will give examples involving patterns of this form in Section \ref{subsec:one-bridge-braids}, but first we will review notation for diagrams for these patterns and show that one of our hypotheses may be dropped in this setting.

By definition, a $(1,1)$ pattern is a pattern admitting a genus-one doubly-pointed bordered Heegaard diagram. Equivalently, a pattern is $(1,1)$ if it admits a so-called $(1,1)$ diagram defined in \cite{Chen2019}. This is a more flexible object to work with compared to a genus-one bordered Heegaard diagram, and we now recall the definition. A $(1,1)$ diagram is a six-tuple $(T^2,\lambda,\mu,\beta,w,z)$
 consisting of a torus $T^2$, three closed curves $\mu$, $\lambda$, and $\beta$ embedded on $T^2$, and two base points $w,z\in T^2$ such that:
 \begin{itemize}
 	\item[(1)] $\mu$ and $\lambda$ intersect geometrically once;
 	\item[(2)] $\beta$ is isotopic to $\mu$ in $T^2$;
 	\item[(3)] $w$ and $z$ are in the complement of $\mu$, $\lambda$, and $\beta$.
 \end{itemize}
A $(1,1)$ diagram encodes a pattern knot $P$ as follows. Attaching a two-handle to $T^2\times[0,1]$ along $\beta\times\{1\}\subset T^2\times \{1\}$ and filling in the resulting $S^2$ boundary-component by a 3-ball produces a solid torus $S^1\times D^2$. The boundary $\partial (S^1\times D^2)$ is parametrized by $(\mu, \lambda)$. Let $l_\beta$ be an oriented arc on $\partial (S^1\times D^2)$ connecting $z$ to $w$ in the complement of $\beta$, and let $l_\alpha$ be an oriented arc on $\partial (S^1\times D^2)$ connecting $w$ to $z$ in the complement of $\mu$ and $\lambda$. Then $P$ is obtained by pushing $l_\beta$ into the interior of the solid torus.  We remark that our convention in this paper is that (1,1) diagram gives the boundary of the solid torus as viewed from inside the solid torus, so pushing into the solid torus means pushing out of the page. 

Any doubly pointed genus-one bordered diagram determines a $(1,1)$-diagram by filing in the boundary with a disk and extending the $\alpha$ arcs across that disk to form the intersecting closed curves $\mu$ and $\lambda$. Conversely it is shown in \cite[Section 2.4]{Chen2019} that one can construct a genus-one bordered Heegaard diagram from a $(1,1)$ diagram by reversing this process, possibly after isotoping $\beta$. Just as a doubly pointed Heegaard diagram can be paired with an immersed curve, one can pair a $(1,1)$ diagram with an immersed curve by identifying the punctured torus containing the immersed curve with a neighborhood of $\mu \cup \lambda$. For a $(1,1)$ diagram obtained directly from a doubly pointed genus-one bordered diagram it is clear that pairing a given immersed curve with either diagram yields the same result. Moreover, if a $(1,1)$ diagram is isotopic to one coming from a bordered diagram, the diagram obtained by pairing an immersed curve with the $(1,1)$ diagram is isotopic to the diagram obtained by pairing the immersed curve with the bordered diagram (we can perform the same isotopy of $\beta$ in the paired diagram). It follows that we can use pairing diagrams of $(1,1)$ diagrams with immersed multicurves, in place of the pairing diagrams of bordered diagrams with immersed curves, to compute the knot Floer chain complex of satellite knots. The corresponding statement for knot Floer chain complexes over $\mathbb{F}[U]$ is in \cite[Theorem 1.2]{Chen2019}, and it holds for $\mathbb{F}[U,V]/UV$ in view of the result in the present paper.

\subsection{Removing the $z$-passable assumption}\label{Subsection, 1 1 pattern knots}

One additional advantage of (1,1) patterns is that the admissibility assumption on $\alpha_{im}$ in Theorem \ref{Theorem, main theorem} can be relaxed. 

\begin{thm}\label{Theorem, pairing theorem restricted to genus-one patterns}
	Let $P$ be a pattern knot admitting a genus-one doubly pointed bordered Heegaard diagram $\mathcal{H}_{w,z}$, and let $\alpha_K$ denote an immersed multicurve for the knot complement of a knot $K$ in the 3-sphere. Then $CFK_{\mathcal{R}}(\mathcal{H}_{w,z}(\alpha_K))$ is bigraded chain homotopy equivalent to a knot Floer chain complex of $P(K)$.
\end{thm}
\begin{proof}
	If $\alpha_K$ is $z$-passable, this follows from Theorem \ref{Theorem, main theorem}. If $\alpha_K$ is not $z$-passable, there is a $z$-passable multicurve $\alpha'_K$ obtained from $\alpha_K$ via the finger moves described in Proposition \ref{Proposition, arranging a curve to be z-adjacent}. Then by Theorem \ref{Theorem, main theorem}, $CFK_{\mathcal{R}}(\mathcal{H}_{w,z}(\alpha'_K))$ is chain homotopy equivalent to $CFK_{\mathcal{R}}(P(K))$. We claim there is a chain isomorphism between $CFK_{\mathcal{R}}(\mathcal{H}_{w,z}(\alpha'_K))$ and $CFK_{\mathcal{R}}(\mathcal{H}_{w,z}(\alpha_K))$, which proves the proposition. To prove the claim, first note that the generators of both chain complexes are in one-to-one correspondence; this is because the finger moves that transform $\alpha_K$ to $\alpha'_K$ are supported away from the intersection points. For genus-one Heegaard diagrams, the differentials of the Floer chain complexes count immersed bigons. Note the finger moves are actually supported away from the $\beta$-curve and do not create any teardrops with $n_z=0$ or $n_w=0$, and hence they deform immersed bigons on $\mathcal{H}_{w,z}(\alpha_K)$ to those on $\mathcal{H}_{w,z}(\alpha'_K)$, setting up a one-to-one correspondence. The claimed isomorphism now follows from the one-to-one correspondences on the generators and differentials. \end{proof}

\subsection{One-bridge braids}\label{subsec:one-bridge-braids}
 One-bridge braids are $(1,1)$ patterns, and they include cables and Berge-Gabai knots as special cases. They were first studied by Gabai \cite{Gabai1990} and Berge\cite{Berge1991} from the perspective of Dehn surgeries. The knot Floer homology of cable knots were studied extensively in the literature \cite{Hedden2005,MR2511910,VanCott2010,MR3134023,MR3217622,Wu2016,Chen2021,Tange2023,Hanselman2023}. The knot Floer homology of Berge-Gaibai satellite knots were studied by Hom-Lidman-Vafaee in \cite{HLV2014}, where they gave a sufficient and necessary condition for such satellite knots to be L-space knots. In this subsection, we apply the main theorem to study the knot Floer chain complexes of one-bridge-braid satellite knots.
\begin{defn}\label{def:one-bridge-braid}
	A knot $P\subset S^1\times D^2$ is called a one-bridge braid if it is isotopic to a union of two arcs $\delta \cup \gamma$ such that (1) $\delta$ is embedded in $\partial S^1\times D^2$ and is transverse to every meridian $\star\times \partial D^2$, and (2) $\gamma$ is properly embedded in a meridian disk $\star\times D^2$.
\end{defn}
We regard two one-bridge braids equivalent if they are isotopic within the solid torus. Our convention follows that in \cite{HLV2014}.

Every one-bridge braid is isotopic to a braid $B(p,q,b)$ specified as follows. Let $p$, $q$, $b$ be integers such that $0\leq b\leq p-1$. Let $B_p$ be the braid group with $p$ strands and let $\{\sigma_i\}_{i=1}^{p-1}$ denote the generators of $B_p$. Then $B(p,q,b)$ is defined to be the braid closure of $(\sigma_{p-1}\sigma_{p-2}\cdots\sigma_1)^q(\sigma_b\cdots\sigma_2\sigma_1)$. We only consider those $p$, $q$, and $b$ such that $B(p,q,b)$ is a knot. See Figure \ref{Figure, example of 1-bridge braid} (left) for an example. Note that we could restrict the value of $b$ to be strictly less than $p-1$, since $B(p,q,p-1)$ is isotopic to $B(p,q+1,0)$; however, we find it convenient to allow presentations with $b=p-1$ so that it will be easier to introduce a different way of describing one-bridge braids that will be useful for us later.  
	\begin{figure}[htb!]
	\centering{
		\includegraphics[scale=0.5]{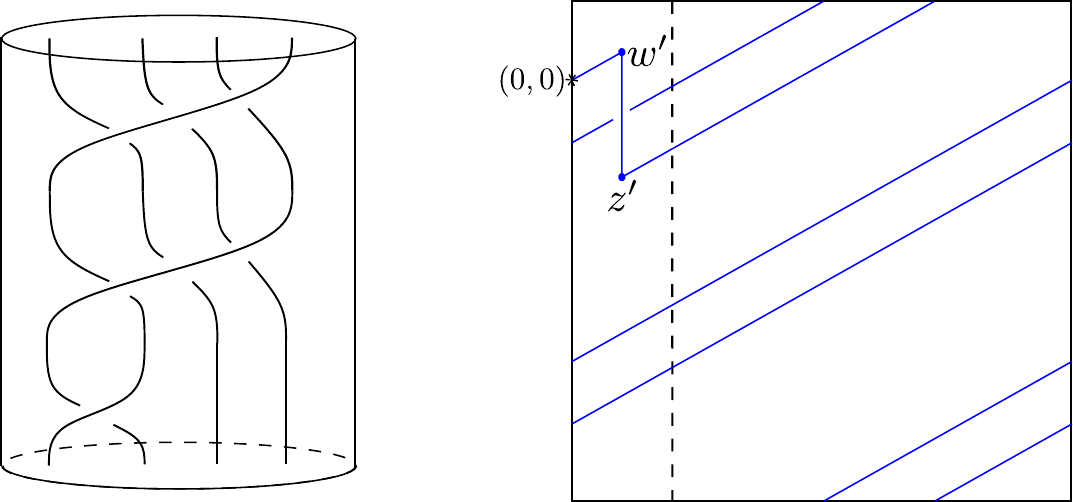}
		\caption{Left: the one-bridge braid $B(4,2,1)$. Right: A projection of $B(4; \frac{9}{16})$ to the boundary torus determined by $p=4$ and the slope $m=\frac{9}{16}$. Note that the torus is drawn as a square such that the vertical direction gives the meridian of the solid torus, and we take the interior of the solid torus to be above the page. From this figure we can observe that $B(4; \frac{9}{16}) = B(4, 2, 1)$.}
		\label{Figure, example of 1-bridge braid}
	}
\end{figure}

Instead of specifying a one-bridge braid by a triple $(p,q,b)$, a one-bridge braid with a single component can also be specified (non-uniquely) using the winding number $p$ and a slope $m$, as we now describe. For each $p$ we exclude slopes in the set $\mathcal{X}_p = \{\frac a b | a,b \in \Z, 1\le b < p\}$; for any other slope $m$ we define a knot $B(p; m)$ in the solid torus by first describing its projection to the boundary torus, which we identify with $(\R/\Z)^2$ such that $\{0\}\times(\R/\Z)$ is a meridian. The projection consists of two embedded segments connecting the points $w' = (\epsilon, m\epsilon)$ and $z' = (\epsilon, m\epsilon - pm)$, a vertical segment $\gamma_{p;m}$ along the curve $\{\epsilon\}\times(\R/\Z)$ that has $w'$ as its top endpoint and $z'$ as its bottom endpoint, and a curve segment $\delta_{p;m}$ of slope $m$ that wraps around the torus in the horizontal direction $p$ times and has $w'$ as its right endpoint and $z'$ as its left endpoint. The knot $B(p; m)$ is obtained by pushing the arc $\gamma_{p;m}$ into the solid torus; it is immediate from Definition \ref{def:one-bridge-braid} that the result is a one-bridge braid. See Figure \ref{Figure, example of 1-bridge braid} (right) for an example. Note that the slopes in $\mathcal{X}_p$ are excluded because they are precisely the slopes for which the curve segment $\delta_{p;m}$ defined above returns to its starting point before wrapping $p$ full times around the torus and is thus not embedded.

To relate the two descriptions of a one-component 1-bridge braid, we can divide the torus with the projection described above into two pieces. The strip $[0,2\epsilon]\times (\R/\Z)$ is called the bridge region and contains a projection of the braid $\sigma_b \cdots \sigma_2 \sigma_1$, and the rest of the torus is called the twist region and contains the projection of $(\sigma_{p-1} \sigma_{p-2} \cdots \sigma_1)^q$. We will define $q(p,m)$ and $b(p,m)$ to be the values of $q$ and $b$ associated with the one-bridge braid $B(p; m)$; that is, they are defined so that $B(p; m) = B(p, q(p,m), b(p,m))$. It is straightforward to check that $q(p,m) = \lfloor pm \rfloor$. We obtain $b(p,m)$ by counting the intersections of $\delta_{p;m}$ with the interior of $\gamma_{p;m}$. 

It is clear that a pair $(p,m)$ for any positive $p$ and any slope $m$ in $\R \setminus \mathcal{X}_p$ gives rise to a 1-bridge braid with one component. It is less obvious that any single-component 1-bridge braid can be represented in this way, but this is indeed true as we now show.

\begin{lem}\label{lem:slope-for-every-b}
For any $p>0, q\in \Z$, and $0\le b\le p-1$ for which $B(p,q,b)$ is a knot, there exists a slope $m$ such that $q(p,m) = q$ and $b(p,m) = b$.
\end{lem}
\begin{proof}
Because $q = q(p,m) = \lfloor pm \rfloor$ we can restrict to slopes $m$ in $[\frac{q}{p}, \frac{q+1}{p})$. We must find a slope in this interval for which $b(p,m) = b$. We will describe how $b(p,m)$ changes as $m$ increases from $q(p,m) = q$, showing that it attains every value $b$ for which $B(p,q,b)$ is a knot.

We first observe that $b(p,m)$ is a locally constant function which jumps at points in $\mathcal{X}_p$. We will label the intersections of $\{\epsilon\}\times \mathbb{R}/\mathbb{Z}$ with $\delta_{p,m}$ by $x_0, x_1, \ldots, x_p$, moving leftward along $\delta_{p,m}$ from  $x_0 = w'$ and $x_p = z'$. The function $b(p,m)$ counts the number of $0 < i < p$ for which $x_i$ lies on the vertical segment between $x_0$ and $x_p$. As $m$ varies continuously the $x_i$ all move continuously, so the number of $x_i$ between $x_0$ and $x_p$ can only change at a slope for which some $x_k$ coincides with either $x_0$ or $x_p$. This in turn happens exactly when either $km$ or $(p-k)m$ is an integer, which means that $m$ is in $\mathcal{X}_p$. We can also see that $b(p,m)$ is non-decreasing on the interval $[\frac{q}{p}, \frac{q+1}{p})$ by observing that when the slope increases by $\Delta m>0$ the point $x_i$ moves downward by $i\Delta m$. Since any given $x_i$ moves downwards it will never leave the interval between $x_0$ and $x_p$ by passing $x_0$, and because $x_i$ moves downward slower than $x_p$ it will never leave the interval by passing $p$. In other words, as $m$ increases from $\frac{q}{p}$ to $\frac{q+1}{p}$ the vertical segment from $x_0$ to $x_p$ grows and swallows more $x_i$, and once in the interval between $x_0$ and $x_p$ no $x_i$ ever escapes this growing interval before $x_p$ returns to $x_0$ at slope $m = \frac{q+1}{p}$. At each slope $n/k$ in $\mathcal{X}_p$ the count $b(p,m)$ must increase since $x_k$ coincides with $x_0$ and thus $x_k$ enters the vertical segment at this slope. In fact, $b(p,m)$ increases by an even number at these points since if $x_k$ coincides with $x_0$ then $x_{p-k}$ coincides with $x_p$. Thus $b(p,m)$ is a non-decreasing locally constant function on $[\frac{q}{p}, \frac{q+1}{p})\setminus \mathcal{X}_p$ that is constant mod 2. The function may jump by more than two; in fact it is straightforward to check that at $m \in \mathcal{X}_p$ it increases by twice the number of pairs $(a,b)$ with $a,b \in \Z$ and $1 \le b < p$ for which $m = \frac{a}{b}$.

To see that $b(p,m)$ realizes every value for which the corresponding one-bridge braid is a knot, we must consider the degenerate projections arising from slopes $m$ in $\mathcal{X}_p$. For each such slope $m$, by taking different perturbations of the the non-embedded arc $\delta_{p,m}$ we can construct a sequence of one-bridge braid projections that realizes all values of $b$ between $b(p, m-\epsilon)$ and $b(p, m+\epsilon)$ for small $\epsilon$. Consider a slope $m = n/k$ in $\mathcal{X}_p$ (with $n$ and $k$ relatively prime), and let $\ell = \lfloor \frac p k \rfloor$. With $x_i$ defined as above we have that $x_i$ coincides with $x_j$ if and only iff $i \equiv j \pmod{k}$. We first assume that $m$ is not a multiple of $\frac 1 p$, so that $x_0$ and $x_p$ do not coincide. On one extreme we will perturb $\delta_{p,m}$ by sliding the $x_i$ points off each other so that within each group of neaby points $x_j$ is above $x_i$ if $j > i$ (see the leftmost projection in Figure \ref{fig:degeneratre-slope-projections} for an example). This is the ordering of the $x_i$ that would arise from a slope slightly smaller than $m$, so the knot arising from this projection is isotopic to $B(p; m-\epsilon)$, in particular it is $B(p,q(p,m),b(p,m-\epsilon))$. For the next projection we swap $x_0$ with $x_k$, extending the vertical arc $\gamma_{p,q}$ upward to reach the new $x_0$ and perturbing the portion of $\delta_{p,m}$ leaving $x_k$ to the right downward with $x_k$ as shown in the second projection in Figure \ref{fig:degeneratre-slope-projections} (the pieces of $\delta_{p,m}$ to the left of $x_0$ and $x_k$ are unaffected). This clearly adds one new crossing with $\gamma_{p,m}$, so the value of $b$ for the corresponding one-bridge braid increases by one. We also observe that this change splits off a closed curve containing a portion of $\delta_{p,m}$. More precisely, if we keep the labels on the remaining $x_i$ the same, the perturbation of $\delta_{p,m}$ now has a closed component that connects $x_1, x_2, \cdots, x_k$ and a component that connects $x_0$, $x_{k+1}$, $x_{k+2}, \ldots, x_p$. It follows that $B(p, q(p,m), b(p,m-\epsilon)+1)$ is a two component link. If $\ell > 1$ we can repeat this procedure, creating a new projection by swapping $x_0$ with the point directly above it (now $x_{2k}$), increasing $b$ by one and splitting off one more closed component that passes through the points $x_{k+1}, x_{k+2}, \ldots, x_{2k}$. We continue in this way, adding one to $b$ and increasing the number of link components by one at each step, until $x_0$ has moved above all other $x_i$ with $i \equiv 0 \pmod k$; at this point we have a projection of $B(p,q(p,m),b(p,m-\epsilon)+\ell)$ which has $\ell$+1 components. We now continue our sequence of projections by sliding $x_p$ downward, interchanging $x_p$ with one other point at a time (first $x_{p-k}$, then $x_{p-2k}$, etc.). Each new projections adds a crossing that increases $b$ by one and decreases the number of link components by one. At the last step $x_p$ is below all other $x_i$ with $i \equiv p \pmod k$ and the link has single component; see the rightmost projection in Figure \ref{fig:degeneratre-slope-projections} for an example. It is clear that this projection is isotopic to the one obtained from a slope slightly larger than $m$, so the last projection in the family arising from $m$ is a projection of $B(p, q(p,m), b(p,m-\epsilon) + 2\ell) = B( p, q(p,m), b(p,m+\epsilon))$. When $m = \frac q p$ is in $\mathcal{X}_p$ we construct a similar family but we start with a projection with the maximum number of components. If $\ell = \gcd(p,q)$ we construct a projection of an $\ell$-component one-bridge braid link by following a line of slope $m$ and closing up the curve and starting a new component each time a curve returns to its starting point. Perturbing the components off each other, it is easy to see that this projection gives the torus link $B(p,q,0) = T(p,q)$. From this projection we can introduce a vertical segment between the first and last point on one of the closed components and expand this, adding crossings with other components. Each time a crossing is added we get a projection of a one-bridge braid where $b$ has increased by one and the number of components has decreased by one, and when we arrive at the single component link $B(p,q,\ell-1)$ it is clear that this the projection is isotopic to that arising from a slope slightly larger than $m = \frac q p$.

	\begin{figure}[htb!]
	\centering{
		\includegraphics[scale=0.4]{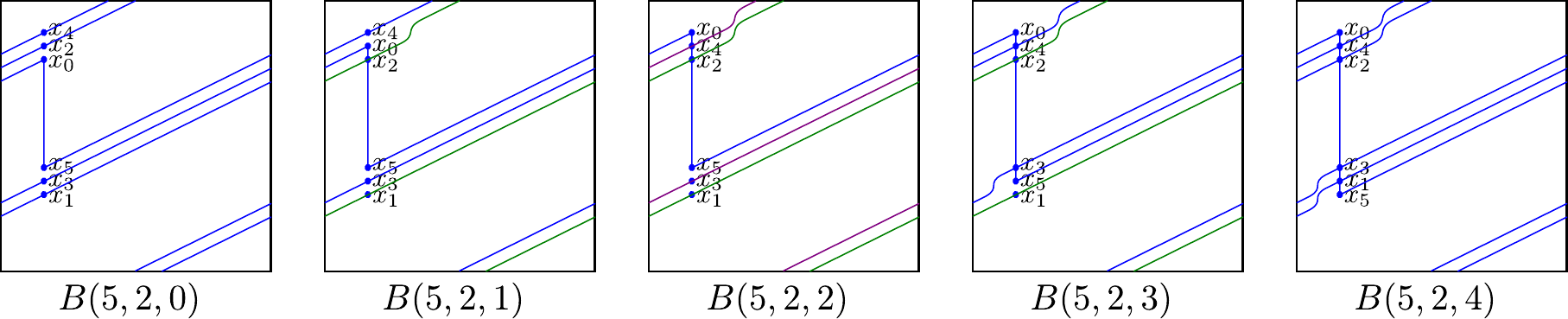}
		\caption{A family of one-bridge braids associated with the degenerate slope $m = \frac 1 2 = \frac 2 4$ and winding number $p = 5$. The leftmost is the projection of a knot that is isotopic to $B(p; m-\epsilon)$ for a sufficiently small $\epsilon > 0$ and the rightmost is the projection of a knot that is isotopic to $B(p;m+\epsilon)$. The value of $b$ associated with the one-bridge braid increases by one at each step, and the number of link components first increases and then decreases so that only the first and last projection give knots.}
		\label{fig:degeneratre-slope-projections}
	}
\end{figure}

To summarize, by varying the slope $m$ from $\frac{q}{p}$ to $\frac{q+1}{p}$ and considering the families of projections described above at the degenerate slopes in $\mathcal{X}_p$ we obtain a projection of $B(p,q,b)$ for every values of $b$ with $0 \le b < p$, and moreover the one-bridge braids $B(p,q,b)$ that only occur at degenerate slopes are precisely those with more than one component. It follows that for any knot $B(p,q,b)$ there is an interval of slopes $m$ for which $b(p,m) = b$.
\end{proof}

\begin{rmk}
While we will not give an closed formula to find an appropriate slope $m$ given any $p$, $q$, and $b$, it is clear from the proof of Lemma \ref{lem:slope-for-every-b} that there is a simple procedure to find the interval of slopes for each value of $b$. We remove any slopes in $\mathcal{X}_p$ from the interval $[\frac{q}{p}, \frac{q+1}{p})$ and then determine $b(p,m)$ on the remaining intervals as follows: the value in the leftmost region is the number of pairs in $\{(a,b) | a,b \in \Z, 0<b<p\}$ for which $\frac a b = \frac q p$ and at each slope in $m \in \mathcal{X}_p$ the value increases by twice the number of pairs in $\{(a,b) | a,b \in \Z, 0<b<p\}$ for which $\frac a b = m$. 
\end{rmk}

 Any single component one-bridge braid in $S^1 \times D^2$ is a (1,1) pattern. To prove Theorem \ref{Theorem, satellite knots with one-bridge braid patterns} in the following section, we will make use of a particular (1,1) diagram for such a one-bridge braid pattern $B(p,q,b)$. The construction of this diagram uses a choice of $m$ for which $B(p; m) = B(p,q,b)$, and we will denote this diagram $\mathcal{H}_{p;m}$. This diagram is isotopic to another diagram $\mathcal{H}'_{p;m}$, which we will describe first.
 
 The diagram $\mathcal{H}'_{p;m} = (T^2, \lambda', \mu', \beta', w', z')$ is defined using the projection for $B(p;m)$ described earlier in this section. We need to choose $\beta'$ and $\mu'$ homologous to meridians of the solid torus while $\lambda'$ will be homologous to a longitude. The basepoints $w'$ and $z'$ are the points defined above. Our goal is to choose $\beta'$ to be disjoint from the arc $\gamma_{p,m}$ and to choose $\lambda'$ and $\mu'$ to be disjoint from the arc $\delta_{p,m}$. We can take $\beta'$ to be the curve $\{0\}\times (\R/\Z)$, for it is in the complement of $\gamma_{p,q}$ and is isotopic to the meridian of the solid torus. We now consider the curves $\mu = \{\frac 1 2\}\times(\R/\Z)$ and $\lambda = (\R/\Z)\times\{\frac 1 2\}$; these are isotopic to the meridian and the longitude of the solid torus, respectively, as desired but they intersect the arc $\delta_{p,m}$. We modify these curves by performing finger moves along $\delta_{p,m}$ in order to eliminate all intersections with $\delta_{p,m}$. More concretely, whenever there is an intersection between $\mu$ or $\lambda$ and $\delta_{p,m}$, we slide that intersection to the left along $\delta_{p,m}$ until it passes over the endpoint $z'$ of $\delta_{p,m}$. We define $\mu'$ and $\lambda'$ to be the curves resulting from these finger moves. It is clear that the diagram $\mathcal{H}'_{p;m}$ encodes $B(p;m)$ since the knot determined by the diagram is the union of $\gamma_{p,m}$ and $\delta_{p,m}$ with $\gamma_{p,m}$ pushed into the interior of the solid torus. An example of a diagram of this form is given on the left of Figure \ref{fig:example-(1,1)-diagrams}.

In $\mathcal{H}'_{p;m}$ the curve $\beta'$ is simple but the curves $\mu'$ and $\lambda'$ are complicated. We can isotope this diagram to produce a second diagram $\mathcal{H}_{p;m} = (T^2, \lambda, \mu, \beta, w, z)$ in which the opposite is true. This is accomplished by sliding $z'$ to the right along the arc $\delta_{p,m}$ until it reaches the point $z$ at = $(-\epsilon, -m\epsilon)$. It is clear from the way $\lambda'$ and $\mu'$ were constructed that this transformation takes these curves back to $\lambda$ and $\mu$. Let $\beta$ denote the image of $\beta'$ under this transformation; that is, $\beta$ is the result of modifying $\beta'$ by finger moves pushing every intersection of $\beta'$ with $\delta_{p,m}$ rightward along $\delta_{p,m}$ until the intersection occurs in a $(2\epsilon)$-neighborhood of the right endpoint of $l'_\alpha$, which is the point $w = w' = (\epsilon, m\epsilon)$. Note that in $\mathcal{H}_{p;m}$ the basepoints $z$ and $w$ are near each other and the midpoint between them is $(0,0)$. An example of a diagram $\mathcal{H}_{p;m}$ is shown on the right of Figure \ref{fig:example-(1,1)-diagrams}. Since $\mathcal{H}_{p;m}$ is isotopic to $\mathcal{H'}_{p;m}$, it is still a (1,1) diagram for $B(p,q,r)$.

\begin{figure}[htb]
\includegraphics[scale=0.6]{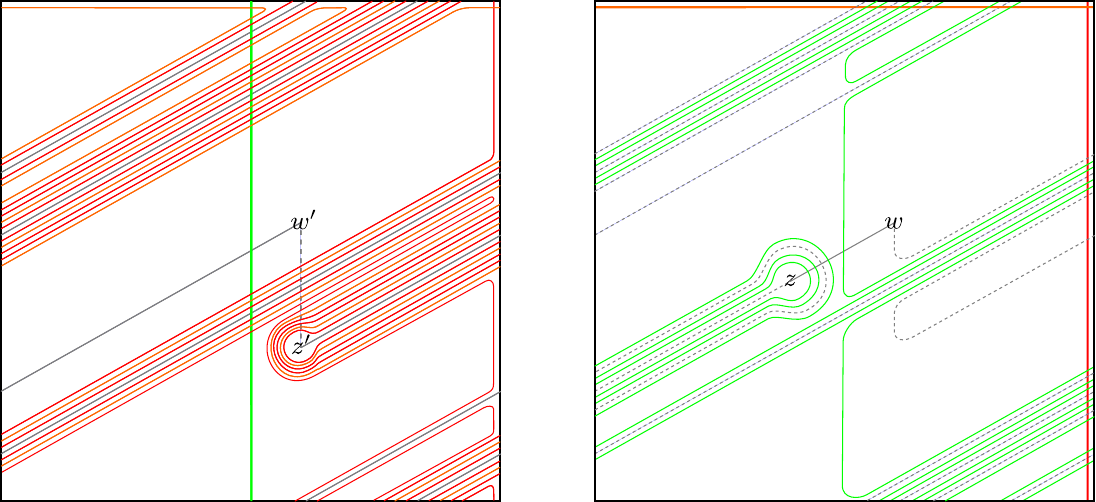}

\caption{Left: The (1,1) diagram $\mathcal{H}'_{4; 9/16}$ associated to the one-bridge braid $B(4;\frac{9}{16}) = B(4,2,1)$. Here the torus is identified with the square $\left[-\frac 1 2, \frac 1 2\right]^2$. $\beta'$ is green, $\mu'$ is red, and $\lambda'$ is orange. The arcs $\delta_{p,m}$ (solid gray) and $\gamma_{p,m}$ (dotted gray) are also shown; these are paths connecting $w'$ and $z'$ in the complement of $\mu'\cup\lambda'$ or of $\beta'$, respectively. Right: The diagram $\mathcal{H}_{4;9/16}$ obtained from $\mathcal{H}'_{4;9/16}$ by an isotopy sliding $z$ along $\delta_{p,m}$. }
			\label{fig:example-(1,1)-diagrams}
\end{figure}

\subsection{Immersed curves for 1-bridge braid satellites}\label{sec:1-bridge-braids-curves}

Given a knot $K$ in the 3-sphere, we use $K_{p,q,b}$ to denote the satellite knot whose pattern knot is $B(p,q,b)$ and whose companion knot is $K$. Let $\gamma_K$ and $\gamma_{K_{p,q,b}}$ be the immersed multicurve associated with $K$ and $K_{p,q,b}$ respectively. Our goal is to describe a way to obtain $\gamma_{K_{p,q,b}}$ from $\gamma_K$. To do so, we will pass to the lifts\footnote{Strictly speaking, the multicurves we consider in $\R^2$ are not lifts of the multicurves in $T^2$ but rather preimages of lifts to the intermediate covering space $(\R/\Z)\times \R$ under the projection map $\R^2 \to (\R/\Z)\times \R$. In particular a nullhomologous curve in $T^2$ lifts to infinitely many curves in $\R^2$. By abuse of terminology we will refer to the resulting multicurves as lifts.} of immersed curves in the universal covering space $\widetilde{T}_{\bullet}$ of the marked torus ${T}_{\bullet}$;  $\widetilde{T}_{\bullet}$ is $\R^2$ with marked points at the integer points $\mathbb{Z}^2$. We require the lifts of the immersed curves to be symmetric in the sense that they are invariant under $180^\circ$ rotation about $(0,\frac{1}{2})$. The lifts of the immersed curves of $K$ and ${K_{p,q,b}}$ are related by a planar transform $f_{p,q,b}:\widetilde{T}_{\bullet} \rightarrow \widetilde{T}_{\bullet}$, defined up to isotopy, which we will now construct. In fact it is most convenient to construct a planar transformation $f_{p; m}: \widetilde{T}_{\bullet} \rightarrow \widetilde{T}$ determined by $p$ and a slope $m \in \R \setminus \mathcal{X}_p$, but we will see that up to isotopy $f_{p; m}$ depends only on $p$, $q(p,m)$ and $b(p,m)$ so we can define $f_{p,q,b}$ to be $f_{p;m}$ for any $m$ such that $q(p,m) = q$ and $b(p,m) = b$.
 
To define $f_{p;m}: \widetilde{T}_{\bullet} \rightarrow \widetilde{T}$, we first define a map $g_{p;m}:\R^2 \to \R^2$ that does not fix the integer lattice. The map $g_{p;m}$ is periodic with horizontal period $p$, by which we mean that if $g_{p;m}(x, y) = (x',y')$ then $g_{p;m}(x+p, y) = (x'+p,y')$, so it suffices to define $g_{p;m}$ on the strip $[-\frac 1 2, p-\frac 1 2] \times \R$. On this strip $g_{p;m}$ is given by an isotopy that fixes the boundary of the strip and slides each integer lattice point $(a,b)$ with $a\neq 0$ to the left along a line of slope $m$ until it hits the vertical line $\{0\}\times \R$, i.e, it stops at $(0,b-am)$. Note that in general the points along the vertical line $\{0\}\times \R$ that are the images of the integer lattice points in $[-\frac 1 2, p-\frac 1 2] \times \R$ are not evenly spaced. We now define $f_{p;m}$ by composing $g_{p;m}$ with three isotopies that together take image of the lattice back to the lattice:
 \begin{enumerate}
	 \item An isotopy that shifts the images of the lattice points under $g_{p;q}$ vertically, preserving the vertical ordering of these points and fixing the points at $(0,n)$ for $n\in \Z$, so that they are evenly spaced---that is, so that they lie in $\{ (0, \frac{n}{p} ) \}_{n\in\Z}$;
	 \item A horizontal compression by a factor of $p$ taking the line the point $(x,y)$ to $(\frac x p, y)$; and
	 \item A vertical stretching by a factor of $p$ taking the point $(x,y)$ to the point $(x, py)$.
 \end{enumerate}
Note that $f_{p;m}$ defines a bijection on the integer lattice even though it takes strips of width $p$ to strips of width $1$, and $f_{p;m}$ is periodic in the sense that if $f_{p;m}(x, y) = (x',y')$ then $f_{p;m}(x+p, y) = (x'+1,y')$. Because of the periodicity of $f_{p;m}$ and of $\gamma_K$ and $\gamma_{K_{p,q,b}}$ it is sufficient to consider the restriction of $f_{p;m}$ to one strip of width $p$, which we may view as a map 
$$f_{p;m}: [-\frac 1 2, p-\frac 1 2] \times \R \to [-\frac{1}{2p}, 1-\frac{1}{2p}] \times \R. $$

Although we chose a slope $m$ to define the map $f_{p;m}$, we observe that different choices of $m$ within the same component of $\R \setminus \mathcal{X}_p$ determine isotopic maps. Given two slopes $m_0$ and $m_1$ in the same component we can vary the slope $m$ from $m_0$ to $m_1$ and note that the vertical ordering on the images of the lattice points under $g_{p;m}$ never changes; if it did there must be some slope for which two lattice points map to the same point, but this only happens for slopes in $\mathcal{X}_p$. Thus we can define $f_{p,q,b}$ to be $f_{p;m}$ for any $m$ for which $B(p;m) = B(p,q,b)$.

We can now restate the method for obtaining $\gamma_{K_{p,q,b}}$ from $\gamma_K$ mentioned in the introduction:

 \PlanarTransformTheorem*

	\begin{figure}[htb!]
		\centering{
			\includegraphics[scale=0.25]{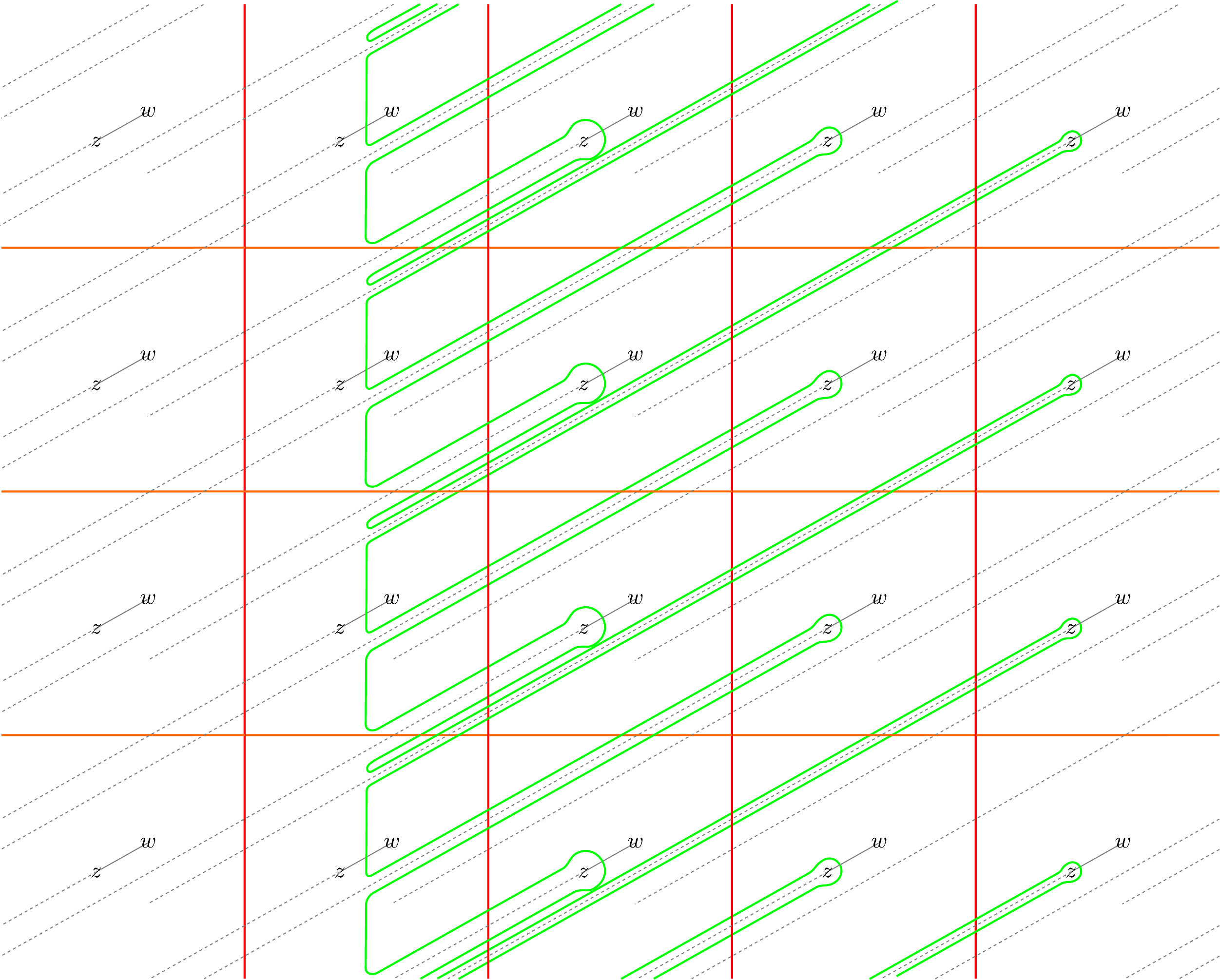}
			\caption{A lift of the $(1,1)$ diagram $\mathcal{H}_{4;9/16}$ for $B(4,2,1)$. The green curve is the lift $\tilde \beta$ of $\beta$. The red vertical lines are lifts of $\mu$ and the orange horizontal lines are lifts of $\lambda$. Lifts of $l'_\alpha$ are shown in gray (the portion over which $z$ slides when constructing $\mathcal{H}_{4;9/16}$ from $\mathcal{H}_{4;9/16}$ is dashed).}
			\label{fig:lift-of-(1,1)-diagram}
		}
	\end{figure}
\begin{proof}
Fixing a slope $m$ for which $B(p;m) = B(p,q,b)$, we consider the (1,1) diagram $\mathcal{H}_{p;m}$ defined in the previous section. Consider the doubly pointed diagram $\mathcal{H}_{p;m}(\gamma_K)=(T^2,\gamma_K,\beta,w,z)$ obtained by pairing $\mathcal{H}_{p;m}$ and $\gamma_K$. By Theorem \ref{Theorem, pairing theorem restricted to genus-one patterns}, the knot Floer chain complex $CFK_\mathcal{R}(\mathcal{H}_{p;m}(\gamma_K))$ is chain homotopy equivalent to $CFK_{\mathcal{R}}(K_{p,q,b})$. 
	
	The complex $CFK_\mathcal{R}(\mathcal{H}_{p;m}(\gamma_K))$ can be computed in the universal cover $\mathbb{R}^2$ of $T^2$ (marked with lifts of $w$ and $z$) by taking the Floer complex of $\tilde{\gamma}_K$ and $\tilde{\beta}$ where $\tilde{\beta}$ and $\tilde\gamma_K$ are lifts of $\beta$ and $\gamma_K$ to $\widetilde T_\bullet$. We pick the lift $\tilde\beta$ of $\beta$ so that it is isotopic to $\{0\}\times \R$ if we ignore the lifts of $z$. An example of the lift of $\mathcal{H}_{p;m}$ to $\mathbb{R}^2$, with the single lift $\tilde{\beta}$ of $\beta$ and all lifts of $\lambda$ and $\mu$, is shown in Figure \ref{fig:lift-of-(1,1)-diagram}.  Lifts of $\delta_{p;m}$ are also shown in the figure. For any $K$, $\tilde{\gamma}_K$ will be a horizontally periodic curve lying in a neighborhood of the lifts of $\lambda$ and $\mu$. Note that from the construction of $\beta$ it is clear that $\tilde\beta$ will never pass the vertical line $\{p-1\}\times \R$, so all intersection points lie in the strip $\left[ -\frac 1 2, p - \frac 1 2\right] \times \R$.
	
	If we were to isotope the lift of $\mathcal{H}_{p;m}$ by sliding each lift of $z$ leftward along the lifts of $\delta_{p;m}$ to the left endpoints of those arcs we would recover a lift of $\mathcal{H'}_{p;m}$. Instead we will consider a different transformation of the plane that slides each pair of nearby lifts of $w$ and $z$ together leftward along the lifts of $\delta_{p;m}$ until the midpoint between them reaches a vertical line $\{ np\} \times \R$ for some integer $n$. This transformation agrees with the map $g_{p;m}$ up to isotopy. By applying appropriate vertical shifts and horizontal and vertical scaling as well, we realize the map $f_{p;m}$. Clearly this transformation of the plane, if applied to both $\tilde \gamma_K$ and $\tilde \beta$, does not affect the complex computed from the diagram. The image of $\tilde \beta$ under this transformation is $\tilde \beta' = \{0\}\times \R$. We thus have that the complex obtained from $f_{p;m}(\tilde \gamma_K)$ by pairing with $\{0\}\times \R$ agrees with the complex $CFK_{\mathcal{R}}(K_{p,q,b})$. Note that $f_{p;m}(\tilde \gamma_K)$ is periodic and, possibly after a homotopy, intersects each line $\{n + \tfrac 1 2\} \times \R$ exactly once. But pairing with $\{0\}\times \R$ gives a bijection between homotopy classes of such curves and homotopy equivalence classes of complexes over $\mathcal{R}$ (this follows from \cite[Theorem 1.2]{Hanselman:CFK}, or from \cite[Theorem 4.11]{hanselman2016bordered} and the well understood relationship between complexes over $\mathcal{R}$ with rank one horizontal and vertical homology and type D structures over the torus algebra). Thus since $f_{p;m}(\tilde \gamma_K)$ and $\tilde{\gamma}_{K_{p,q,b}}$ determine the same complex we must have that $f_{p;m}(\tilde{\gamma}_K)$ is homotopic to $\tilde{\gamma}_{K_{p,q,b}}$.	
\end{proof}

	\begin{figure}[htb!]
	\centering{
		\includegraphics[scale=0.75]{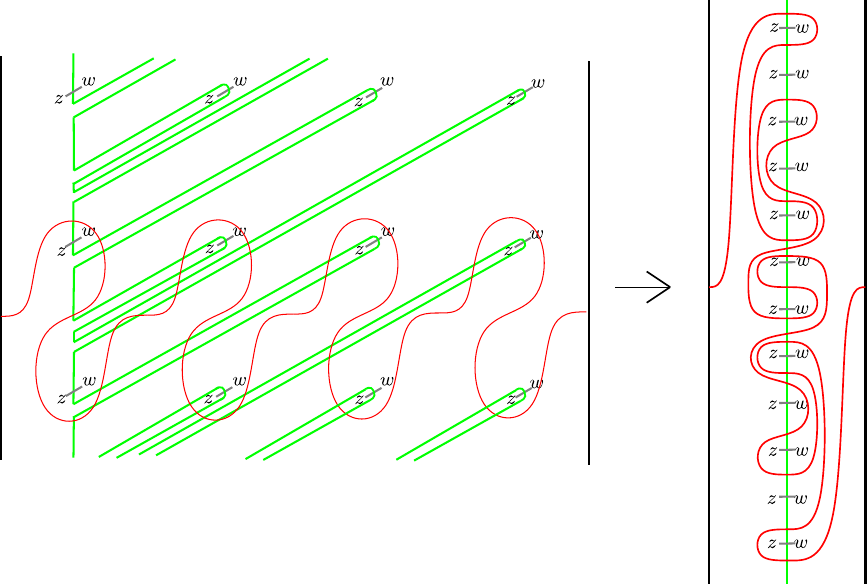}
		\caption{The planar transform $f_{4,2,1}$ acting on the immersed curve of the right-handed trefoil $T_{2,3}$. On the left is a lift of the pairing diagram $\mathcal{H}_{4;9/16}(\gamma_K)$ to $\R^2$, restricted to $\left[-\frac 1 2, 4- \frac 1 2\right]\times \R$. Applying $f_{4,2,1}$ pulls the green curve straight and the image of $\tilde \gamma_K$ is homotopic to the curve on the right.  This curve (repeated horizontally) is $\tilde{\gamma}_{K_{4,2,1}}$.}
		\label{fig:K_4,2,1-example}
	}
\end{figure}

An example illustrating Theorem \ref{Theorem, satellite knots with one-bridge braid patterns} is shown in Figure \ref{fig:K_4,2,1-example}. Let $K$ be the right handed trefoil $T_{2,3}$ and consider the satellite knot $K_{4,2,1}$; we fix the slope $m = \frac{9}{16}$, which is consistent with this value of $q$ and $b$. The figure shows a lift of the pairing diagram $\mathcal{H}_{4;9/16}(\gamma_K)$. The planar transformation $f_{4,2,1} = f_{4;9/16}$ takes $\tilde \beta$ to a vertical line with $z$'s on the left and $w$'s on the right and the image of $\tilde\gamma_K$ under this transformation, one period of which is shown on the right in the figure, is homotopic to $\tilde\gamma_{K_{4,2,1}}$ as a curve in $\R^2 \setminus \Z^2$.

\subsection{L-space slopes, $\tau$, and $\epsilon$ for one-bridge braid satellites}\label{sec:one-bridge-tau}

Theorem \ref{Theorem, satellite knots with one-bridge braid patterns} gives uniform access to several results in the literature. It generalizes the cabling transformation of immersed curves by the second author and Watson \cite{Hanselman2023}, which correspond to the case when $b=0$. This cabling transformation gave simple reproofs of earlier cabling formulas for the $\tau$-invariant and the $\epsilon$-invariant \cite[Theorems 1 and 2]{MR3217622} and of an L-space knot criterion for cables \cite{Hom2011}. We can now extend these results, with essentially the same proofs, to all one-bridge braid satellites; see Theorem \ref{Theorem, L space knot criterion} below for the L-space gluing criterion, Theorem \ref{thm:epsilon-for-one-bridge-braids} for the $\epsilon$ formula, and Theorem \ref{thm:tau-for-one-bridge-braids} for the $\tau$ formula. In addition to cables, other special cases of one-bridge braid patterns have been studied previously. In \cite{HLV2014}, Hom-Lidman-Vafaee proved the aforementioned L-space criterion for Berge-Gaibai knots, which is a proper subset of one-bridge braids. In Example 1.4 of \cite{Hom2016}, Hom gave a sufficient condition for satellite knots with 1-bridge-braid patterns to be L-space knots. Theorem \ref{Theorem, L space knot criterion} can be viewed as a generalization of both these results. Theorem \ref{thm:tau-for-one-bridge-braids} also recovers a recent formula of Bodish for $\tau$ of a family of (1,1) satellites \cite[Theorem 1.1]{Bodish}; note that the pattern denoted $P_{p,1}$ there is the one-bridge braid $B(p,1,2)$.

We first state the L-space knot criterion for one-bridge braid satellites.
 \begin{thm}\label{Theorem, L space knot criterion}
 	$K_{p,q,b}$ is an $L$-space knot if and only if $K$ is an $L$-space knot and $\frac{q}{p}\geq 2g(K)-1$.
 \end{thm}   
 \begin{proof}
 The proof is a straightforward generalization of the proof of Theorem 9 in \cite{Hanselman2023}, using Theorem \ref{Theorem, satellite knots with one-bridge braid patterns} instead of the cabling transformation. The key idea is that an L-space knot is one for which the immersed curve moves monotonically downward from the first to the last time it intersects the vertical line through the punctures, i.e. there is no vertical backtracking. The planar transformation $f_{p;m}$ compresses multiple periods of the periodic curve $\tilde \gamma_K$ by sliding along lines of slope $m$, and there will be no vertical backtracking in the result if and only if $\gamma_K$ has no vertical backtracking and the highest point of $\tilde \gamma_K$ along one vertical line ends up below the lowest point of $\gamma_K$ on the  previous vertical line; this last condition occurs exactly when $m > 2g(K) - 1$. Finally, we observe that $m \in [\frac{q}{p}, \frac{q+1}{p})$ is greater than $2g(K)-1$ if and only if $\frac{q}{p} \geq 2g(K) - 1$. (Note that when $\frac{q}{p} = 2g(K) - 1$, $m$ must be greater than $\frac{q}{p}$ in order for $B(p,q,b)$ to be a knot.)
 \end{proof}
 
 We next derive a formula for $\epsilon$ of one-bridge braid satellites. Recall that $\tilde\gamma_K$ has one non-compact component which is homotopic to $\R\times\{0\}$ if the punctures are ignored and that we view as being oriented left to right. Also recall that $\epsilon$ records whether this component turns downward ($\epsilon = 1$), turns upward ($\epsilon = -1$), or continues straight ($\epsilon = 0$) after the first time (moving left to right) it crosses the vertical line $\{0\}\times \R$.

 \begin{thm}\label{thm:epsilon-for-one-bridge-braids}
 If $\epsilon(K) = \pm 1$ then $\epsilon(K_{p,q,b}) = \epsilon(K)$. If $\epsilon(K) = 0$ then
 $$\epsilon(K_{p,q,b}) = \begin{cases}
 1 &  \textnormal{ if $q > 1$  or if $q = 1$ and $b > 0$}\\
 0 & \textnormal{ if $q \in \{0,-1\}$ or if $(q,b) \in \{(1,0), (-2,p-1) \}$}\\
 -1 & \textnormal{ if $q < -2$ or if $q = -2$  and $b < p-1$}
 \end{cases} $$
 \end{thm}   
\begin{proof}
The proof is essentially the same as the proof of Theorem 3 in \cite{Hanselman2023}. If the non-compact component of $\tilde \gamma_K$ turns either upward or downward, it is clear that this property is preserved by the operation $f_{p;m}$ for any $m$. If $\epsilon(K) = 0$, then the relevant component of $\tilde \gamma_K$ is a horizontal line. In this case, $\epsilon(K_{p,q,b}) = 0$ if and only if all the lattice points initially above $\tilde \gamma_K$ remain above all lattice points initially below $\tilde \gamma_K$ after applying $f_{p;m}$; since the lattice points moving the most are those on the line $\{p-1\}\times\R$, which move vertically by $(p-1)m$, this means that $-\frac{1}{p-1} < m < \frac{1}{p-1}$. For other slopes $\tilde \gamma_{K_{p,q,b}}$ turns downward or upward depending on whether $m$ is positive or negative. Finally, since the only point of $\mathcal{X}_p$ in $(-\frac{1}{p-1}, \frac{1}{p-1})$ is 0, it is simple to check that $\left(q(p,m), b(p,m)\right)$ is $(-2, p-1)$ on  $(-\frac{1}{p-1}, -\frac{1}{p})$, $(-1, 0)$ on $(-\frac{1}{p}, 0)$, $(0, p-1)$ on $(0, \frac 1 p)$, and $(1, 0)$ on $(\frac{1}{p}, \frac{1}{p-1})$.
\end{proof}

Note that $B(p,1,0)$ and $B(p,0,p-1)$ are both isotopic to the torus knot $T(p,1)$ in the boundary of the solid torus and $B(p,-1,0)$ and $B(p,-2,p-1)$ are both isotopic to $T(p,-1)$, so the only satellites with one bridge braid patterns for which $\epsilon = 0$ are $(p, \pm 1)$-cables.

Finally, we compute $\tau$ for one-bridge braid satellites.  Recall that  $\tau(K)$ measures the height of the first intersection of the non-trivial component of $\tilde\gamma_K$ with the vertical line $\{0\}\times\R$; the first intersection occurs between heights $\tau(K)$ and $\tau(K)+1$, while by symmetry the last intersection with $\{0\}\times\R$ occurs between heights $-\tau(K)$ and $-\tau(K) +1$. It follows that $2\tau-1$ is the difference between the height of the lattice point immediately below the first intersection and the height of the lattice point immediately above the last intersection. It is not difficult to identify points on $f_{p;m}(\tilde \gamma_K)$ that give the first and last intersection with $\{0\}\times \R$. The only step that requires some care is computing the height difference between these points; the following lemma will be helpful. We use $y(p)$ to denote the $y$-coordinate of a point $p$ in $\R^2$.

\begin{lem}\label{lem:parallelogram-lattice-points}
For any $p>0$, $q\ge 0$ and $b \in \{0, \ldots, p-1\}$ for which $B(p,q,b)$ is a knot, if $A = (0,0)$ and $B = (p-1,0)$ then
$$y(f_{p,q,b}(A)) - y(f_{p,q,b}(B)) = pq - q + b.$$
\end{lem}
\begin{proof}
Choose a slope $m$ in $\R\setminus \mathcal{X}_p$ such that $f_{p,q,b} = f_{p;m}$; note that $m>0$ since $q \ge 0$. The difference in height between $f_{p;m}(A)$ and $f_{p;m}(B)$ is one less than the number of points of $\{0\}\times \Z$ between $f_{p;m}(A)$ and $f_{p;m}(B)$, inclusive. This latter quantity is the same as the number of integer lattice points contained in the parallelogram (with boundary included) with vertices at $g_{p;m}(A) = A = (0,0)$, $g_{p;m}(B) = (0, -m(p-1))$, $B = (p-1,0)$, and $C = (p-1, m(p-1))$, since the lattice points in this parallelogram are precisely those that map to either $g_{p;m}(A)$, $g_{p;m}(B)$, or the interval between them under $g_{p;m}$.

\begin{figure}
	\includegraphics[scale=0.60]{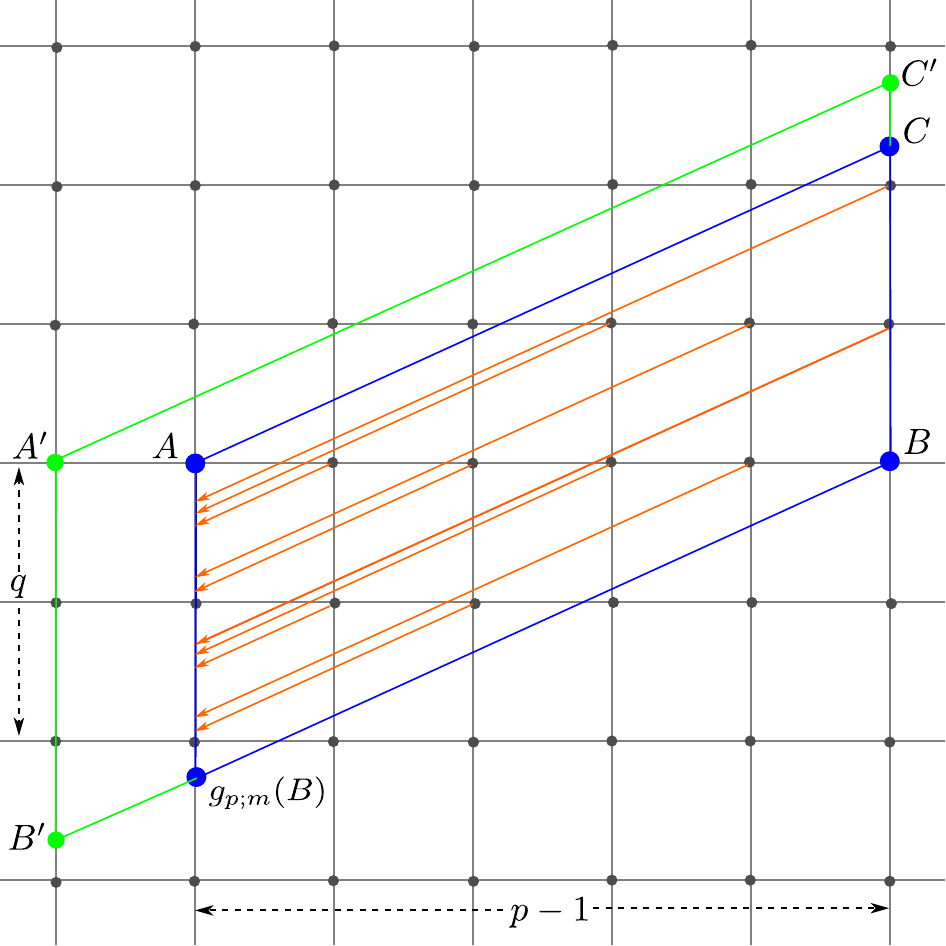}
	\caption{The number of lattice points in the blue parallelogram is the same as the number of lattice points between $f_{p;m}(A)$ and $f_{p;m}(B)$, inclusive. We count these lattice points by finding the number in the larger green parallelogram, which can be more simply stated in terms of $p$, $q$, and $b$, and subtracting the excess. }
	\label{fig:tau-parallelogram}
\end{figure}

To count lattice points in the the parallelogram with corners at $A$, $g_{p;m}(B)$, $B$, and $C$, we first count lattice points in the larger parallelogram with corners at $A' = (-1, 0)$, $B' = (-1, -mp)$, B, and $C' = (p-1, pm)$ (see Figure \ref{fig:tau-parallelogram}). We claim that this closed region contains $q(p+1) + b + 2$ lattice points. To see this, note that if we apply the transformation $g_{p;m}$ translated leftward by one unit (that is, the transfomation $g'_{p;m}$ defined such that $g'_{p;m}(x,y)$ is one unit to the left of $g_{p;m}(x+1,y)$) then no lattice points enter or leave this parallelogram and all lattice points in the parallelogram apart from those on its right edge end up on the left edge of the parallelogram. Since $m \in [\frac q p, \frac{q+1}{p}]$, $B'$ lies between the lattice points at $(-1, -q)$ and $(-1, -q-1)$. From the method of computing $b(p,m)$ described in Section \ref{subsec:one-bridge-braids}, the number of these points ending up between $B'$ and $(-1, -q)$ is $b$. In addition, for each $-q \le i \le -1$ there are $p$ lattice points that end up on the half open segment $\{(-1, i+t)\}_{t\in[0,1)}$. Combining these along with the point $A'$, we see that there are $qp + b + 1$ lattice points taken to the left edge of the parallelogram by $g'_{p;m}$. Adding the $q+1$ lattice points on the right edge gives $pq + q + b +2$. To obtain the number of lattice points in the smaller parallelogram, we remove from this the $q+1$ lattice points points on the the left edge from $A'$ to $B'$. We also subtract the number of lattice points in the trapezoid with corners $A'$, $A$, $C$, and $C'$, not counting $A'$ or any along the edge $AC$; this trapezoid intersects the $q$ horizontal lines $\R\times\{i\}$ for $1\le i \le q$ and intersects each in a half open interval of length one that must contain exactly one lattice point, so the trapezoid contains $q$ lattice points. Thus the smaller parallelogram contains $pq - q + b + 1$ lattice points, and subtracting one gives the height difference.
\end{proof}
Note that the vertical translation the formula in Lemma \ref{lem:parallelogram-lattice-points} also give the height difference between $f_{p,q,b}( (0, n) )$ and $f_{p,q,b}( (p-1, n) )$ for any integer $n$.

 \begin{thm}\label{thm:tau-for-one-bridge-braids}
 If $\epsilon(K) = \pm 1$ then $\tau(K_{p,q,b}) = p\tau(K) + \frac{(p-1)(q \mp 1) + b}{2}$. If $\epsilon(K) = 0$ then
 $$\tau(K_{p,q,b}) = \begin{cases}
 \frac{(p-1)(q-1)+b}{2} &  \textnormal{ if $q > 1$  or if $q = 1$ and $b > 0$}\\
 0 & \textnormal{ if $q \in \{0,-1\}$ or if $(q,b) \in \{(1,0), (-2,p-1) \}$}\\
 \frac{(p-1)(q+1)+b}{2} & \textnormal{ if $q < -2$ or if $q = -2$  and $b < p-1$}
 \end{cases} $$
 \end{thm}   
\begin{proof}
The proof is similar to that of Theorem 4 in \cite{Hanselman2023}. We first assume that $q \ge 0$; the case of $q < 0$ follows from this by taking mirrors. We consider cases based on the value of $\epsilon(K)$.

If $\epsilon(K) > 0$, consider the points $A = (0, \tau(K))$, $B = (p-1, \tau(K))$, and $B' = (p-1, 1-\tau(K))$. When the non-compact component of $\tilde \gamma_K$ is pulled tight the first intersection with $\{0\}\times\R$ occurs just above $A$ and the last intersection with $\{p-1\}\times \R$ occurs just below $B'$; it is clear that after applying $f_{p,q,b}$ the first intersection with $\{0\}\times\R$ occurs just above $f_{p,q,b}(A)$ and the last intersection with $\{0\}\times\R$ occurs just below $f_{p,q,b}(B')$; see Figure \ref{fig:tau-cases}. We now have that
\begin{align*}
2\tau(K_{p,q,b}) - 1 &= y( f_{p,q,b}(A) ) - y( f_{p,q,b}(B') )  \\
&= \left[ y( f_{p,q,b}(A) ) - y( f_{p,q,b}(B) ) \right] + \left[ y( f_{p,q,b}(B) ) - y( f_{p,q,b}(B') )  \right].
\end{align*}
The first term in this sum is $pq - q + b$ by Lemma \ref{lem:parallelogram-lattice-points} and the second term is $p(2\tau(K)-1)$ since $f_{p,q,b}$ scales the distance between lattice points in the same column by a factor of $p$. Solving for $\tau(K_{p,q,b})$ gives
$$\tau(K_{p,q,b}) = 2 p \tau(K) + \frac{(p-1)(q-1)+b}{2}.$$

If $\epsilon(K) < 0$, we instead consider the points $A = (0, \tau(K)+1)$, $B = (p-1, \tau(K)+1)$, and $B' = (p-1, -\tau(K))$. When the non-compact component of $\tilde \gamma_K$ is pulled tight the first intersection with $\{0\}\times\R$ occurs just \emph{below} $A$ and the last intersection with $\{p-1\}\times \R$ occurs just \emph{above} $B'$; it is clear that after applying $f_{p,q,b}$ the first intersection with $\{0\}\times\R$ occurs just below $f_{p,q,b}(A)$ and the last intersection with $\{0\}\times\R$ occurs just above $f_{p,q,b}(B')$; see Figure \ref{fig:tau-cases}. It follows that
\begin{align*}
2\tau(K_{p,q,b}) + 1 &= y( f_{p,q,b}(A) ) - y( f_{p,q,b}(B') )  \\
&= \left[ y( f_{p,q,b}(A) ) - y( f_{p,q,b}(B) ) \right] + \left[ y( f_{p,q,b}(B) ) - y( f_{p,q,b}(B') )  \right].
\end{align*}
Once again the first term in this sum is $pq - q + b$ by Lemma \ref{lem:parallelogram-lattice-points} and the second term is now $p(2\tau(K)+1)$. Solving for $\tau(K_{p,q,b})$ gives
$$\tau(K_{p,q,b}) = 2 p \tau(K) + \frac{(p-1)(q+1)+b}{2}.$$

If $\epsilon(K) = 0$, consider the points $A = (0, 0)$, $B = (p-1, 0)$, and $B' = (p-1,1)$. The non-compact component of $\tilde \gamma_K$ is homotopic to the horizontal line that passes just above $A$ and $B$. If $q = 0$ or $(q,b) = (1,0)$ then $f_{p,q,b} = f_{p;m}$ for some slope $m$ with $0 < m < \frac{1}{p-1}$; in this case it is clear that no lattice points cross the horizontal line just above $A$ when $g_{p;m}$ is applied, so the image under $f_{p;m}$ of this line is still homotopic to a horizontal line and $\tau(K_{p,q,b}) = 0$. If $q > 1$ or if $q =1$ and $b>0$ then $f_{p,q,b} = f_{p;m}$ for some $m > \frac{1}{p-1}$. We can homotope the non-compact component of $\tilde \gamma_K$ so that it passes just above $A$ and just below $B'$, as in Figure \ref{fig:tau-cases}, and it is clear that the first intersection with $\{0\}\times \R$ of the image under $f_{p;m}$ occurs just above $f_{p;q}(A)$ and the last occurs just below $f_{p;q}(B')$, so 
\begin{align*}
2\tau(K_{p,q,b}) - 1 &= y( f_{p,q,b}(A) ) - y( f_{p,q,b}(B') )  \\
&= \left[ y( f_{p,q,b}(A) ) - y( f_{p,q,b}(B) ) \right] + \left[ y( f_{p,q,b}(B) ) - y( f_{p,q,b}(B') )  \right] \\
&= [pq - q + b] + [-p].
\end{align*}
From this we find that $\tau(K_{p,q,b}) = \frac{(p-1)(q-1) + b}{2}$.

Finally, we check the case of $q\le 0$ using the fact that the mirror of $B(p,q,b)$ is $B(p, -q-1, p-b-1)$ and the fact that mirroring flips the sign of $\tau$ and $\epsilon$. If $\epsilon(K) = \pm 1$ then $\epsilon(\overline{K}) = \mp 1$ and

\begin{figure}[h!]
	\includegraphics[scale=0.50]{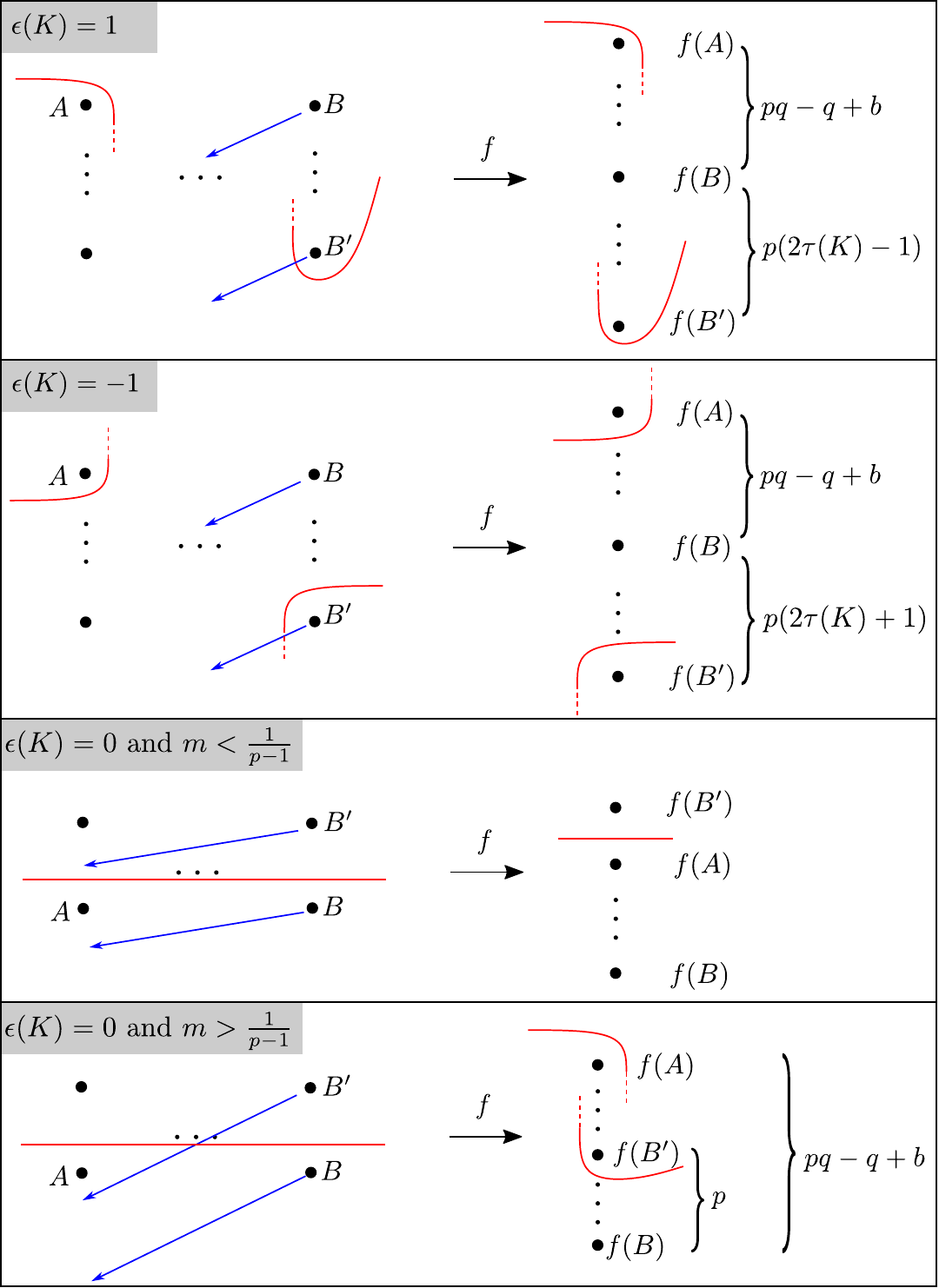}
	\caption{The first intersection of the non-compact component of $\tilde\gamma_K$ with $\{0\}\times\R$ and the last intersection with $\{p-1\}\times\R$ give rise to the first and last intersection with $\{0\}\times\R$ after applying $f = f_{p,q,b}$}
	\label{fig:tau-cases}
\end{figure}

\begin{align*}
\tau(K_{p,q,b}) &= -\tau( \overline{K}_{p, -q-1, p-b-1} ) = - p\tau(\overline{K}) - \frac{ (p-1)( (-q-1) \pm 1 ) + (p-b-1)}{2} \\
&= p\tau(K) - \frac{ (p-1)( -q \pm 1 ) -b}{2} =  p\tau(K) + \frac{ (p-1)( q \mp 1 ) + b}{2} \\
\end{align*}
as desired. A similar computation, which we omit, checks the case that $\epsilon(K) = 0$ when $q < 0$.
\end{proof}

\subsection{Mazur satellites}\label{sec:mazur}
We have seen that for one-bridge braid patterns, the immersed curve invariant for a satellite knot can be obtained from that of the companion knot by performing a diffeomorphism in a covering space. Unfortunately this is not always the case, even for $(1,1)$-satellites. Consider the Mazur pattern $M$, pictured along with its genus one doubly pointed Heegaard diagram in Figure \ref{Figure:MazurPattern}. We can use Theorem \ref{Theorem, main theorem} to compute the complex $CFK_{\mathcal{R}}$ associated with $M(T_{2,3}$, the Mazur satellite of the right handed trefoil; the resulting complex is shown in Figure \ref{fig:mazur-of-trefoil}. The immersed multicurves representing this complex, also shown in the figure, can be obtained from the complex following the algorithm in \cite{Hanselman:CFK}. Note that the resulting curve has more then one component, even though the curve for the trefoil is connected, indicating that there is no hope of recovering this curve by a plane transformation. It is an interesting question whether there is some more complicated geometric operation to recover the immersed multicurve for a Mazur satellite directly from the immersed multicurve for the companion, although we do not pursue this in the present paper. 

\begin{figure}[htb!]
	\begin{center}
		\includegraphics[scale=0.40]{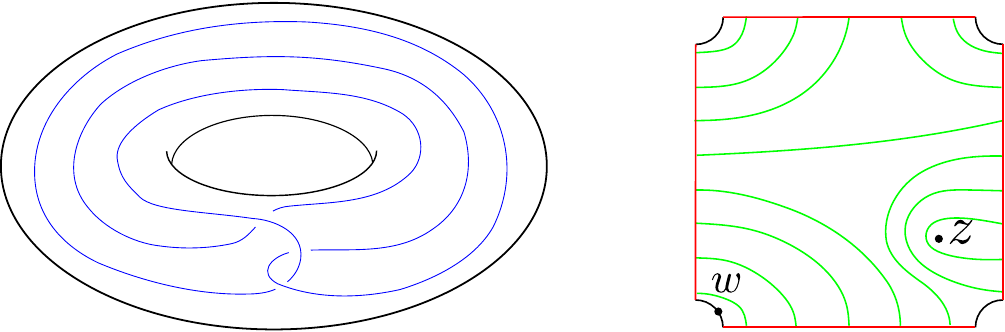}
		\caption{A doubly pointed bordered Heegaard diagram (right) for the Mazur pattern (left).}
		\label{Figure:MazurPattern}
	\end{center}
\end{figure}

\begin{figure}[htb!]
	\begin{center}
		\includegraphics[scale=0.65]{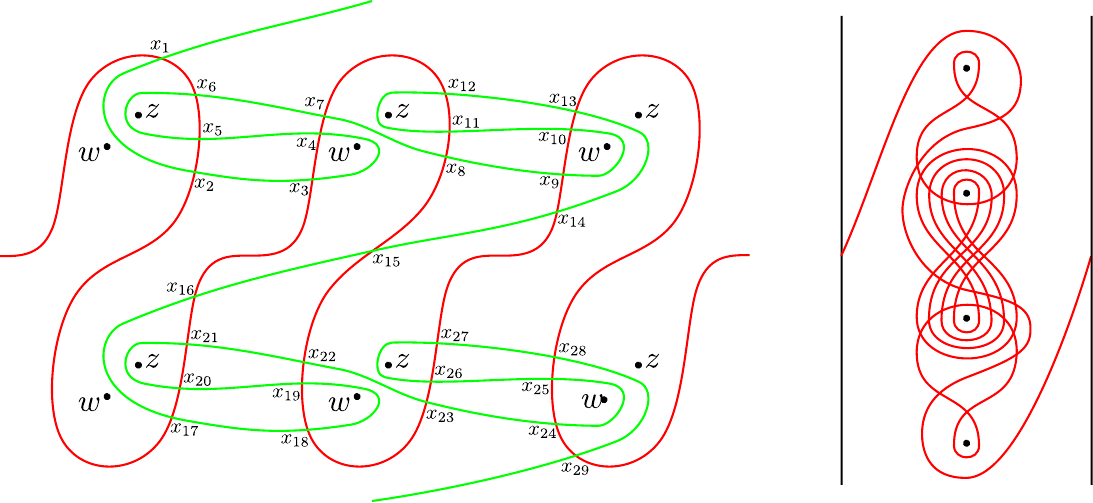}
		
\phantom{line}

\begin{tabular}{lll}
	\hline
	$\partial(x_1) = 0$ 	\hspace{25mm}	& $\partial(x_{11}) = Ux_8 +  Vx_{12}$ 	& $\partial(x_{21}) = Ux_{16}$			\\
	$\partial(x_2) = Vx_1$ 			& $\partial(x_{12}) = Ux_7 + Ux_{15}$	& $\partial(x_{22}) = Ux_{15} + Ux_{23}$	\\
	$\partial(x_3) = V^2 x_{16}$		& $\partial(x_{13}) = Ux_{14}$			& $\partial(x_{23}) = 0$				\\
	$\partial(x_4) = Ux_3 + Vx_7$ 		& $\partial(x_{14}) = 0$	\hspace{25mm}	& $\partial(x_{24}) = Vx_{29}$			\\
	$\partial(x_5) = Ux_2 + Vx_6$ 		& $\partial(x_{15}) = 0$				& $\partial(x_{25}) = Ux_{24}+Vx_{28}$	\\
	$\partial(x_6) = Ux_1$			& $\partial(x_{16}) = 0$				& $\partial(x_{26}) = Ux_{23}+Vx_{27}$	\\
	$\partial(x_7) = 0$				& $\partial(x_{17}) = Vx_{16}$			& $\partial(x_{27}) = U^2 x_{14}$		\\
	$\partial(x_8) = Vx_7 + V x_{15}$ 	& $\partial(x_{18}) = Vx_{15} + Vx_{23}$	& $\partial(x_{28}) = Ux_{29}$			\\
	$\partial(x_9) = Vx_{14}$			& $\partial(x_{19}) = Ux_{18}+Vx_{22}$	& $\partial(x_{29}) = 0$				\\
	$\partial(x_{10}) = Ux_{9} + Vx_{13}$ & $\partial(x_{20}) = Ux_{17}+Vx_{21}$	& 								\\
	\hline
\end{tabular}
	
		\caption{(Top left): A lift of the immersed genus one Heegaard diagram obtained by pairing a (1,1)-diagram for the Mazur pattern with the immersed curve for the right handed trefoil. (Bottom): The chain complex $CFK_{\mathcal{R}}$ computed from this diagram. (Top right): The immersed multicurve representing this complex.}
		\label{fig:mazur-of-trefoil}
	\end{center}
\end{figure}

Despite this difficulty, Theorem \ref{Theorem, main theorem} can be useful in peforming computations with the Mazur pattern and other (1,1) patterns. We will demonstrate this by reproving formulas for the $\epsilon$-invariant and a $\tau$-invariant formula of Mazur satellite knots derived by Levine \cite{MR3589337}. Levine derived these formulas by first computing the bimodule $\widehat{CFDD}(X_M)$ of the exterior $X_M$ of the Mazur pattern using the arc-slides algorithm that is developed in \cite{Lipshitz2014b} and is implemented in a Python script by \cite{Zhan}. In theory, it suffices to analyze $\widehat{CFDD}(X_M)\boxtimes \widehat{CFA}(S^3\backslash \nu(K))$ to derive both formulas. However, this approach is hindered by its computational complexity since $\widehat{CFDD}(X_M)$ is large. Instead, by taking box-tensor products of the bimodule with appropriate bordered invariants, Levine partially computed the hat-version knot Floer chain complexes of Mazur satellite knots $M(K)$ to obtain the $\tau$-invariant formula, and then further deduced the $\epsilon$-invariant of $M(K)$ by computing and examining the hat-version knot Floer chain complexes of $(2,\pm 1)$-cables of the Mazur satellite knots. In this subsection, we present an alternate proof for both formulas using the immersed-curve technique. While our approach is ultimately built on computing $\widehat{CFDD}(X_M)\boxtimes \widehat{CFA}(S^3\backslash \nu(K))$ too, we circumvent the complexity in obtaining the bimodule and in performing and simplifying the box-tensor product. Instead we use \ref{Theorem, main theorem} to analyzes the $\mathbb{F}[UV]/UV$-version knot Floer chain complex of Mazur satellite knots using immersed curves diagrams.

\begin{thm}[Theorem 1.4 of \cite{MR3589337}]\label{Theorem: epsilon_of Mazur_satellites}
	Let $M$ be the Mazur pattern. Then $$
	\epsilon(M(K))=
	\begin{cases}
		0 \ \ \ \epsilon(K)=0\\
		1 \ \ \ \text{otherwise}, 
	\end{cases}
	$$
	and 
	$$
	\tau(M(K))=
	\begin{cases}
		\tau(K)+1 \ \ \ \tau(K)>0\ \text{or}\ \epsilon(K)=-1\\
		\tau(K) \ \ \ \ \ \ \ \ \text{otherwise}. 
	\end{cases}
	$$
\end{thm}

In the proof of Theorem \ref{Theorem: epsilon_of Mazur_satellites}, we will examine the knot Floer chain complex defined from the pairing diagram of the bordered Heegaard diagram $\mathcal{H}_M$ for $M$ shown in Figure \ref{Figure:MazurPattern} and the immersed curve $\alpha_K$ of the companion knot $K$. For convenience, we will be working with the lifts of the curves in the universal cover of the doubly-marked torus; the curves will be denoted by $\tilde{\beta}_M$ and $\tilde{\alpha}_K$ throughout the proof. 

Moreover, we only need the portion of $\alpha_K$ that indicates the value of $\tau(K)$ and $\epsilon(K)$. Specifically, in the universal cover $\mathbb{R}^2$, the lines $\mathbb{Z}\times \mathbb{R}$ divides $\tilde{\alpha}_K$ into segments, and there is only one segment connecting $\{i\}\times\mathbb{R}$ and $\{i+1\}\times \mathbb{R}$ for each $i$; we call it \textit{the unstable segment} as it corresponds to the unstable chain (defined in \cite[Theorem 11.26]{LOT18}). The sign of the slope of the unstable segment is equal to the sign of $\tau(K)$, and the minimum number of intersection points between the unstable segment and the horizontal grid lines $\mathbb{R}\times \mathbb{Z}$ is equal to $2|\tau(K)|$. The invariant $\epsilon(K)$ can be read off from the subsequent segment of the unstable segment as we traverse the unstable segment from left to right: If the next segment is a cap that turns right, $\epsilon(K)=1$; if the next segment is a cap that turns left, then $\epsilon(K)=-1$; otherwise, $\epsilon(K)=0$. Note that by the symmetry of knot Floer chain complexes, the segment preceding the unstable segment is determined by the segment after the unstable segment. Apart from the unstable segment and the segments preceding and trailing it, the rest of ${\alpha}_K$ will not affect the proof of Theorem \ref{Theorem: epsilon_of Mazur_satellites}.   

\begin{proof}[Proof of Theorem \ref{Theorem: epsilon_of Mazur_satellites}]
	When $\epsilon(K)=0$, $K$ is $\epsilon$-equivalent to the unknot $U$, and hence the $\tau$- and $\epsilon$-invariant of $M(K)$ coincide with those of $M(U)$ \cite[Definition 1 and Proposition 4]{Hom2014}. Since $M(U)$ is isotopic to the unknot,  $\epsilon(M(K))=0$ and $\tau(M(K))=0$.
	
	\begin{figure}[!]
		\begin{center}
			\includegraphics[scale=0.55]{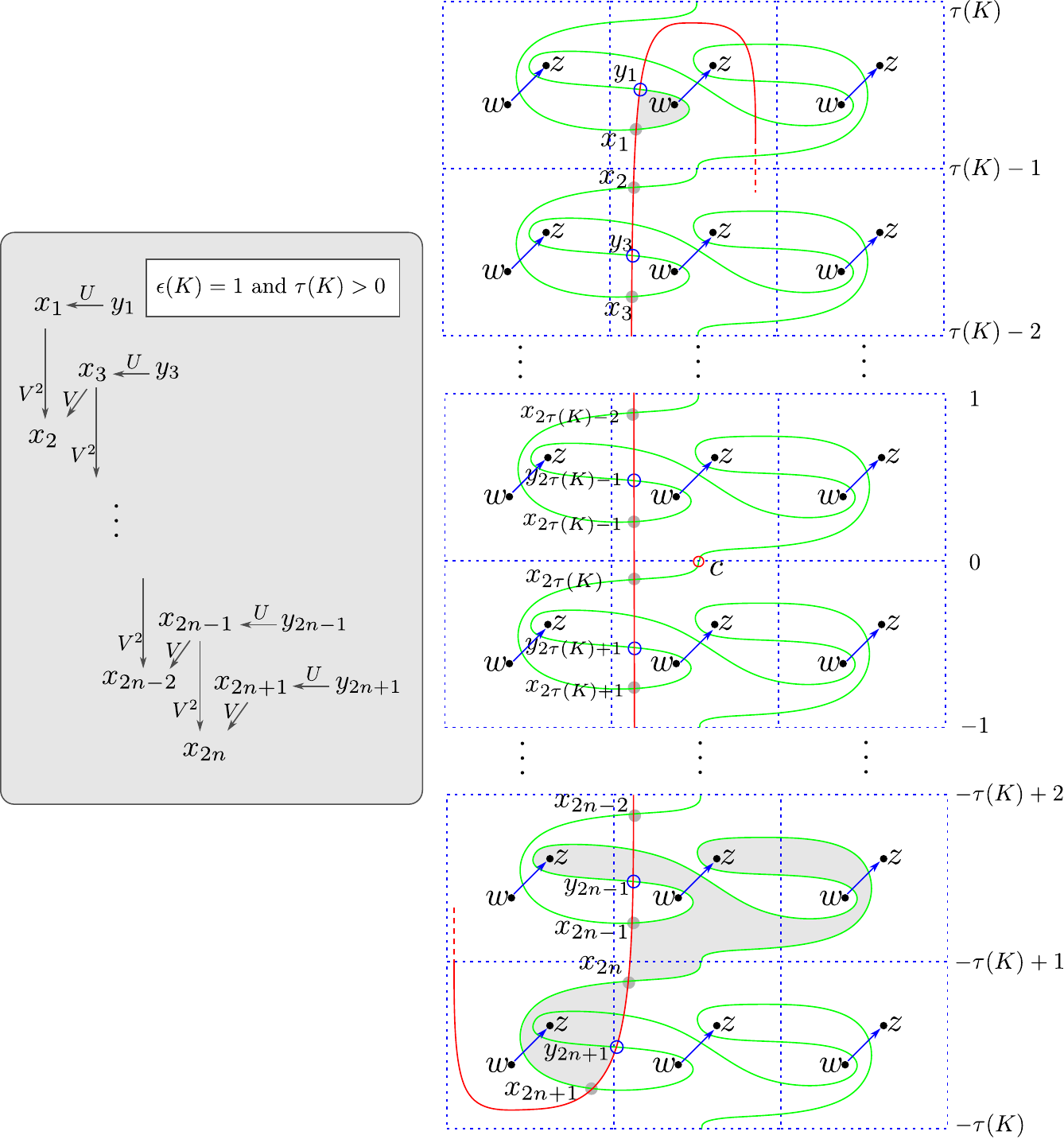}
			\caption{The pairing diagram and relevant differentials for $M(K)$ when $\epsilon(K)=1$ and $\tau(K)>0$. As an illustration for finding the differentials, some bigons contributing to the listed differentials are shaded. The y-coordinates of the horizontal grid lines are labeled.}
			\label{Figure:MazurSatellite0}
		\end{center}
	\end{figure} 
	When $\epsilon(K)\neq 0$, we will discuss the case $\epsilon(K)=1$ and the case $\epsilon(K)=-1$ separately. Within each case, we will further separate the discussion into two sub-cases depending on the value of $\tau(K)$. We shall inspect the chain complex defined from the pairing diagram, and at the cost of an isotopy of the curves in the pairing diagram, we may assume every chain complex is reduced, i.e., each bigon in the pairing diagram contributing to a differential will cross some base point. The differentials obtained from bigons crossing $z$ are called the \textit{vertical differentials}, and the arrows are labeled by a power of $V$ specifying the multiplicity of $z$. Likewise, \textit{horizontal differentials} refer to those obtained from bigons crossing $w$ and the arrows are labeled by an appropriate power of $U$. 
	
	\begin{figure}[!]
		\begin{center}
			\includegraphics[scale=0.52]{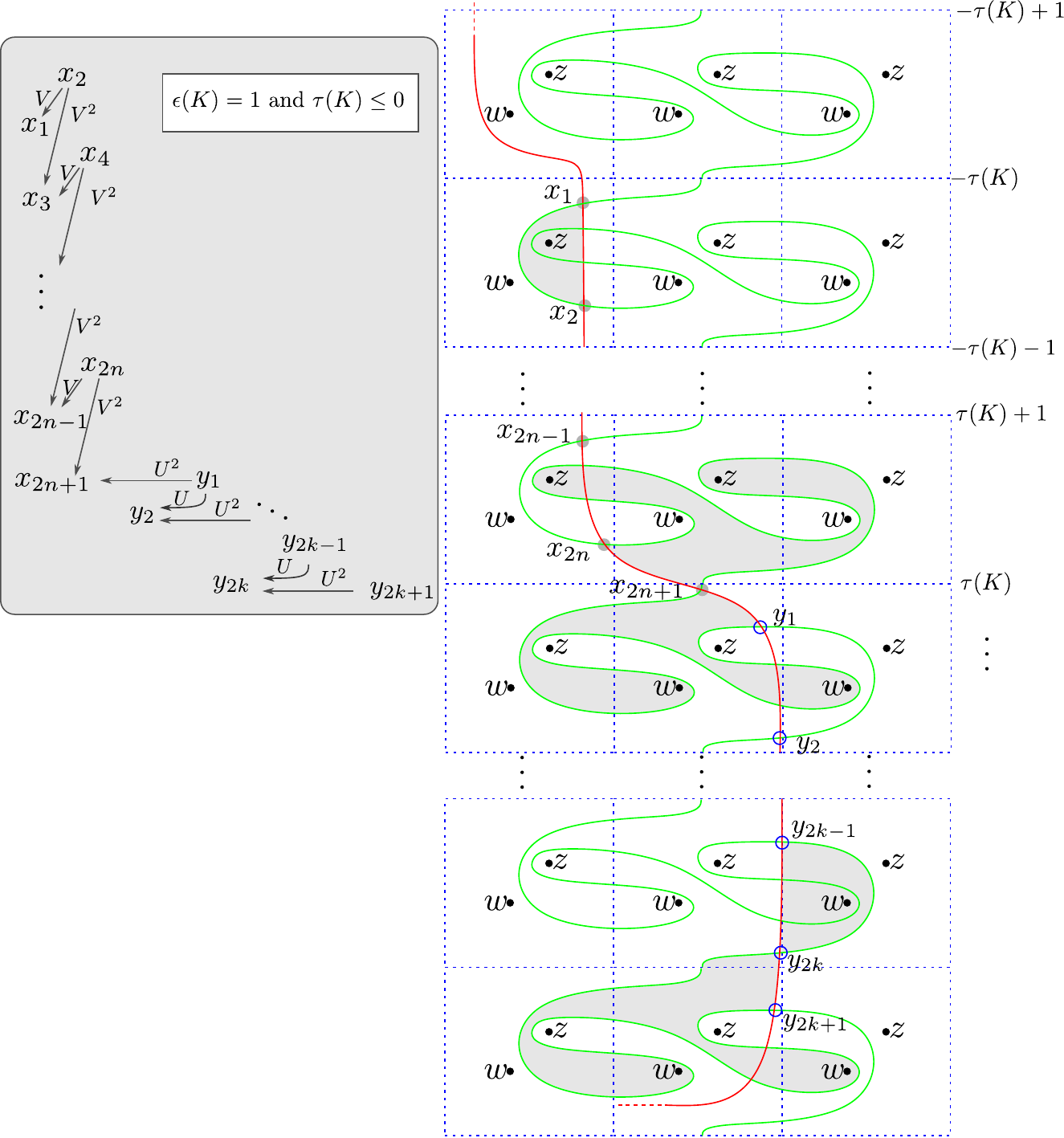}
			\caption{The pairing diagram and relevant differentials for $M(K)$ when $\epsilon(K)=1$ and $\tau(K)\leq 0$.}
			\label{Figure:MazurSatellite1}
		\end{center}
	\end{figure} 

	\begin{figure}[h]
	\begin{center}
		\includegraphics[scale=0.53]{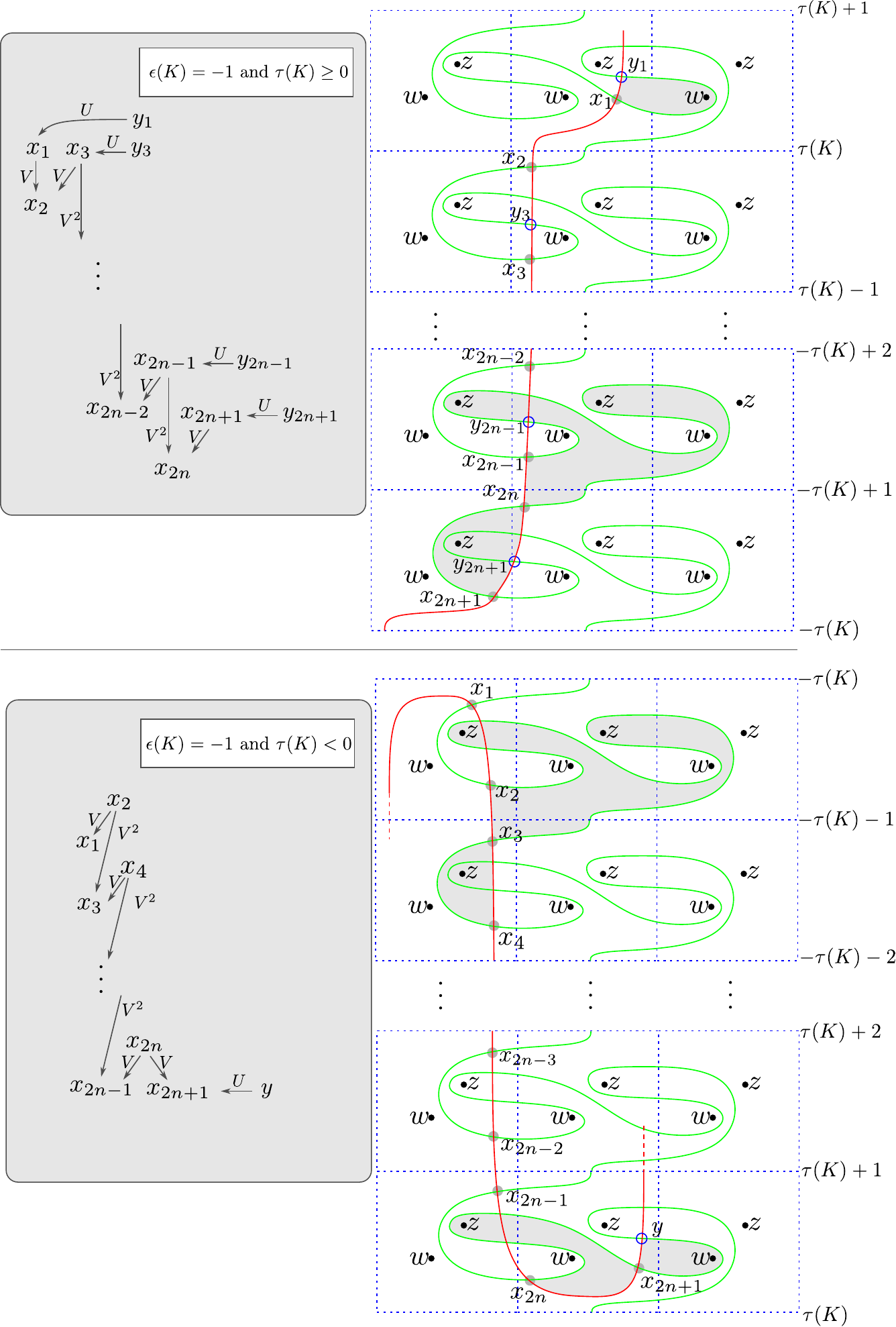}
		\caption{The pairing diagrams and relevant differentials for $M(K)$ when $\epsilon(K)=-1$, separated into two cases by $\tau(K)$.}
		\label{Figure:MazurSatellite2}
	\end{center}
	\end{figure}
	We begin with the case $\epsilon(K)=1$. When $\tau(K)>0$, the pairing diagram is shown in Figure \ref{Figure:MazurSatellite0}. By direct inspection, the intersection points $\{x_i\}_{i=1}^{2n+1}$ (with the vertical differentials) form a non-acyclic subcomplex of the hat-version knot Floer chain complex $\widehat{CFK}(\mathcal{H}_M(\alpha_K))$: the cycle $\sum_{i=0}^n x_{2i+1}$ survives in the homology $\widehat{HF}(S^3)$, and this cycle is the distinguished element of some vertically simplified basis in the sense of \cite[Section 2.3]{MR3217622}. Note that there is a horizontal arrow from $y_i$ to $x_i$ for each odd $i$, which we write $\partial^{horz}(y_i)=U x_i$. Since $U(\sum_{i=0}^n x_{2i+1})=\partial^{horz}(\sum_{i=0}^n y_{2i+1})$, the distinguished element $\sum_{i=0}^n x_{2i+1}$ is a boundary with respect to the horizontal differential. Therefore, $\epsilon(M(K))=1$ by \cite[Lemma 3.2 and Definition 3.4]{MR3217622}. Also, from the aforementioned sub-complex, it is standard to see that the $\tau$-invariant of $M(K)$ is equal to the Alexander grading of $x_1$, which we denote by $A(x_1)$. To get $A(x_1)$, we apply an algorithm given in \cite[Lemma 4.1]{Chen2019}. Specifically, up to an isotopy, the planar pairing diagram admits an involution (which swaps the $w$'s and $z$'s) given by a 180-degree rotation about a center point $c\in \tilde{\beta}_M\cap \tilde{\alpha}_K$. Let $l_{c,x_1}$ denote the path on $\tilde{\beta}_M$ from $c$ to $x_1$, and let $\delta_{w,z}$ denote the set of equivariant short arcs in the complement of $\tilde{\alpha}_K$ connecting $w$ to $z$ within each fundamental region. Then $A(x_1)$ is equal to the algebraic intersection number $l_{c,x_1}\cdot \delta_{w,z}$. To express $l_{c,x_1}\cdot \delta_{w,z}$ in terms of $\tau(K)$, note that the unstable segment stretches across $2\tau(K)$ fundamental regions vertically, with its midpoint sharing the same height with $c$. For a clearer discussion, we parametrize the plane so that $c$ is the origin $(0,0)$, the side length of each fundamental region is $1$; see Figure \ref{Figure:MazurSatellite0}. Observe that $l_{c,x_1}$ consists of $\tau(K)-1$ copies of lifts of $\beta_M$, together with a path from the point $(0,\tau(K)-1)$ to $x_1$. Each copy of $\beta_M$ intersects the $\delta_{w,z}$'s algebraically once, and the path from the point $(0,\tau(K)-1)$ to $x_1$ intersects the $\delta_{w,z}$'s algebraically twice; in sum, $l_{c,x_1}\cdot \delta_{w,z}=\tau(K)+1$. So, $\tau(M(K))=\tau(K)+1$. When $\tau(K)\leq 0$, the corresponding pairing diagram is shown in Figure \ref{Figure:MazurSatellite1}. The intersection points $\{x_i\}_{i=1}^{2n+1}$ generate a sub-complex of $\widehat{CFK}(\mathcal{H}_M(\alpha_K))$: the cycle represented by $x_{2n+1}$ survives in $\widehat{HF}(S^3)$ and is the distinguished element of a vertically simplified basis. One can see that $x_{2n+1}$ is a boundary with respect to the horizontal differential since $\partial^{horz}(\sum_{i=1}^{k}y_i)=U^2x_{2n+1}$, implying $\epsilon(M(K))=1$ by \cite[Lemma 3.2 and Definition 3.4]{MR3217622}. Using a similar argument as in the previous case, one may show $\tau(M(K))=A(x_{2n+1})=l_{c,x_{2n+1}}\cdot \delta_{w,z}=\tau(K)$.
	
	When $\epsilon(K)=-1$, the proof is similar to the case when $\epsilon(K)=1$. The pairing diagrams are given in Figure \ref{Figure:MazurSatellite2}.
	When $\tau(K)\geq 0$, $\sum_{i=0}^{n}x_{2i+1}$ is the distinguished element of some vertically simplified basis. Since $\partial^{horz}(\sum_{i=0}^{n}y_{2i+1})=U\sum_{i=0}^{n}x_{2i+1}$, we have $\epsilon(M(K))=1$. For the $\tau$-invariant, we have $\tau(M(K))=A(x_1)=l_{c,x_1}\cdot\delta_{w,z}=\tau(K)+1$.
	When $\tau(K)< 0$, $x_{2n+1}$ is the distinguished element of some vertically simplified basis. As $x_{2n+1}=U \partial^{horz}(y) $, we know $\epsilon(M(K))=1$. Finally, $\tau(M(K))=A(x_{2n+1})=l_{c,x_{2n+1}}\cdot\delta_{w,z}=\tau(K)+1$.
	
\end{proof}
\bibliography{satellite}
\bibliographystyle{alpha}
\end{document}